%% file: main_monograph_EMS_Barthelme_Mann_vector.tex
\documentclass{book}
\usepackage[lang = american]{ems-book} %

\usepackage{subcaption}
\usepackage{pinlabel}

\newtheorem{theorem}{Theorem}[section]
\newtheorem{corollary}[theorem]{Corollary}
\newtheorem{proposition}[theorem]{Proposition}
\newtheorem{definition}[theorem]{Definition}
\newtheorem{lemma}[theorem]{Lemma}
\newtheorem{claim}[theorem]{Claim}
\newtheorem{conjecture}[theorem]{Conjecture}
\newtheorem{exercise}[theorem]{Exercise}
\newtheorem*{theorem*}{Theorem}
\newtheorem*{proposition*}{Proposition}
\newtheorem*{definition*}{Definition}
\newtheorem*{lemma*}{Lemma}
\newtheorem*{claim*}{Claim}
\newtheorem*{corollary*}{Corollary}
\newtheorem*{convention*}{Convention}
\newtheorem{observation}[theorem]{Observation}

\theoremstyle{definition}
\newtheorem{convention}[theorem]{Convention}
\newtheorem{example}[theorem]{Example}

\newtheorem*{notation}{Notation}

\theoremstyle{remark}
\newtheorem{rem}[theorem]{Remark}
\newtheorem{remark}[theorem]{Remark}
\newtheorem{construction}[theorem]{Construction}
\newtheorem*{rem*}{Remark}

\newcommand{\wt}[1]{\widetilde{#1}}

\newcommand\bR{\mathbb R}
\newcommand\R{\mathbb R}

\newcommand\bZ{\mathbb Z}
\newcommand\Z{\mathbb Z}
\newcommand\bN{\mathbb N}
\newcommand\bC{\mathbb C}
\newcommand\bH{\mathbb H}

\newcommand{\bD}{\mathbb D}
\newcommand{\bT}{\mathbb T}

\newcommand\eps{\varepsilon}

\DeclareMathOperator{\id}{id}

\newcommand{\flow}{ \varphi }
\newcommand{\hflow}{ \tilde{\varphi} }

\newcommand\orb{ \mathcal O_{\flow} }

\newcommand\fs{\mathcal F^{s} }
\newcommand\hfs{\widetilde{\mathcal F}^{s} }
\newcommand\fu{\mathcal F^{u} }
\newcommand\hfu{\widetilde{\mathcal F}^{u} }
\newcommand\bfs{\bar{\mathcal F}^{s} }
\newcommand\bfu{\bar{\mathcal F}^{u} }

\newcommand\fss{\mathcal F^{ss} }
\newcommand\fuu{\mathcal F^{uu} }
\newcommand{\leafs}{\Lambda(\hfs)}
\newcommand{\leafu}{\Lambda(\hfu)}

\newcommand{\cC}{\mathcal{C}}

\newcommand{\cF}{\mathcal{F}}
\newcommand{\cJ}{\mathcal{J}}

\newcommand{\cO}{\mathcal{O}}
\newcommand{\cP}{\mathcal{P}}
\newcommand{\cR}{\mathcal{R}}

\newcommand{\cT}{\mathcal{T}}

\newcommand{\cW}{\mathcal{W}}

\newcommand{\Pbound}{\partial P}

\newcommand{\Fixbar}{\overline{\mathrm{Fix}}}
\newcommand{\Fix}{\mathrm{Fix}}

\newcommand{\lG}{\leq_G}

\newcommand{\PSL}{\mathrm{PSL}}
\newcommand{\SL}{\mathrm{SL}}
\newcommand{\Homeo}{\mathrm{Homeo}}

\newcommand{\Stab}{\mathrm{Stab}}

\begin{document}

\frontmatter

\title{Pseudo-Anosov Flows: \\a Plane Approach}
\author{Thomas Barthelmé and Kathryn Mann\\ \vspace{.2cm} 
with illustrations by Federico Salmoiraghi}
\maketitle

\chapter*{Foreword} 
\input{foreword}

\tableofcontents
\mainmatter

\chapter{From pseudo-Anosov flows to the orbit space}\label{chap_pA_to_orbit}

\input{pAtoOrbitSpace_vector}

\chapter{Bifoliated planes and Anosov-like actions}\label{chap_bifoliated_planes}

\input{PlanesAndActions_vector}

\chapter{Actions on lower-dimensional spaces}\label{chap_low_dimension}

\input{low_dimensional_vector}

\chapter{A closing lemma}\label{chap_closing_lemma}

\input{using_closing_lemma_vector}

\chapter{Examples of orbit spaces}\label{chap_examples_orbit_spaces}

\input{ExamplesOfOrbitSpaces_vector}

\chapter{Features in the plane and the 3-manifold}\label{chap_features_and_3manifold}

\input{Plane3ManifoldDictionary_vector}

\backmatter

\bibliographystyle{ems}
\bibliography{monograph_ref}

\end{document}

%% file: foreword.tex
This monograph is an introduction and invitation to the study of pseudo-Anosov flows on 3-manifolds from the topological perspective, a subject with roots in both classical dynamics and in three-manifold topology.   
To set the scene, we start by telling a brief, though necessarily incomplete, story of this history.  

\section*{Roots in dynamics...}
The broad theme of studying qualitative aspects of dynamical systems through the global comportment of orbits can be traced back to the work of Poincaré.  This work was picked up and further developed in the late nineteenth and early twentieth century by %
many other influential mathematicians, including Lyapunov, Hadamard, Birkhoff, Hedlund, and Hopf. 
An important theme early on was to find and study explicit examples of systems exhibiting complicated topological behavior (such as the homoclinic tangles of Poincaré) and/or ergodic properties (such as topological transitivity, topological mixing, ergodicity, etc\dots).  Geodesic flow on surfaces with constant negative curvature was among the first systems that was studied and shown
to exhibit such ergodic 
properties. See \cite{Hed39} for a historical survey giving the state of the art at the end of the 1930s, and \cite{FH19} and references therein for a modern introduction.

In the 1960s an important shift towards axiomatization and abstraction took place, which gave rise to the modern field of {\em hyperbolic dynamics}.
On one side, the work of Anosov gave an {\em abstract framing} of uniformly hyperbolic dynamics through his definition of Anosov systems (introduced as {\em U-systems}).  This extended the results known for geodesic flow on surfaces of constant negative curvature to the setting of variable negative curvature, and higher dimensional manifolds.  On the other side, Smale introduced the concept of \emph{structural stability} as a key organizing  principle\footnote{It is now known that structural stability \emph{characterizes} hyperbolic systems, thanks to the cumulative work of Robbin, Robinson, Ma\~n\'e and Hayashi, see \cite{Hay01}.}. As well as promoting this perspective, Smale's 1966 ICM address \cite{Sma67} outlined a rigorous framework for the study and classification of Anosov and Axiom A diffeomorphisms and flows, making precise the idea that the ``qualitative study" of dynamical systems should mean the study and classification up to {\em orbit equivalence} and describing a program to do so.   This is still a central organizing principle today, and is the notion of equivalence we consider in this text.  

Smale's classification program for Anosov \emph{diffeomorphisms} was successfully resolved in low (co)dimension, by John Franks \cite{Fra70} and Sheldon Newhouse \cite{New70}  early on, as well as in the case of infra-nilmanifolds.  There, the picture is quite rigid: every Anosov diffeomorphism on an infra-nilmanifold is conjugate to an algebraic automorphism \cite{Man74}.\footnote{It is still an open question whether this classification holds without the topological assumption on the manifold.} For Anosov flows however, many exotic, non-algebraic examples on 3-manifolds were discovered in the late 70s, early 80s, hinting at the richness of the subject, and the difficulty of carrying out a classification. 

Starting in the 1980s, Anosov diffeomorphisms and Anosov flows were often used as a fundamental test case to investigate dynamical or ergodic-theoretical properties of flows, with the geodesic flows of negatively curved manifolds being of particular interest.  This grew into a rich and still active research area; the reader interested in the history and development can consult \cite[\S 0.3]{FH19}, as well as the text \cite{FH19} itself for a thorough introduction to the field.  
In this text we focus primarily on the {\em topological} properties of these flows and their associated structures, and exclusively on 3-manifolds.   
\section*{...and in geometric topology}

In parallel with the work of Smale and Anosov, another motivation for the study of pseudo-Anosov flows was developing, coming from Thurston's work on the classification and geometerization of 3-manifolds.   Already implicit in the work of Nielsen (see e.g., \cite{Nie44}) was the idea of a {\em pseudo-Anosov homeomorphism} of a surface: a homeomorphism with two invariant singular foliations, one contracted and one expanded.  These can be seen as generalizations of linear hyperbolic maps of the torus.   In the 1970s, Thurston showed that the mapping torus of a pseudo-Anosov homeomorphism of a surface is a 3-manifold that admits a hyperbolic structure and he conjectured\footnote{Thurston wrote: ``Does every hyperbolic 3-manifold have a finite-sheeted cover which
fibers over the circle? This dubious-sounding question seems to have a definite
chance for a positive answer.'' \cite{Thu82}} 
(now proved by Agol \cite{Ago13}) that, up to passing to covers, all closed hyperbolic 3-manifolds arise in this way. Such a mapping torus comes equipped with a {\em suspension flow}; a pseudo-Anosov flow transverse to the fibers. 

Thurston also showed that, for a fixed hyperbolic
3-manifold $M$, different realizations of $M$ as mapping tori of pseudo-Anosov homeomorphisms can be organized into {\em faces} of the Thurston norm ball, a finite-sided polyhedron in $H_2(M,\bR)$.  
Extending and explaining Thurston's work, Fried \cite[Expos\'e 14]{FLP79} described how each fibered face of the Thurston norm ball in fact gives a canonical flow: for each such face $F$, there is a flow on $M$ such that the first return map to {\em any} fiber of a fibering in $F$ is a pseudo-Anosov homeomorphism.    These flows were the first examples referred to as ``pseudo-Anosov flows".

Thurston and Fried's suspension examples were studied by Mosher, who later extended the scope of the theory to a more general class of pseudo-Anosov flows which included not only suspensions, but flows that were {\em locally} modeled on the suspension picture \cite{Mos92a, Mos92b}.  This general definition using a local model was also present in the dynamics literature of the time, in the study of {\em expansive flows} (see e.g., the description of ``circle prong singularities" in \cite{Pat93, IM90} and a brief discussion in \cite[\S 2.3]{Pol87}.) 
This definition, or rather a precise topological reframing of it, is what we take to be the basic object of study in this text (see Definition \ref{def_topPA}).

By the late 1990s, thanks in large part to work of Thierry Barbot and Sergio Fenley,  the idea emerged to study pseudo-Anosov and Anosov flows on 3-manifolds from a uniform perspective, that would give insight simultaneously to the dynamical problems and classification program of Smale, and the Thurstonian geometric-topological perspective.   Their work makes essential use of a tool called the {\em orbit space} of the flow, reducing the topological study of Anosov flows to questions of discrete groups acting on a plane with transverse foliations. 

The idea to use the leaf spaces of the foliations to understand topological properties of flows was already present in the work of Verjovsky \cite{Ver74}, as well as in works of Plante, Thurston, and Ghys.  Orbit spaces appear also in the study of quasi-geodesic flows on 3-manifolds pioneered by Calegari \cite{Cal06} (and later, Frankel, see e.g.,~\cite{Fra13}), and are becoming an increasingly important tool.  As explained below, they will play a key role in this work.

\section*{What we do here}

In this text we (re)-tell the theory of pseudo-Anosov flows on 3-manifolds with the orbit space as the central character; via a streamlined framework called {\em Anosov-like group actions}.  This brings a simplified and unified perspective, and allows for the import of tools from 1- and 2-dimensional topological dynamics.   In so doing, we are able to give a relatively efficient presentation (or re-presentation) of foundational results in the field which appeared over several decades and many papers.

We hope that the reader, whether a beginning or more advanced topologist or dynamicist, finds this a welcoming invitation to the rich theory of pseudo-Anosov flows.

\section*{Scope of this text}

This text stays well on the topological side, and thus, we do not treat the analytic and ergodic aspects of Anosov flows (which was in fact Anosov's original motivation!). The book of Hasselblatt and Fisher \cite{FH19} is a good reference for that side of the theory.
  
We also focus strictly on 3-manifolds.  While there has been significant work done on the dynamics of Anosov flows on higher dimensional manifolds, the {\em topological} side is poorly understood.  
There is a relative sparsity of examples compared to those furnished by surgery in dimension 3 -- a codimension-two surgery-type construction is given in \cite{BBG21} but little else has been done.   For codimension one flows 
(those where either the weak-stable or unstable foliation is of codimension one), a famous conjecture of Verjovsky states that there are in fact very few examples: any smooth codimension one Anosov flow in dimension greater than or equal to $4$ is orbit equivalent to a suspension flow. Counterexamples in dimension $4$ were very recently built in \cite{FMP26}, but the situation in higher dimensions remains unclear.  

Within the theory of pseudo-Anosov flows, we treat questions of structure and classification of the flows themselves, rather than the related theory of flows transverse or quasi-transverse to foliations (or laminations) in 3-manifolds.  There are many important questions in this direction, growing out of the Thurstonian perspective described above, but these are beyond the scope of this text. Similarly, the theory of Birkhoff sections, while important, is absent here. 

Finally, we have intentionally focused on foundational results, rather than attempting to describe the state of the art of current research.  As such, we do not touch at all on some currently very active areas such as the relationship between Anosov flows and (bi)-contact structures, or the story of veering triangulations.

\section*{Acknowledgements} 

The authors are grateful to the many people that, over the years, have helped us understand, and then develop some of, the theory presented in this text. These were our teachers, our co-authors, and friends, and aside from the intrinsic interest of the subject, it is the enjoyment of the countless conversations and human interactions that made us want to share what we learned, and hopefully make it more accessible to the next generations.

The authors thank Isaac Broudy, Sergio Fenley, Lingfeng Lu, Rafael Potrie and Chi Cheuk Tsang for their comments and corrections on earlier versions of this text.

TB extends special thanks to Christine Barthelmé for her constant support, and in particular for taking on a lot of extra home work in the Fall of 2024 when most of this text was developed.   TB is partially supported by the NSERC Discovery (RGPIN-2024-04412) and Alliance International programs (ALLRP 598447 - 24). 

KM thanks Cornell University, Queen's University, the SLMath institute, and the Institute de Mathématiques de Jussieu, all of which were gracious hosts at times while this manuscript was developed, and also thanks the numerous colleagues who made each place feel like home. Special thanks to the Queen’s University Principal’s Development Fund for Visiting Scholars for their support, and the attendees of the 2025 Berstein Seminar in topology at Cornell.  KM was also partially supported by NSF CAREER grant DMS 1844516, grant DMS 2505228, and a Simons Foundation fellowship.

Finally, while most of this project was done in the time when LLMs could be barely trusted to solve first year calculus exercises, it is now completed in a different era, so we finish this foreword with the now mandatory AI disclosure: All of the mathematics contain in this text have been human-generated, and all the text was (painstakingly) produced and typed by the authors. However, an LLM was used to position the labels on the vectorized figures based on our hand-drawn drafts, as well as to do a typographical check.

%% file: pAtoOrbitSpace_vector.tex
\section{Definitions}

\subsection{Anosov and topological Anosov flows}

We start with the definitions of the main objects of study in this monograph. Throughout this text, $M$ will always refer to a closed (i.e., compact and without boundary) $3$-manifold, equipped with a Riemannian metric.  Thus, it makes sense to speak of distances, tangent spaces, and norms of tangent vectors.  
 
 \begin{definition} \label{def_smoothAF}
 A $C^1$ flow  $\flow^t \colon M \rightarrow M$ is called a \emph{(smooth) Anosov flow} if there is a splitting of the tangent bundle ${TM = \bR X \oplus E^{ss} \oplus E^{uu}}$ 
 preserved by $D\flow^t$ and two constants $a,b >0$ with the following properties:
\begin{enumerate}[label = (\roman*)]
 \item $X$ is the generating vector field of $\flow$;
 \item \label{item_smooth_contract} For any $v\in E^{ss}$ and $t>0$,
    \begin{equation*}
     \lVert D\flow^t(v)\rVert \leq be^{-at}\lVert v \rVert \, ;
    \end{equation*}
  \item \label{item_smooth_expand} For any $v\in E^{uu}$ and $t>0$,
    \begin{equation*}
     \lVert D\flow^{-t}(v)\rVert \leq be^{-at}\lVert v \rVert\, .
    \end{equation*}
\end{enumerate}
\end{definition}
If $b=1$, we say the metric on $M$ is {\em adapted} to the flow.  Such a metric is also sometimes called a \emph{Lyapunov metric}.  Using an averaging trick, one can show that any manifold with a smooth Anosov flow $\flow$ admits an adapted metric for $\flow$. See \cite[Prop 5.1.5]{FH19} for a proof. 

In this text, we will mainly be interested in flows up to \emph{orbit equivalence}.

\begin{definition}\label{def_orbit_equivalence}
Two flows  $\flow$ and $\psi$ on a manifold $M$ are \emph{orbit equivalent} if there is a homeomorphism $h \colon M\to M$ 
taking (unoriented) orbits of $\flow$ to orbits of $\psi$.  

Equivalently, the flow $h^{-1}\circ \psi^t\circ h$ is a {\em time-change} of $\flow^t$, meaning that there is a continuous real-valued function $\tau \colon \R \times M \to \R$ such that
\[
 \flow^{\tau(t,x)}(x) = h^{-1}\circ \psi^t\circ h(x)
\]
holds for all $x\in M$ and $t\in \R$.
\end{definition}
The fact that $h$ is a homeomorphism implies that either $\tau(t,x)>0$ whenever $t>0$, or $\tau(t,x)<0$ whenever $t>0$. 
 In the first case, we say that $h$ preserves direction of the flow; in the second we say it is direction reversing.\footnote{The usual definition of orbit equivalence in dynamical systems assumes that it is direction-preserving, in our sense. For reasons that will appear more clearly in later chapters, it is more natural from our point of view to also include direction-reversing orbit equivalences in the definition.}

\begin{remark}
Orbit equivalence is not to be confused with the notion of \emph{topological conjugacy}, which asks for a homeomorphism $h$ such that $\flow^t = h^{-1}\circ \psi^t\circ h$.  
Topological conjugacy is generally much more restrictive.  For instance, we will see later that all of the geodesic flows for negatively curved metrics on a fixed surface are orbit equivalent.  However, they cannot be topologically conjugate if the metrics have different length spectra. 
\end{remark}

Studying flows up to orbit equivalence amounts to understanding the qualitative picture of a flow given by the trajectories of its orbits.  This idea has roots going back to the work of Poincaré and Birkhoff; but it is the definition of orbit equivalence which provides a formal framework.  
Another motivation to consider Anosov flows up to orbit equivalence is due to the following {\em structural stability} theorem of Anosov: 

\begin{theorem}[Structural stability \cite{Ano67}] \label{thm_structural_stability} 
Let $\flow$ be an Anosov flow.  There exists a $C^0$ neighborhood of $\flow$ in the space of maps $M \times \R \to M$ such that any Anosov flow in this neighborhood is orbit equivalent to $\flow$.  Furthermore, if $\psi$ is any $C^1$ flow that is sufficiently close to $\flow$ in the $C^1$ topology, then $\psi$ is Anosov and is orbit-equivalent to $\flow$.  
\end{theorem} 

As noted by Smale \cite{Sma67}, structural stability implies there are only countably many Anosov flows up to orbit equivalence, and one could hope to completely classify them by discrete algebraic or numerical invariants.  %

In addition to the foliation by orbits, Anosov flows preserve other 1-dimensional and also 2-dimensional foliations.  
\begin{proposition}[Anosov \cite{Ano63}] \label{prop_integrable}
 If $\flow$ is a smooth Anosov flow, then the distributions $E^{ss}$, $E^{uu}$, $E^{ss} \oplus \R X$, and $E^{uu} \oplus \R X$ are uniquely integrable, hence tangent to foliations. 
\end{proposition}

The associated foliations to $E^{ss}$, $E^{uu}$, $\R X \oplus E^{ss} $, and $\R X \oplus E^{uu}$
are denoted respectively by $\fss$, $\fuu$, $\fs$, and $\fu$, and called the strong stable, strong unstable, (weak) stable, and (weak) unstable foliations, respectively.

Proposition \ref{prop_integrable} follows from the {\em stable manifold theorem}, see e.g., \cite[Theorem 6.1.1]{FH19}. 
Part of the proof consists in characterizing leaves of $\fs$ and $\fu$ in terms of asymptotic behavior of orbits.  By definition, each leaf is a union of orbits; the $\fs$ (resp.~$\fu$) leaves are precisely those where orbits converge in positive (resp.~negative) time, as follows:
\begin{proposition}[See, e.g., \cite{FH19} \S 6.1] \label{prop_characterize_stable}
Two orbits $\{ \flow^t(x) : t \in \R\}$ and $\{ \flow^t(y) : t \in \R\}$ lie on a common leaf of $\fs$ if and only if they converge in forward-time, in the sense that there exists a continuous, strictly increasing 
reparameterization $\sigma$ such that $d(\flow^{\sigma(t)}(y), \flow^t(x)) \to 0$ as $t \to \infty$. 
\end{proposition} 

As a consequence, if  $h$ is a direction-preserving orbit equivalence between flows $\flow$ and $\psi$, then $h$ maps weak stable and unstable leaves of $\flow$ to weak stable and unstable (respectively) leaves of $\psi$.   If instead $h$ is direction reversing, then it maps weak stable leaves of $\flow$ to weak unstable leaves of $\psi$, and weak unstable of $\flow$ to weak stable of $\psi$.    

Orbit equivalence does not typically preserve the strong stable and strong unstable foliations of Anosov flows.  
Since we are primarily concerned with the study of flows up to orbit equivalence, we will focus on the weak foliations and their subfoliations by orbits; ignoring the strong stable/unstable. This naturally leads to the {\em a priori} weaker definition of {\em topological Anosov flow}, which takes as starting point the behavior of orbits along the leaves of 2-dimensional (weak) foliations.   
A topological Anosov flow is a special case of a {\em pseudo-Anosov flow}, as defined in the next section, and it was the study of such pseudo-Anosov flows which historically motivated this topological perspective.  See Remark \ref{rem_other_definitions} for further comments on the definitions.  

The definition of topological Anosov flow that we present here is constructed to parallel that of smooth Anosov flow, but without reference to strong distributions or to derivatives: the invariant splitting of the tangent bundle is replaced with the existence of a pair of transverse, 2-dimensional foliations; and the contraction and expansion of strong-stable vectors in items \ref{item_smooth_contract} and \ref{item_smooth_expand} of smooth Anosov flows is replaced by {\em local} conditions (rather than {\em infinitesimal}) which say that nearby orbits on stable leaves approach in the future (item \ref{item_TAF_forward_asymptotic} below), but not the past (item \ref{item_TAF_backwards_expansivity}), and vice-versa along unstable leaves.  

We will use the following standard definitions and notation.  

\begin{notation}
Given a foliation $\cF$ of a manifold $M$ and a point $x \in M$, we denote by $\cF(x)$ the leaf of $\cF$ containing $x$, and $B_\eps(x) \subset M$ the $\eps$-ball about $x$.  
The {\em $\eps$-local leaf of $x$}, denoted $\cF_{\eps}(x)$, is the connected component of $\cF(x) \cap B_\eps(x)$ containing $x$.\footnote{Note that some authors, especially in the literature on expansive flows, use a different convention, and define the $\eps$-local stable leaf of $x$ as the points whose whole forward orbits stay at distance less than $\epsilon$ from the forward orbit of $x$.} Similarly the {\em $\eps$-local orbit} of $x$ is the connected component of $\flow^t(x) \cap B_\eps(x)$ containing $x$.  
\end{notation}
In the definition of topological Anosov flow, we will need to use reparameterizations of orbits.  We define a {\em standard reparameterization} to be a continuous, increasing homeomorphism $\sigma\colon \bR \to \bR$ with $\sigma(0)=0$.  

 \begin{definition}\label{def_TAF}
A \emph{topological Anosov flow} is a flow $\flow^t \colon M \to M$ (recall $M$ is a closed 3-manifold) that satisfies the following:
\begin{enumerate}[label = (\roman*)]
\item \label{item_TAF_weak_foliations}  There are two continuous, topologically transverse, $2$-dimensional foliations $\cF^{s}$ and $\cF^{u}$ whose leaves are saturated by orbits of $\flow$ and intersect along orbits of $\flow$
\item \label{item_TAF_forward_asymptotic}  For any $\eps>0$ sufficiently small, there exists $\delta>0$ such that if 
$y \in \cF_\delta^{s}(x)$ (resp.~$y \in \cF_\delta^{u}(x)$), there is a standard reparameterization $\sigma$ with $d\left(\flow^t(x), \flow^{\sigma(t)}(y)\right) < \eps$ for all $t>0$ (resp.~$t<0$) and tends to 0 as $t\to +\infty$ (resp.~$t\to -\infty$). 
\item \label{item_TAF_backwards_expansivity}
For any $\eps>0$ sufficiently small, there exists $\delta >0$ such that, if $y \in \cF^s_\eps(x)$ (resp.~$y \in\cF^u_\eps(x)$) is not in the same $\eps$-local orbit as $x$, then for any standard reparameterization $\sigma$, there exists $t<0$ (resp.~$t>0$) with $d(\flow^t(x), \flow^{\sigma(t)}(y)) > \delta$. 
\end{enumerate}
 \end{definition}
 
The foliations $\fs$ and $\fu$ are not required to have any additional regularity, even along the leaves.  While many definitions of foliations require some leafwise-manifold structure, leaves of $\fs$ and $\fu$ may in general not even be rectifiable.  {\em Topologically transverse} means that every point in $M$ has a neighborhood with a chart to $\R^3$ taking the foliations $\cF^s$ and $\cF^u$ to a pair of transverse foliations by planes.  
 
 \begin{remark}[Topological flows and time change]
Any continuous time-change of a smooth Anosov flow is a topological Anosov flow, and any smooth time-change of a smooth Anosov flow is a smooth Anosov flow by \cite[Theorem 5.1.16]{FH19}, but it is a long-standing open question to know whether any topological Anosov flow is orbit equivalent to a smooth Anosov flow. This question was answered positively by Mario Shannon \cite{Sha21} for transitive Anosov flows, that is, flows that have a dense orbit.  Transitivity is discussed further in Section \ref{sec_transitive}.

This is another motivation for taking topological Anosov flow as our primary object of study: it allows for modifications by continuous time change as in the definition of orbit equivalence.  Thus, in this text, we will primarily consider topological Anosov flows, and not worry about whether they are orbit equivalent to smooth ones or not. 
\end{remark}

\subsection{Pseudo-Anosov flows} 

There is a further generalization of Anosov flows on 3-manifolds which has received significant attention since the work of Thurston, due to the relationship of these flows and the topology and geometry of 3-manifolds.  This is the class of {\em pseudo-Anosov flows}. 
Roughly speaking, pseudo-Anosov flows look like Anosov flows everywhere except in a neighborhood of at most finitely many periodic orbits (called singular orbits), where the flow instead is locally modeled on a (semi)-branched cover over a periodic orbit of an Anosov flow.   

To give the precise definition, we first describe the topological picture near a singular orbit.  See Figure \ref{fig_model_prong} for an illustration. 
\begin{definition} \label{def_model_prong}
For $p \geq 3$, a {\em 2-dimensional $p$-prong} is a pair of singular foliations $\cF^+, \cF^-$ on an open disc $D$, homeomorphic (via a foliation-preserving homeomorphism) to a neighborhood of the origin in $\bR^2$ with foliations lifted from the horizontal and vertical foliations of $\bR^2 \cong \bC$ under the semi-branched cover $z \mapsto z^{p/2}$.

A {\em 3-dimensional local $p$-prong} is a pair of 2-dimensional singular foliations on $D \times (0,1)$ obtained by taking the product of a 2-dimensional $p$-prong $\cF^+, \cF^-$ on $D$ with $(0,1)$; that is to say the leaves are of the form $l^{\pm} \times (0,1)$ where $l^\pm$ is a leaf of $\cF^\pm$. 
\end{definition}

\begin{figure}[h]
\centering 
\includegraphics[width=5cm]{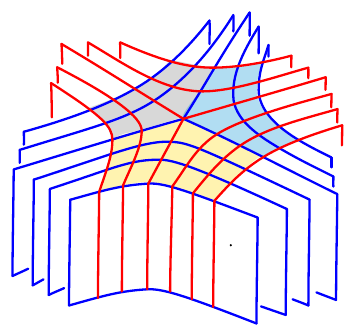}
\caption{A 3-dimensional 3-prong singularity. 
}
\label{fig_model_prong}
\end{figure}

The reader familiar with the Nielsen--Thurston theory of surface homeomorphisms (as in \cite{FLP79}) will notice that a model $p$-prong is a local picture of the mapping torus of a pseudo-Anosov homeomorphism of a surface near a fixed $p$-prong singularity.  Early informal definitions of pseudo-Anosov flow often described them as ``flows locally modeled on pseudo-Anosov suspensions".  The definition below is a precise formulation of this in the topological setting.  

Our convention in this text is that Anosov flows are a special case of pseudo-Anosov flows, namely, those with no singular orbits. 
If the foliations $\fs$ and $\fu$ in the definition below are nonsingular, one recovers the definition of topological Anosov flow presented above, so a pseudo-Anosov flow is a strict generalization. 

\begin{definition}\label{def_topPA}
A (topological) \emph{pseudo-Anosov} flow is a flow $\flow^t \colon M \to M$ satisfying the following conditions.  
\begin{enumerate}[label = (\roman*)]
\item \label{item_PAF_foliations} There are two continuous, topologically transverse, $2$-dimensional foliations $\cF^{s}$ and $\cF^{u}$ whose leaves are saturated by orbits of $\flow$ and intersect along orbits of $\flow$.  Singularities lie along a finite (possibly empty) collection of periodic orbits $\alpha_1, \dots, \alpha_n$. 
Each $\alpha_i$ is locally homeomorphic to the singular locus $0 \times (0,1)$ of a model $p_i$-prong, for some $p_i \geq 3$.
\item  \label{item_PAF_converges} 
For any $\eps>0$ sufficiently small, there exists $\delta>0$ such that if 
$y \in \cF_\delta^{s}(x)$ (resp.~$y \in \cF_\delta^{u}(x)$), there is a standard reparameterization $\sigma$ with $d\left(\flow^t(x), \flow^{\sigma(t)}(y)\right) < \eps$ for all $t>0$ (resp.~$t>0$) and tends to 0 as $t\to +\infty$ (resp.~$t\to -\infty$).
\item \label{item_PAF_backwards_expansivity} 
For any $\eps>0$ sufficiently small, there exists $\delta >0$ such that, if $y \in \cF_\eps^s(x)$ (resp.~$y \in\cF_\eps^u(x)$) is not in the same $\eps$-local orbit as $x$, then for any standard reparameterization $\sigma$, there exists $t<0$ (resp.~$t>0$) with $d(\flow^t(x), \flow^{\sigma(t)}(y)) > \delta$. 
\end{enumerate}
\end{definition}
As in the definition of (smooth) Anosov flow, the property of being a pseudo-Anosov flow is independent of the choice of metric on $M$, provided $M$ is compact, and it is preserved under lifting to covers.  

\begin{remark} \label{rem_other_definitions} 
Many historical definitions of pseudo-Anosov flow or topological Anosov flow omit condition \ref{item_PAF_backwards_expansivity} (see, e.g., \cite{Mos92a,  BF13} and papers referencing these), or state a weaker form of it as in \cite{BFP23}.  This appears to be an oversight.  Indeed, Example \ref{ex_blowup_Anosov} shows that there are flows which satisfy (in fact strengthenings of) conditions \ref{item_PAF_foliations} and \ref{item_PAF_converges} but also have the property that certain orbits along stable leaves converge in both forwards {\em and} backwards time, a behavior which has always been assumed impossible.  In particular, Example \ref{ex_blowup_Anosov} does not admit a Markov partition, a property often asserted or tacitly assumed true for pseudo-Anosov flows. 

Some instances of the definition of pseudo-Anosov flow in the literature require the existence of a Markov partition (e.g.,~\cite{AT24}), which resolves the issue described above.\footnote{Thanks to Theorem \ref{thm_expansive_implies_pA} and \cite{Iak24}, pseudo-Anosov flows in the sense of Definition \ref{def_topPA} always do admit Markov partitions.} Other definitions, such as those in \cite{FM01, Cal07} or the definition of ``smooth pseudo-Anosov flow" in \cite{AT24} require some smoothness\footnote{A desire for smoothness is also mentioned in  \cite{Mos92a} but not formalized into the definition.}, typically asking for the flow to be smooth Anosov away from the singular locus with at least some tameness of behavior as one approaches the singularities, again preventing backwards and forwards convergence. However, like the topological Anosov case, it remains an open question whether topological pseudo-Anosov flows are always orbit equivalent to such ``smooth" ones, though it is proved in the transitive case in \cite{AT24}. 
\end{remark} 

We reiterate that $\fs$ and $\fu$ are not required to have any regularity or leafwise regularity. 
However, a result of Calegari \cite{Cal01} says that, up to conjugacy, one may take {\em either one} of $\fs$ or $\fu$ to have smooth leaves -- without necessarily obtaining any regularity of orbits -- by choosing a suitable triangulation of $M$ and performing an isotopy. 
It is unknown under what circumstances (aside from assuming transitivity of the flow and using the result of \cite{Sha21,AT24}) one can make both foliations simultaneously even leafwise $C^1$.  
Nonetheless, the fact that these foliations are required to be topologically transverse means that each point has a local flow-box neighborhood in the following sense.  

\begin{definition} \label{def_flow_box}
Let $\flow$ be a pseudo-Anosov flow on $M$.  A {\em flow-box neighborhood} of a point $p \in M$ on a nonsingular orbit is a neighborhood $B$ of $p$ admitting a homeomorphism to $[-1,1]^3 \subset \R^3$ sending $p$ to the origin, connected components of leaves of $\fs \cap B$ to the foliation by planes of fixed $x$-coordinate, and connected components of leaves of $\fu \cap B$ to the foliation by planes of constant $y$-coordinate.  We call these connected components of $\fs \cap B$ and $\fu \cap B$ the (stable and unstable) \emph{sheets} of the flow box $B$.

A flow box neighborhood of a point on a singular orbit is a neighborhood locally homeomorphic, via a foliation preserving homeomorphism, to a local model prong described in Definition \ref{def_model_prong}.  
\end{definition} 

The next few lemmas give a series of consequences of Definition \ref{def_topPA}, strengthening the conditions listed there.  
The first two combined improve condition \ref{item_PAF_converges} to say that any points on a common leaf of $\fs$ converge towards each other in forward time (up to reparameterization) and eventually lie in a common local stable leaf.  

\begin{lemma}[Global forward convergence along stable leaves]  \label{lem_all_converge}
Suppose $x \in \cF^s(y)$, and let $\delta >0$.  Then, there exists a standard reparameterization $\sigma \colon \R \to \R$ and $T \in \R$ such that, for all $t>T$, we have $d(\flow^t(x), \flow^{\sigma(t)}(y)) < \delta$.  
\end{lemma} 
By symmetry, the statement holds replacing $\cF^s$ with $\cF^u$ and $t>T$ with $t<T$, with the same proof.  

\begin{proof}
Let $\delta >0$ be given, and let $\eps$ be a constant small enough to apply condition \ref{item_PAF_converges}.   
Take a path in $\cF^s(x)$ between $x$ and $y$. Cover it with (finitely many) flow-box neighborhoods each of diameter at most $\eps$, and use this to build a sequence of points $x=x_0, x_1, \ldots, x_n=y$ in $\cF^s(x)$ with $x_i \in \cF^s_\eps(x_{i-1})$.   Iteratively applying condition \ref{item_PAF_converges}, we can find a standard reparameterization $\sigma_i$ and time $T_i$ so that $d(\flow^{\sigma_{i-1}(t)}(x_{i-1}), \flow^{\sigma_{i}(t)}(x_{i})) < \delta /n$ for all $t > T_i$.  Let $\sigma = \sigma_n$ and $T = \max_i\{T_i\}$.  Then $d\left(\flow^t(x), \flow^{\sigma(t)}(y)\right) < \delta$ for all $t >T$.   
\end{proof} 

While Lemma \ref{lem_all_converge} says that, up to reparameterization, two orbits on a stable leaf converge in the future, distance is measured in $M$, so {\em a priori} the orbits could remain on different local stable leaves as they approach each other.   Our next goal is to show this is not the case.  

\begin{lemma}[Global convergence along the leaf]  \label{lem_same_local_stable}
For any $\delta >0$, if $x \in \cF^s(y)$, then there exists a standard reparameterization $\sigma$ and some $T \in \bR$ such that $\flow^{\sigma(t)}(y) \in \cF^{s}_{\delta}(\flow^t(x))$ for all $t>T$. 
\end{lemma} 

\begin{proof}
Let $\eps$ be small enough so that conditions \ref{item_PAF_converges} and \ref{item_PAF_backwards_expansivity} apply and cover $M$  with flow boxes $U_i$ of diameter $< \eps$, small enough so that they are contained in foliation charts for $\cF^s$. Thus, if $U_i \cap U_j \neq \emptyset$ then each sheet of $\fs$ in $U_i$ meets at most one sheet of $\fs$ in $U_j$. 
Let $\delta_0$ be the constant given by condition \ref{item_PAF_backwards_expansivity}, and pick $\delta < \delta_0/2$ small enough so that if $d(x, y) < 2\delta$, then $x$ and $y$ are contained in a single flow box $U_i$ and there exists some $z \in \fs_\eps(x) \cap \fu_\eps(y) \cap U_i$.  Such a $\delta$ can be found using the Lebesgue number lemma of the cover and continuity of the foliations.  

Now let $x, y$ be given with $x \in \cF^s(y)$.  
If $x$ and $y$ are in the same orbit, then it is easy to find a standard reparameterization $\sigma$ satisfying the desired conditions, so we assume they are in distinct orbits.  By Lemma \ref{lem_all_converge} there is a standard reparameterization $\sigma$ such that $d(\flow^t(x),\flow^{\sigma(t)}y) \to 0$.  

Up to replacing $x$ and $y$ with points further along their orbits, we can assume that
$d(\flow^t(x),\flow^{\sigma(t)}(y))<\delta$ for all $t\geq 0$.  
If $x, y$ are on the same sheet of some flow box $U_i$, then we can find a sequence of flow boxes so that $\flow^t(x)$ and $\flow^{\sigma(t)}y$ are always in a common box along the sequence, and thus always on the same sheet.  Since $d(\flow^t(x),\flow^{\sigma(t)}(y)) \to 0$, continuity of the foliations means we have $\flow^t(x) \in \cF^s_\delta(\flow^{\sigma(t)}y)$ for all sufficiently large $t$.  

So we need only treat the case where $x$ and $y$ lie in a flow-box $U_i$ but not on the same local sheet. We will show this gives a contradiction.  By our choice of constants, there exists $z \in \fs_\eps(x) \cap \fu_\eps(y) \cap U_i$.  Since $z \in \fs_\eps(x)$, there exists a reparameterization $\tau$ and $T>0$ such that $d(\flow^{\tau(t)}(z), \flow^t(x)) < \delta$ for all $t > T$.   Again replacing $x, y, z$ with points further along their orbits, we can assume that $d(\flow^{\tau(t)}(z), \flow^t(x)) < \delta$ for all $t>0$.  

Since $x$ is not in the same sheet of $U_i$ as $y$, the point $z$ is not on the same local orbit as $y$.  Thus, by condition \ref{item_PAF_backwards_expansivity}, (applied to the standard reparameterization $\sigma \circ \tau^{-1}$) there exists $t>0$ such that $d(\flow^{\tau(t)}(z),\flow^{\sigma(t)}(y)) > \delta_0$.  This contradicts the triangle inequality, since $d(\flow^{\tau(t)}(z), \flow^t(x)) < \delta$ and we chose $\delta < \delta_0/2$.   See Figure \ref{fig_divergence} for a schematic illustration.  
\end{proof} 

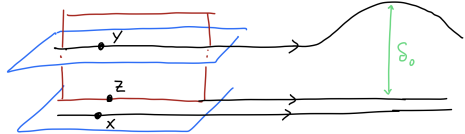
\begin{figure}[h]
\centering 
\labellist
\small \hair 2pt
\pinlabel {$y$} at 61.124 100.142
\pinlabel {$z$} at 61.124 53.142
\pinlabel {$x$} at 61.124 21.142
\pinlabel {$\delta_0$} at 329.124 77.142
\endlabellist
\includegraphics[width=7cm]{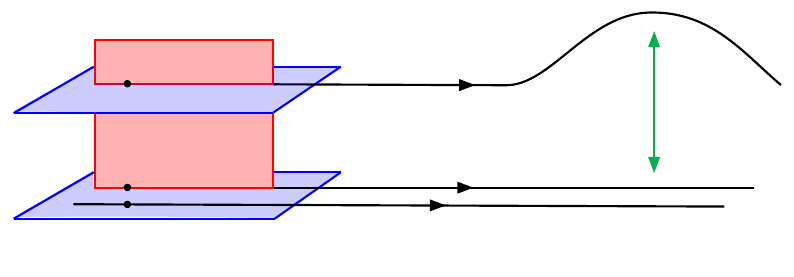}
\caption{If $x, y$ are sufficiently nearby but not on the same local sheet of $\fs$, then $x$ fellow-travels a point $z$ which does not fellow-travel $y$.} 
\label{fig_divergence}
\end{figure}

\begin{remark}  \label{rem_forward_converge_implies_local_stable}
Taking a little more care with continuity of the foliations and compactness, one can use the proof above to show the following:  
For any $\delta$, there exists $\eta<\delta$ so that if $x, y \in M$ are such that for some standard reparameterization we have $d(\flow^t(x), \flow^{\sigma(t)}(y)) < \eta$ for all $t>0$, then $\flow^t(x) \in \cF^s_\delta(\flow^{\sigma(t)}(y))$ for all $t>0$.  

In particular, this shows that leaves of $\cF^s$ are exactly the maximal sets of orbits which converge (in $M$ or equivalently on local sheets) in the future.  
\end{remark}

\subsection{Expansive flows}
In this section we describe a different definition of pseudo-Anosov flow, which is equivalent to the one above thanks to the work of  Inaba and Matsumoto \cite{IM90}, and  Paternain  \cite{Pat93}.  
This comes from the classical notion of {\em expansive flows}, introduced by Bowen and Walters \cite{BW72}.  

\begin{definition}\label{def.expflow}
A nonsingular flow $\flow^t \colon M \to M$ is \emph{expansive} if for every $\eps>0$ there exists $\delta>0$ with the following property:  If $x,y \in M$ and there exists a standard reparameterization $\sigma$ with $d(\flow^t(x), \flow^{\sigma(t)}(y)) \leq \delta$ for all $t \in \bR$, then $y = \flow^s(x)$ for some $|s|< \eps$. 
\end{definition}

The study of expansive maps or homeomorphisms of metric spaces predates that of expansive flows; Definition \ref{def.expflow} was the first formalization to the flows setting.  These soon became recognized as an important class of examples generalizing Smale's Axiom A systems.  
We will prove below that pseudo-Anosov flows are examples of expansive flows: 

\begin{proposition}  \label{prop_pa_implies_expansive}
Any pseudo-Anosov flow in the sense of Definition \ref{def_topPA} is expansive. 
\end{proposition} 

What is truly remarkable is that the converse is also true: thus expansivity is a complete characterization of pseudo-Anosov flows. Some authors have even taken expansive as a {\em definition} of pseudo-Anosov flow, see for instance \cite{Sha21, Iak24}. 
\begin{theorem} [Inaba--Matsumoto \cite{IM90} Theorem 1.5;  Paternain \cite{Pat93}, Lemma 7] \label{thm_expansive_implies_pA}
Let $\flow^t \colon M \to M$ be an expansive flow on a closed $3$-manifold with no stationary points.  Then $\flow$ is pseudo-Anosov in the sense of Definition \ref{def_topPA}. 
\end{theorem} 
We note that this theorem does not have an obvious generalization in higher dimensions; there are many expansive flows on higher dimensional manifolds which can have quite wild singular behavior.  The main content of the proof of Theorem \ref{thm_expansive_implies_pA} is to show that expansive flows have invariant singular foliations (the existence of stable and unstable sets was already established by Keynes and Sears \cite{KS81}), and that the topology around the singularities is quite constrained.   
The proof is too involved to repeat here; instead we show only the easy direction stated in Proposition \ref{prop_pa_implies_expansive} that pseudo-Anosov implies expansive. 

\begin{proof}[Proof of Proposition \ref{prop_pa_implies_expansive}]
The proof is similar to that of Lemma \ref{lem_same_local_stable}.  
Let $\flow$ be a pseudo-Anosov flow, and let $\eps>0$ be given.  
Fix $\eps_0 < \eps$ small enough so that condition \ref{item_PAF_backwards_expansivity} applies with this choice of $\eps_0$ (recall that it specifies ``for all sufficiently small...").  Additionally, using uniform continuity of $\flow$, we can ensure that $\eps_0$ is small enough so that for all $x \in M$, if $y$ lies in the $\eps_0$-local orbit of $x$, then we have $y = \flow^s(x)$ where $|s|<\eps$. 

 Let $\delta(\eps_0)$ be the constant obtained from condition \ref{item_PAF_backwards_expansivity} with this input, and fix $\delta = \delta(\eps_0)/2$.  Up to shrinking $\delta$ even further, we can assume also as in the proof of Lemma \ref{lem_same_local_stable}, that if $d(p, q)<\delta$ then they lie in a common foliation chart given by a flow box, and $\cF^s_{\eps_0}(p) \cap \cF^u_{\eps_0}(q)$ also contains a point in this flow box. 
 
Suppose as a first easy case that $x, y \in M$ satisfy $y \in \cF^s_{\eps_0}(x)$.  If there exists a standard reparameterization $\sigma$ such that $d(\flow^t(x), \flow^{\sigma(t)}(y)) < \delta < \delta(\eps_0)$ for all $t$, then in particular this is true for all $t<0$ so we conclude by item
\ref{item_PAF_backwards_expansivity}  that $y$ lies in the same $\eps_0$-local orbit as $x$, and so by our choice of constants $y = \flow^s(x)$ where $|s|< \eps$.   A similar argument applies if $y \in \cF^u_{\eps_0}(x)$.  

Passing now to the general case, suppose that $x, y \in M$ are any two points such that there exists a standard reparameterization such that $d(\flow^t(x), \flow^{\sigma(t)}(y)) < \delta$ for all $t$.   Assume that they are not on the same local unstable or stable leaf, and hence not in the same local orbit; we will derive a contradiction.  
In particular, $d(\flow^t(x), \flow^{\sigma(t)}(y)) < \delta$ is true for $t =0$, so we can find $z \in \cF^u_{\eps_0}(x) \cap \cF^s_{\eps_0}(y)$, not in the same local orbit as $x$.   Then by condition \ref{item_PAF_converges} there exists $T<0$ and a reparameterization $\tau$ so that 
$\flow^{\tau(t)}(z) \in \cF^u_\delta(\flow^t (x))$ for all $t < T$.  

Let $z' = \flow^{\tau(T)}(z)$, $y' = \flow^{\sigma(T)}(y)$ and $x' = \flow^T(x)$.  Then by the triangle inequality we have $d(\flow^{\tau(t')}(z), \flow^{\sigma(t')}(y))< \delta(\epsilon_0)$ for all $t'<0$.  So by condition \ref{item_PAF_backwards_expansivity}, $z'$ and $y'$ lie on the same $\eps_0$-local orbit, thus $x'$ and $y'$ lie on the same local unstable sheet of a flow box in our cover.  
Since neighboring foliation charts meet along single sheets, it follows that $x$ and $y$ lie in the same unstable sheet of a flow box as well, contradicting our assumption that they were not in the same local unstable leaf.  
\end{proof}

As well as giving a very concise definition, the expansive property has the benefit of giving a checkable condition in order to verify that a given flow is pseudo-Anosov, somewhat in the spirit of the {\em Alekseev cone field criterion}\footnote{See e.g., \cite[Proposition 5.1.7]{FH19}.} for smooth Anosov flows.   
 In essence, the cone field criterion requires the existence of two ``constant angle" cone fields $C^s$ and $C^u$ defined at each tangent space, so that $D\flow^t$ maps the closure of $C^u$ into $C^u$ and decreases norms of vectors in $C^s$ for $t>0$ (possibly up to multiplicative error); and similarly for $t<0$ it maps the closure of $C^s$ into $C^s$ and decreases norms of vectors in $C^u$.    Thus, one does not need to find any {\em invariant} object to use this criterion.   Similarly, the definition of expansive does not ask for an invariant object, but rather for a notion of local expansion.

We conclude this section by noting that in order to distinguish a nonsingular topological Anosov flow from a singular pseudo-Anosov flow, it suffices to find a single invariant foliation.  

\begin{proposition}[Barthelm\'e--Fenley--Potrie, \cite{BFP23} Theorem 5.9]
Let $\flow$ be a topological pseudo-Anosov flow on $M$.  If $\flow$ preserves each leaf of some 2-dimensional nonsingular foliation on $M$, then it is a topological Anosov flow. \footnote{Note that there is a typo in the statement of Theorem 5.9 of \cite{BFP23}, and that $\flow$ preserves each leaf of a foliation is a necessary hypothesis, as can be immediately seen by considering the foliations by transverse surfaces in the mapping torus of a pseudo-Anosov homeomorphism.}
\end{proposition}
\begin{proof}
Let $\cF$ be a 2-dimensional topological foliation whose leaves are saturated by orbits of $\flow$. We need to show that this is incompatible with having a singular periodic orbit. So assume that $\flow$ admits a singular periodic orbit $\alpha$.  
Let $L$ be the leaf of $\cF$ containing $\alpha$. Consider $L_l$ one connected component of $L\smallsetminus\{\alpha\}$. Since $\alpha$ is periodic, invariance of $L$ by the flow implies that $L_l$ must be contained in either $\cF^s(\alpha)$ or $\cF^u(\alpha)$. Similarly, denoting by $L_r$ the other connected component of $L\smallsetminus\{\alpha\}$, we deduce that $L_r$ is also contained in either $\cF^s(\alpha)$ or $\cF^u(\alpha)$. (Note that at this point, a priori $L_l$ could be in a stable leaf and $L_r$ in an unstable, or vice versa, but this does not matter for the proof.)
Now consider a small transverse disk $D$ around a point $x$ on $\alpha$. Because the leaves $\cF^s(x)$ and $\cF^u(x)$ are singular, a leaf of $\cF$ that is near $L$ on at least one side of it must intersect a ray of both $\cF^s(x)\cap D$ and $\cF^u(x)\cap D$. Such a leaf cannot be left invariant by $\flow$, as the image of such a leaf under the first return map to $D$ would have to intersect itself.  
\end{proof}

\section{Examples}\label{sec_flows_examples}

We will now describe some examples of Anosov and pseudo-Anosov flows.  In the 1960s, the only known examples were the  {\em suspensions} and {\em geodesic flows}.   Starting in the late 1970s, many new examples of Anosov flows were constructed on  3-manifolds using surgery techniques.  Such constructions take as input an existing flow and make a modification to produce another flow, typically on a different 3-manifold.   Later, starting with work of Bonatti and Langevin \cite{BL94} a new strategy was introduced, gluing ``building blocks" together, to produce a flow out of elementary pieces rather than modifying a previously known example.  

After describing geodesic and suspension flows, we give a few samples of flows obtained by surgery and gluing techniques.   For simplicity, and for historical reasons, the constructions we describe will involve Anosov flows, but many generalize easily to the pseudo-Anosov setting.

 \subsection{Suspension flows}\label{sec_example_suspensions}
 
If $N$ is a manifold and $f\colon N \to N$ is a homeomorphism, the {\em mapping torus} $N_f$ of $f$ is the quotient 
of $N \times \bR$ by the equivalence relation generated by $(x, t) \sim (f(x), t-1)$.  

The {\em suspension flow $\flow_f$ of $f$} is the flow on $N_f$ generated by the vector field $\partial/\partial t$ in the $\bR$-coordinate direction.
 The time $k$ map $\flow_f^k\colon N_f \to N_f$, for $k \in \bZ$, preserves the slice $N \times \{0\}$, and its restriction to this slice agrees with $f^k$.   Thus, $\flow_f$ can be thought of as a ``continuous interpolation" of the iterates of $f$ on $N \times \{0\} \subset N_f$.  
To produce an example which is Anosov, one needs a map $f$ that has invariant expanding/contracting directions, as in the example below.  

\begin{example}[The suspension flow of a hyperbolic linear map] \label{ex_suspension_AF}

The linear action of a matrix $A \in \SL(2,\Z)$ on $\R^2$ preserves the integer lattice, so descends to a map $\bar{A}$ of the torus $\bT^2 := \bR^2/\bZ^2$. Since $\det(A) = 1$, this induced map is a homeomorphism.  If $A$ has eigenvalues $\lambda, \lambda^{-1}$ with $|\lambda| \neq 1$, then $\bar{A}$ is an Anosov diffeomorphism, the discrete analog of an Anosov flow.\footnote{The reader can guess the definition, or modify Definition \ref{def_smoothAF} by deleting all appearances of $X$ and restricting $t$ to take integer values.} 

Let $E_\lambda$ and $E_{\lambda^{-1}}$ denote the two eigendirections of $A$, and suppose that $\lambda>1$. 
The tangent space to any point of the mapping torus $\bT^2_{\bar{A}}$ splits naturally as $\R^2 \times \R$ and 
the 1-dimensional, constant, distributions 
\[E^{uu} := (E_\lambda, 0) \text{ and } E^{ss}:= (E_{\lambda^{-1}}, 0)
\] 
are well defined everywhere and invariant under the suspension flow $\flow_{\bar{A}}$.  Similarly, the 2-dimensional distributions 
$E^u$ and $E^s$, spanned by $E^{uu}$ (resp.~$E^{ss}$) together with the direction of the flow, are tangent to invariant planar foliations.  With these foliations, $\flow_{\bar{A}}$ is a topological Anosov flow.  
\end{example} 

\begin{exercise}
Verify that the suspension $\flow_{\bar{A}}$ is indeed a topological Anosov flow.
\end{exercise}

To check that $\flow_{\bar{A}}$ satisfies the smooth Anosov property (Definition \ref{def_smoothAF}) one must first describe a Riemannian metric structure on $\bT^2_{\bar{A}}$.   We will define one that is particularly convenient, and in so doing, we will also see that the flow is {\em algebraic} in the sense that it can be written as the action of a one-parameter subgroup of a Lie group on a quotient of that Lie group.

Let $\{ x_1, x_2\}$ denote the basis for $\R^2$ consisting of unit eigenvectors of $A$, with $A(x_1) = \lambda x_1$.   The standard Euclidean metric $d x_1^2 + d x_2^2$ descends to the torus quotient $\bT^2$ so defines a (Euclidean) Riemannian metric there.  Now on $\bT^2 \times \R$ we fix the metric 
\[ ds^2 = \lambda^{2t} d x_1^2 + \lambda^{-2t} d x_2^2 + dt^2.\]
The quotient map $(x, t) \sim (A(x), t-1)$ is an isometry with respect to this metric, so the metric descends to the quotient. It is now easy to see that $\flow_{\bar{A}}$ is Anosov and $E^{ss}$ and $E^{uu}$ as defined above are the strong stable and unstable distributions. 
 
 In fact, all transformations of the form 
 \[ (x_1, x_2, t) \mapsto ( \lambda^{-s} x_1 +a  , \lambda^{s} x_2+ b, s+t )\]
 for $(a, b, s) \in \R^3$ are isometries of this metric.  The equation above is the formula for multiplication $(a,b, s) \cdot  (x_1, x_2, t)$ in the Lie group $\mathrm{Sol}$; and what we have just described is a left-invariant metric defining a Sol-structure on $\bT^2_{\bar{A}}$ for which the suspension flow is given by {\em right} multiplication by the one-parameter subgroup $(0, 0, t)$.

Example \ref{ex_suspension_AF} can be generalized to the pseudo-Anosov setting by replacing $\bT^2$ by a higher genus surface $S$ and $\bar{A}$ by a {\em pseudo-Anosov homeomorphism} -- a homeomorphism $f\colon S \to S$ leaving invariant two transverse singular foliations with isolated prong singularities, with an expansion/contraction property analogous to that in Definition \ref{def_topPA}. See, e.g., \cite{FM12}.   Such examples were the motivation for the definition of pseudo-Anosov flow.  In fact, the earliest definition of pseudo-Anosov flow (as in Fried's article \cite[Exposé 14]{FLP79}) requires pseudo-Anosov flows to be suspensions. It was only much later in the work of Mosher that the picture was generalized to include non-suspension examples.

 \subsection{Geodesic flow}\label{sec_example_geodesic_flow}

Let $\Sigma$ be a manifold equipped with a Riemannian metric.  The {\em geodesic flow}\footnote{Often this is called the ``geodesic flow on $\Sigma$'', despite the fact that the flow is actually {\em on} the manifold $T^1(\Sigma)$.} is the flow on the unit tangent bundle $T^1\Sigma$ defined by the property that $\phi^t(v)$ is the unit tangent to the geodesic $\{ \exp(sv): s \in \R\}$ at point $\exp(tv)$, where $\exp$ denotes the Riemannian exponential map.   
For surfaces of constant negative curvature, such flows are algebraic in the same sense as the suspension flow described above.  

\begin{example}[Geodesic flow of a hyperbolic surface]
 If $\Sigma$ is a surface of constant curvature $-1$, then $T^1\Sigma$ can be isometrically identified with the (left) quotient of $\PSL(2, \bR) \cong T^1(\bH^2)$ by a discrete subgroup $\Gamma \cong \pi_1(\Sigma)$; where $\PSL(2,\bR)$ is equipped with the standard left-invariant metric.  
 Geodesic flow on $T^1 \Sigma \cong  \Gamma \backslash \PSL(2, \bR)$ corresponds to right multiplication by the 1-parameter subgroup $G^t=\begin{pmatrix} e^{t/2} & 0 \\ 0 & e^{-t/2} \end{pmatrix}$. 
 
It is easy to verify that the matrices $H_+ = \left( \begin{smallmatrix} 0 & 0 \\ 1 &0  \end{smallmatrix} \right) $ and $H_- =\left( \begin{smallmatrix} 0 & 1 \\ 0 &0  \end{smallmatrix} \right)$ in $\mathfrak{sl}(2,\bR)$, the Lie algebra of $\PSL(2,\bR)$, define distributions invariant by the derivative of $G^t$ and tangent vectors in these distributions are respectively expanded and contracted at an exponential rate. That is, the geodesic flow is an Anosov flow. 
A much more in-depth description of the geodesic flow of hyperbolic surfaces can be found in  \cite[Chapter 2]{FH19}.
 \end{example}
 
Despite geodesic flows being clearly dependent on the choice of metric, we generally say \emph{the} geodesic flow for a hyperbolic or negatively curved surface of fixed topological type.  
This is because any two geodesic flows on the same negatively curved surface are orbit equivalent, as shown in \cite{Ghy84,Gro00}.\footnote{Gromov's paper cited here was circulated as a preprint around the same time as Ghys' (independent) results, but Gromov's work was not published until much later.} 
One can easily see this using structural stability of Anosov flows (see, e.g., \cite[Theorem 5.3.6]{FH19}), as follows.  Notice first that two distinct geodesic flows a priori live on two different manifolds, the unit tangent bundle of $\Sigma$ associated with two different metrics. To remedy this, one may define instead the geodesic flows on the \emph{homogenized bundle} (or \emph{bundle of directions}) $H\Sigma = (T\Sigma \smallsetminus \{0\}) / \R^+$ which is topologically the same as $T^1 \Sigma$.  
 Now the space of negatively curved Riemannian metrics on a surface is connected (one may travel in the conformal class to the hyperbolic metric and then travel in the Teichm\"uller space).  Structural stability implies that open connected sets in the space of Anosov flows are orbit equivalent, hence the result.

 P.~Tomter \cite{Tom68} classified algebraic flows, showing that all examples are either supensions or arise as a tower of torus bundles over a homogeneous geodesic Anosov flow.  In dimension 3, this implies that the only algebraic examples are the geodesic and suspension flows already described.   Thus, any new examples require different techniques of construction.

 \subsection{The examples of Handel and Thurston} \label{sec_example_Handel_Thurston}
 
 Our next two examples are modifications of geodesic and suspension flows (respectively) by surgery techniques.  These were the first historical examples of Anosov flows which were not orbit equivalent to algebraic ones.  

\example[Handel and Thurston's construction]
We describe the family of Anosov flows constructed in \cite{HT80}. 

Let $\Sigma$ be a fixed, orientable, hyperbolic surface, and let  $\flow$ denote the geodesic flow on $T^1\Sigma$.  Fix a simple closed geodesic $c$ and let $T$ be the preimage of $c$ under the projection 
$T^1\Sigma \to \Sigma$.  Topologically, $T$ is a torus.  
This torus contains exactly two periodic orbits of $\flow$, representing the geodesic $c$ with each orientation, and these two periodic orbits are isotopic, essential simple closed curves in $T$. We can trivialize $T = S^1 \times S^1$ so that these periodic orbits are the first coordinate direction, and the second is the unit tangent fiber.  
The union of these curves divides $T$ into two regions, on which the flow is transverse to $T$ but in opposite directions.  

An illustration of $T$, showing the two periodic orbits, and the intersection of the foliations $\cF^u$ and $\cF^s$ with $T$ is shown in Figure \ref{fig_HT_torus}.  

\begin{figure}[h]
\centering 
\labellist
\small \hair 2pt
\pinlabel {$\Sigma_1$} at 51.591 193.580
\pinlabel {$\Sigma_2$} at 158.591 193.580
\pinlabel {$c$} at 105.591 172.580
\pinlabel {$0$} at 115 255
\pinlabel {$\pi$} at 97.591 224.580
\pinlabel {$2\pi$} at 226.591 263.580
\pinlabel {$\pi$} at 228.591 215.580
\pinlabel {$0$} at 229.591 165.580
\pinlabel {$c$} at 296.591 152.580
\pinlabel {{\color{blue} $\mathcal F^s$}} at 374.591 266.580
\pinlabel {{\color{red} $\mathcal F^u$}} at 225 190
\pinlabel {$\mathrm{out}\ \Sigma_1$} at 377 238.580
\pinlabel {$\mathrm{in}\ \Sigma_1$} at 377 188.580
\pinlabel {$\begin{pmatrix}1&n\\0&1\end{pmatrix}$} at 198.591 84.580
\pinlabel {$\longrightarrow$} at 198.591 54.580
\endlabellist
\includegraphics[width=\linewidth]{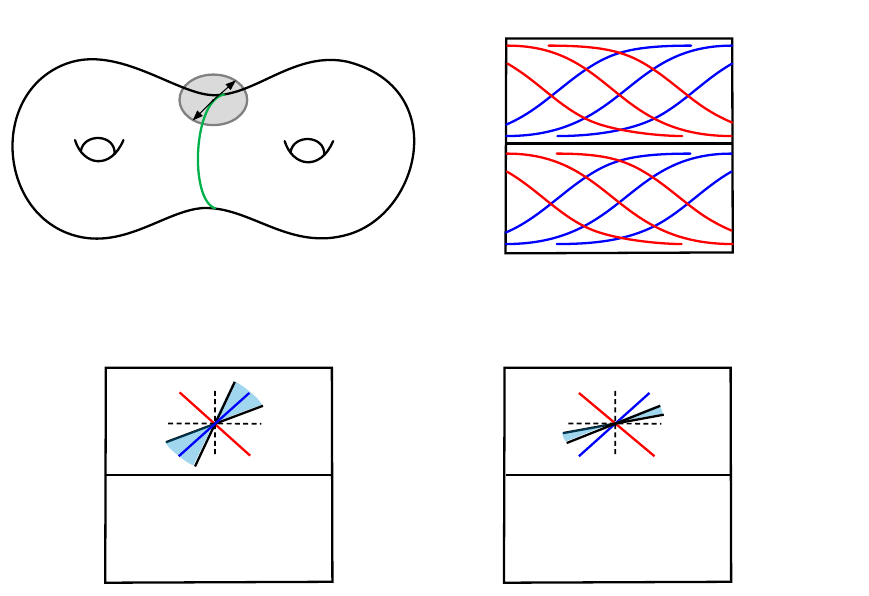} %
\caption{The Handel--Thurston surgery}\label{fig_HT_torus}
\end{figure}

The next step is to perform a surgery by cutting $T^1\Sigma$ along $T$ and regluing.  
For simplicity, we will describe the case where $c$ is a separating geodesic, but the general proof is exactly the same, treating the two sides of $T$ rather than the two connected components of $T^1\Sigma \setminus T$.  
Thus, assuming $c$ is separating, let $S_1$ and $S_2$ denote the two connected components of $T^1\Sigma \setminus T$, and glue $S_1$ to $S_2$ by a shear of the form $\left( \begin{smallmatrix} 1 & n \\ 0 & 1 \end{smallmatrix} \right)$, for some $n \in \bZ$ in our local coordinates, to obtain a new 3-manifold $M_n$.   This gluing map preserves the periodic orbits on $T$ and matches the transverse components.  One can also check explicitly in coordinates that, after re-gluing, the vector field generating $\flow$ remains a continuous vector field on $M_n$.  The question is whether, or more precisely {\em for which} $n \in \bZ$, this vector field generates an Anosov flow.  

There is a uniform argument to show that for any $n > 0$ (in the correct choice of local coordinates), the flow on $M_n$ is Anosov.  The proof involves a computation in coordinates which we do not carry out in detail.  However, for those familiar with the cone field criterion, we give an indication of why the positivity or negativity of $n$ is important.   

Fix coordinates on $T$ as in Figure \ref{fig_HT_torus}.  Consider, for concreteness, a point $x$ where the flow is orthogonal to $T$ and moving from $S_1$ to $S_2$.  
Working with a metric so that $E^{ss}$, $E^{uu}$ and the vector field generating the flow are all orthogonal, consider a cone  $C_x$ about the stable direction.  If $n$ is positive, the image of $C_x$ under the gluing map $\left( \begin{smallmatrix} 1 & n \\ 0 & 1 \end{smallmatrix} \right)$ remains in the same octant in these coordinates, so will again be contracted towards the stable direction as one applies the flow.  However, if $n$ is negative, the image of $C_x$ may contain $E^{uu}$, posing a problem.   

However, for any fixed $n<0$, one may choose a hyperbolic metric on $\Sigma$ so that the surgery by $\left( \begin{smallmatrix} 1 & n \\ 0 & 1 \end{smallmatrix} \right)$ described above {\em does} result in an Anosov flow.   In fact, this will be true for any hyperbolic metric where the length of the curve $c$ is sufficiently short.  The idea is as follows:  by a standard result in hyperbolic geometry, as the length of a curve $c$ approaches $0$, the length of a thin collar-neighborhood tube about $c$ approaches infinity; thus, for any fixed $t_0$ one can arrange so that any two consecutive intersections of a geodesic flowline with the torus $T$ are separated by a trajectory of time at least $t_0$. This allows one to ensure that the orbits spend a long enough time away from the surgery so that the cone criterion is satisfied. 
This argument is done carefully, in a more general context, in \cite[Section 2.3]{BGP16}.  

Provided that $n \neq 0$, the manifold $M_n$ is not Seifert fibered.  Rather, it is a graph manifold with one or two Seifert fibered pieces glued along the torus $T$.  Thus, the flow on it cannot be orbit equivalent to a geodesic flow or suspension flow, since geodesic flows are all supported on Seifert manifolds, and the mapping torus of an Anosov diffeomorphism is an exceptional graph manifold whose unique Seifert component is a product $\bT^2\times I$. Thus, we have produced a non-algebraic example.

  \begin{rem} \label{rem_neige_gluing}
 The technique of Handel and Thurston has been vastly generalized, and is now  included as a case of the general construction result of Neige Paulet.  See \cite{Pau23} for a general treatment of such gluing theory.  
 \end{rem}

 \subsection{The DA-construction and Franks--Williams example}\label{sec_example_Franks_Williams}
We now describe a construction that modifies a suspension flow to produce a new example.  The interesting feature of the new flow is that there is a separating torus in the 3-manifold which is transverse to the flow, which implies that the flow is not transitive, i.e., it does not have a dense orbit.  We now know many nontransitive examples on 3-manifolds, but the construction was quite surprising at the time.  Whether all Anosov diffeomorphisms were transitive is a question raised by Smale \cite{Sma67}; for flows the question is already implicit in Anosov \cite{Ano67}. The diffeomorphism case is still open in higher dimensions.  

The idea of the Franks--Williams example is to take a suspension of a slightly modified version of a linear Anosov map, where the modification is done so that there is a torus transverse to the suspension flow which separates the manifold.  The suspension flow thus constructed is not Anosov, but the lack of hyperbolicity lies only on one side of the torus.  So, one can cut along the torus, remove the bad side, and glue in another copy of the ``good" side. 

To make this precise, we apply the \emph{DA} (Derived from Anosov) construction of Smale \cite{Sma67}, which was later fleshed out by Williams in \cite{Wil70}.  

\begin{example}[Suspension of a DA map of the torus]
Let $A$ be a linear Anosov map of the torus, and let $x$ be a fixed point of $A$.  
Let $h$ be a diffeomorphism supported on a small neighborhood of $x$, on which $x$ is the only fixed point and $h$ repels radially away from $x$.  

Consider now the suspension flow $\psi$ of the map $h \circ A$.  The orbit of $x$ is periodic for this flow, since it is a fixed point of $h \circ A$, but is now a repeller (as long as $h$ is chosen strong enough to counterbalance the contraction of $A$ on the stable leaf), and has a tubular neighborhood $V$ whose boundary $T$ is a transverse torus.  Since $h$ repels radially, it preserves the stable manifold of $x$ and so the weak stable manifolds of the new suspension flow are contained in those of the suspension of $A$.  We may choose the transverse torus $T$ so that the intersection of $\cF^s$ leaves with $T$ forms a foliation with two closed leaves, coming from $\cF^s(x)$, and two Reeb components in the complement, as shown in Figure \ref{fig_DA}. 

 Finally, and less obviously, one can verify that if $h$ is chosen correctly then $M \smallsetminus \cup_{t\geq0}  \psi^t(V)$ is a {\em hyperbolic attractor} for $\psi$.\footnote{An {\em attractor} is a set of the form $\bigcap_{t>0} \flow^t(U)$ where $U$ is a nonempty, proper subset such that $\flow^t(U) \subset U$ for all $t>0$, and $\flow^T(\bar{U}) \subset U$ for some $T>0$.  A set is called {\em hyperbolic} for a flow $\flow$ if the restriction of $\flow$ to the set has invariant expanded/contracted tangent directions as in the definition of a smooth Anosov flow. See \cite{FH19} for further background.}
  \end{example} 
 
 \begin{remark}
 By working with a local section in the neighborhood of a periodic orbit rather than the global section above, one can generalize this construction to any Anosov flow.  Precisely, if  $\flow$ is an Anosov flow on a $3$-manifold $M$ and $\alpha$ is a periodic orbit, one    
can modify $\flow$ so that $\alpha$ becomes a repelling periodic orbit, again by ``pushing away'' from $\alpha$ along its stable foliation, and satisfying the same properties listed above. 
\end{remark}

\begin{figure}[h]
\begin{subfigure}[b]{0.45\textwidth}
\centering 
\scalebox{0.4}{
\includegraphics[width=283.015bp]{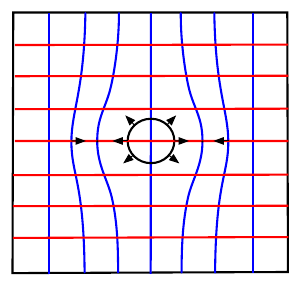}}
\caption{A repelling DA bifurcation on $\alpha$}
\end{subfigure}
\quad
\begin{subfigure}[b]{0.45\textwidth}
\centering 
\scalebox{0.4}{
\includegraphics[width=341.018bp]{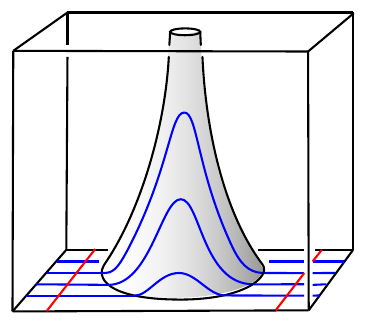}}
\caption{The transverse torus for a DA on a suspension}
\end{subfigure}
\caption{The DA bifurcation}\label{fig_DA}
\end{figure}

Using this we describe the Franks--Williams construction from \cite{FW80}. 
 
 \begin{example}\label{ex_Franks_Williams}
Let $\psi_1$ be the suspension flow of the DA map $h \circ A$ described above, supported on the mapping torus $M$ of $h \circ A$. Let $V_1$ be the tubular neighborhood of the periodic orbit of $x$, with transverse torus boundary $T_1$ chosen as in the description of the construction.  For simplicity, we adjust the metric on $M$ so that $T_1$ is normal to the direction of flow and the flow has unit speed.   Let $M_1 = M \setminus V_1$.  
Now, take another copy $M'$ of $M$ and let $V_2$ and $T_2$ denote the images of $V_1$, $T_1$ in this copy.  Consider the flow $\psi_2^t := \psi_1^{-t}$ on $M'$.  Here the torus $T_2$ is transverse (and still normal) and its intersection with the \emph{unstable} foliation of $\psi_2$ has two closed leaves and two Reeb components, and $M' \smallsetminus \cup_{t\leq0}  \psi_2^t(V_2)$ is a hyperbolic {\em repeller}.  

Let $M_2 = M' \smallsetminus V_2$, and glue $M_1$ to $M_2$ along their boundary tori, so that the image of the stable Reeb foliation of $\psi_1$ on $T_1$ is transverse to the unstable foliation of $\psi_2$ on $T_2$.  The vector fields generating $\psi_1$ and $\psi_2$ agree on the glued tori, so define a global flow on the glued manifold $M_1 \cup_{T_1} M_2$.  
Since orbits of the new flow intersect the torus at most once, and thus no additional recurrence is created, a result due to Ricardo Ma\~n\'e \cite{Man77}, shows that the transversality of the glued stable and unstable manifolds is in fact sufficient to ensure that this new flow is Anosov. See \cite[Theorem 1.2]{FW80}.  
\end{example}

 \subsection{Dehn--Fried surgery}\label{sec_example_Dehn_Fried}
 Topologically, the construction of Handel and Thurston in Section \ref{sec_example_Handel_Thurston} involved a {\em Dehn surgery} along a torus in the 3-manifold $T^1(\Sigma)$.  
  Goodman \cite{Goo83} developed a more general method of Dehn surgery along annuli, applicable to any Anosov flow, using an annulus transverse to the flow and contained in the neighborhood of a periodic orbit.  This construction readily generalizes to pseudo-Anosov flows as well, and in fact working in a neighborhood of a periodic orbit is not even strictly necessary; a modern exposition of this general construction is given in \cite{Tsa24}.  Goodman's construction was originally carried out in the smooth category, but works for topological flows as well.  
 
 Around the same time,  Fried \cite{Fri83} described a different method of modifying Anosov flows by Dehn surgery.  His construction does not obviously give a smooth Anosov flow (even when the input flow is smooth), but only a topological one.  However, the construction has the advantage of being easy to visualize.\footnote{For transitive flows, the fact that Fried surgeries produce flows orbit equivalent to smooth ones is shown by Mario Shannon in \cite{Sha21}.  In fact, he shows that Fried surgery on a smooth flow is always orbit equivalent to Goodman's surgery construction, which is by definition smooth.   As of writing the question has still not been resolved for nontransitive Anosov flows.} 
 We describe this surgery procedure now.  
 
 To begin, recall that a 
  $(p,1)$-\emph{Dehn twist} on the torus $\bT^2 = \R^2/\Z^2$ is any homeomorphism isotopic to the map $f\colon \bT^2 \rightarrow \bT^2$ given by
  \[
   f(s, t) = (s + pt , t ) .
  \]

  \begin{figure}[h]
    \labellist
\small \hair 2pt
\pinlabel {$\longrightarrow$} at 236.155 68.885
\endlabellist
\includegraphics[width=8cm]{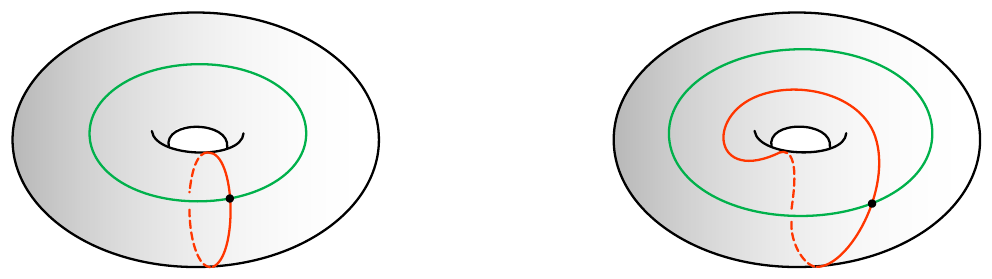}
   \caption{A $(1,1)$-Dehn twist} \label{fig_Dehn_twist}
  \end{figure}
 
Note that this map is dependent on the choice of coordinates on $\bT^2$, specifically on its factorization as $S^1 \times S^1$.  
 
 An \emph{integral Dehn surgery} on a $3$-manifold $M$ is the action of removing a solid torus $\bD^2 \times S^1$ from $M$, fixing coordinates on the boundary $\bT^2 =  \partial \bD^2 \times S^1$ so that the first factor (called the {\em meridian}) bounds a disc, and 
gluing it back in via a $(p,1)$ Dehn-twist on its boundary.   
 Note that, while the meridian is canonically defined, the other $S^1$ factor (called the {\em parallel}) is not; any curve intersecting the meridian once can be used to define the splitting into meridian--parallel coordinates.  
 
\begin{rem}  \label{rem_collapse_curves}
Since $\bD^2 \times S^1$ retracts to the circle $\{x\} \times S^1$ for any $x \in \bD^2$, the surgery described above is topologically equivalent to the following construction: {\em remove} the interior of the solid torus $\bD^2 \times S^1$ from $M$, foliate the boundary torus $S^1 \times S^1$ by a family of curves representing $(p, 1)$ in $\pi_1(S^1 \times S^1) \cong \bZ \times \bZ$, and collapse each curve in the foliation to a point.  
\end{rem}

 \begin{example}[Fried's surgery \cite{Fri83}]
  Let $\flow$ be an Anosov flow on a manifold $M$ and let $\alpha$ be a periodic orbit of $\flow$. Assume for simplicity that $\fs(\alpha)$ and $\fu(\alpha)$ are both annuli (for a generic periodic orbit, one or both could be a M\"obius band, and hence a tubular neighborhood of the orbit could be a solid Klein bottle).
  Fix a local linearization of $\flow$ in a neighborhood of this orbit, that is to say, choose coordinates in which the flow locally agrees with the suspension flow of a hyperbolic linear map $\left( \begin{smallmatrix} \lambda & 0 \\ 0 & 1/\lambda \end{smallmatrix} \right)$, for $\lambda >1$, near the origin in $\bR^2$.  
    
   Let $M^\ast$ be the manifold obtained by blowing up $M$ along the normal bundle of $\alpha$. That is, if $V$ is a tubular neighborhood of $\alpha$ and $(r,\theta, t)$ are polar coordinates on $V$ such that $\alpha = (0,0,t)$, then $M^\ast$ is obtained by replacing each point $x \in \alpha$ by the circle of normal tangent directions $\left\{(r,\theta,t) \mid r = 0, \, \theta \in \R/\Z \right\}$. 
  The manifold $M^{\ast}$ has one torus boundary $T_\alpha$ and contains $M\smallsetminus\alpha$ in its interior. The flow $\flow$, which is defined on the interior of $M^\ast$, extends continuously to give a flow $\flow_\ast^t$ on $M^\ast$.  See Figure \ref{fig_blow_up_Fried} for an illustration of the behavior of orbits of this flow.  
There are exactly four periodic orbits of $\flow_\ast^t$ on $T_\alpha$, two attracting and two repelling. 
  
  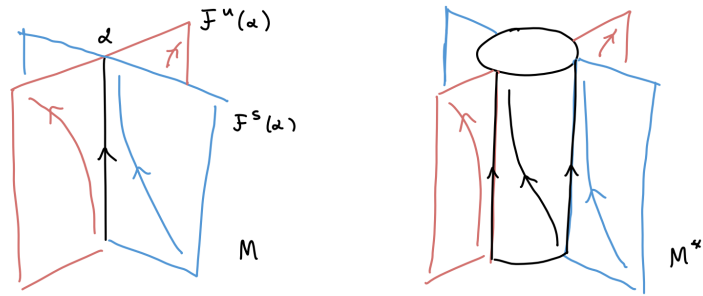
\begin{figure}[h]
    \labellist
\small \hair 2pt
\pinlabel {$\alpha$} at 62 182
\pinlabel {$\mathcal F^u(\alpha)$} at 128.596 188.920
\pinlabel {$\mathcal F^s(\alpha)$} at 149.596 104.920
\pinlabel {$M$} at 140 26
\pinlabel {$M^*$} at 384.596 26
\endlabellist
\includegraphics[width=10cm]{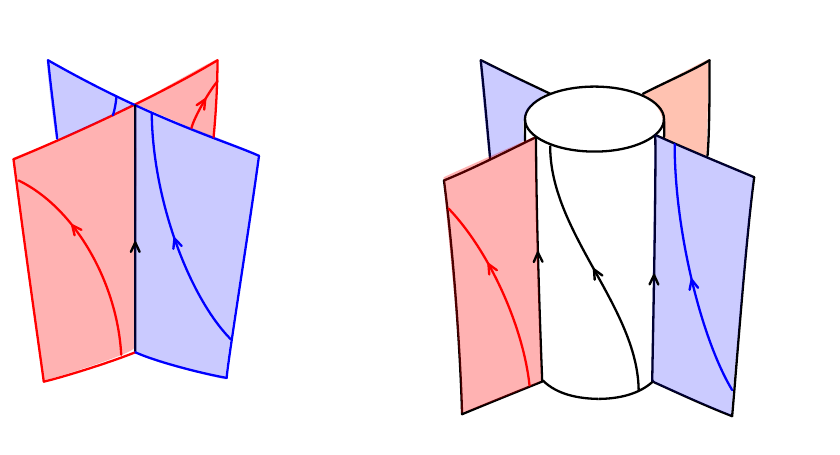}
   \caption{Blowing-up the orbit $\alpha$} \label{fig_blow_up_Fried}
  \end{figure}

  On $T_\alpha$, choose a foliation $\mathcal C$ by circles, transverse to $\phi_\ast^t$, and such that each circle leaf intersects each of the periodic orbits only once.   Moreover, for any choice of $p \in \bZ$, one can find such a curve that projects onto $p \alpha$, by choosing this to wind around $p$ times in the direction of $\alpha$ (with orientation) in some connected component of the complement of the periodic orbits.  
 
 Now, let $M_{\mathcal C}$ be the manifold obtained by identifying every circle leaf of $\mathcal{C}$ to a point.    By Remark \ref{rem_collapse_curves}, this defines a  $(p,1)$-Dehn surgery on $M$, where the meridian of $T_\alpha$ is chosen to be a curve $ \{r=0, \, t=\text{const.}\}$ and the parallels are the curves $ \{r=0, \, \theta=\text{const.}\}$.
    
 The flow $\phi_\ast^t$ naturally descends to a flow $\phi_{\ast \mathcal C}^t$ on $M_{\mathcal C}$.  To check that this flow is \emph{topologically} Anosov is not too difficult.  This is because the flow and the foliations of $M$ and $M_{\mathcal C}$ are identical outside of the orbit $\alpha$, and the new orbit $\alpha_{\mathcal C}$ has, by construction, invariant weak stable and unstable leaves with the correct behavior of orbits.  
 \end{example}

 \begin{rem}[Flows on atoroidal 3-manifolds]
An irreducible 3-manifold $M$ is said to be {\em geometrically atoroidal} if it contains no $\pi_1$-injectively immersed torus, 
and algebraically atoroidal if $\pi_1(M)$ does not contain any subgroup isomorphic to $\bZ \times \bZ$.  By the Perelman--Thurston Geometrization Theorem,  this is also equivalent to either being spherical (having finite fundamental group) or admitting a hyperbolic metric. 
The suspension, geodesic, Handel--Thurston, and Franks--Williams flows are all supported on manifolds with essential tori.  But Fried (or Goodman) surgery can be used to produce a wealth of examples on atoroidal 3-manifolds.  This is thanks to the {\em hyperbolic Dehn surgery} theorem of Thurston (see Chapter 5 of Thurston's Princeton University Notes, now published in \cite{Thu22}, for a description of hyperbolic Dehn surgery.)

Starting with an Anosov flow $\flow$ on any 3-manifold $M$, take a periodic orbit or  finite union of periodic orbits $\{\alpha_i\}$ with the property that any torus in $M \setminus (\bigcup_i \alpha_i)$ is homotopic into the boundary.  Then, Thurston's theorem implies that for all but finitely many choices of slopes of Dehn surgeries $(p_i, 1)$ on the orbits $\alpha_i$ as in Fried's construction, the resulting manifold will be atoroidal, and hence Fried's surgery produces an Anosov flow on an atoroidal 3-manifold.  As a concrete example, one may start with geodesic flow on $M = T^1\Sigma$ and remove any collection of filling geodesics. 
 \end{rem}

We conclude this section with a famous open conjecture about Fried surgeries; the Fried--Ghys conjecture.  Fried posed this as a question, in Ghys' version it is a conjecture, giving a suggestion for a very simple classification of transitive Anosov flows up to orbit equivalence \emph{and} Fried surgeries.  
\begin{conjecture} \label{ghys_fried_conjecture}
Let $\flow $ be a transitive Anosov flow with transversally orientable foliations. Then $\flow$ is obtained from a suspension of $\left(\begin{smallmatrix} 2 & 1\\ 1 &1\end{smallmatrix}\right)$ via (a finite number of) Dehn--Fried surgeries.\end{conjecture} 
In \cite{DS19}, Dehornoy and Shannon show that the geodesic flows on the unit tangent bundle of negatively curved 2-dimensional orbifolds can all be obtained by surgery from suspension Anosov flows.  Using also a result of Minikawa explained in \cite{DS19}, this shows that Ghys' conjecture is true for all geodesic flows, as well as all suspensions.\footnote{Note that the result as stated in \cite{DS19} says that Ghys--Fried conjecture holds for all algebraic Anosov flows, but they use a too restrictive definition of algebraic Anosov flow in that article. The conjecture is still open for the algebraic Anosov flows which are finite covers of geodesic flows.}

 \subsection{The Bonatti--Langevin example}\label{sec_example_bonatti_langevin}
 All the surgery constructions discussed thus far started with an Anosov flow (and more generally, could be done starting with a pseudo-Anosov flow), then modified the supporting manifold while somehow carrying the flow along, thus producing a new example on a new manifold.  
 
 The following example, constructed in 1994 by  Bonatti and Langevin \cite{BL94}, takes a different approach: rather than starting with an existing (pseudo)-Anosov flow, they define, from scratch, a partial flow on a manifold with boundary and then glue up the boundary components.  This is done so that the resulting flow is Anosov and  has the following interesting features: it is transitive\footnote{The example built in \cite{BL94} is in fact made to be \emph{volume-preserving}. It is now known that \emph{any} transitive Anosov flow is orbit equivalent to a volume-preserving one \cite{Asa08}.}, but has a (nonseparating) torus transverse to the flow, and moreover every orbit of the flow \emph{except for one} intersects this transverse torus.
 
The construction can be roughly described in the following way.  
Consider first the suspension flow (on $\bR^2 \times S^1$) of a hyperbolic linear map $ \left( \begin{smallmatrix} \lambda & 0 \\ 0 & 1/\lambda \end{smallmatrix} \right)$, and take an ``octagonal" tubular neighborhood of the periodic orbit $\{ 0\} \times S^1$ consisting of eight annuli, four of which are transverse to the flow (meeting the coordinate axes in $\bR^2$); with two of these on which the flow is outgoing and two incoming and four of which are tangent to the flow.   See Figure \ref{fig_BL} for an illustration, first as a neighborhood of the periodic orbit in standard linear coordinates, then in a different system of coordinates, normalized so that flowlines on the tangent annulus sides of the octagons are horizontal.  

We next glue the opposite tangent sides together in pairs respecting flowlines, as indicated in Figure \ref{fig_BL} (right).  In particular, the direction of the flow means that opposite sides are glued with a rotation.  
 The result is a manifold $M_{BL}$ homeomorphic to $\Sigma \mathbin{\widetilde{\times}} S^1$, the orientable circle bundle over $\Sigma$, where $\Sigma$ is the complement of two discs in $\bR P^2$.  %
The manifold has two boundary components, each a torus.  The flow induced from the suspension is outgoing on one torus and incoming on the other.  Thus, gluing the torus boundary components together produces a closed manifold with a flow.  As in the Handel--Thurston example, the relevant question is which gluing maps lead to an {\em Anosov} flow.  

In Bonatti and Langevin's work \cite{BL94}, they explicitly prove that the gluing map which exchanges the two coordinate directions (given by the $S^1$ factor and the factor with constant $S^1$ coordinate) on the boundaries leads to an Anosov flow.  More generally, the needed condition to ensure that the resulting flow is Anosov is to choose the gluing map so that the intersection of the stable leaf of the hyperbolic orbit with the incoming torus and the intersection of its unstable leaf with the exiting torus are not isotopic in the glued torus.  This extension was originally shown in \cite{Bar98}, then extended far beyond that specific construction in \cite[section 8]{BF13}, see also Proposition \ref{prop_gluing_fatgraphs} for a more precise statement.

\begin{figure}[h]
\labellist
\small \hair 2pt
\pinlabel {$\mathrm{in}$} at 260 95
\pinlabel {$\mathrm{out}$} at 358 100
\pinlabel {$\mathrm{out}$} at 430 200
\pinlabel {$\mathrm{in}$} at 522 210
\pinlabel {$\mathrm{in}$} at 465.204 80
\pinlabel {$\mathrm{out}$} at 556 110
\pinlabel {$A'$} at 482.204 211.080
\pinlabel {$B$} at 547 189
\pinlabel {$B'$} at 429.204 56.080
\pinlabel {$A$} at 499 27
\endlabellist
\includegraphics[width=12cm]{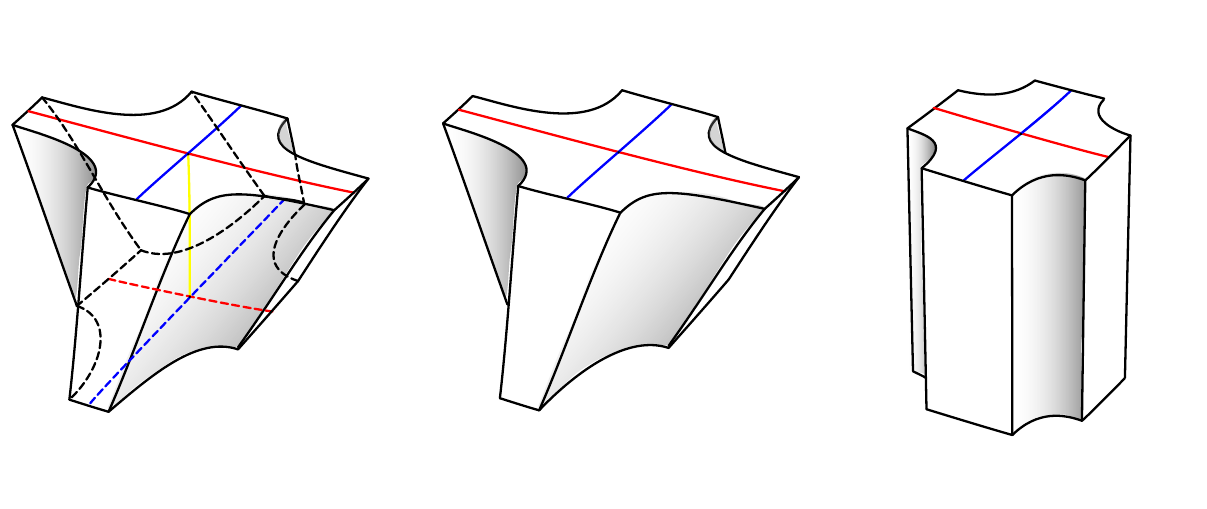}
\caption{A local model of a hyperbolic orbit, in two different coordinate views (middle and right).  The ``block" is obtained by gluing the opposite pairs of tangent annuli together with a rotation, as indicated in the rightmost figure.}
 \label{fig_BL}
\end{figure}
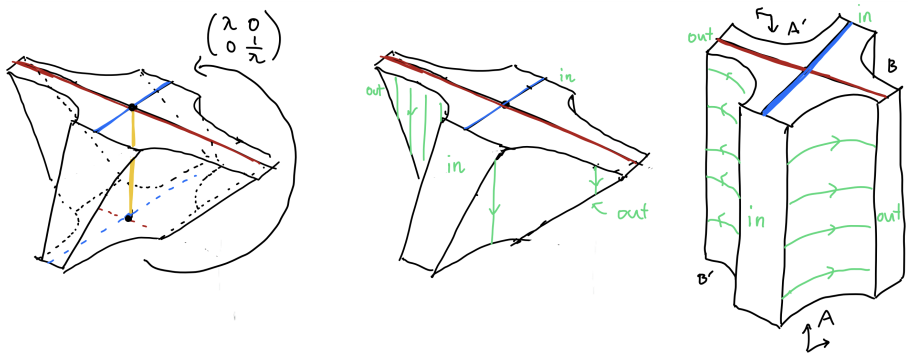

\begin{rem}[Generalizations of Bonatti--Langevin]
Like the Franks--Williams example, the example of Bonatti--Langevin consists of a flow on a (possibly disconnected) manifold with boundary that has a hyperbolic maximal invariant subset, a boundary consisting of transverse tori, and a choice of gluing of the boundaries to obtain an Anosov flow in the glued manifold. This strategy has been widely extended by  Bonatti,  Béguin and  Yu \cite{BBY17}, and then by  Paulet \cite{Pau24}, to give conditions on the gluings to ensure the creation of an Anosov flow on the glued manifold.  

In another direction, Barbot and Fenley \cite{BF13,BF15} generalized the ``building block" construction of this example to give a generalized procedure for gluing a collection of many such blocks together to produce flows on graph manifolds; and showed that all of the so-called {\em totally periodic (pseudo)-Anosov flows} are obtained from such constructions.  We discuss these more in Chapter \ref{chap_features_and_3manifold}.
\end{rem}

There have been many other techniques developed since the 1990s to produce Anosov and pseudo-Anosov flows.  In addition to the ones already mentioned above, we note of particular interest for Anosov flows the generalization of Goodman surgery by 
 Foulon and  Hasselblatt \cite{FH13}, which produces Anosov flows which are Reeb flows of contact structures. This technique was later extended further by  Salmoiraghi \cite{Sal23,Sal24} to obtain Anosov flows via surgery on bi-contact structures.

Singular pseudo-Anosov flows are even more common than Anosov ones. An unpublished work of Gabai and Mosher from the early 90s explains how to build pseudo-Anosov flows that are transverse to finite-depth foliations. In particular, they showed that any hyperbolic manifold with positive first Betti number admits a pseudo-Anosov flow. 
More recently, the theory of {\em veering triangulations} developed by Agol and Guéritaud \cite{Ago11, Gue16} has emerged as a powerful tool to produce examples of pseudo-Anosov flows. See, e.g., \cite{AT24,FSS19,LMT23,SS23}.

\subsection{A cautionary non-example}  \label{sec_non-ex} 
The following non-example illustrates the nuance in the ``backwards expansivity" condition, item \ref{item_PAF_backwards_expansivity}, of the definition of pseudo-Anosov flow.   By modifying a nontransitive flow, we produce an example of a flow $\psi$ with foliations $\cF^+$ and $\cF^-$ where orbits converge in the future along $\cF^+$ leaves, and in the past along $\cF^-$ leaves (satisfying condition \ref{item_PAF_converges}) but there are some orbits on $\cF^+$ leaves which also converge in the past.  
The construction uses a general tool called {\em blowing up a leaf}, which replaces a leaf of a foliation with a closed interval worth of leaves, and can be performed while keeping this foliation topologically transverse to another.  

\begin{construction}[Blowing up along a leaf -- see \cite{Cal07} Example 4.14 ] \label{const_blowup}
Suppose $\cF^s$ and $\cF^u$ are two topologically transverse foliations on a 3-manifold $M$.   Consider an atlas of foliation charts $f_i\colon U_i \to D^2 \times [0,1]$, (where $D^2 \cong [0,1]^2$ is a closed disc) so that the image of $\cF^s$-leaves are horizontal slices, and $\cF^u$ leaves are transverse vertical planes. 
Fix a leaf $l$ of $\cF^s$, and let $X_i \subset [0,1]$ denote the subset such that $f_i(U_i \cap l) = D^2 \times X_i$.  
For each chart $U_i$ let $h_i\colon [0,1] \to [0,2]$ be a strictly increasing function with a jump discontinuity at each point of $X_i$ and continuous on the complement of $X_i$.  Thus, there is a continuous, increasing function $\bar{h}_i$, which collapses each interval in the complement of the image of $h_i$ to a point, and such that $\bar{h}_i h_i$ is the identity map on $[0,1]$.   

We now use $D^2 \times [0,2]$ as foliation charts with the following transition maps:  For each of the original transition maps $\psi_{ij}\colon I_i \to I_j$, where $I_i \subset [0,1]$ is the interval of leaves of $U_i$ meeting the chart $U_j$, we modify this to be defined on the blown-up domain  by sending $y \in h_i(I_i \setminus X_i)$ to $h_j \psi_{ij} \bar{h}_i(y)$, and extending linearly over the blown-up intervals in the complement.  We leave the gluing maps unchanged in the $D^2$ coordinate directions, so that there is again a transverse vertical foliation $\cF^u_*$, isomorphic to $\cF^u$.
This gives a pair of topologically transverse foliations $\cF^s_*$ and $\cF^u_*$ on a manifold homeomorphic to $M$.  The local maps $\bar{h}_i$ glue together to define a global map from the leaf space of $\cF^s_*$ to that of $\cF^s$, realized by a continuous, surjective map $H\colon M \to M$ sending leaves to leaves.  This map $H$ is a homotopy equivalence, and injective except on the preimage of $l$, on which it collapses an interval worth of leaves to $l$.     
\end{construction} 

\begin{example} \label{ex_blowup_Anosov}
Let $\flow$ be a nontransitive Anosov flow on a closed 3-manifold $M$.  Choose a leaf $l\in \cF^u$ such that all the orbits on $l$ are wandering. Such a leaf exists since $\flow$ is assumed to be non-transitive. Let $\fs_*$ and $\fu_*$ be the foliations obtained from $\fs$ and $\fu$ by blowing up the leaf $l$ as in the above construction.  Let $\psi$ be any parameterization of a flow whose orbits are the intersection of leaves of $\fs_*$ and $\fu_*$.
We claim that $\psi$ satisfies all the conditions of Definition \ref{def_topPA} except for condition \ref{item_PAF_backwards_expansivity}.

By construction, $\psi$ is at the intersection of two topologically transverse foliations $(\fu_*, \fs_*)$, so satisfies item \ref{item_PAF_foliations}. 
Thus, what we need to show is that item \ref{item_PAF_converges}, i.e., the forward convergence along stable leaves and backwards convergence along unstable leaves, still holds, and that item \ref{item_PAF_backwards_expansivity} (backwards expansivity) does not.

The construction of the blowup ensures that the dynamics of the flow is unchanged along unstable leaves: the map $H$ (from Construction \ref{const_blowup}) sends each leaf of $\fu_*$ to a leaf of $\fu$, preserving the foliation by orbits.  Thus, any two orbits on the same leaf of $\fu_*$ converge as $t \to -\infty$.  
We next show that for any orbit $\alpha$ in $l$, each orbit of $H^{-1}(\alpha)$ converges both in the past and future, and then that all orbits on $\fs_*$ leaves converge in the future.

To do this, suppose we are given some $\delta >0$. Let $K$ be a compact subset of $l$ such that for all $x \in l \smallsetminus K$, the set $H^{-1}(x)$ (which is an interval in $M$) has diameter less than $\delta$.    Take $x \in l$ and let $\tau_x := H^{-1}(x)$. 
Since all orbits in $l$ are wandering, the distance between the compact set $K$ and the basic sets of $\flow$ is strictly positive. Thus, for $|t|$ sufficiently large, $\flow^t(x) \cap K = \emptyset$, so the diameter of $\psi^t(\tau_x)$ is less than $\delta$.  It follows that for any $y, z \in H^{-1}(x)$, we can find a standard reparameterization $\sigma$ such that $d(\psi^t(y), \psi^{\sigma(t)}(z))< \delta$ for all $t$ sufficiently large.  This gives the desired backwards and forwards convergence of orbits in $H^{-1}(\alpha)$.  

To show the forwards convergence required by condition \ref{item_PAF_converges}, we use the following fact given by the blowup construction: For any $r>0$, there exists a compact set $K_r \subset l$ such that if a transversal $\tau$ to $\cF^u$ does \emph{not} intersect $K_r$, then the transversal $\tau_\ast = H^{-1}(\tau)$ to $\cF^u_\ast$ is such that $\mathrm{diam}(\tau_\ast) \leq  \mathrm{diam}(\tau)+r$.

Consider $x,y$ on distinct orbits in the same stable leaf of $\fs_\ast$, and $\tau_\ast$ a segment between $x$ and $y$ and transverse to $\fu_\ast$. Let $\tau = H(\tau_\ast)$.
Pick a $r>0$ and associated $K_r\subset l$. As above, since $l$ is a leaf of a wandering orbit, there exists $T_r$ such that, if $|t|>T_r$, then $\flow^t(\tau)\cap K_r = \emptyset$.
Thus, for $t>T_r$, $\mathrm{diam}(\psi^t(\tau_\ast)) \leq  \mathrm{diam}(\flow^t(\tau))+r$. Since $\tau$ was on a stable leaf, and $r$ is arbitrary, we deduce that, for any fixed $\epsilon>0$, we can find a standard reparameterization $\sigma$ such that $d(\psi^t(x), \psi^{\sigma(t)}(y))< \epsilon$ for all $t$ sufficiently large.

\end{example} 

 In the above example, it appears to be necessary to blow up a wandering leaf in order to ensure that the blown up flow still satisfies the forwards convergence (as opposed to just having a flow whose orbits in the stable foliation stay a bounded distance apart). We do not know of a construction that works in the transitive case.

\section{The orbit space of a pseudo-Anosov flow}\label{sec_pA_to_orbit_space}

Understanding a flow $\varphi$ on $M$ up to orbit equivalence amounts to understanding the topology of the foliation of $M$ by orbits of $\flow$.  While orbits (and their stable and unstable leaves) are arranged in a complicated way in $M$, the picture is much nicer when lifted to the universal cover.  Thus, it can be advantageous to work $\pi_1(M)$-equivariantly in $\wt{M}$ rather than in $M$ itself.    In this section we describe the global picture of leaves and orbits lifted to $\wt M$, and then use this to pass to a convenient quotient called the {\em orbit space}.

\subsection{Topology of stable, unstable, and orbit foliations in $\wt M$}  
\begin{convention*} 
In this section, we fix a closed $3$-manifold $M$, a pseudo-Anosov flow $\flow$ on $M$, and denote by $\hflow$ the lift of $\flow$ to the universal cover  $\wt M$.  The (possibly singular) foliations $\fs$ and $\fu$ associated with $\flow$ lift to $\hflow$-invariant foliations on $\wt M$, which we denote by $\hfs$ and $\hfu$. 
  \end{convention*} 

Our first goal is to describe the topology of leaves of $\fs$ and $\fu$ in $M$, along with their lifts to $\wt M$, and their foliations by orbits.  
To start, we prove a lemma describing the orbit foliation on the universal cover of a leaf of $\cF^s$ or $\cF^u$.  Note that we do not know yet that leaves are $\pi_1$-injectively mapped into $M$ (we will prove this later), so {\em a priori} $\wt l$ may not agree with the lift of $l$ to $\wt M$.  
\begin{lemma} \label{lem_trivial_foliated_lift}
Let $l$ be a nonsingular leaf of $\fs$ or $\fu$, or a connected component of the complement of the closed orbit in a singular leaf.  Then the lift of the orbit-foliation of $l$ to the universal cover $\wt l$ of $l$ is a trivial foliation of the plane.   
\end{lemma}
A {\em trivial foliation of the plane} means one that is homeomorphic to the standard foliation of $\R^2$ by vertical lines.  

\begin{proof}  
Without loss of generality, by reversing the flow, we can assume $l$ is a leaf of $\fs$. Since $l$ is foliated by orbits, it cannot be $S^2$ or $\bR P^2$, so $\wt l$ is homeomorphic to the plane.   It suffices to show that the space of (lifted) orbits in $\wt l$ is homeomorphic to $\bR$, i.e., any pair of orbits in $\wt l$ can be separated by a third.  
Let $\wt o$ and $\wt o'$ be orbits in $\wt l$.   Suppose for contradiction they cannot be separated.  We will show that this implies that they can be reparameterized to converge in the future, and then use this to in fact separate them.  

Since $\wt o$ and $\wt o'$ cannot be separated, 
 then there are leaves in $\wt l$ limiting on to both.  In particular, for any $\delta$ we can find a point $\wt x \in \wt l$ such that (up to switching $\wt o$ and $\wt o'$), $\wt x$ is distance $\delta$ from some point $y$ in $\wt o$, and for some $t_0 >0$ we have $\hflow^{t_0}(\wt x)$ distance at most $\delta$ from a point $y'$ in $\wt o'$.  
Project these to $x, o$ and $o'$ in $l$.  By definition of pseudo-Anosov flow, provided $\delta$ is chosen sufficiently small, we can find a standard reparameterization $\sigma$ so that $d(\flow^t(x), \flow^{\sigma(t)}(y))< \epsilon$ for all $t>0$ and goes to zero (for any desired choice of $\epsilon$)  and a standard reparameterization $\tau$ such that $d(\flow^t(x), \flow^{\tau(t)}(y'))< \epsilon$ for all $t>t_0$.
Lemma \ref{lem_same_local_stable} then implies that the forward rays of the orbits of $x, o$ and $o'$ can be reparameterized to all stay in the same local leaf.  Note that (as can be seen by covering $M$ with flowboxes) one can always assume that local leaves are simply connected.   
Thus, the union of such local leaves between the forward ends of $x,$  $o$ and $o'$ can be lifted to a strip $S$ homeomorphic to $[0,1] \times [0, \infty)$ in $\wt l$, bounded by a small transverse segment $\tau$ in $\wt l$, and by positive rays of either $\wt o, \wt x$ or $\wt o'$.  In particular, $\wt o$ and $\wt o'$ converge (up to reparameterization) in the future, and can be separated by some leaf in the strip $S$.  
\end{proof} 

Next we describe the possible topology of leaves in $M$. 
For this we briefly recall (without proof) the classification of foliations on closed surfaces.  A detailed account of foliations on surfaces can be found in \cite[Chapter 1]{HH86}.  
\begin{proposition}[Foliations on surfaces] \label{prop_foliations_torus}
If $S$ is a closed, orientable surface admitting a foliation, then $S$ is a torus. The foliation contains a closed (possibly empty) set $C$ of closed leaves, which are pairwise homotopic.  If $C$ is nonempty, any connected component $A$ of the complement of $C$ is an annulus foliated either as a {\em Reeb component} or the suspension of an increasing homeomorphism of the interval.  See 
Figure \ref{fig_torus_foliations}. Finally, $C = \emptyset$ if and only if the foliation is homeomorphic to the suspension of a circle homeomorphism with irrational rotation number.  
\end{proposition}

\begin{figure}[h]
\includegraphics[width=10cm]{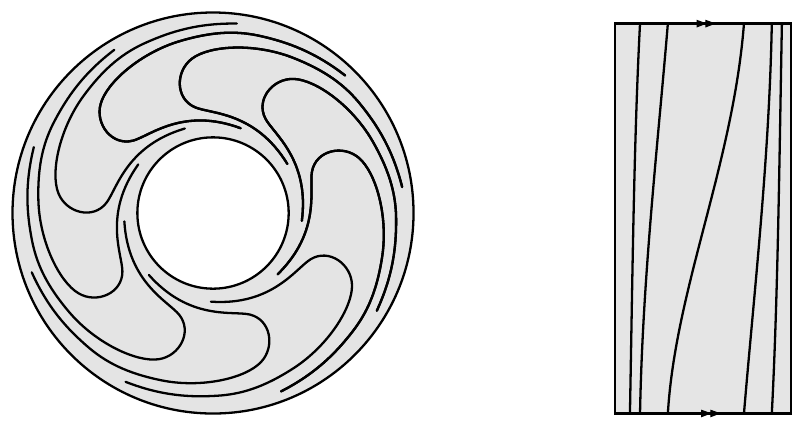}
\caption{Foliations of an annulus with closed boundary leaves}
 \label{fig_torus_foliations}
\end{figure}

\begin{lemma}[Leaves of $\fs$ and $\fu$] \label{lem_no_compact_leaves}
$\fs$ and $\fu$ have no compact leaves, and any leaf contains at most one periodic orbit.  Topologically, any nonsingular leaf of $\fs$ or $\fu$ is either a plane, an open annulus, or an open M\"obius band.   For singular leaves, each connected component of the complement of the periodic orbit is an open annulus or M\"obius band.  
\end{lemma} 

\begin{proof}
Up to switching direction of the flow, we need only treat leaves of $\fs$; let $l$ be such a leaf.  
First, if $l$ contains two periodic orbits, then since they are compact sets, there is some positive minimum distance between them, contradicting condition \ref{item_PAF_converges} (forward convergence of orbits) in the definition of pseudo-Anosov flow.   
Convergence can also be used to show that leaves cannot be compact, as follows.  A compact nonsingular leaf is a torus or Klein bottle; by passing to a double cover of the leaf we can consider only the torus case.  By  Proposition \ref{prop_foliations_torus}, if the foliation of this torus leaf contained a closed orbit, then either it would consist entirely of closed orbits, contradicting forward convergence, or we could find an orbit accumulating onto a closed leaf in the past, contradicting expansivity.  The remaining possibility is that the orbit-foliation of $l$ is the suspension of a circle homeomorphism with irrational rotation number; meaning that the first-return map to a closed transversal gives a homeomorphism of the circle without periodic points. A standard argument shows that this also fails forward convergence.  A similar argument applies to singular leaves, showing they also cannot be compact.\footnote{In the smooth Anosov case, there is an easy argument for non-compactness of leaves, since the contraction of the strong foliations directly implies that the area of a compact leaf would decrease under the flow, which is absurd.}

To describe singular and noncompact leaves, we return to the picture given by Lemma \ref{lem_trivial_foliated_lift}.  Suppose first for concreteness that $l$ is nonsingular.  Let $\wt l$ be the universal cover of $l$. Then $\pi_1(l)$ acts by deck transformations preserving the foliation of $\wt l$ by orbits, and so acts also on the space of orbits in $\wt l$, which is homeomorphic to $\mathbb{R}$.  Suppose some orbit $o$ is fixed by a nontrivial deck transformation $g$, so $o$ descends to a closed orbit in $\wt l / g$. Since $g$ acts by isometries, if it also fixed a second orbit $o'$, then $o$ and $o'$ would be some minimum Hausdorff distance apart in $\wt M$, again contradicting forward-convergence in $\wt M$ (on which the lifted flow is also pseudo-Anosov).  
Thus, the action of $\pi_1(l)$ on the space of orbits has the property that each element acts with at most one fixed point.   A classical theorem of H\"older and Solodov (see Theorem \ref{thm:solodov} stated below) implies that $\pi_1(l)$ is isomorphic to a subgroup of affine transformations of $\mathbb{R}$, so in particular solvable. By the classification of surfaces (and the fact that $l$ is noncompact), only the plane, open annulus and open M\"obius band, have a solvable fundamental group, hence $l$ must be one of these possibilities so either $\pi_1(l)$ is trivial or $\pi_1(l) \cong \bZ$. 
In the singular case, the same argument shows that each element of $\pi_1(l)$ preserves only the singular orbit, hence $\pi_1(l) \cong \mathbb{Z}$.  

\end{proof} 

Combined with the results above, the following important proposition will allow us to deduce many topological properties of the foliations in $\wt M$.  

\begin{proposition}[Verjovsky, Fenley--Mosher, Paternain]  \label{prop_closed_transversal}
Any closed loop transverse to $\cF^s$ or to $\cF^u$ is not nullhomotopic.  
\end{proposition} 

In the case of smooth Anosov flows, Proposition \ref{prop_closed_transversal} can be deduced from Haefliger's theorem, proved by Haefliger in the $C^2$ setting, and generalized by Solodov \cite{Sol82}  to lower regularity.  A proof in the $C^0$ case can be found in \cite[Ch. VII]{HH86}.  This theorem states that a
foliation which admits a nullhomotopic closed transversal always has a leaf with {\em one-sided holonomy}, a condition which is easy to rule out using the convergence of orbits in the future (or past) along stable (unstable) leaves.  This approach is taken in Verjovsky's thesis \cite{Ver74}. 
The general case of this theorem is proved (for expansive flows) by Paternain in \cite[Lemma 9]{Pat93}, using similar considerations.   

In place of Haefliger's theorem, one can instead use the closely related, but stronger, theorem of Novikov \cite{Nov65}, as generalized to $C^0$ foliations by Solodov \cite{Sol82}, and generalized to essential laminations by Gabai and Oertel \cite{GO89}.  This theorem states directly that a 3-manifold admitting a foliation (or more generally, an {\em essential lamination}) without compact leaves has no nullhomotopic closed transversal.  Lemma \ref{lem_no_compact_leaves} shows $\fs$ and $\fu$ have no compact leaves, so in the Anosov case, this immediately gives Proposition \ref{prop_closed_transversal}.  In the pseudo-Anosov case, Fenley and Mosher \cite{FM01} show that one can open up the foliation $\fs$ (or $\fu$) along singular leaves to give an essential lamination, then apply the work of  \cite{GO89} to give an independent proof of Proposition \ref{prop_closed_transversal}.  See Proposition 4.2 in \cite{FM01}.   Proofs of Novikov's theorem (in the differentiable setting) can be found in \cite{Cal07} or \cite[Ch.~6]{RS17}. 

From Novikov's theorem and its generalizations, one can also deduce the following.  

\begin{proposition}[Inaba--Matsumoto \cite{IM90}]
If $M$ admits a pseudo-Anosov flow, then $\pi_2(M) = 0$. 
\end{proposition}
We will later show, using the orbit space, that $\wt M \cong \bR^3$, giving a different proof that $M$ is aspherical.  See Corollary \ref{cor_aspherical}.

Using Proposition \ref{prop_closed_transversal}, we prove the following. 

 \begin{proposition}[Topological properties of orbits and foliations] \label{prop_top_properties}
 \begin{enumerate}[label=(\roman*)]
 \item  \label{item_no_contractible_orbit} No periodic orbit of $\flow$ is null-homotopic. 
 \item \label{item_generated_periodict} The fundamental group of any non-simply connected leaf of $\fs$ or $\fu$ is generated by a periodic orbit. 
 \item \label{item_unique_intersection} Any pair of leaves of $\hfs$ and $\hfu$ are either disjoint or intersect along a unique orbit.
 \item \label{item_properly_embedded} In $\wt M$, nonsingular leaves of $\hfs$ and $\hfu$ are properly embedded planes, connected components of the complement of singular orbits in singular leaves are properly embedded half-planes, 
 and orbits of $\hflow$ are properly embedded lines.  
 \end{enumerate}
\end{proposition} 

 \begin{proof}
Suppose $\gamma$ is a periodic orbit of $\flow$. 
On the stable leaf of $\gamma$, each other orbit must, by definition, accumulate onto it in the future. Thus, one may push $\gamma$ off itself along its stable leaf to obtain a closed curve $c$ in a leaf of $\fs$ transverse to the flow direction, as in Figure \ref{fig_making_transversal}, left.  
Such a curve is necessarily transverse to $\fu$, which, by Proposition \ref{prop_closed_transversal}, implies that $\gamma$ is not null-homotopic.

For the second item, we use the description of the topology of leaves of $\fs$ and $\fu$ given by Lemma \ref{lem_no_compact_leaves}.  The singular case follows immediately from Lemma \ref{lem_no_compact_leaves} since singular leaves always contain a periodic orbit.
Now suppose that $l$ is a nonsingular leaf with nontrivial fundamental group. By Lemma \ref{lem_no_compact_leaves}, $l$ is either an annulus or a M\"obius band, and it is foliated by orbits. Moreover, by Lemma \ref{lem_trivial_foliated_lift} the foliation of $l$ by orbits is trivial in the universal cover of $l$. 
 In the case of a M\"obius band, this directly implies that there is a closed orbit, since orbits are oriented, and the generator of the fundamental group is orientation reversing on the leaf space of the orbit foliation. 
 Consider now the case where $l$ is an annulus. Since the foliation is trivial in $\wt l$, either there is a periodic orbit, in which case we are done, or all orbits on this annulus go from one end (in the topological sense) of it to the other.  If $\gamma$ is a loop representing a core of the annulus and transverse to $\flow$, then it is transverse also to $\fu$.  The forwards-asymptotic property of orbits implies that, as we apply $\flow$ to $\gamma$, it becomes smaller and will eventually be contained in a ball smaller than the injectivity radius of $M$, so is contractible in $M$, a contradiction with Proposition \ref{prop_closed_transversal}.

  \begin{figure}[h]
   \begin{center}
    \labellist
\small \hair 2pt
\pinlabel {$\gamma$} at 96 45.434
\pinlabel {$\rightarrow$} at 108 45.434
\pinlabel {{\color[HTML]{00A448}$c$}} at 122 45.434
\pinlabel {{\color[HTML]{0000FF}$l^s$}} at 360 170
\pinlabel {$\tau^s$} at 350 123.434
\pinlabel {$\tau^u$} at 362.230 87
\pinlabel {{\color[HTML]{FF0000}$l^u$}} at 407.230 88.434
\endlabellist
\includegraphics[width=10cm]{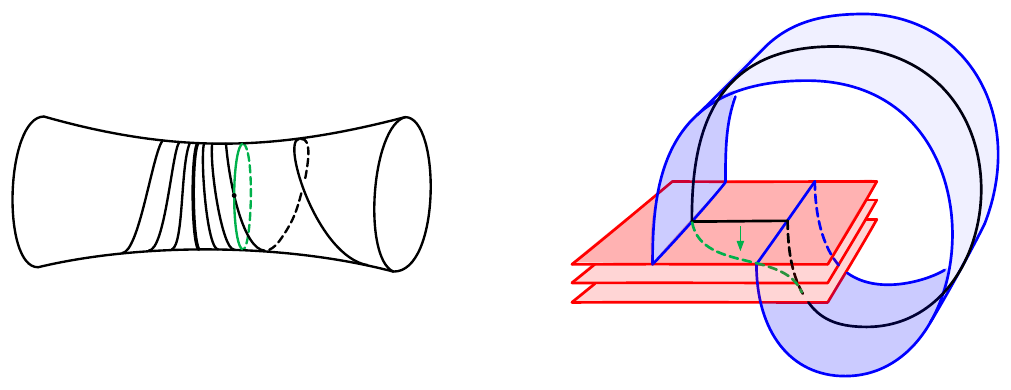}
   \end{center}
   \caption{Pushing a closed orbit, and a segment to produce closed transversals} \label{fig_making_transversal}
  \end{figure}
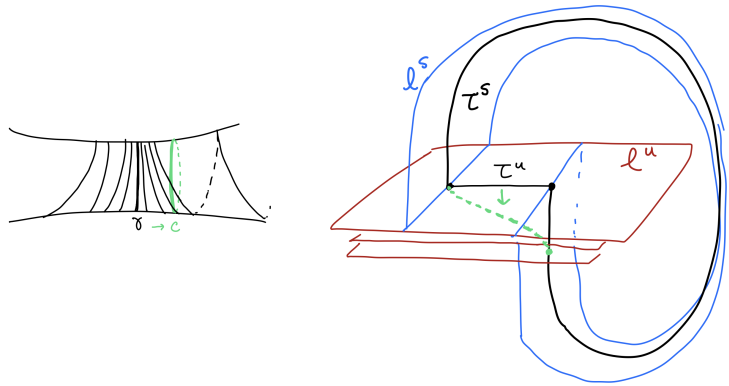

For the third item, 
suppose $l^s$ and $l^u$ are leaves of $\hfs$ and $\hfu$, respectively, which intersect along two distinct orbits $\alpha$, $\beta$.   
Consider a closed path $\tau$, transverse to the flow, made up of a segment $\tau^s$ from $\alpha$ to $\beta$ along  $l^s$ and a segment $\tau^u$ from $\beta$ to $\alpha$ along $l^u$.   Shortening $\tau^s$ (and modifying $\tau^u$ to replace with a different path along an unstable leaf, to replace $l^u$ if needed), we assume $\tau^s$ intersects $l^u$ only at its endpoints, and does not intersect any other leaf of $\hfu$ twice.  

Since $\tau^s$ is transverse to $\hfu$ it defines a local transverse orientation for the leaves of $\hfu$.  Since it intersects each leaf of $\hfu$ once (with the exception of the double intersection of the endpoints), even if the leaf $l^u$ is singular $\tau^s$ only intersects it along one face.  Thus, this local orientation makes sense, and implies that the two ends of $\tau^s$ intersect $l^u$ from opposite sides of $l^u$ (or the face of intersection).  This means that, in a local foliation chart near $\tau^u$, we can then push $\tau^u$ to make it transverse to $\hfu$, as shown in Figure \ref{fig_making_transversal}. This again contradicts Proposition \ref{prop_closed_transversal}.

 Finally, item \ref{item_generated_periodict} and our description of leaves in Lemma \ref{lem_no_compact_leaves} imply that a nonsingular leaf $\wt l$ is a topological plane and connected components of the complement of a singular orbit in a singular leaf are half-planes.  If one such leaf or half-leaf $\wt l$ of, say, $\hfs$ was not properly embedded, then it would accumulate on itself, and thus intersect some leaf of $\hfu$ along more than one orbit, which we just showed was impossible.  Finally, the fact that the orbits trivially foliate $\wt l$ shows that orbits are also properly embedded.  
 \end{proof}  
 
 \begin{remark}
 Notice that item \ref{item_unique_intersection} does not say that every leaf of $\hfs$ intersects every leaf of $\hfu$. In fact it follows from results proved in the next chapter (Theorem \ref{thm:trichotomy} and Proposition \ref{prop:trivial_affine}) that if \emph{one} leaf of $\hfs$ intersects every leaf of $\hfu$, then the flow is a suspension of an Anosov diffeomorphism on the torus. 
 \end{remark}

Now that we know leaves of $\hfs$ and $\hfu$ are planes, we wish to describe how they are arranged in $\wt M$. This is captured by the {\em leaf space}. 
 
\begin{definition}  \label{def_leaf_space}
If $\cF$ is a foliation of a manifold $X$, the 
{\em leaf space} of $\cF$, denoted $\Lambda(\cF)$, is the quotient space obtained by identifying each leaf to a single point, equipped with the quotient topology.  
\end{definition} 
In the case of $\leafs$ and $\leafu$, we call these the \emph{stable} and \emph{unstable leaf spaces} of $\flow$.

 \begin{proposition} \label{prop_leaf_spaces}
 If $\flow$ is Anosov, then the leaf spaces $\leafs$ and $\leafu$ are connected, simply connected, not necessarily Hausdorff, $1$-manifolds.   In general, the leaf spaces $\leafs$ and $\leafu$ are \emph{non-Hausdorff trees} in the sense of \cite{Fen03}.
 \end{proposition}
 Non-Hausdorff trees are called {\em order trees} in Gabai--Oertel's work \cite{GO89}.   An introduction to non-Hausdorff 1-manifolds can be found in \cite{HR57}, and in addition to the references above, non-Hausdorff trees are also studied in \cite{RS01}. 
  
 The idea of the proof of Proposition \ref{prop_leaf_spaces} in the Anosov case is simple: to show $\leafs$ is 1-dimensional, we need to find local charts to $\bR$.  Since leaves of $\hfu$ intersect leaves of $\hfs$ at most once by Proposition \ref{prop_top_properties} item \ref{item_unique_intersection}, given a leaf $l^s$ in $\hfs$, 
 we may fix some $l^u$ in $\hfu$ intersecting $l^s$ and use a small transverse segment $\tau$ in $l^u$ as a local chart.  
 To show that it is simply connected, one uses the fact that closed transversals are not nullhomotopic. See, for instance, \cite[Appendix D.1]{CC03} which gives a more detailed proof of a more general result of Palmeira \cite{Pal78}.  The singular case is treated in Gabai-Oertel \cite{GO89} or Fenley \cite{Fen03}.

To describe the non-Hausdorff points of leaf spaces, we introduce the following terminology.  
 \begin{definition}[Nonseparated and branching leaves] \label{def_non_separated}
 Let $\cF$ be a foliation.  Two distinct leaves $l, l'$ of $\cF$ are \emph{nonseparated} if for any open neighborhoods $U$ and $U' \subset \Lambda(\cF)$ of $l$ and $l'$ respectively, we have $U \cap U' \neq \emptyset$.    
A leaf $l$ is called \emph{branching} if there exists a leaf $l'$ distinct from $l$, such that $l$ and $l'$ are nonseparated.
 \end{definition}
 
 As a simple example, consider the universal cover of the {\em Reeb foliation} of the torus, that is, the foliation with a single closed orbit and a Reeb component in the complement.   Topologically, this can be described as the foliation of $\bR^2$ whose leaves are curves of the form $y = \tan^2(x) + c$ together with the vertical leaves $x = k\pi/2$ for $c \in \bR$ and $k$ an odd integer. See Figure \ref{fig_tan_square}.
 Here, adjacent vertical leaves are nonseparated.  Note that this example also shows that ``nonseparated with" is not a transitive relation.   
Taking a product with $\bR$ gives a higher dimensional version of this example, where nonseparated leaves occur in a foliation of $\R^3$ by planes.  
\begin{figure}[h]
\includegraphics[width=6cm]{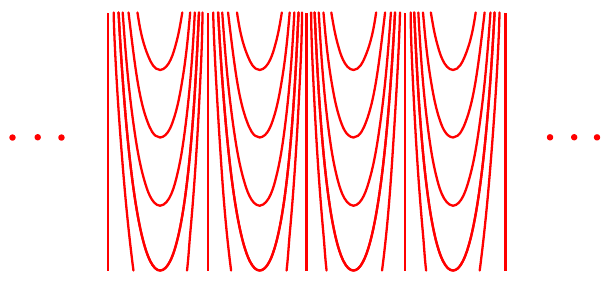}
\caption{The universal cover of a Reeb foliation of the torus.}\label{fig_tan_square}
\end{figure}

It will be useful to have the following sequential description of branching leaves. %
\begin{proposition}  \label{prop_sequential_definition}
Let $\cF$ be a foliation of $\bR^2$.  The following are equivalent  
\begin{enumerate} 
\item $l$ and $l'$ are nonseparated in $\Lambda(\cF)$,
\item There exists a sequence of points  $x_n$ converging to a point $x$ on $l$ and $x'_n$ on the same leaf as $x_n$, such that $x_n'$ converges to a point $x'$ on $l'$.  Furthermore, if $y$ is any other point on $l$, then one can find also a sequence $y_n$ on the same leaf as $x_n$ converging to $y$.  
\end{enumerate} 
\end{proposition} 
The proof is an exercise in the definition of foliations and the quotient topology. 
 
\begin{proposition}\label{prop_nonseparated_and_intersection} 
If $l, l'$ are nonseparated in $\hfs$ (resp.~$\hfu$) then no leaf of $\hfu$ (resp.~$\hfs$) intersects both.
\end{proposition}

\begin{proof}
Assume that $l,l'$ are nonseparated leaves of $\hfs$ but that there exists a leaf $f$ of $\hfu$ intersecting both. Call $\gamma$, resp.~$\gamma'$, the orbit obtained as the intersection of $f$ and $l$, resp.~$l'$. (The orbits are well-defined thanks to Proposition \ref{prop_top_properties}.)
By Lemma \ref{lem_trivial_foliated_lift}, the foliation of $f$ by orbits is trivial, so we can find a segment $\tau$ in $f$ from $\gamma$ to $\gamma'$ that is transverse to this foliation by orbits.  Now consider a sequence $l_n$ of leaves of $\hfs$ limiting to $l$ and $l'$. Then, for $n$ large enough, $l_n$ will intersect $\tau$ in two distinct points, but this contradicts the fact that $f\cap l_n$ is a single orbit.  
\end{proof}

 \subsection{The orbit space}
To describe the way the spaces of leaves of $\hfs$ and $\hfu$ interact in $\wt{M}$, we will use the {\em orbit space} (defined below), which reduces the 3-dimensional picture to a 2-dimensional one.  
We continue to use the notation from the previous section, using $\flow$ to denote a pseudo-Anosov flow, $\fs$ and $\fu$  its weak stable and unstable foliations, and we write $\hflow, \hfs, \hfu$ for their respective lifts to $\wt M$.  

 \begin{definition} 
The \emph{orbit space} of $\flow$, denoted $\orb$, is the quotient 
of $\wt M$ by the relation ``being on the same orbit of $\hflow$'', 
equipped with the quotient topology.
 \end{definition}

Since leaves of $\hfs$ and $\hfu$ are unions of orbits, the leaves of each of these foliations descend to give partitions of the quotient space $\orb$.  In fact, we will show that the situation is much nicer than this: 

 \begin{theorem} \label{thm_orb_is_plane}
$\orb$ is homeomorphic to $\bR^2$ and $\hfs, \hfu$ descend to a pair of transverse 1-dimensional foliations on $\orb$, with isolated prong singularities.  
\end{theorem}
This is proved by Barbot \cite{Bar95a} and Fenley \cite{Fen94} in the case of Anosov flows (where there are no prong singularities), and Fenley--Mosher \cite{FM01} in the pseudo-Anosov case.  
 
 \begin{proof} 
 First, we show that $\orb$ is a 2-manifold, i.e., has local charts to $\R^2$.   
 Given an orbit $o$ of $\hflow$, let $U$ be a disc transverse to $\hflow$, intersecting $o$, and contained in a flow-box neighborhood of a point in $o$ so that its intersection with $\hfs$ and $\hfu$ gives a trivial product foliation if $o$ is nonsingular, or a standard prong if $o$ is singular.  
  Provided that no orbit of $\hflow$ intersects $U$ twice, we can use projection to $U$ to give a local chart from a neighborhood of $o$ to $\R^2$. 
  
 To show that no orbit intersects twice, we argue by contradiction. Suppose some orbit $o'$ intersects $U$ at two points. Since singular orbits are isolated, up to shrinking the discs we consider we can assume without loss of generality that $o'$ is nonsingular. 
First, we argue that the two intersections of $o'$ in $U$ are inside a trivially foliated region of $U$. This is obvious if $o$ is nonsingular so $U$ itself is trivially foliated.  Otherwise, either $o'\subset \hfs(o)$, or $o'$ is in exactly one of the connected components of $\wt M\smallsetminus \hfs(o)$. In the latter case, we deduce that all the intersections of $o'$ with $U$ are in the same connected component of $U\smallsetminus \hfs(o)$, so we can take a trivially foliated rectangle $U'$ in $U$ that contains two hits of $o'$. In the first case, $o'$ must be in exactly one of the connected components of $\wt M\smallsetminus \hfu(o)$, and the same argument applies.

Since there are no closed orbits of $\hflow$, at least one of the local stable or local unstable manifolds of the two intersection points of $o'$ with $U'$ differ.  For concreteness, suppose the unstable intersection points differ.  Fix a transverse orientation for the leaves of $\hfu$ in $U'$, %
and use this to coherently 
push $o'$ off its unstable leaf and obtain a transverse segment to $\hfu$. This can then be closed up within $U'$ to a closed transversal to $\hfu$, contradicting Proposition \ref{prop_closed_transversal}.

Since $\orb$ is the quotient of $\wt M$ by properly embedded lines, it is connected, and the local transversal charts described above give it the structure of a locally trivial fibration, so it is simply connected.  
With these properties, to conclude $\orb$ is homeomorphic to $\bR^2$, it remains only to show that $\orb$ is Hausdorff and not $S^2$. 

Let $o$ and $o'$ be distinct orbits.  We will find some leaf $l$ of either $\hfs$ or $\hfu$ which separates them, in the sense that $o$ and $o'$ lie in distinct connected components of $\wt M \smallsetminus l$.  Since leaves are unions of orbits, this implies that $o$ and $o'$ are in disjoint open sets in the quotient $\orb$.  

First, observe that if $o$ and $o'$ lie on a common (say, stable) leaf $l$, then the fact that the orbit foliation of $l$ is trivial (Lemma \ref{lem_trivial_foliated_lift}), implies also that $o$ and $o'$ are separated in $\orb$ by the image of any unstable leaf intersecting $l$ between $o$ and $o'$.  
Next, if $\hfs(o)$ and $\hfs(o')$ are distinct but separated in the space of stable leaves, any stable leaf separating them will separate $o$ and $o'$.  If they are instead distinct and nonseparated, we claim that $\hfu(o)$ and $\hfu(o')$ are separated:  if not, a leaf sufficiently close to $\hfs(o)$ and $\hfs(o')$ would intersect both $\hfu(o)$ and $\hfu(o')$, contradicting Proposition \ref{prop_nonseparated_and_intersection}.    Thus, $o$ and $o'$ can be separated by a leaf of $\hfu$.   This is what we needed to show.  %

Finally, we argue that the projections of $\hfs, \hfu$ define singular foliations of $\orb$.  For this, we can simply use the local charts constructed at the beginning of this proof.  Each chart maps $\hfs$ and $\hfu$ leaves to local vertical/horizontal foliations of $\R^2$, or to a local picture of a model $p$-prong about a singular point.   This defines the desired foliation atlas.  
Thus, we have shown that $\orb$ is a Hausdorff, connected, simply connected 2-manifold, with two transverse foliations, possibly with isolated prong singularities.  From this it follows that $\orb$ is $\bR^2$, since $\chi(S^2) >0$, so it does not admit a nonsingular foliation, nor a foliation with prongs (which contribute negative Euler characteristic).  
 \end{proof}
 
 Since we showed earlier that orbits of $\hflow$ are properly embedded lines, 
Theorem \ref{thm_orb_is_plane} has the following consequence:
 \begin{corollary}  \label{cor_aspherical}
 If $M$ admits a pseudo-Anosov flow, then $\wt M \cong \bR^3$. 
 \end{corollary} 
In the case of Anosov flows, this follows also from a result of Palmeira \cite{Pal78}, which states that any simply connected manifold $M^n$ foliated by planes of dimension $n-1$ is necessarily homeomorphic to $\bR^{n}$. In the case of pseudo-Anosov flows, it also follows from \cite{GO89}, since the singular foliations of pseudo-Anosov flows induce essential laminations of the manifold.

In Chapter \ref{chap_examples_orbit_spaces}, we will give some concrete descriptions of the foliations that occur in the orbit spaces of the various flows described in Section \ref{sec_flows_examples}.  To do this accurately, we need some of the structure theory developed for these examples of {\em bifoliated planes} from Chapter \ref{chap_bifoliated_planes}. However, a reader who would like some visual images to keep in mind may wish to flip through Chapter \ref{chap_examples_orbit_spaces}, or at least its figures, before returning here.

\subsection{The action of $\pi_1(M)$ on the orbit space }

Thanks to Theorem \ref{thm_orb_is_plane}, we now have the main character of this text:  The orbit space $\orb$ is a topological plane equipped with two transverse 1-dimensional foliations, which are the projections of $\hfs$ and $\hfu$; we
denote these projected foliations by $\bfs$ and $\bfu$.   By construction, the projection gives an identification of $\hfs$ leaves with $\bfs$ leaves which descends to a (tautological) homeomorphism of leaf spaces.  Thus, we have
$\Lambda(\bfs) = \Lambda(\hfs)$ and $\Lambda(\bfu) = \Lambda(\hfu)$, and we will often pass back and forth between these spaces.    

The action of $\pi_1(M)$ on $\wt M$ by deck transformations preserves the lifted foliations and orbits of the flow, so descends to an action by homeomorphisms of $\orb$ preserving the foliations $\bfs$ and $\bfu$.  
We call this the \emph{action induced by $\flow$}.  

Observe that this picture is preserved by orbit equivalence, namely: 
\begin{observation}
Suppose $f\colon M \to M$ is an orbit equivalence between $\flow$ and $\psi$, and let $f_\ast$ denote its induced action on $\pi_1(M)$.  Then $f$ induces a map $\bar{f}\colon \mathcal{O}_\flow \to \mathcal{O}_\psi$ sending foliations to foliations and
such that, for each $g\in \pi_1(M)$ 
we have $\bar f(g \cdot o) = f_\ast(g) \cdot \bar f(o)$ for all $o \in \mathcal{O}_\flow$.
\end{observation} 

The following fundamental result of Barbot, based on an extension of arguments of Ghys in \cite{Ghy92},  gives the converse to this statement.  Thus, understanding the induced action of $\pi_1(M)$ on $\orb$ is enough to understand \emph{everything} about a flow, up to orbit equivalence.   

\begin{theorem}[Barbot \cite{Bar95a}, Théorème 3.4]\label{thm_action_determines_OEflow} 
Let $\flow$ and $\psi$ be  pseudo-Anosov flows on $M$.  Suppose there exists an isomorphism $f_\ast \colon \pi_1(M) \rightarrow \pi_1(M)$ and a homeomorphism  $\bar f \colon \orb \rightarrow \cO_{\psi} $ that is $f_\ast$-equivariant, i.e., 
 \[ \bar f(g \cdot o) = f_\ast(g) \cdot \bar f(o)\] 
 for all $o \in \mathcal{O}_\flow$ and $g\in \pi_1(M)$.  Then $\flow$ and $\psi$ are orbit equivalent by a homeomorphism $f$ inducing $f_\ast$ and $\bar f$.  
\end{theorem} 

\begin{remark} 
In fact, the assumption that $\flow$ and $\psi$ are supported on the same manifold $M$ is not necessary.  By Corollary \ref{cor_aspherical}, any 3-manifold supporting a pseudo-Anosov flow is aspherical, hence a $K(\pi, 1)$ space.  Thus, the existence of an isomorphism between fundamental groups {\em implies} the manifolds are homeomorphic.
\end{remark} 

In the hypotheses of Theorem \ref{thm_action_determines_OEflow} we also did not assume that $\bar f$ sends foliations to foliations, but this is in fact a consequence of $\pi_1(M)$ equivariance.  
We show this below, using some elementary observations about the dynamics of the action of $\pi_1(M)$ on leaves which we demonstrate at the beginning of the next section. They do not depend on Theorem \ref{thm_action_determines_OEflow} so we import them here as a black box for the sake of thematic continuity.   

\begin{proposition} \label{prop_conjugacy_preserves_foliations}
If $\psi$ and $\flow$ are pseudo-Anosov flows, and  $\bar f \colon \orb \rightarrow \cO_{\psi} $ is a homeomorphism equivariant with respect to the actions of $\pi_1(M)$ on $\orb$ and $\cO_{\psi}$ respectively, then $\bar f$ sends leaves of the stable foliation of $\flow$ to leaves of either the stable or unstable foliation of $\psi$.    
\end{proposition} 
 
\begin{proof} 
Let $\bfs_\flow$ and $\bfs_\psi$ denote the stable foliations for $\flow$ and $\psi$ respectively, and $\bfu_\flow$, $\bfu_\psi$ the unstable foliations.   Proposition \ref{prop_density} (taken here as a black box) shows that the set of leaves of $\fs$ that are fixed by nontrivial elements of $\pi_1(M)$ forms a dense subset of $\orb$.   A leaf fixed by a nontrivial element $g$ can be characterized (via Proposition \ref{prop_hyperbolic}) as a connected subset of $\orb$ containing a single fixed point for $g$ which is either a topological attractor or repeller (depending on the orientation of $g$), in the sense that it contains a unique fixed point $x$ and is equal to the set of points $y$ such that $g^n(y) \to x$ as $n \to \infty$ (in the attractor case) or as $n \to -\infty$ (for a repeller).  

If actions on $\orb$ and $\cO_\psi$ are conjugate by a map $\bar f$ (up to applying some automorphism $f_\ast$ of $\pi_1(M)$)  then $\bar f$ sends points fixed by $g$ to those fixed by $f_\ast(g)$, and preserves attractors and repellers.  
This shows that $\bar f$ sends a dense subset of leaves of $\bfs_\flow$ to leaves of either $\bfs_\psi$ or $\bfu_\psi$ (connectedness of the plane implies that one fixed leaf of $\bfs_\flow$ is sent to one of $\bfs_\psi$ if and only if they all are, otherwise they are all globally swapped).  By continuous extension, we conclude $\bar f$ sends foliations to foliations, as claimed. 
\end{proof} 

The dynamics of the action of $\pi_1(M)$ on $\orb$ also allows one to detect whether the orbit equivalence between $\flow$ and $\psi$ is direction preserving or not: $\bar f$ sends $\bfs_\flow$ to $\bfs_\psi$ when $f$ preserves the orientation of orbits, and to $\bfu_\psi$ when $f$ reverses direction.  
Theorem \ref{thm_action_determines_OEflow} is in part why we choose to consider undirected orbit equivalence.  
Forgetting the labels $\bfs$ and $\bfu$ of the two foliations on $\orb$, one obtains the topological picture of a bifoliated plane with an action of the group $\pi_1(M)$.  Barbot's theorem says this data is enough to recover the flow up to unoriented orbit equivalence. 

The proof given in \cite{Bar95a} treats only the Anosov flow case, but in fact applies equally well to the pseudo-Anosov setting. 
To show this, and for the sake of completeness, we include an outline.   Since it uses a sophisticated result on classifying spaces of foliations, we also give an alternative low-tech argument (for the Anosov case) which does not require any machinery.  

\begin{proof}[Sketch of Barbot's proof of Theorem \ref{thm_action_determines_OEflow}]
This proof uses the theory of classifying spaces for the transverse holonomy groupoids of foliations, in the sense of Haefliger \cite{Hae84}.  The foliations we consider are the 1-dimensional foliations of $M$ by orbits of $\flow$ and $\psi$.  

Associated to any foliation is a topological groupoid called the {\em transverse holonomy} groupoid. One definition of this groupoid is 
as the space of triples $(x, \tau, y)$ where $x$ and $y \in M$ lie in a common leaf (in our case, orbit) and $\tau$ is an orbit-segment between them, equipped with the quotient relation where $(x, \tau, y)$ is considered equivalent to $(x, \tau', y)$ if the germs of the holonomy along the paths $\tau$ and $\tau'$ agree.  Composition in the groupoid is via concatenation where defined.  (Equivalent, alternative definitions, such as that obtained by holonomy of a global section, and further discussion can be found in \cite{Hae84}.)   In the case of 1-dimensional foliations, an element $(x, \tau, y)$ is equivalent to some other only when there exists a closed orbit with finite holonomy group.  Thus, in the case of pseudo-Anosov flows, there are no identifications.  This means that $M$ as a foliated space is itself a {\em classifying space} for the orbit foliation of any pseudo-Anosov flow it supports.

Now, suppose that $\flow$ and $\psi$ satisfy the hypotheses of Theorem \ref{thm_action_determines_OEflow}.  The equivariant homeomorphism $\bar{f}$ of orbit spaces can easily be shown to induce an equivalence of transverse holonomy groupoids for the orbit foliations of $\flow$ and $\psi$.  The uniqueness of classifying spaces then implies that the pairs $(M, \flow)$ and $(M, \psi)$ are related in the following sense:  There exists a continuous map $h\colon M \to M$, which is a homotopy equivalence inducing $f_\ast$ on $\pi_1(M)$,  and sending orbits of $\flow$ to orbits of $\psi$ bijectively.  

To improve the map $h$ to an orbit equivalence, one needs to modify it to be injective along orbits.  For this one can use an averaging trick.  
Lift $h$ to a continuous map $\wt h\colon \wt{M} \to \wt{M}$.  Since stable leaves are characterized as maximal sets of orbits that converge in the future, and unstable leaves similarly via convergence in the past (see Remark \ref{rem_forward_converge_implies_local_stable}), it follows that $\wt h$ sends stable leaves for $\hflow$ to either stable or unstable leaves (globally) for $\wt \psi$.  As a further consequence, for sufficiently large $T$, the image of every orbit segment of the form 
$\{ \hflow^t(x) : t \in [0, T]\}$ under $\wt h$ has distinct endpoints.  One then defines $f\colon M \to M$ by ``averaging" $\wt h$ along orbit segments of length $T$, then projecting to $M$.  Since orbits are $C^1$ curves, one can differentiate this map along the orbits and check by hand that it has nowhere zero derivative, hence is injective. 
For further details, see \cite[p.257]{Bar95a}. 
\end{proof}

For Anosov flows, one can give an alternative proof of Barbot's theorem by showing that any (topological) Anosov flow is orbit equivalent to a ``model'' flow that can be described only using the induced action.  This is done in \cite[Section 2]{BFM25}.  For the sake of simplicity, here we only treat the case of smooth, \emph{transversally orientable} Anosov flows, i.e., flows on orientable $3$-manifolds such that both their stable and unstable foliations admit transverse orientations. Equivalently, those are the flows such that the induced actions on the orbit space preserve a choice of orientation for both $\bfs$ and $\bfu$ in $\orb$.

\begin{theorem}\label{thm_induced_action_model_flow}
Let $\pi_1(M)\curvearrowright (\orb,\bfs,\bfu)$ be the induced action of a transversally orientable Anosov flow. Fix a coherent orientation on leaves of the foliation $\bfs$.
For all $x\in \orb$, let $\bfs_+(x)$ be the connected component of $\bfs(x)\smallsetminus \{x\}$ on the positive side of $x$ and define
\[
W^s_+ :=\{ (x,y) \in \orb \times \orb \mid y \in \bfs_+(x)\}.
\]
Let $\Psi$ be the $1$-dimensional foliation of $W^s_+$ where the leaf through a point $(x_0,y_0)\in W^s_+$ is the set $\{ (x_0,y) \in \orb \times \orb \mid y \in \bfs_+(x_0)\}$.
Then $W^s_+$ inherits a natural (coordinate-wise) action of $\pi_1(M)$ and,
\begin{enumerate}
\item $W^s_+/\pi_1(M)$ is a $3$-manifold homeomorphic %
 to $M$.
\item The foliation $\Psi$ of $W^s_+$ descends to a $1$-dimensional foliation $\psi$ of $W^s_+/\pi_1(M)$, and $\flow$ is orbit equivalent to any flow-parametrization of $\psi$.
\end{enumerate}
\end{theorem}
The same holds replacing $\fs$ with $\fu$, and defining the analogous set $W^u_+$.  

Using proposition \ref{prop_conjugacy_preserves_foliations}, Barbot's Theorem \ref{thm_action_determines_OEflow} follows directly in the case treated by Theorem \ref{thm_induced_action_model_flow}:  if $\flow$ and $\psi$ are Anosov flows with conjugate induced actions, then the conjugacy sends $\bfs_\flow$ to either $\bfs_\psi$ or $\bfu_\psi$.  The model $W^s_+/\pi_1(M)$ obtained using $\flow$ will thus be orbit equivalent to either the model $W^s_+/\pi_1(M)$ or $W^u_+/\pi_1(M)$ constructed for 
$\psi$, and hence the two flows are orbit equivalent to each other.  

\begin{proof}
The proof will use strong stable foliations for $\flow$ and an adapted metric, so applies to smooth Anosov flows, or any flow orbit equivalent to one that admits these.\footnote{This is in fact not a restriction: every topological Anosov flow is orbit equivalent to one with a strong stable foliation, see \cite[Proposition 5.3]{Pot25}.}

With this set-up, we define a map $h\colon W^s_+ \to \wt{M}$ as follows:  A point $(x,y) \in W^s_+$ specifies two orbits $x$ and $y$, on the same weak-stable leaf.  Let 
$h(x,y)$ be the point on the orbit $x$ such that the distance along $\wt \cF^{ss}$ to the orbit $y$ is exactly 1 unit.  
This specifies a unique point thanks to our choice of adapted metric, where strong stable leaves are uniformly contracted under the flow.  
It is easy to see that $h$ is bijective and continuous, with continuous inverse, and thus a homeomorphism.  

We now show that $h$ is $\pi_1(M)$--equivariant, where the action of $\pi_1(M)$ on $W^s_+$ is the (diagonal) action induced from the orbit space.  Given $\gamma \in \pi_1(M)$ and $(x,y) \in W^s_+$, by definition $h(\gamma(x), \gamma(y))$ is the point on the orbit of $\gamma(x)$ whose distance along the (lifted) strong stable foliation is 1 from the orbit $\gamma(y)$.  Since deck transformations act by isometries on $\wt{M}$ and preserve foliations, this is simply the image under $\gamma$ of the point on orbit $x$ distance $1$ from the orbit $y$, in other words equal to $\gamma( h(x,y))$.  

Thus, $h$ descends to a homeomorphism $\bar{h}\colon W^s_+/\pi_1(M) \to M$.  The constant first-coordinate foliation on $W^s_+$ is invariant under $h$ and its image under $\bar{h}$ is exactly the foliation by orbits of $\flow$.  This completes the proof.  
\end{proof}

\section{Dynamics of the action of $\pi_1(M)$ on $\orb$}\label{sec_induced_action}
Thanks to Theorem \ref{thm_action_determines_OEflow}, one should at least in theory be able to recover all topological properties of a pseudo-Anosov flow from the dynamics of the induced action of $\pi_1(M)$ on the orbit space.   In this section we describe a few of the essential properties of this action, which we will later use to construct an explicit dictionary between features of the orbit space and topological properties of the flow. 

The first important property of the action is on the level of individual elements.  

\begin{proposition} \label{prop_closed_orbit_fixed_point}
Let $\gamma$ be a nontrivial element of $\pi_1(M)$.  The induced action of $\gamma$ on $\orb$ has a fixed point if and only if $\flow$ has a periodic orbit in the free homotopy class of $\gamma$. 
\end{proposition} 

\begin{proof} 
First observe that the action of $\gamma$ fixes an orbit $\alpha \in \orb$ if and only if the deck transformation $\gamma$ of $\wt M$ preserves the orbit $\alpha$ of $\hflow$. Suppose this is the case.  Since deck transformations act by isometries of $\wt M$ and $\flow$ is nonsingular, the action of $\gamma$ translates along $\alpha$.  
 Up to replacing $\gamma$ with a conjugate, we may consider the basepoint of $\pi_1(M)$ to be on $\alpha$; so $\alpha$ descends to a closed curve in $M$ (which is by definition an orbit of $\flow$) that represents $\gamma$ in $\pi_1(M)$.  

Conversely, if there is some periodic orbit in the free homotopy class of $\gamma$, each one of its lifts to $\wt M$ is an orbit of $\hflow$ invariant by some conjugate of $\gamma$, depending on the lift.  Thus, it is fixed by this conjugate in $\orb$.  Since the property of acting with a fixed point is a conjugacy invariant, $\gamma$ acts on $\orb$ with a fixed point as well. 
\end{proof} 

To illustrate the power of the orbit space perspective, we use Proposition \ref{prop_closed_orbit_fixed_point} to now prove the main theorem of \cite{Fen97}.  This result was also independently proved in \cite[Th\'eor\`eme F]{Bar95b}. 

\begin{corollary}[Closed orbits are primitive]\label{cor:closed_orbits_primitive}
Let $\gamma \in \pi_1(M)$ and suppose $\flow$ has a periodic orbit freely homotopic to $\gamma^n$ for some $n$.  Then $\flow$ has a closed orbit freely homotopic to $\gamma$.
\end{corollary}

\begin{proof} 
We prove the contrapositive: Suppose the free homotopy class of $\gamma$ does not contain a closed orbit of $\flow$.  By Proposition \ref{prop_closed_orbit_fixed_point}, $\gamma$ acts freely (without fixed points) on $\orb \cong \R^2$.  By the classical Brouwer plane translation theorem \cite{Bro12} (see \cite{Fra92} for a proof), $\gamma^n$ also acts freely for all $n\neq 0$. Thus, $\gamma^n$ does not represent the free homotopy class of a periodic orbit.  
\end{proof} 

We next describe the dynamics of an element of $\pi_1(M)$ around a fixed point. Recall a homeomorphism of a topological space is called {\em topologically contracting} if there is a point $x \in X$ such that $g^n(y) \to x$ as $n \to \infty$ for all $y \in X$, and {\em topologically expanding} if $g^{-1}$ is topologically contracting. In particular, $x$ is the unique fixed point for the action of $g$.  

\begin{proposition}\label{prop_hyperbolic}
Let $g\in \pi_1(M)$, $g\neq \id$. If $g$ fixes a leaf of $\bfs$ or $\bfu$, then it fixes a unique orbit $\alpha$ in that leaf.  
If the orientation of $g$ agrees with the orientation on $\alpha$ given by $\flow$, then $g$ is topologically expanding\footnote{This is not a typo: When the direction of $g$ matches that of $\alpha$, then the action of $g$ is \emph{expanding} on the stable leaf and contracting on the unstable.} on $\bfs(\alpha)$ and topologically contracting on $\bfu(\alpha)$.
\end{proposition} 

\begin{proof} 
Suppose $g$ fixes a leaf of $\bfs$; then it also fixes the corresponding leaf $\wt l$ of $\hfs$.  Thus, calling $l$ the projection of $\wt l$, we have that $\pi_1(l)$ is nontrivial, so by Proposition \ref{prop_top_properties}, it is generated by a periodic orbit in $l$.  Thus, $g$ has a fixed point. 
By Lemma \ref{lem_no_compact_leaves}, $l$ contains at most one periodic orbit, so $g$ has a unique fixed orbit on $l$.  The same applies for leaves of $\bfu$.  It remains to show that this action is contracting or expanding, as determined by the orientation.  

Suppose $\alpha$ is an orbit of $\hflow$ invariant under $g$, and $g$ translates along $\alpha$ in the direction of the flow, and
call $T>0$ the period of $\alpha$ for $\flow$.
First suppose that $\alpha$ is nonsingular.  
Consider a flow box $B$ containing a segment of $\alpha$ centered about a point $x$.  
Let $\beta_1, \beta_2$ be the two orbits bounding this flow box in $\hfs(x)$.  
By Lemma \ref{lem_all_converge}, 
the distance from $\beta_i$ to $\hflow^{nT}(x)$ goes to $0$ as $n \to \infty$.  Since $g$ translates along $\alpha$ by isometries, for sufficiently large $n$, we have that $\beta_i$ is in the interior of $g^n(B)$.  Translating this to the orbit space, we conclude that $g^{n}$ {\em expands} the interval of orbits between $\beta_1$ and $\beta_2$ corresponding to $B \cap \hfs(x)$.   Since we already know that $g$ has a unique fixed point in $\hfs(x)$, this shows it is topologically expanding.  
The singular case follows directly from the definition of the local model prong.   %
\end{proof}

In order to establish more properties of the induced action, we need the classical \emph{Anosov closing lemma}, generalized to the pseudo-Anosov setting: 
\begin{proposition}[Pseudo-Anosov closing lemma]\label{prop_pAclosing_lemma}
Let $\flow$ be a pseudo-Anosov flow on $M$.
For every $\epsilon>0$, and every $x\in M$, there exists $\delta=\delta(x,\epsilon)>0$ such that, if $d(\flow^T(x),x)<\delta$ for some large enough $|T|$, then there exists a \emph{closed} orbit $\beta$ of $\flow$ such that the orbit segment between $x$ and $\flow^T(x)$ 
is Hausdorff distance at most $\epsilon$ from $\beta$.  
Moreover, if $B$ is a compact set that does not intersect any singular orbits, then $\delta$ can be chosen uniformly on $B$.
\end{proposition}
\begin{rem}\label{rem_pA_closing_dependancy_delta}
As we will see in the proof, one can in fact be more precise about the dependency of $\delta(x,\epsilon)$ on $x$: One can take $\delta(x,\epsilon) = \min\{\delta_0(\epsilon), d(x)\} $ where $\delta_0(\epsilon)$ depends only on the flow and $\epsilon$, and $d(x)$ depends on the flow and the distance between $x$ and the closest singular orbit of the flow.
\end{rem}

In particular, when the flow has no singular orbits (i.e., is Anosov), then $\delta$ can be chosen uniformly on $M$.

Note that in the presence of singularities, the dependence of $\delta$ on $x$ is indeed necessary, as one can sometimes exhibit a segment of orbit that starts and ends very close to a singularity, but on different sides of it, and are not shadowed by any periodic orbit. This behavior happens for instance in the examples of \cite[Propositions 8.8 and 8.9]{BBM24b}.

The closing lemma is due to Anosov \cite{Ano63} for smooth Anosov flows, and is proved in greater generality (i.e., for hyperbolic sets that can be a proper subset of the ambient manifold) in many standard references, such as %
 \cite{FH19}, usually as a consequence of the shadowing property. For smooth pseudo-Anosov flows, the closing lemma was proved in \cite{Man98}, albeit with the stated (but unnecessary for that result) additional condition that the flow be quasi-geodesic. While no different in spirit to its smooth counterpart, it is difficult to find a statement and proof for topological pseudo-Anosov flows in the literature. 
In keeping with the spirit of this text, we present a proof in the orbit space.  

The proof uses critically the following lemma.  

\begin{lemma} \label{lem_bound_distance}
Let $\hflow$ be the lift of a pseudo-Anosov flow on a compact manifold $M$ to $\wt M$.  
For any $\epsilon>0$, there exist $\delta>0$ and $T>0$ such that if $x, y \in \wt M$ satisfy $d(x,y) < \delta$ and $d(\hflow^s(x), \hflow^r(y)) < \delta$ for some $s, r >T$, then there exists a standard reparameterization $\sigma$ such that $d(\hflow^t(x), \hflow^{\sigma(t)}(y))< \epsilon$ for all $t \in [0, s]$.  %
\end{lemma}
For smooth Anosov flows, the proof is a straightforward exercise: If $x, y$ lie on the same orbit, this is trivial. If not we may assume without loss of generality (up to reversing time) that they lie on different stable leaves.  The Anosov property implies that the distance between $\hflow^t(x)$ and the closest point on the orbit of $y$ is bounded below by $c \lambda^t$ for some $c>0, \lambda>1$, from which the Lemma easily follows.  

The pseudo-Anosov or topological case requires more work to get a substitute, coarser version of this uniform divergence.  This can be done by taking a section of the flow and considering the return maps, as in Lemma 4 of \cite{Iak24}.  See also the proof of \cite[Proposition 6]{Iak24} which essentially gives the closing lemma in this context.

\begin{proof}[Proof of proposition \ref{prop_pAclosing_lemma}]

First, if $x$ is on a periodic orbit $\alpha$, choose $\delta$ small enough so that the $\delta$-neighborhood is embedded in $M$.  Then, if $d(\flow^t(x), x) < \delta$ for any large $t$, the segment from $x$ to $\flow^t(x)$ will consist of a nearly integral number of loops around $\alpha$, hence we can use this power of $\alpha$ as the nearby periodic orbit.  
In particular, this applies to any point on a singular orbit of $\flow$, so we now fix some $x\in M$ that is not on a singular orbit. 
Let $\eps$ be given and let $\delta(\eps)$ be the constant from Lemma \ref{lem_bound_distance}.  
Consider a small topological rectangle $R$ of diameter at most $\delta(\eps)$, containing $x$, transverse to $\flow$, with its sides tangent to stable and unstable leaves and not intersecting any singular orbits.  %

Let $\delta$ be such that any point $y$ with $d(x, y)< \delta$ has a nearby point $\flow^s(y) \in R$, for some $s$ very small. (Note that the dependence of $\delta$ on $x$ appears here: the closer $x$ is to a singular orbit, the smaller $\delta$ will have to be.)
Suppose that there exists $T>0$ such that $d(x, \flow^T(x))<\delta$, but $d(x, \flow^s(x))>\delta$ for some $0<s<T$.  
In particular, up to slightly adjusting $T$ we can assume that $\flow^T(x) \in R$.   Let $c$ be the closed curve obtained by following the orbit from $x$ to $\flow^T(x)$ and then closing it via a path in $R$. Let $g\in \pi_1(M)$ be an element freely homotopic to $c$ and take lifts $\wt R\subset \wt M$ and $\wt x \in\wt R$ so that the lifted orbit through $\wt x$ intersects $g\wt R$.

If $T$ is large enough, the forward-asymptotic property of orbits on leaves of $\hfs$ implies that the forward orbit through any point in $\wt R\cap \hfs(\wt x)$ will intersect $g\wt R$. Similarly, the backward orbit through any point in $g\wt R\cap \hfu(\wt\flow^T(\wt x))$ will intersect $\wt R$.

Let $R_{\cO}$ denote the projection of $\wt R$ to $\orb$; so $g R_{\cO}$ is the projection of $g \wt R$. 
Let $\alpha$ be the orbit through $\wt x$.
By definition, $R_{\cO}$ is a trivially foliated set in $\orb$ bounded by four segments along leaves of $\bfs$ and $\bfu$. 
Since $\alpha$ meets both $R_{\cO}$ and $g R_{\cO}$, these have nonempty intersection.  
The fact that all points in $\wt R \cap \hfs(\wt x)$ flow forward to intersect $g\wt R$ implies that the rectangle $gR_{\cO}$ contains the segment $\bfs(\wt x)\cap R_{\cO}$.  The set of leaves of $\bfu$ which intersect $gR_{\cO}$ forms an interval in the leaf space $\Lambda(\bfu)$, which we just argued {\em contains} the interval of leaves that intersect $R_{\cO}$.  Thus, $g^{-1}$ acts as a contraction on this interval, so by the Brouwer fixed point theorem it has a fixed leaf in the interval.  
Similarly, we may find a fixed leaf for $g$ among the set of $\bfs$ leaves meeting $R_{\cO}$, on which $g$ (rather than $g^{-1}$) acts as a contraction.  See Figure \ref{fig_closing}.

The intersection of these two fixed leaves, necessarily in $R_{\cO}$, gives a fixed orbit $\beta$ for $g$.  By construction, the projection of $\beta$ to $M$ ``shadows" the orbit segment between $x$ and $\flow^T(x)$.  By Lemma \ref{lem_bound_distance}, 
the distance between $\beta$ and this segment is controlled by the size of $R$. %

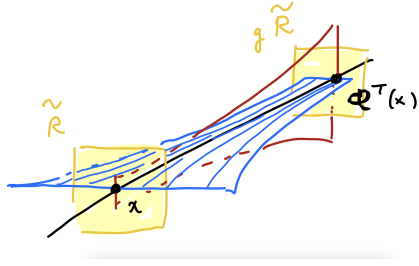
\begin{figure}[h]
\centering{
\labellist
\small \hair 2pt
\pinlabel {$\widetilde R$} at 57 80
\pinlabel {$g\widetilde R$} at 200 145
\pinlabel {$x$} at 100.932 22.540
\pinlabel {$\varphi^T(x)$} at 280.932 95.540
\endlabellist
\includegraphics[width=6cm]{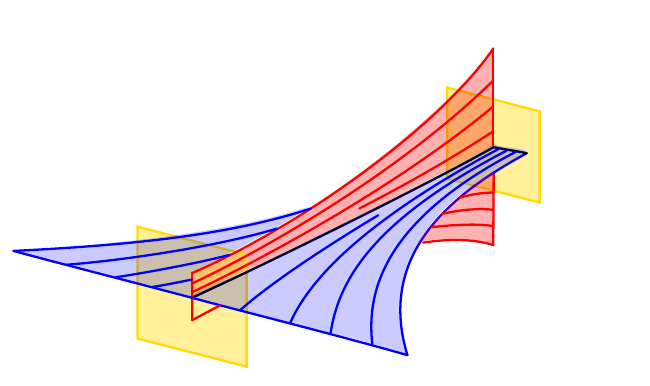}}
\caption{Divergence and convergence of orbits in the proof of the closing lemma} \label{fig_closing}
\end{figure}

\end{proof}

Translating the closing lemma to a statement in the language of the orbit space gives the following.  
\begin{proposition} \label{prop_closing_translation}
Let $\flow$ be a pseudo-Anosov flow.  
Each non-singular point $x \in \orb$ has a neighborhood basis $U_i$ in $\orb$, with the property that for each $U_i$ there is a smaller neighborhood $V_i \subset U_i$ such that, if $g(V_i) \cap V_i \neq \emptyset$, then $g$ has a fixed point in $U_i$.

Similarly, each singular point $p\in \orb$ has a pair of neighborhood bases $V_i \subset U_i$ in $\orb$, such that if $R$ is any connected component of $V_i\smallsetminus (\bfs \cup \bfu)(p)$ and $gR \cap R \neq \emptyset$, then $g$ either fixes $p$, or has a fixed point inside the connected component of $U_i\smallsetminus (\bfs \cup \bfu)(p)$ containing $R$ or adjacent to it.  
\end{proposition}

The proof for nonsingular points is a direct translation of Proposition \ref{prop_pAclosing_lemma} to the language of the orbit space.  For singular points, one needs to use the observation of Remark \ref{rem_pA_closing_dependancy_delta}: to construct neighborhoods $V_i$, in each quadrant $Q$ of a singular orbit $p$, take a point $x_Q$ that is much closer to the singular orbit than $\delta_0$ (as in Remark \ref{rem_pA_closing_dependancy_delta}), and use a rectangular cross-section neighborhood $R_Q$ of $x_Q$ with $p$ on the boundary contained in the quadrant $Q$.  Now run the argument of Proposition \ref{prop_pAclosing_lemma} to find a larger interval of $\bfs$-leaves mapped into itself by flowing forward following an orbit segment, and an interval of $\bfu$ mapped into itself by flowing backward.  The intersection of fixed leaves will lie in the quadrant $Q$ or an adjacent one.   We leave the remaining details as an exercise.

Using the closing lemma, we can now prove the following.  
\begin{proposition} \label{prop_density}
The union of leaves in $\bfs$ (resp.~$\bfu$) with nontrivial stabilizer is dense in $\orb$.
 Moreover, if the flow $\flow$ is transitive, then the union of \emph{points} in $\orb$ fixed by nontrivial elements of $\pi_1(M)$ is dense. 
\end{proposition}

\begin{remark} We give an independent proof of this result below, but note that for smooth Anosov flows, it is a direct consequence of Smale's {\em spectral decomposition theorem} \cite[Theorem 5.2]{Sma67}  which implies in particular that having a dense set of periodic orbits is in fact {\em equivalent} to being transitive -- see Theorem \ref{thm_characterize_transitive_flows}.  
In the topological and pseudo-Anosov setting this equivalence is shown in \cite{BBM24b}, and proved here in Theorem \ref{thm_characterization_transitive}.  
\end{remark} 

\begin{proof}[Proof of Proposition \ref{prop_density}]
 Let $\tilde x$ be any point in $\wt M$.   Let $x$ denote its projection to $M$ and $\bar{x}$ its projection to $\orb$.  Fix a neighborhood $R$ of $\bar{x}$ in $\orb$, and let $\tilde{R}$ be the full preimage of $R$ in $\wt M$.  
We will find a leaf of $\bfs$ with nontrivial stabilizer that intersects $R$.  A symmetric argument can be used to show that some leaf of $\bfu$ with nontrivial stabilizer also intersects $R$.  

First, if $x$ is on a periodic leaf (for instance, on a singular orbit) there is nothing to show, so we assume this is not the case.  Thus, we can take $R$ to be a rectangular, trivially foliated neighborhood defined by an interval in each of the leaf spaces.  
Since $M$ is compact, the forward orbit $\{\flow^t(x):  t>0\}$ has accumulation points. 
Assume as a first case that no accumulation point is on a singular orbit.  Thus, we can find a compact set $B$ containing the forward orbit of $x$, and bounded away from the union of singular orbits in $M$.  Fix $\epsilon$ small enough so that the $\epsilon$-ball about $\wt x$ is contained in $\tilde{R}$, and choose $\delta$ as in the pseudo-Anosov closing lemma, uniform on $B$.  
 
Because $\{\flow^t(x):  t>0\}$ has accumulation points, we can find arbitrarily large $s>0$ and $T >0$ such that $d(\flow^{s+T}(x), \flow^{s}(x)) < \delta$.  Thus, by the pseudo-Anosov closing lemma, there exists a periodic orbit $\beta$ which 
$\epsilon$-shadows the orbit segment between $\flow^{s}(x)$ and $\flow^{s+T}(x)$, so has stable and unstable leaves meeting an $\epsilon$-ball around $\flow^{s}(x)$.   
Since orbits {\em diverge} along unstable leaves, $\tilde{R}$ will meet the stable leaf of any point passing through the $\epsilon$-neighborhood of $\flow^{s}(x)$, provided $s$ is chosen sufficiently large.  This gives the desired conclusion.  

If instead $\{\flow^t(x):  t>0\}$ has an accumulation point $y$ on a singular orbit $\alpha$, a similar strategy applies but one does not need to apply the closing lemma since $\alpha$ is already periodic: one can find arbitrarily large $t$ such that the local unstable leaf of some $\flow^{t}(x)$ intersects the local stable leaf of $y$, arbitrarily close to $y$. Hence, one can find a lift of $\alpha$ whose stable leaf projects to meet $R$ in $\orb$.

It remains only to prove the last assertion in the statement of the proposition. Assume that $\flow$ is transitive. Consider a point $x\in M$  whose forward-orbit under $\flow$ is dense\footnote{It is classical that the existence of a dense orbit is equivalent to the existence of a forward-dense (and backwards-dense) orbit, see, e.g., \cite[Proposition 1.6.9]{FH19}.}. 
Take any small $\eps$-ball $U = B_\eps(z)$ in $M$; we will find a periodic orbit passing through $U$.  If there is already a singular orbit through $U$, we are done.   By density of $\flow^t(x)$, there is some point $y = \flow^t(x)$ in the smaller ball $B_{\eps/2}(z)$. Choose a constant $\delta$ for the pseudo-Anosov closing Lemma so that $d(\flow^T(y),y) < \delta$ implies the existence of a periodic orbit that $\eps/2$-shadows the segment from $y$ to $\flow^T(y)$.  
This gives the existence of a periodic orbit that passes through $U$.  Since $U$ was arbitrary, we conclude the  set of periodic orbits is dense in $M$.  Proposition \ref{prop_closed_orbit_fixed_point} now gives the desired conclusion. 
\end{proof}

The properties established thus far were all, in effect, the translations of classical facts about hyperbolic dynamical systems to the orbit space of a flow.   By contrast, the next result is a comparatively recent major discovery of Fenley \cite{Fen98}.

\begin{theorem}[Fenley \cite{Fen98}, Theorem B]\label{thm_branching_are_periodic}
Each branching leaf of $\bfs$ or $\bfu$ has nontrivial stabilizer.
\end{theorem}

\begin{proof}
We do the proof for a branching leaf in $\bfs$, the other case being symmetric.
Let $l$ be a branching leaf, and let $l'$ be a leaf nonseparated with $l$ in $\bfs$.
Since $l$ and $l'$ are nonseparated, we can find some $l^s$ containing points $\alpha$ and $\beta$ of $\orb$ sufficiently close to points on $l$ and $l'$ so that $\bfu(\alpha) \cap l \neq \emptyset$ and $\bfu(\beta) \cap l' \neq \emptyset$.  
Let $S \subset \Lambda(\bfu)$ be the set of leaves intersecting both $l$ and $l^s$.     

By Proposition \ref{prop_nonseparated_and_intersection} no leaf of $\bfu$ can intersect both $l$ and $l'$, so the $\bfu$ saturation of $l'$ is disjoint from $S$.  
Thus, there exists a unique leaf $l^u \in \partial S$ such that $l$ and $l'$ lie in different connected components of $\orb \setminus l^u$. See Figure \ref{fig_non_separated_setup}.  
We will first show that $l^u$ is fixed by some nontrivial element $g$, and then use its relationship with $l$ to show that a power of $g$ fixes $l$. 

\begin{figure}[h]
\centering{
\labellist
\small \hair 2pt
\pinlabel {{\color[HTML]{0000FF}$l$}} at 21 120
\pinlabel {{\color[HTML]{0000FF}$l'$}} at 293 125
\pinlabel {{\color[HTML]{0000FF}$l^s$}} at 293 60.580
\pinlabel {{\color[HTML]{FF0000}$l^u$}} at 120 90
\pinlabel {$\alpha$} at 68 70
\pinlabel {$\beta$} at 247 70
\pinlabel {$l^u\cap l^s$} at 132 49
\endlabellist
\includegraphics[width=8cm]{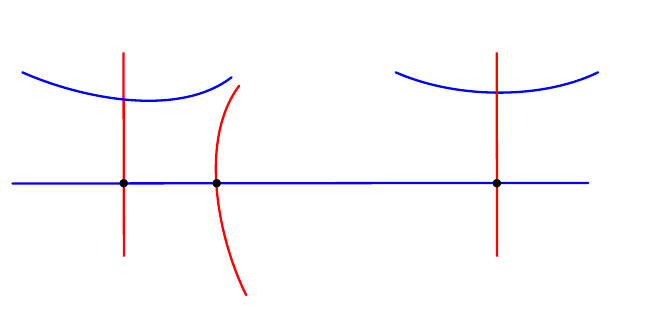}}
\caption{The set-up for the proof of Theorem \ref{thm_branching_are_periodic}} \label{fig_non_separated_setup}
\end{figure}
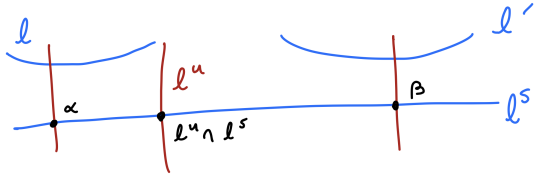

To show $l^u$ is fixed, we need to leave $\orb$ to work in $\wt M$.   Fix a point $p \in l^u \cap l^s$ in $\wt M$.  
Let $\delta >0$ be the constant given by property \ref{item_PAF_backwards_expansivity} of Definition \ref{def_topPA} (for some fixed choice of $\eps$).  It follows from this property, together with compactness of $M$ and the fact that orbits are properly embedded lines in $\wt M$, that there exists a sequence $t_n \to - \infty$ so that no point on the orbit $\beta$ nor on the orbit $\alpha$ of $\hflow$ lies distance less than $\delta$ from $\hflow^{t_n}(p)$. Note that while this is obvious for smooth (pseudo)-Anosov flows, or flows admitting strong foliations, proving it in full generality requires more work, but also follows readily from the existence of Markov partitions as in \cite{Iak24}. 

Consider the sequence $\{ \hflow^{t_n}(p): n \in \bN\}$.  Since $M$ is compact, after passing to a subsequence and reindexing, the projections of these points to $M$ converge to a limit.  
Passing to a further subsequence if needed, we can assume that they also converge in a finite cover $\hat M$ of $M$ where $\pi_1(\hat{M})$ acts on $\orb$ preserving orientation. 
Let $p_\infty$ be a lift of this limit to $\wt M$, thus there exist $g_n \in \pi_1(M)$ preserving orientation, such that 
\[ p_n := g_n \wt \flow^{t_n}(p) \] 
converges to $p_\infty$.    We can pass to a subsequence so that all $p_n$ lie in a single quadrant of $p_\infty$ (which will be needed if $p_\infty$ is singular), and up to a further subsequence, we can assume that $g_n g_m^{-1}$ all preserve local orientations of $\hfs$ and $\hfu$ in this quadrant.

If some points $p_n$ and $p_m$, for $m \neq n$ lie on the same leaf of $\hfu$, then this leaf is necessarily periodic, stabilized by $g_n g_m^{-1}$.   Since the $\hfu$ leaf containing $p_n$ is the image of $l^u$ under $g_n$, we conclude that $l^u$ is periodic, which is what we needed to show.  
So, we will assume that the unstable leaves of $p_n$ and $p_m$ are all distinct.  From this we will derive a contradiction.  

Let $B$ be a small flow box neighborhood of $p_\infty$ for $\hflow$, of diameter at most $\delta/2$, 
and choose $n, m$ large so that $p_n = g_n \hflow^{t_n}(p) \in B$ and similarly $p_m \in B$.   Since $g_n$ acts by isometries, we know that no point on the orbit $g_n \beta$ or $g_n \alpha$ lies in $B$; similarly for $g_m \beta$ and $g_m \alpha$. 

Up to relabeling $n, m$ we assume that $p_m$ is on the side of $\hfu(p_n)$ containing $g_n\beta$.  Since $g_m \alpha \cap B = \emptyset$ and $\hfu(g_m \alpha) \cap g_m(l) \neq \emptyset$, it follows that $g_m(l) \cap \hfu(p_n) \neq \emptyset$.  See figure \ref{fig_non_separated}. 
Since $g_n \beta \cap B = \emptyset$, it follows that $g_n(l')$ is not contained in the region between $\hfu(p_n)$ and $\hfu(p_m)$.  Thus, $g_m(l)$ %
separates $g_n(l)$ from $g_n(l')$.  This contradicts the nonseparation of $l$ and $l'$.  

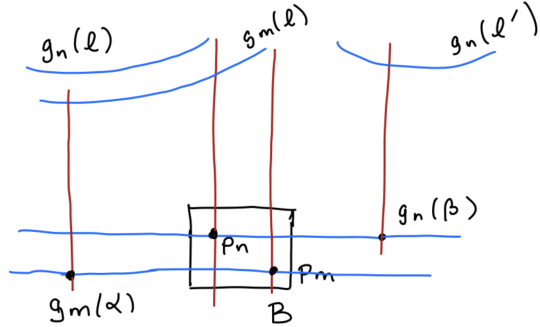
\begin{figure}[h]
\centering{
\labellist
\small \hair 2pt
\pinlabel {$g_n(l)$} at 63 152
\pinlabel {$g_m(l)$} at 50 119
\pinlabel {$g_n(l')$} at 270 180
\pinlabel {$g_n(\beta)$} at 255 80
\pinlabel {$g_m(\alpha)$} at 30 34
\pinlabel {$p_n$} at 124.029 59.460
\pinlabel {$p_m$} at 166.029 43.460
\pinlabel {$B$} at 150 19.460
\endlabellist
\includegraphics[width=8cm]{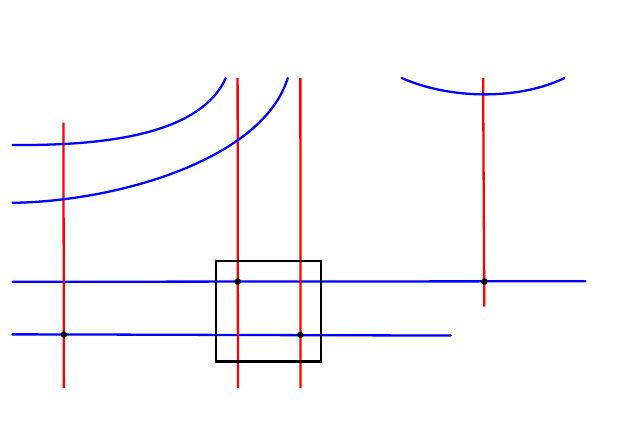}}
\caption{Configuration of images of leaves under $g_n$ and $g_m$.  This picture indicates (with a 2-dimensional cartoon) the configuration of leaves near $B$ in $\wt M$ that contradicts the fact that $l$ and $l'$ are nonseparated.} \label{fig_non_separated}
\end{figure}
This gives the desired contradiction and we conclude that $l^u$ is fixed by some element $g \in \pi_1(M)$.

Now we need to show that a power of $g$ fixes $l$.  
Up to replacing $g$ with a power, we can assume that $g$ fixes each half-leaf of $\bfu$ and $\bfs$ through its unique fixed point $z$ on $l^u$.  
Note that the configuration of $l$ and $l^u$ is symmetric in the sense that $l^u$ is a boundary leaf of the unstable-saturation of $l$.  
In fact, it can be characterized as the {\em unique} boundary leaf with the property that it has non-empty intersection with each stable leaf through any sufficiently small transversal to $\bfs$ based along the ray from $z$ towards $l$. 
This uniqueness means that if $g$ fixes $l^u$ and all its rays, it must fix $l$ as well.   (This configuration is called a {\em perfect fit} and will be formalized and studied in the next chapter -- see Lemma \ref{lem:perfect_fit_unique}.)
This concludes the proof. 
\end{proof}

Having established some essential properties of the action of $\pi_1(M)$ on $\orb$, our goal in the next chapter is to study an abstract framing of these: a set of axioms for groups acting on the plane which captures the key features of orbit space actions.  We then will use these axioms, along with arguments in one- and two-dimensional topology and dynamics, to develop much of the known structural theory of pseudo-Anosov flows in a more general context.  

With this in mind, we collect the relevant results (that we just proved!) in the following statement, for future reference.  

\begin{theorem}\label{thm_induced_action_satisfy_axioms}
The action on $\orb$ induced by a pseudo-Anosov flow $\flow$ satisfies the following properties:
\begin{enumerate}[label=(\roman*)]
\item \label{item_induced_action_hyperbolic} Let $g\in \pi_1(M)$, $g\neq \id$. If $g$ fixes a leaf $l$ of $\bfs$ or $\bfu$, then it fixes a unique point $x\in l$, and $g$ is topologically expanding on the leaf $\bfs(x)$ and contracting on $\bfu(x)$, or vice versa. 
\item \label{item_induced_action_density} The union of leaves in $\bfs$ (resp.~$\bfu$) with nontrivial stabilizer is dense in $\orb$.
\item \label{item_induced_action_prong}  Each singular leaf of $\bfs,\bfu$ is fixed by some nontrivial element of $\pi_1(M)$.
\item \label{item_induced_action_nonseparated_are_periodic}  If $l$ is a branching leaf then  there exists a nontrivial element of $\pi_1(M)$ that fixes $l$.
\end{enumerate}
\end{theorem}
\begin{proof} 
Item \ref{item_induced_action_prong} is a direct consequence of the definition of singular leaf, which must necessarily contain a closed orbit, and Proposition \ref{prop_closed_orbit_fixed_point}.   The other items were established in Propositions \ref{prop_hyperbolic} (for item \ref{item_induced_action_hyperbolic}), \ref{prop_density} (for item \ref{item_induced_action_density}), and Theorem \ref{thm_branching_are_periodic} (for item \ref{item_induced_action_nonseparated_are_periodic}).  
\end{proof}

%% file: PlanesAndActions_vector.tex
In Chapter \ref{chap_pA_to_orbit} we saw that a pseudo-Anosov flow on a 3-manifold $M$ gives rise to an action of $\pi_1(M)$ on the orbit space, a plane with two topologically transverse, possibly singular foliations.  Moreover, 
by Theorem \ref{thm_action_determines_OEflow}, the topological picture of the plane with these foliations and action completely determines the flow.

In this chapter we develop a theory of general {\em Anosov-like group actions on bifoliated planes}.  
Definition \ref{def_action} axiomatizes the essential dynamical features exhibited by 3-manifold groups acting on orbit spaces of pseudo-Anosov flows.  Working from this abstract definition, we develop a general theory of such group actions using one- and two-dimensional dynamics.  Later, we translate our results back to the context of pseudo-Anosov flows. 

\section{Bifoliated planes and leaf spaces}  \label{sec_bifoliated_intro}
To begin, we introduce and recall some basic terminology around topologically transverse foliations.  

\begin{definition} 
A {\em bifoliated plane} is a topological plane $P$ with a pair of transverse, possibly singular, topological foliations $(\cF^+, \cF^-)$.  Singularities are required to be properly embedded $p$-prongs for $p \geq 3$; as in definition \ref{def_model_prong}. 
\end{definition}

Recall that, for a point $x \in P$, the leaves of $\cF^+$ and $\cF^-$ through $x$ are denoted by $\cF^+(x)$ and $\cF^-(x)$ respectively.  We also use the following standard terminology. 

\begin{definition}[Face, ray, quadrant] \label{def_face_ray_quadrant}
A {\em face} of a singular leaf $l$ is a properly embedded $\R$ that bounds a single connected component of $P \setminus l$; each $n$-prong singularity has $n$ distinct faces of each foliation.  
A {\em ray} of $\cF^\pm(x)$ is a properly embedded copy of $[0,\infty)$ based at $x$.
A {\em quadrant} of $x$ is a connected component of $P \setminus (\cF^+(x) \cup \cF^-(x))$.  Note that singular points have more than four quadrants.   
\end{definition} 

Recall from Definition \ref{def_leaf_space} that the {\em leaf space} of a foliation $\cF$ on a manifold $X$, denoted $\Lambda(\cF)$, is the space obtained from $X$ by identifying each leaf to a single point, equipped with the quotient topology. For a nonsingular foliation of the plane, this is a (possibly non-Hausdorff) connected, simply connected, 1-manifold.  See, e.g. Haefliger-Reeb \cite{HR57} for a proof.  
If $\cF$ has prong singularities, but no two on a common leaf, this space is no longer a manifold, but rather an {\em order tree} in the sense of Gabai--Oertel \cite{GO89} and  Gabai-Kazez \cite{GK97}, or a {\em non-Hausdorff tree} in the terminology of Fenley \cite{Fen03}.    

Recall also from Definition \ref{def_non_separated} that two leaves of a foliation $\cF$ are called nonseparated if they are not members of disjoint open subsets of $\Lambda(\cF)$, in which case we say each leaf is a \emph{branching leaf}.  Nonseparated leaves $l$ and $l'$ have the property that there exists a sequence of leaves $l_n$ eventually entering every open neighborhood of $l$ and of $l'$; we say in this case that $l$ and $l'$ are {\em limits} of the sequence. 

\begin{exercise} 
Exhibit an example of a bifoliated plane with a set of $n$ leaves in $\cF^+$ that are pairwise non-separated; for $n = 2, 3, \ldots$ as well as $n = \infty$.  Can there be an uncountably infinite such collection? 
\end{exercise}

\begin{definition} Two bifoliated planes are {\em isomorphic} if there is a homeomorphism between them which takes the  foliations of one to the foliations of the other.  An {\em automorphism} of a bifoliated plane is a self-isomorphism, i.e. a homeomorphism that preserves each of the foliations $\cF^+$ and $\cF^-$.  
\end{definition} 

One easy method to produce a large class of examples of (nonsingular) bifoliated planes is as follows.
\begin{example}  \label{example_bifoliated}
Let $U$ be an open, path connected, semi-locally simply connected subset of $\R^2$, equipped with the standard coordinate foliations.  Lift these foliations to the universal cover $\wt{U}$ of $U$.  The result is a topological plane, with a pair of transverse foliations.
\end{example}

\begin{exercise}   
Give examples of nonsingular bifoliated planes which are not isomorphic to those obtained in Example \ref{example_bifoliated}. An interesting, perhaps challenging, problem is to characterize all of the possible bifoliated planes that can be obtained using this construction.
\end{exercise}

\section{Anosov-like actions}
Here and going forward, when a group acts on a bifoliated plane, we always assume that the action is by automorphisms. 

\begin{definition}[Anosov-like action] \label{def_action}
An action of a group $G$ on a bifoliated plane $(P, \cF^+, \cF^-)$ is called \emph{Anosov-like} if it satisfies the following: 
\begin{enumerate}[label = (A\arabic*)]
	\item\label{Axiom_A1} 
	If a nontrivial element of $G$ fixes a leaf $l \in \cF^\pm$, then it has a fixed point $x \in l$ and is topologically expanding on one leaf through $x$ and topologically contracting on the other. 
	\item\label{Axiom_dense} The union of leaves of $\cF^{+}$ that are fixed by some nontrivial element of $G$ is dense in $P$, as is the union of leaves of $\cF^{-}$ that are fixed by some nontrivial element of $G$.  
	\item \label{Axiom_prongs_are_fixed} Each singular point is fixed by some nontrivial element of $G$.
	\item \label{Axiom_nonseparated} If $l$ is a branching leaf of $\cF^+$ or $\cF^-$, then some nontrivial element $g\in G$ fixes $l$.
	\end{enumerate}
\end{definition}
\begin{remark}
These axioms are not quite the same as those appearing in previous work in the literature.  The paper \cite{BFM22} considered only transitive flows, so had a longer more restrictive set of axioms which describe {\em transitive} Anosov-like actions.  The paper \cite{BBM24b} gives the list above, but with an additional fifth property restricting the topology of the foliations on the plane which we will not use here -- but see Remark \ref{rem_no_ideal_quad}.  
\end{remark} 

As an easy consequence of the definition, observe that Axioms \ref{Axiom_prongs_are_fixed} and \ref{Axiom_A1} together imply that there is at most one singularity on any given leaf.   Additionally, if $g \in G$ fixes some point $x$, then a finite power of $g$ necessarily preserves each of the rays through $x$, and hence preserves each quadrant of $x$.  Thus, Axiom \ref{Axiom_prongs_are_fixed} gives the existence of elements preserving each ray through a singular point, and Axiom \ref{Axiom_A1} combined with \ref{Axiom_nonseparated} gives elements fixing each ray through a point on a branching leaf.  

Admitting an Anosov-like action places strong constraints on the topology of the bifoliation. 
\begin{exercise} Give examples of bifoliated planes from the construction of Example  \ref{example_bifoliated} which do not admit any Anosov-like action of a group.   Can you find any examples which do?  
\end{exercise}

However, there are indeed many examples; the paradigm case coming from pseudo-Anosov flows. 
Restating Theorem \ref{thm_induced_action_satisfy_axioms} in the language of Anosov-like actions, we obtain:
\begin{theorem}
Let $M$ be a compact 3-manifold with pseudo-Anosov flow $\flow$.  Then the induced action of $\pi_1(M)$ on the orbit space $\orb$ is Anosov-like.  
\end{theorem} 

We will also see below a few special examples of Anosov-like group actions which do not arise from pseudo-Anosov flows, and in fact for which the group acting is not isomorphic to the fundamental group of a 3-manifold.   As another source of examples and non-examples, the definition of pseudo-Anosov flow can also be generalized to non-compact 3-manifolds, and many (but not all) induce Anosov-like actions on their orbit spaces.

\subsection{A first taste of relating dynamics to group structure} 

The existence of the invariant foliations (as well as their one-dimensional leaf spaces) allows one to use tools from one-dimensional dynamics to study Anosov-like actions.   The following result of H\"older, generalized by Solodov, is a classical result in this direction.  

\begin{theorem}[H\"older] \label{thm:holder}
Suppose $G$ is a group acting freely on $\bR$.  Then, there is a faithful {\em translation number homomorphism} semi-conjugating the action of $G$ to an action by translations.  In particular $G$ is abelian.  
\end{theorem} 

\begin{theorem}[Solodov]  \label{thm:solodov}
Suppose $G$ is a group of orientation-preserving homeomorphisms of the line with the property that each nontrivial element has at most one fixed point.  Either $G$ has a global fixed point and is abelian; or $G$ is isomorphic to a subgroup of $\mathrm{Aff}(\R)$ and the action is semi-conjugate to one by affine transformations. 

Moreover, if the action of $G$ is minimal, one can replace ``semi-conjugate" by {\em conjugate}.\footnote{Recall that an action of a group on a topological space is said to be {\em minimal} if every orbit is dense.   Two actions $\rho_1, \rho_2$ of $G$ on the line are {\em semi-conjugate} if there is a monotone increasing map $h\colon \R \to \R$ such that $\rho_1(g) \circ h = h \circ \rho_2(g)$ for all $g\in G$.  See \cite{Nav11} for further introduction to 1-dimensional dynamics. }
\end{theorem} 
See \cite[Section 2.2.4]{Nav11} for a proof and further discussion. 

Using Theorem \ref{thm:solodov}, we describe the structure of stabilizers of leaves under an Anosov-like action.   
\begin{lemma} \label{lem:common_fixed_point}
Suppose $G$ acts Anosov-like on a bifoliated plane $(P, \cF^+, \cF^-)$.  Let $l$ be a leaf of either $\cF^+$ or $\cF^-$.  
Then the stabilizer of $l$ in $G$ acts on $l$ with a global fixed point, and the finite-index subgroup of the stabilizer that preserves each ray 
 through that fixed point is abelian.  
\end{lemma}
We will later show that leaf stabilizers are not only virtually abelian but virtually cyclic.  See Proposition \ref{prop:cyclic_stabilizers}. 

\begin{proof} 
First, suppose that $l$ is a nonsingular leaf, and consider the subgroup $H$ of the stabilizer of $l$ which acts on $l$ preserving orientation.  By Axiom \ref{Axiom_A1}, each nontrivial element of this group acts with exactly one fixed point. By Theorem \ref{thm:solodov},  either the group acts with a global fixed point or is semi-conjugate to a group of affine transformations. In this latter case, in the absence of a global fixed point, there is a nontrivial subgroup acting by translations, which would contradict Axiom \ref{Axiom_A1}.   Thus, we conclude that $H$ acts with a global fixed point and is abelian.  Finally, since $H$ is of index at most two in the stabilizer of $l$, any element not in $H$ must fix this unique global fixed point as well. 

In the singular case, Axiom \ref{Axiom_A1} and Axiom \ref{Axiom_prongs_are_fixed} imply that the singular point is unique and fixed by every element of the stabilizer of $l$.  
Consider the finite index subgroup $H$ of the stabilizer that preserves each ray based at the singular point.  The restriction of the action to any such ray is a faithful, free action of $H$ on $\R$.  Theorem \ref{thm:holder} now says that $H$ is abelian. 
 \end{proof}
 
 We also have that point stabilizers can be generated by {\em primitive} elements of $G$: 
  \begin{proposition} \label{prop_primitive_stabilizer}
 Let $x \in P$.  If $g$ fixes all rays through $x$ and $h^n =g$ then $h(x)=x$.
 \end{proposition} 
 The proof is a direct application of the Brouwer plane translation theorem, exactly as in Corollary  \ref{cor:closed_orbits_primitive} (primitivity of closed orbits in their free homotopy class).

\section{Trivial and skew planes} \label{sec_trivial_skew}
We begin by classifying the Anosov-like actions on the simplest example of a bifoliated plane.   
\begin{definition} A bifoliated plane is {\em trivial } if it is isomorphic to $\R^2$ with the standard coordinate foliations.  
\end{definition} 
The trivial plane is the orbit space of a suspension Anosov flow as in Example \ref{ex_suspension_AF}.  
We next show that all Anosov-like actions on the trivial plane resemble this, in that they are conjugate to actions by affine transformations.  %

\begin{proposition}[Characterization of Anosov-like actions on trivial planes] \label{prop:trivial_affine} 
An action of a group $G$ on a trivially bifoliated plane is Anosov-like if and only if the following conditions are satisfied 
\begin{enumerate} 
\item The action on $P$ is conjugate to an action on $\bR^2$ by affine homeomorphisms of the form $(x,y) \mapsto (ax+b, cy+d)$ where $a \neq -1$, $c \neq -1$, $|a|>1$ if and only if $|c|<1$ and if $a=c=1$ then either $b=d=0$ or $b\neq 0$ and $d\neq 0$,
\item At least one map has $a \neq 1$, and 
\item The action on each coordinate factor is minimal.
\end{enumerate} 
\end{proposition}

Slightly weaker versions of this proposition are proved in \cite[Th\'eor\`eme 2.7]{Bar95a} and \cite{BFM22}.  

\begin{proof} 
Suppose $(P, \cF^+, \cF^-)$ is a trivial plane and $G$ acts Anosov-like on $P$.  We identify with the trivially foliated $\R^2$, so that $\cF^+$ and $\cF^-$ are the coordinate foliations.  Since $G$ preserves these foliations, each element acts on $\R^2$ by maps of the form $g(x, y) = (g_1(x), g_2(y))$.  

If, for some nontrivial $g \in G$, the first-coordinate map $g_1$ has two distinct fixed points, say $x_1$ and $x_2$, then $g$ fixes each of the vertical leaves $\{x_i\} \times \bR$. It follows that $g$ has a fixed point of the form $(x_i, y_i)$.  If $y_1=y_2$, then there are two fixed points on the horizontal leaf $\bR \times \{y_1\}$, otherwise both points $(x_1, y_1)$ and $(x_1, y_2)$ are fixed, again contradicting Axiom \ref{Axiom_A1}.  We conclude $g_1$ (and, by symmetry, $g_2$) each have at most one fixed point.  

Axiom \ref{Axiom_dense} now implies that $G$ does not have a global fixed point or globally fixed leaf.  Thus, by Theorem \ref{thm:solodov}, after restricting to the subgroup preserving orientation, the 
induced action on the first-coordinate is semi-conjugate to an action by affine transformations.   The fact that this semi-conjugacy is in fact a {\em conjugacy} also follows from Axiom \ref{Axiom_dense}: If the semi-conjugacy is not a conjugacy, the action is necessarily \emph{not} minimal, possessing an invariant closed set.  An interval complementary to this set cannot contain a repelling fixed point of an element, contradicting  Axiom \ref{Axiom_dense}.  
The same argument applies to the second coordinate action, giving the desired conjugacy (applied on each factor), for the index at most two subgroup preserving orientation.  It is now easy to see that $G$ itself is conjugate to a group of affine transformations.  %

Axiom \ref{Axiom_A1} is satisfied by maps of the form $(x,y) \mapsto (ax+b, cy+d)$ only when $|a| = |c| = 1$ and in that case either the element is the identity or translates both factors or 
$|a|>1$ and $|c|<1$, or  vice versa. 
To finish the proof of the converse statement, we note that Axioms \ref{Axiom_prongs_are_fixed} and \ref{Axiom_nonseparated} are vacuously satisfied on trivial planes, so any such action satisfying Axiom \ref{Axiom_dense} will be Anosov like.  This is ensured by having at least one element with a fixed point, along with minimality of the action on each coordinate factor, which ensures a dense set of leaves fixed by the conjugates of this element.  
\end{proof} 

As an application we prove the following theorem, originally shown in the case of Anosov flows by Solodov (in unpublished work) and Barbot \cite{Bar95a}, using work of Plante \cite{Pla81}. 

\begin{theorem}\label{thm_trivial_implies_suspension}
Let $\flow$ be a pseudo-Anosov flow on a compact, 3-manifold $M$ with $\orb$ the trivial plane.  Then $M$ is the mapping torus of a linear  hyperbolic map of the torus, and 
$\flow$ is orbit equivalent to the suspension flow. Furthermore, this orbit equivalence can be taken to be isotopic to the identity.
\end{theorem}

\begin{proof}
Let $G = \pi_1(M)$ and consider the action of $G$ on $\orb$, which is (after conjugacy) an affine action as described in Proposition \ref{prop:trivial_affine}.   The diagonal structure given by Proposition \ref{prop:trivial_affine} implies that the derivative homomorphism of $G$ has image in an abelian linear group.  Since $G$ is not a group of translations, nor does it have a global fixed point (both contradict the definition of Anosov-like), $G$ is not abelian and so the derivate homomorphism has nontrivial kernel.  Denote this kernel by $T$; it is the translation subgroup.  
Since $T$ acts freely on $\orb$, the pseudo-Anosov closing lemma (Proposition \ref{prop_closing_translation}) implies that $T$ must be a discrete group of translations.  Thus it is either isomorphic to $\bZ$ or $\bZ^2$.  We show next it is isomorphic to $\bZ^2$. 

By Proposition \ref{prop:trivial_affine} $G$ contains some element $A$ with nontrivial linear part.  Choose coordinates on $\orb$ so that the fixed point of $A$ is the origin, i.e. $A$ is linear and diagonal.  Since $T$ is normal and the conjugate by $A$ of translation by $t$ is translation by the vector $A(t)$, we have that $T \subset \bR^2$ is an $A$-invariant subgroup.  Since it is discrete and nontrivial and $A$ is diagonal, we conclude that it is isomorphic to $\bZ^2$ and $A$ is conjugate into $\mathrm{SL}(2, \bZ)$.  

Thus, the subgroup $G'$ generated by $T$ and $A$ is isomorphic to the fundamental group of the mapping torus of $A$.  Since this is a subgroup of $\pi_1(M)$ and $M$ is an aspherical, closed $3$-manifold, hence a $K(\pi, 1)$ space, we have that $G'$ is finite index in $G$.  
Since $G'$ contains the full translation subgroup of $G$, we conclude that $G$ is generated by $T$ and a root $A^{1/k}$ of $A$ in $\mathrm{SL}(2, \bZ)$. 
We conclude that $M$ is the mapping torus of $A^{1/k}$ and the orbit space action agrees with that of the suspension flow. By Barbot's theorem \ref{thm_action_determines_OEflow} this shows that $\flow$ is orbit equivalent to the suspension via a map isotopic to identity.  
\end{proof}

\begin{exercise} 
Using the classification above, give explicit examples of Anosov-like actions on trivial planes (in terms of generators of the group) that do not arise as orbit-space actions of Anosov flows on 3-manifolds\footnote{An example is given in \cite{BFM22}.}.    Can you find examples with non-cyclic point stabilizers?  What other pathological behavior can occur? 
\end{exercise}

The next simplest example of a bifoliated plane which admits an Anosov-like action is that called the {\em skew plane}.  It is the example that arises as the orbit space of geodesic flow on a surface, as we describe later in Example \ref{ex_geodesic_flow_orbit_space}.

\begin{definition}\label{def_skew} A bifoliated plane is {\em skew} if it is isomorphic to the diagonal strip in $\R^2$ bounded by the lines $y =x $ and $y = x+1$, equipped with the standard coordinate foliations.  
\end{definition} 

\begin{figure}[h]
\labellist
\small \hair 2pt
\pinlabel {$(0,0)$} at 202.827 65.784
\pinlabel {$(1,1)$} at 380 244.116
\endlabellist
\includegraphics[width=6cm]{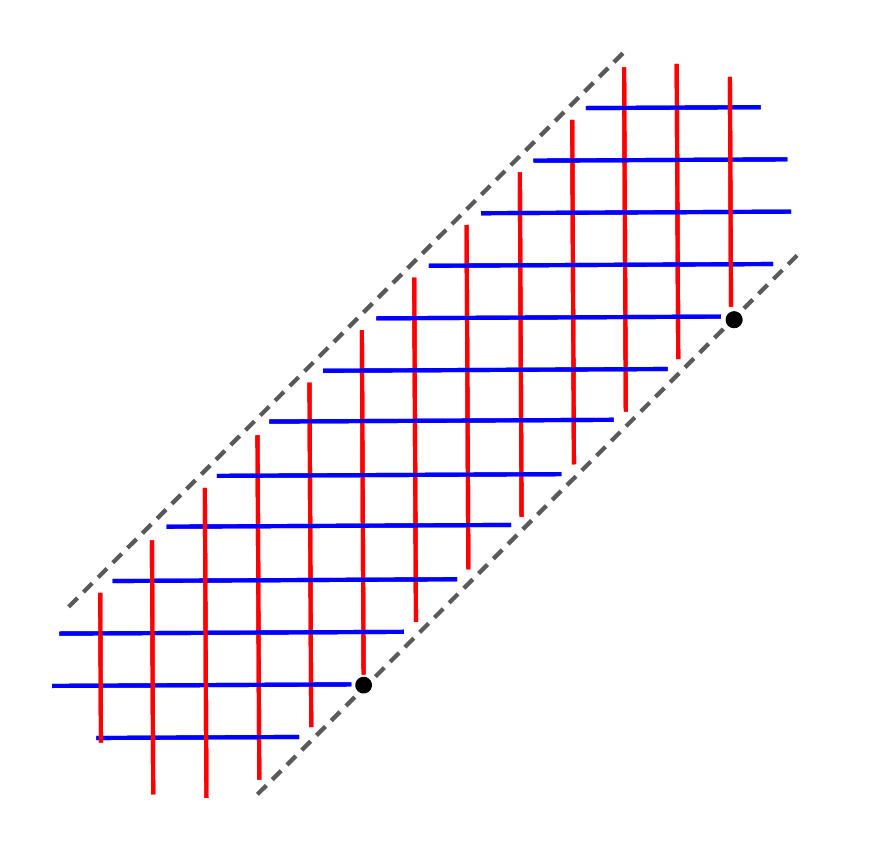}
\caption{The skew bifoliated plane}
\end{figure}

Considering the constant coordinate of each vertical (or horizontal) leaf gives a natural identification of each leaf space with 
 $\R$.  
 
We now describe the structure of automorphisms of the skew plane.  Consider, for concreteness, a leaf $y=c$ of the horizontal foliation.  The set of leaves of the vertical foliation which intersect $y=c$ forms an open interval, bounded by the leaves $x=c-1$ and $x=c$.  
If $g$ is an automorphism of the skew plane, then it preserves intersections.  Thus, the induced homeomorphism of the horizontal leaf space permutes the sets of leaves of the form $y = \Z + c$.   In the case where an automorphism $g$ preserves orientation on this leaf space, this implies that $g$ commutes with integer translations; if $g$ reverses orientation it will instead anti-commute.
Conversely, it is not hard to show that any group $G$ of homeomorphisms of $\R$ that commute or anti-commute with integer translation (depending on whether orientation is preserved or not) induces an action on the skew plane, where $f \in G \subset \Homeo(\R)$ acts on the skew plane by  $f \cdot (x,y) = (f(x), f(y))$ if $f$ is increasing and by $f \cdot (x,y) = (f(x), f(y)+1)$ if $f$ is decreasing.  

The fact that the translation $(x,y)\mapsto (x+1,y+1)$ is almost central in the group of automorphisms of the skew plane is a useful structural feature. For this reason it has a special name: 
\begin{definition}\label{def_stepup_half_stepup}
The map $\eta\colon (x,y)\mapsto (x+1,y+1)$ on the skew plane is called the \emph{one-step up map}. Notice that $\eta$ has a square root, $\eta_{1/2}$, defined by $(x,y)\mapsto (y,x+1)$; it is sometimes called the \emph{half-step up map}. \footnote{Confusingly, the half-step up map has also been called a one-step up map in some instances in the literature.}
\end{definition}

We can now give a characterization of Anosov-like actions on skew planes, in parallel with Proposition \ref{prop:trivial_affine}.  For this, we leave the proof as an exercise.  

\begin{exercise}[Characterization of Anosov-like actions on skew planes] \label{ex_charac_skew}
Show that $G$ has an Anosov-like action on a skew plane if and only if the induced action of $G$ on $\bR$ (via the canonical identification with either leaf space) has the following four properties: 
\begin{itemize} 
\item Each element commutes (if it preserves orientation) or anti-commutes (otherwise) with the map $x \mapsto x+1$,
\item Each orientation-preserving element either acts freely, or has exactly two fixed points in $[0,1)$, one attracting and one repelling
\item At least one element is not a translation, and
\item The action is minimal.
\end{itemize} 
\end{exercise}  

\begin{proof}[Hint:]
The first property comes from the structure of the plane, the second two from Anosov-like dynamics, and minimality corresponds to density of fixed leaves combined with the first 3 properties.
\end{proof}

\begin{rem}
There are many examples of Anosov-like actions on skew planes, both induced from Anosov flows and not. 
For instance, if $H$ is any subgroup of $\PSL(2, \R)$ without parabolic elements, then any lift of $H$ to the universal covering group $\wt{\PSL}(2, \R) \subset \Homeo_+(\R)$ that acts minimally will induce an Anosov-like action. When $H$ is discrete and cocompact, then this construction corresponds to the induced action of a geodesic flow (see Section \ref{ex_geodesic_flow_orbit_space}).

Lifts of the M\"obius-like groups of  Kova\v{c}evi\'c in \cite{Kov99} give examples which are {\em not} isomorphic to lifts of subgroups of $\PSL(2, \R)$.  Many other examples, which are more difficult to describe explicitly, come from Anosov flows on 3-manifolds obtained by surgery constructions.  
\end{rem} 

As a parallel to Theorem \ref{thm_trivial_implies_suspension}, 
the following theorem of Thurston
 characterizes which Anosov-like actions on skew planes come from Anosov flows. 

\begin{theorem}[Thurston \cite{Thu97}] \label{thm_extended_convergence}
Call an {\em extended convergence group} a subgroup of $\Homeo_+(\R)$ such that the induced diagonal action on the space 
$\{(x, y, z) \in \R^3 : x < y < z < x+1 \}$ is properly discontinuous and cocompact. 

Then the actions on the skew plane induced by extended convergence groups are precisely those which arise as orbit space actions (on skew planes) from Anosov flows on compact 3-manifolds.  
\end{theorem}

A picture showing the relationship between the space of triples 
\[ \mathcal{T}:= \{(x, y, z) \in \R^3 : x < y < z < x+1 \}\] and the orbit space of the flow is given in Section \ref{sec_slithering}.  

One direction of Theorem \ref{thm_extended_convergence} (at least for smooth Anosov flows) follows from the construction used in Theorem \ref{thm_induced_action_model_flow}, as follows:  if $\flow$ is a smooth Anosov flow with skew orbit space, then the finite index 
subgroup of $\pi_1(M^3)$ preserving leaf-space orientations acts properly discontinuously and cocompactly on the space $W^s_+ \subset \orb \times \orb$ defined in  Theorem \ref{thm_induced_action_model_flow}.  Supposing, without loss of generality that the stable foliation is vertical, then $W^s_+$ can be identified with the space of triples $\mathcal{T}$ defined above by associating $(x,y,z) \in \mathcal{T}$ to the pair $((x,y), (x,z)) \in W^s_+ \subset \orb \times \orb$.   
The converse statement is more involved; we do not discuss the proof here and refer to \cite{BFM25} instead. 

Finally, we note a very different characterization of Anosov flows whose orbit spaces are skew planes, due to Marty.  While contact structures are beyond the scope of this work, contact (and bi-contact) structures have a very close relationship with Anosov flows.  Marty's result characterizes the contact flows as precisely the skewed ones:   
\begin{theorem}[Marty \cite{Mar24}] 
An Anosov flow on a 3-manifold $M$ is orbit equivalent to the Reeb flow of some contact structure on $M$ if and only if its orbit space is the skew plane.  
\end{theorem}

\section{Perfect fits and lozenges} 
In this section $P$ always denotes a bifoliated plane, with foliations $\cF^+$ and $\cF^-$.  

\begin{definition} 
Let $X \subset P$ be a subset.  The $\cF^+$ {\em saturation} of $X$, denoted $\cF^+(X)$, is the union of all leaves $\cF^+(x)$, for $x \in X$.  Similarly, $\cF^- (X):= \bigcup_{x \in X} \cF^-(x)$.  
\end{definition}  
Note that, if $X \subset P$ is $g$-invariant for some $g \in G$, then $\cF^+(X)$ and $\cF^-(X)$ are also $g$-invariant, as are the boundaries of $\cF^+(X)$ and $\cF^-(X)$.  Here one can take ``boundary" to mean either the union of leaves in the boundary of $\cF^\pm(X)$ considered in the leaf space of $\cF^\pm$, or its topological boundary in $P$.\footnote{The leaf-space boundary and topological boundary can differ only in the case where a boundary leaf is singular, in which case one may see only some {\em faces} of the prong leaf in the topological boundary.}

In our discussion of automorphisms of the skew plane in the previous section, we used the fact that the saturation of a leaf $l$ was an interval, bounded on either side by a ``limit leaf" which did not intersect $l$.   The following definition generalizes this limit idea.  

 \begin{definition}[Perfect fit of leaves]
 Two leaves $l^+, l^-$, of $\cF^+$ and $\cF^-$ respectively, make a {\em perfect fit} if $l^+$ is contained in the boundary of the saturation $\cF^+(l^-)$, and symmetrically $l^-$ is in the boundary of $\cF^-(l^+)$.
 \end{definition} 

The following definition gives a version of perfect fit applicable to rays, rather than full leaves.  Note that two leaves make a perfect fit in the sense above if and only if they contain rays which make a perfect fit in the sense below.  

 \begin{definition}[Perfect fit of rays]
 Two rays $r^+$ and $r^-$ contained in leaves of $\cF^+$ and $\cF^-$, respectively, make a \emph{perfect fit} if there are arcs $\tau^+, \tau^-$ such that 
 \begin{enumerate} 
 \item $\tau^\pm$ is based at a point of $r^\mp$ and transverse to $\cF^\mp$, and 
 \item every leaf $l^\pm$ of $\cF^\pm$ that intersects the interior of ${\tau}^\mp$ intersects $r^\mp$.  
 \end{enumerate} 
\end{definition} 
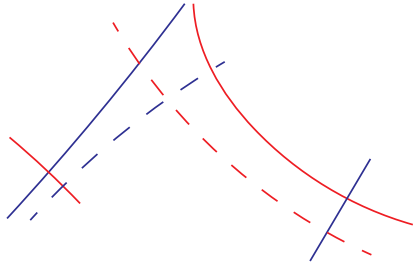
\begin{figure}[h]
\labellist
\small \hair 2pt
\pinlabel {$r^-$} at 112 140
\pinlabel {$r^+$} at 375 160
\pinlabel {$\tau^+$} at 59.680 24.646
\pinlabel {$\tau^-$} at 409.786 23.927
\endlabellist
\includegraphics[width=6cm]{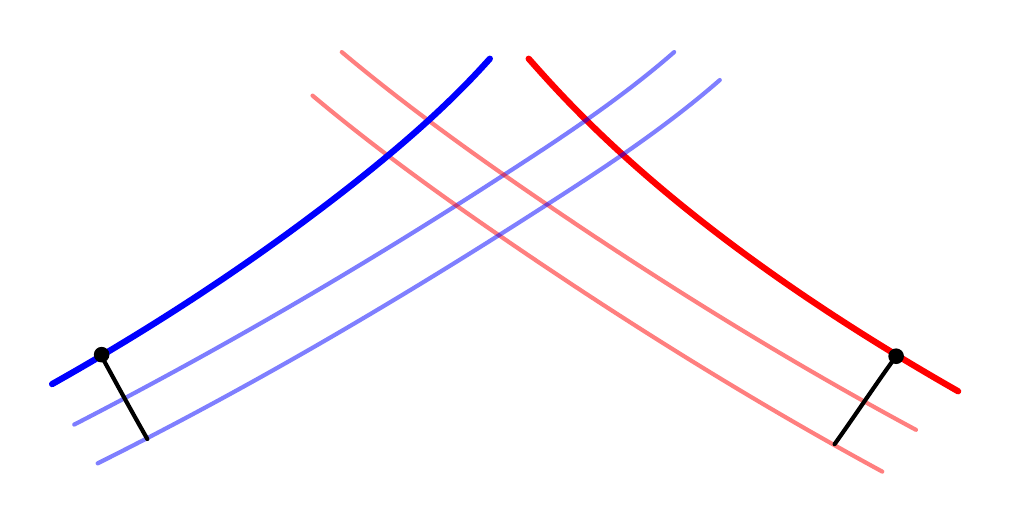}
\caption{A perfect fit of rays. %
}
\end{figure}

Note that if some transverse arcs $\tau^\pm$ satisfy the properties of the definition above, then any sufficiently small transverse arcs on the same side of $r^\pm$ will also.  
Extending this terminology, we say that a leaf $l$ makes a perfect fit with a ray $r$ if some sub-ray of $l$ does.  

\begin{lemma}[Uniqueness of perfect fits] \label{lem:perfect_fit_unique}
If $l_1$ and $l_2$ make a perfect fit with a ray $r$ of a leaf $l$, then $l$ separates $l_1$ from $l_2$ in $P$.  Consequently, there is at most one leaf making a perfect fit with $r$ on each side.  
\end{lemma}

\begin{proof}[Proof of Lemma \ref{lem:perfect_fit_unique}]
Suppose $l_1$ and $l_2$ make a perfect fit with a ray $r$ of a leaf $l$.  For concreteness say $l \in \cF^+$, and therefore $l_1$ and $l_2$ are in $\cF^-$.  

Let $\tau$ be a transverse arc to $\cF^+$ based on $l$.  If some leaf $l^+$ through the interior of $\tau$ meets both $l_1$ and $l_2$, we can (up to relabeling) assume $l_1$ separates $\tau \cap l^+$ from $l_2 \cap l^+$, as in Figure \ref{fig_unique_perfect_fit}, left.   This means that $l_2$ does not lie in the boundary of $\cF^-(l)$, so cannot make a perfect fit with $r$.  
It follows that the transverse arcs based on $l$ furnished by the definition of ``makes a perfect fit with $l_1$" and ``makes a perfect fit with $l_2$" must be on opposite sides of $l$ as in Figure \ref{fig_unique_perfect_fit} right.  We conclude that $l$ separates $l_1$ from $l_2$.  \qedhere

\begin{figure}[h] 
\labellist
\small \hair 2pt
\pinlabel {$l$} at 43.290 111.363
\pinlabel {$\tau$} at 90 33
\pinlabel {$l_1$} at 338.744 79.907
\pinlabel {$l_2$} at 405.844 184.038
\pinlabel {$l_1$} at 557.359 186.207
\pinlabel {$l_2$} at 746.753 186.207
\pinlabel {$l$} at 663.420 134.142
\pinlabel {$\tau_1$} at 579.004 90
\pinlabel {$\tau_2$} at 714.285 90
\endlabellist
\includegraphics[width=9cm]{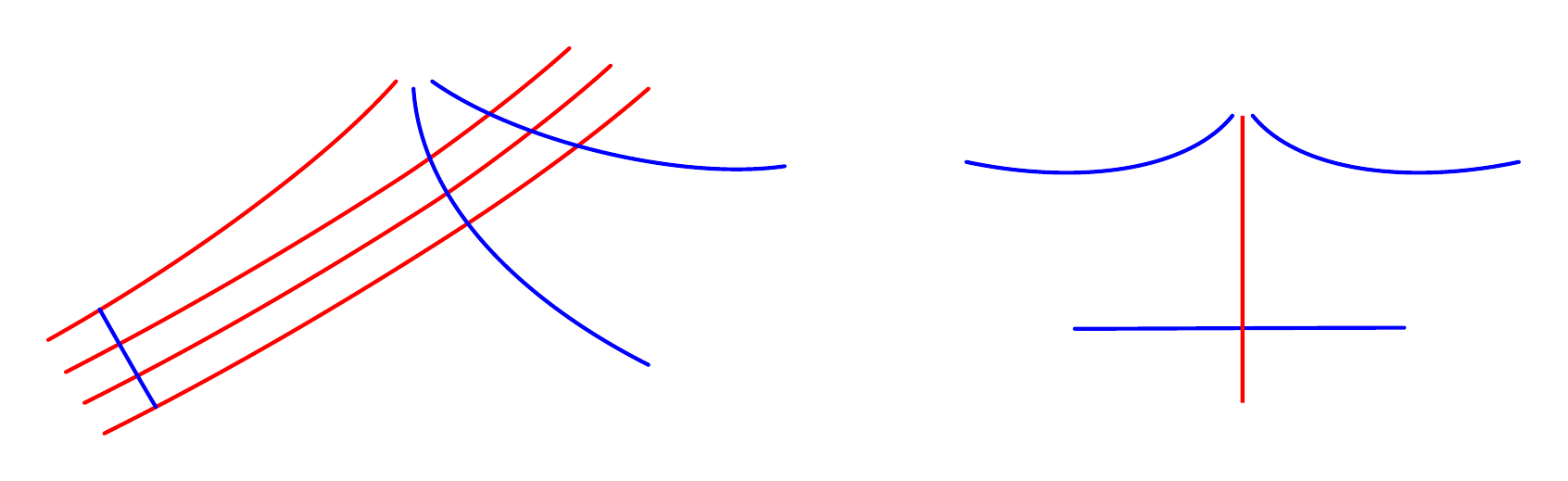}
\caption{The configuration on the left is impossible. The leaf $l$ must separate $l_1$ from $l_2$.}
\label{fig_unique_perfect_fit}
\end{figure}
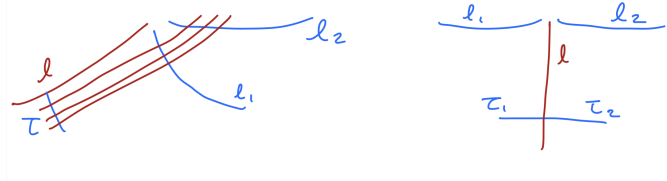
\end{proof} 

If, for some $x$ and $y$, both of the leaves $\cF^\pm(x)$ make perfect fits with $\cF^{\mp}(y)$ the resulting configuration is called a {\em lozenge}.\footnote{The term ``lozenge" refers to the shape in our sketched drawings.  Technically, a lozenge is a rhombus with all sides the same length, here the term is used very loosely to suggest a rhombus-like configuration.}
Precisely, we have: 

\begin{definition}[Lozenges]
Suppose $x,y \in P$ are points such that there exist open rays $r_x^{+}$ and $r_y^{-}$ of the leaves $\cF^+(x)$ and $\cF^-(y)$, respectively, that make a perfect fit, and also a pair of open rays $r_x^{-}$ and $r_y^{+}$ of $\cF^-(x)$ and $\cF^+(y)$ which make a perfect fit.  Then the open set
\begin{equation*}
 L := \lbrace p \in P \mid \cF^-(p) \cap r_x^+ \neq \emptyset \text{ and } \cF^+(p) \cap r_x^- \neq \emptyset \rbrace
\end{equation*}
is called a {\em lozenge} with {\em corners} $x$ and $y$, and
the rays $r_x^{\pm}$ and $r_y^{\pm}$ are called its \emph{sides}. 
\end{definition} 
\begin{figure}[h]
\labellist
\small \hair 2pt
\pinlabel {$x$} at 97 146.607
\pinlabel {$y$} at 278 104.847
\pinlabel {$r_x^+$} at 145 212.146
\pinlabel {$r_x^-$} at 125 79.908
\pinlabel {$r_y^-$} at 255.347 188.366
\pinlabel {$r_y^+$} at 237.358 63.088
\endlabellist
\includegraphics[width=6cm]{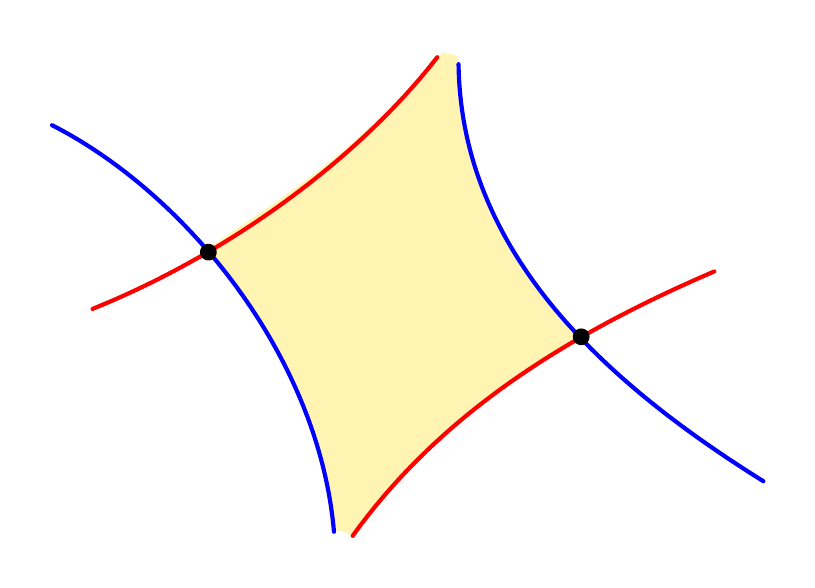}
\caption{A lozenge.}
\end{figure}
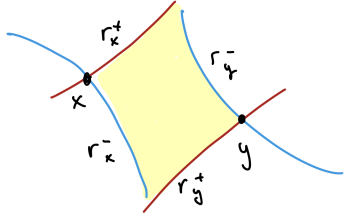

A lozenge is by definition an open subset of $P$.  The closure of a lozenge in $P$ (i.e., the union of the lozenge with its sides and corners) is called a {\em closed lozenge}. 
We will sometimes refer to the two perfect fits made by the rays of a lozenge as its {\em ideal corners}. 

\begin{rem}
The interior of a lozenge is \emph{trivially foliated}, meaning that it is isomorphic to an open rectangle foliated by horizontal and vertical segments.  However, the closed lozenge is not isomorphic to a closed rectangle with vertical and horizontal leaves! 
We will see time and time again in this text the appearance of trivially foliated open sets whose closure is not trivially foliated.  Understanding the boundaries of such sets is a key tool in the study of bifoliated planes.  
\end{rem}

Next, we record for future use a few basic properties about the possible arrangement of lozenges in a bifoliated plane.  The first is a simple condition which prevents a point from being the corner of a lozenge in a particular quadrant.  
 
   \begin{lemma}[Non-corner criterion] \label{lem_no_corner_criterion}  
  Let $x\in P$, and let $r^{\pm}_x$ be two rays bounding a quadrant $Q$ of $x$. 
  If $r^+_x$ and $r^-_x$ intersect a pair of leaves making a perfect fit, or if they each intersect leaves of a singular point, then there are no lozenges in $Q$ with $x$ as a corner.
  \end{lemma} 

This criterion is in fact both necessary and sufficient, provided $x$ is fixed by a nontrivial element of a group acting Anosov-like on $P$.  See Exercise \ref{ex_noncorner_converse}. 
  
  \begin{proof}
   Assuming the set-up of the lemma, let $l^{\pm}$ denote the leaves of $\cF^{\pm}$ that intersect $r^{\mp}_x$ and either make a perfect fit or are the leaves of a prong.  Suppose that there exists some leaf $l_1^+$ making a perfect fit with $r^-_x$ and a leaf $l_1^-$ making a perfect fit with $r^+_x$ in $Q$, as shown in Figure \ref{fig:no_corner_lem}.   
   
The leaf $l_1^+$ must be inside the half-space $P\smallsetminus l^+$ that does not contain $x$, and $l_1^-$ is inside the half-space $P\smallsetminus l^-$ that does not contain $x$.  By assumption, these two half-spaces are disjoint, so these leaves cannot intersect, as would be required to form the other corner of a lozenge in $Q$.  We conclude that $x$ cannot be a corner of a lozenge in $Q$. 

  \begin{figure}[h]

     \centerline{ \mbox{
\labellist
\small \hair 2pt
\pinlabel {$x$} at 99.887 97.500
\pinlabel {$x$} at 533.030 97.500
\pinlabel {$l^-$} at 106.620 310.758
\pinlabel {$l^+$} at 239.032 31.278
\pinlabel {$l^-$} at 534.152 313.003
\pinlabel {$l^+$} at 664.320 27.910
\endlabellist
\includegraphics[width=8cm]{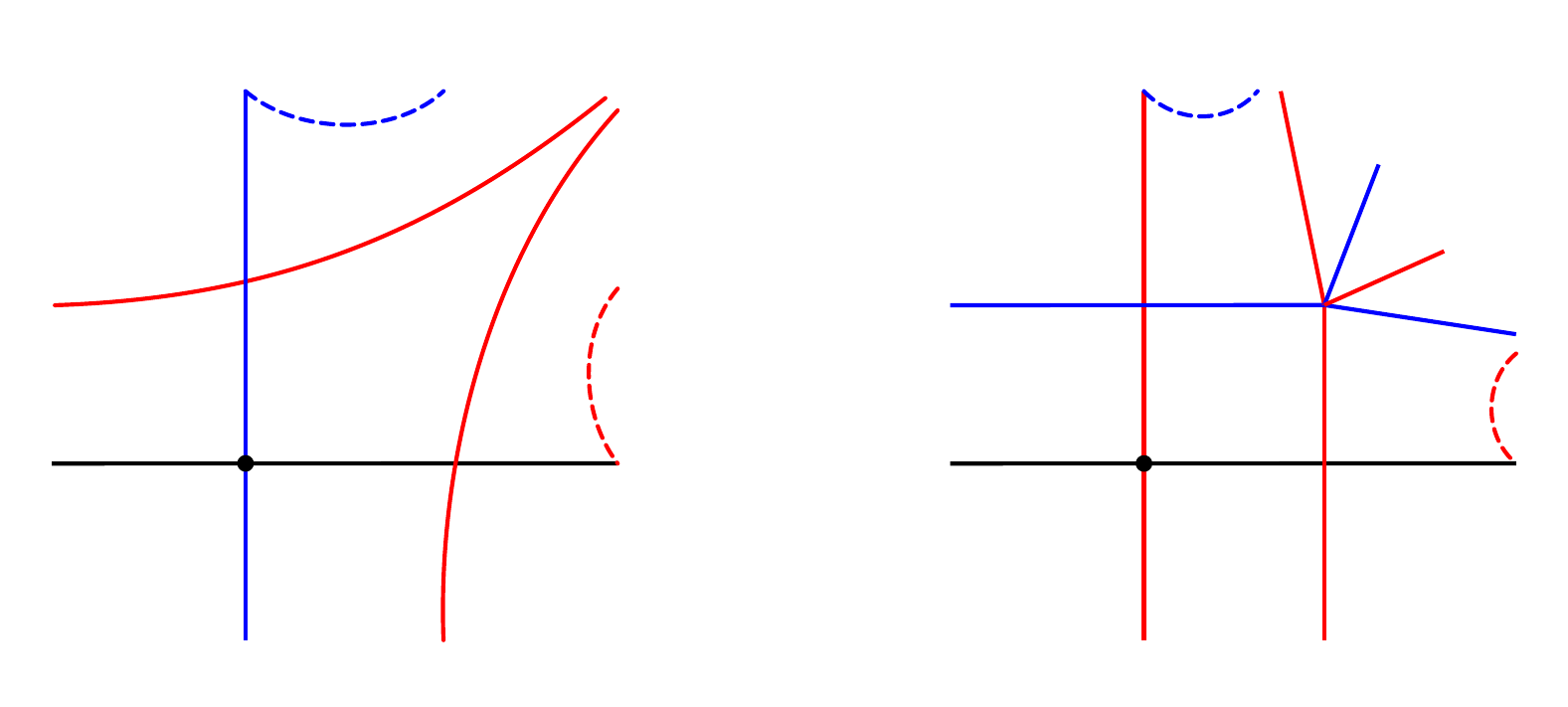}} }
\caption{The dotted leaves making perfect fits cannot intersect to form a lozenge with $x$ as a corner.}
 \label{fig:no_corner_lem} 
\end{figure}
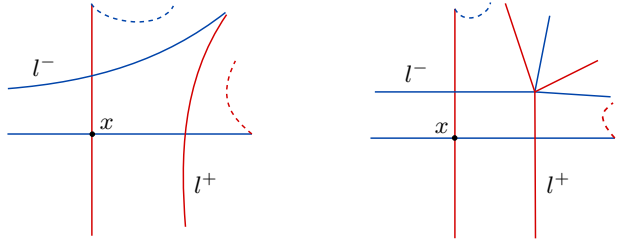
  \end{proof}
  
The next lemma describes how lozenges can intersect each other, see Figure \ref{fig:two_types_intersection}.   
\begin{lemma} \label{lem_markovian_or_corner} 
Let $L, L'$ be two lozenges.  If $L \cap L' \neq \emptyset$, then (up to exchanging the labels of $\cF^+$ and $\cF^-$), either 
\begin{itemize} 
\item $L \subset \cF^+(L')$ and $L' \subset \cF^-(L)$ or 
\item One corner of $L$ lies in $L'$ (and thus one corner of $L'$ lies in $L$).  
\end{itemize} 
Moreover, the two possibilities are exclusive unless $L=L'$
\end{lemma} 
 \begin{figure}[h]
     \centerline{ \mbox{
\labellist
\small \hair 2pt
\pinlabel {$L$} at 235.018 124.730
\pinlabel {$L'$} at 117.518 142
\pinlabel {$L$} at 554.112 109.771
\pinlabel {$L'$} at 596.735 182.266
\endlabellist
\includegraphics[width=10cm]{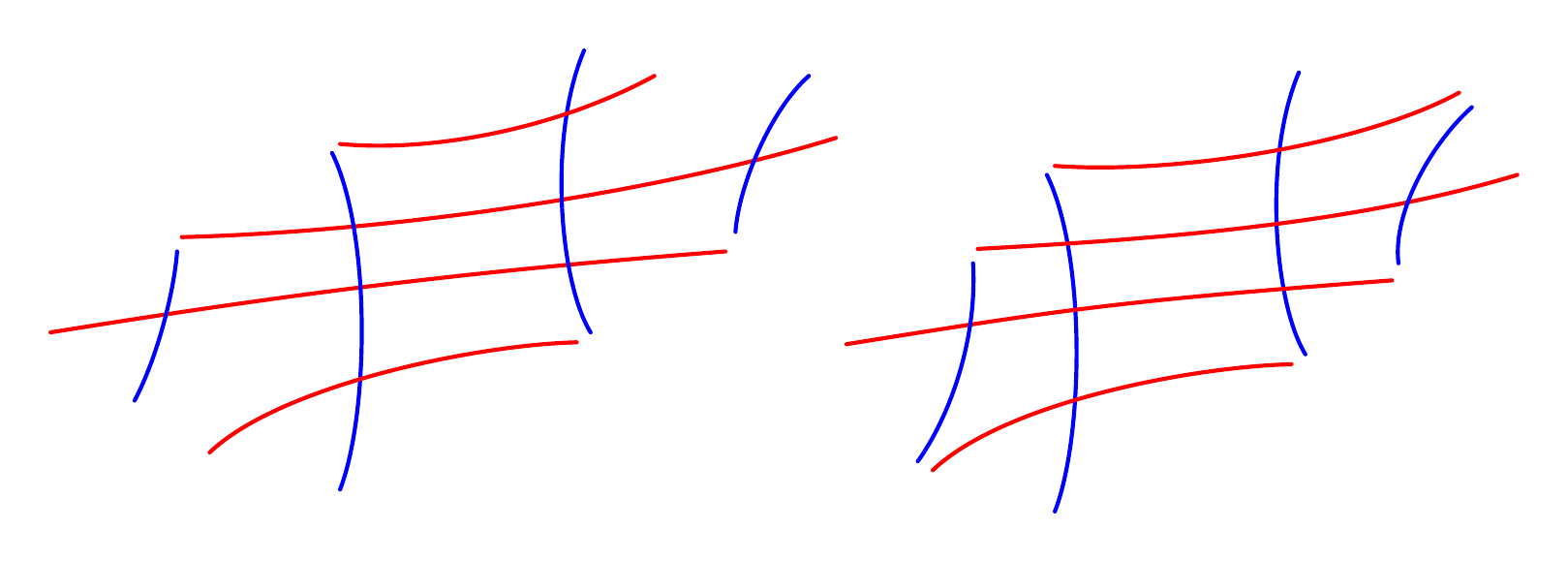}} }
\caption{The two types of intersections of lozenges, up to symmetry.}
 \label{fig:two_types_intersection} 
\end{figure}

\begin{proof} 
Let $l^\pm_1, l^\pm_2$ denote the sides of $L$ that lie in $\cF^\pm$, with $l^+_i$ and  $l^-_i$ making perfect fits. %
Assume $L \cap L' \neq \emptyset$.   

First suppose $L \subset \cF^+(L')$, so $l_i^+$ each intersect $L'$.  Since the leaves $l^-_i$ make perfect fits with $l^+_i$, it follows that $l^-_i$ cannot meet the interior of $L'$, and instead one lies on each side.  Thus, $L' \subset \cF^-(L)$.  
The same argument applies if $L \subset \cF^-(L')$, showing that this implies $L' \subset \cF^+(L)$,  and symmetrically if we reverse the roles of $L$ and $L'$.  

So, the remaining case to treat is (up to relabeling) that $l_1^+ \cap L' \neq \emptyset$ but $l_2^+ \cap L' = \emptyset$; and similarly exactly one (not both) of $l_i^-$ meet $L'$, as well as the symmetric condition with the roles of $L$ and $L'$ reversed.  From this it follows that one corner of $L$ lies in $L'$, and one corner of $L'$ is in $L$. 
\end{proof} 

\subsection{Lozenges in the presence of Anosov-like actions} 
Up to this point, we have not assumed that $P$ admits an Anosov-like group action, nor have we used any dynamics.  We will now see how the dynamics given by the axioms influence the structure of the bifoliation.  

We start with an easy observation about group actions satisfying Axiom \ref{Axiom_A1}.
\begin{observation} \label{obs:all_rays} 
Suppose that a group $G$ acts on $(P, \cF^+, \cF^-)$ satisfying Axiom \ref{Axiom_A1}.
If an element $g\in G$ fixes two distinct points $x,y$, then $g$ fixes all the rays through $x$ and $y$.
\end{observation} 

\begin{proof}
If $g$ fixes $x$, but not all of its rays, then it nontrivially permutes the quadrants of $x$.  This permutation preserves or reverses cyclic order (and preserves foliations) so cannot preserve any quadrant.  Thus, the only possibility for fixed points of $g$ lie on the leaves $\cF^\pm(x)$; but $x$ is the unique fixed point on these leaves by 
 Axiom \ref{Axiom_A1}.
\end{proof}

The next important lemma gives a condition for the existence of lozenges.
\begin{lemma} \label{lem:fixed_is_corner} 
Suppose a group $G$ acts on $(P, \cF^+, \cF^-)$ satisfying Axiom \ref{Axiom_A1}.  
Suppose a nontrivial element $g \in G$ fixes $x$ and each of its rays, and some ray $r$ of $\cF^\pm(x)$ makes a perfect fit with a leaf $l$.  Then $g(l) = l$ and $x$ is the corner of a lozenge whose opposite corner lies on $l$. 
\end{lemma} 

\begin{proof} 
To fix notation, we assume that $r$ is a ray in $\cF^+(x)$.  Suppose $g$ fixes all rays through $x$, and hence preserves quadrants of $x$.  By uniqueness of perfect fits (Lemma \ref{lem:perfect_fit_unique}), we deduce that $g(l)=l$. 
 By Axiom \ref{Axiom_A1}, $g$ fixes a unique point $y$ on $l$ and by Observation \ref{obs:all_rays}, it also fixes all the rays through $y$. 
Consider the ray $r'$ of $\cF^+(y)$ which contains a segment $\tau'$ so that leaves passing through $\tau'$ meet $r$, as given by the definition of perfect fit. 

We claim $r'$ makes a perfect fit with a ray of $\cF^-(x)$, forming a lozenge with corners $x$ and $y$.  
First note that $r' \cap \cF^-(x) = \emptyset$, because any point of intersection would need to be fixed by $g$, contradicting uniqueness of fixed points given by Axiom \ref{Axiom_A1}.  

Next, let $l_0^-$ be a leaf intersecting $r$; we will show it also intersects $r'$.  (Thus, $r$ can be used as the required transverse segment to $\cF^-(x)$ in the definition of perfect fit.)  Up to replacing $g$ with $g^{-1}$, we may assume that $g$ is topologically contracting on $\cF^+(x)$.   Since $r$ and $l$ make a perfect fit, there exists a leaf $l_1^-$ intersecting $r$ and $r'$.  Then for all $n$ we have $g^n(l_1^-) \cap r' \neq \emptyset$ and $g^n(l_1^-) \cap r \neq \emptyset$, and as $n \to \infty$ the point $g^n(l_1^-) \cap r$ approaches $x$.  Thus, for some $n$, the leaves $g^n(l_1^-)$ and $g^{n+1}(l_1^-)$ lie on either side of $l_0^-$, forming a rectangle bounded on either side by $r$ and $r'$.  We conclude $l_0^- \cap r' \neq \emptyset$ as desired.  

A symmetric argument shows that each leaf that intersects $l$ between $y$ and the perfect fit with $r$ also intersects $\cF^-(x)$; thus, we have found a lozenge with corners $x$ and $y$.  
\end{proof}

As an application of Lemma \ref{lem:fixed_is_corner}, we prove that the skew and trivial planes are the only bifoliated planes where both leaf spaces are homeomorphic to $\bR$ which admit Anosov-like actions.  Later, in Theorem \ref{thm:trichotomy} we will improve this statement to show that, in the presence of an Anosov-like action, even {\em one} leaf space being a Hausdorff manifold implies the plane is skew or trivial.  

\begin{proposition}\label{prop:both_R}
Let $(P, \cF^+,\cF^-)$ be a bifoliated plane where the leaf spaces of $\cF^+$ and $\cF^-$ are both homeomorphic to $\R$.  If $P$ admits an Anosov-like action of a group, then $P$ is either trivial or skew. 
\end{proposition}

\begin{proof}
Fix identifications of $\Lambda(\cF^+)$ and $\Lambda(\cF^-)$ with $\bR$.  Then, one can realize $P$ (isomorphically) as a subset of the trivially foliated $\bR \times \bR = \Lambda(\cF^+) \times \Lambda(\cF^-)$ by assigning a point $p$ to $(\cF^+(p), \cF^-(p))$.  
The image of $P$ in $\bR^2$ under this realization is bounded above and below by the graphs of the functions 
 $s, i: \Lambda(\cF^+) \to \Lambda(\cF^-) \cup \{\pm\infty\}$  defined by 
\[ s(l) = \sup\{ l^- \in \Lambda(\cF^-) : l \cap l^- \neq \emptyset \}\]
\[ i(l) = \inf \{ l^- \in \Lambda(\cF^-) : l \cap l^- \neq \emptyset \}.\] 

Suppose first that $s$ is not identically $\infty$, so that some leaf $l_0$ has $s(l_0)$ a leaf.  Fix an orientation on $\Lambda(\cF^+)$ so that $s(l_0)$ lies to the right of $l_0$. %
Since $\Lambda(\cF^-) \cong \bR$, no leaf of $\cF^+$ to the left of $l_0$ can intersect the horizontal line $y = s(l_0)$, else there would be another horizontal leaf at height $s(l_0)$ and hence $\Lambda(\cF^-)$ would not be $\bR$.  See figure \ref{fig:both_R}.   Thus, the function $s$ is increasing.  In particular, if $s$ is not always finite, then the set on which it takes the value $\infty$ is an interval of the form $[a, \infty)$ in $ \Lambda(\cF^+)$.  However, this implies that the leaf $a$ is $G$-invariant, which easily contradicts Axiom \ref{Axiom_dense} (take any $g$ fixing some leaf crossing $a$ with fixed point not on $a$, and apply iterates of $g$ to move $a$).   We conclude that either $s \equiv \infty$, or $s$ is always finite and is weakly monotone (increasing, with our choice of orientation).   A similar argument applies to $i$.  

  \begin{figure}[h]
      \centerline{ \mbox{
\labellist
\small \hair 2pt
\pinlabel {$l$} at 99.126 232.934
\pinlabel {$l_0$} at 381.345 20.456
\pinlabel {$s(l_0)$} at 506.495 209.418
\endlabellist
\includegraphics[width=8cm]{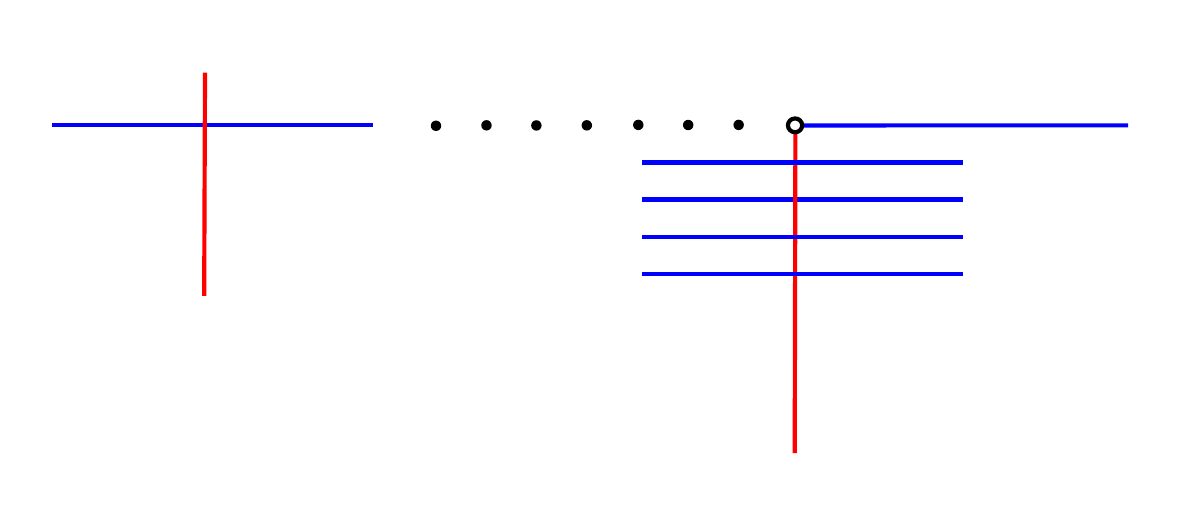}} }
\caption{If $l$ intersects $y=s(l_0)$, then some horizontal leaves through $l$ and $l_0$ are non-separated}
 \label{fig:both_R} 
\end{figure}

We next claim the following: 
\begin{claim}   \label{lem_s_pf}
For any $l \in \Lambda(\cF^+)$ fixed by a nontrivial element $g \in G$, if $s(l)$ or $i(l)$ is a leaf (rather than $\pm \infty$), then it makes a perfect fit with $l$.  Consequently, $l$ is the side of a lozenge, and its other $\cF^+$-side is a leaf $l'$ making a perfect fit with $i(l')$.  
\end{claim} 

\begin{proof} 
We prove the assertion for $s(l)$, the other case being symmetric.  
By definition $s(l)$ lies in the boundary of $\cF^-(l)$.  Note also that, up to replacing $g$ with a power that preserves orientations of leaf spaces, $s(l)$ is invariant by $g$.  We need to show that $l$ is in the boundary of $\cF^+(s(l))$.  

Assume for contradiction that this is not the case.  Let $l'$ denote the boundary leaf of $\cF^+(s(l))$ on the side of $l$.  
Since $s(l)$ is $g$-invariant, and $g$ preserves orientations, we have $g(l')=l'$.  Let $x$ and $x'$ denote the fixed points of $g$ on $l$ and $l'$ respectively.  By definition of $s(l)$, there exists a sequence of leaves intersecting $l$ that limit to $s(l)$, so we can find some leaf $k \in \cF^-$ intersecting both $l$ and $l'$ above the points $x$ and $x'$.   Then, either as $n \to \infty$ or $n \to -\infty$, the sequence $g^n(k)$ 
limits to both $\cF^-(x)$ and $\cF^-(x')$, contradicting that $\Lambda(\cF^-) \cong \bR$ so has no nonseparated leaves.    See Figure \ref{fig_s_lozenges}. We conclude that $s(l)$ makes a perfect fit with $l$, so by Lemma \ref{lem:fixed_is_corner}, $l$ and $s(l)$ form two sides of a lozenge.  If $l'$ denotes the other $\cF^+$ side of this lozenge, then it must make a perfect fit with a leaf below it, necessarily $i(l')$.  This proves the claim.  \end{proof} 

 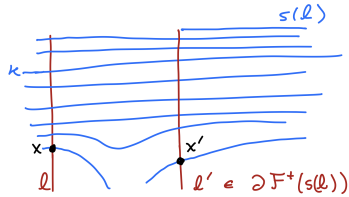
\begin{figure}[h]
     \centerline{ \mbox{
\labellist
\small \hair 2pt
\pinlabel {$l$} at 71.830 53.522
\pinlabel {$x$} at 66.650 87
\pinlabel {$x'$} at 351 87
\pinlabel {$l'\in\partial\mathcal F^+(s(l))$} at 410 47
\pinlabel {$k$} at 37.514 217.032
\pinlabel {$s(l)$} at 408.516 299.110
\endlabellist
\includegraphics[width=7cm]{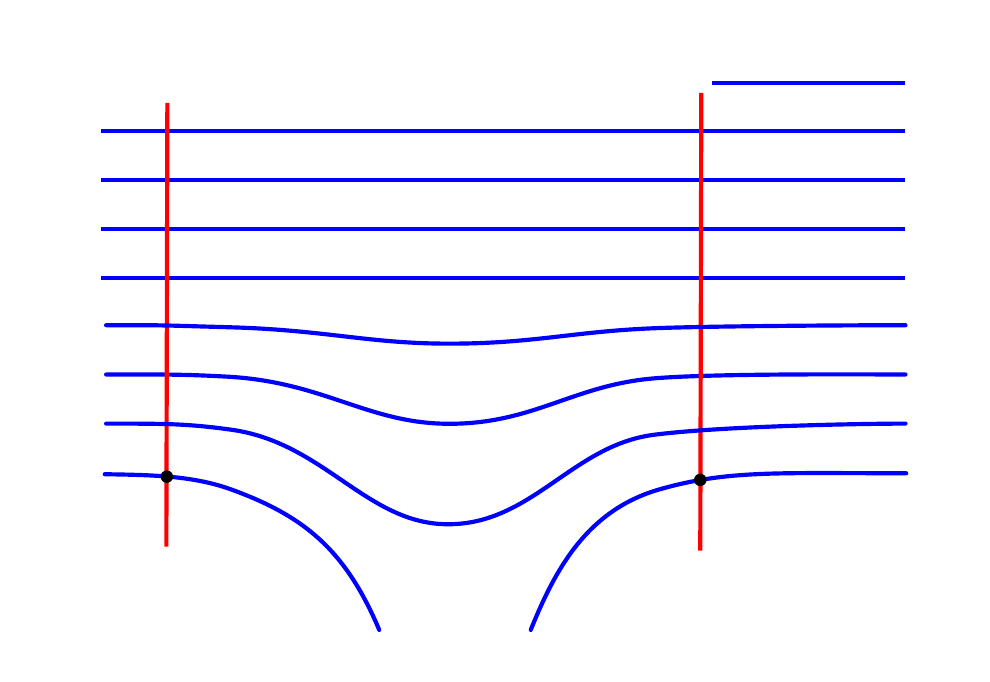}} }
\caption{If $s(l)$ doesn't make a perfect fit with $l$, then we find branching leaves.}
 \label{fig_s_lozenges} 
\end{figure}

As a consequence of the claim, if $s$ is finite, then $i$ is finite.  Furthermore, for any lozenge with sides $l < l'$ in $\cF^+$, arranged according to our orientation so that $s(l)$ lies to the right of $l$, then $i(l')$ lies to the left of $l'$.  Consequently, if $k$ is a leaf between $l$ and $l'$, it must meet both sides of the lozenge so we have $s(l)<s(k)$ and $i(k)<i(l')$.    
This shows that $s$ and $i$ are both increasing functions.  Finally, we argue that this also implies they are strictly increasing.  For concreteness we do the argument for $s$.  If there was a nontrivial open set on which $s$ was constant,  Axiom \ref{Axiom_dense} says that we could find a leaf $l$ fixed by some element of $G$ in this set, but we just argued that $s(k)>s(l)$ for all $k$ to the right of $l$, a contradiction.   The argument for $i$ is exactly the same.  

We can now quickly conclude the proof.   Either $s \equiv \infty$ and $i \equiv -\infty$, in which case the plane is trivial, or up to choice of orientation, the realization of $P$ in $\bR^2$ is bounded above by the graph of $s$ and below by the graph of $i$, which are both strictly increasing, unbounded functions.  After a reparameterization of $\bR$, we obtain the desired homeomorphism with the skew plane.  
\end{proof}

\section{Nonseparated leaves and lines of lozenges} 

In this section we study the structure of sets of pairwise nonseparated leaves in bifoliated planes.  In Proposition \ref{prop:4weak4strong} we show that, in the presence of an Anosov-like group action, these always arise as sides of adjacent lozenges and give an important strengthening of Axiom \ref{Axiom_nonseparated}.  

Note that ``nonseparated with" is not a transitive relation: it is possible to have $l_0$ nonseparated with $l_1$ and $l_1$ nonseparated with $l_2$, but $l_0$ and $l_2$ are separated.   This can occur, for example when $l_1$ is a prong singularity, or when $l_0$ and $l_2$ are on opposite sides of $l_1$.  
However, for any collection of {\em pairwise nonseparated leaves}, there does exist a sequence $l_n$ eventually entering every open neighborhood of each.   This allows us to put a linear order on the leaves in any such set, as follows.  

 \begin{observation} \label{obs:nonseparated_implies_pf}
Suppose $(P, \cF^+,\cF^-)$ is a bifoliated plane and $l$ and $l'$ are nonseparated in $\Lambda(\cF^+)$.  Consider the set $S$ of leaves of $\cF^-$ that separate $l$ from $l'$ in $P$, as illustrated in Figure \ref{fig_linear_order}.
If $l_n$ is a sequence of leaves converging to a union of leaves with both $l$ and $l'$ in the limit, then each leaf of $S$ will intersect $l_n$ for all sufficiently large $n$, hence $S$ inherits a linear order.  The supremum and infimum of $S$ are leaves of $S$ that make a perfect fit with $l$ and $l'$ (respectively, up to reversing the order). 
 \end{observation}
 
   \begin{figure}[h]
     \centerline{ \mbox{
\labellist
\small \hair 2pt
\pinlabel {$l$} at 69.797 192.033
\pinlabel {$l'$} at 579.591 193.002
\pinlabel {$l_n$} at 642.588 123.257
\pinlabel {$S$} at 325.663 25.421
\endlabellist
\includegraphics[width=8cm]{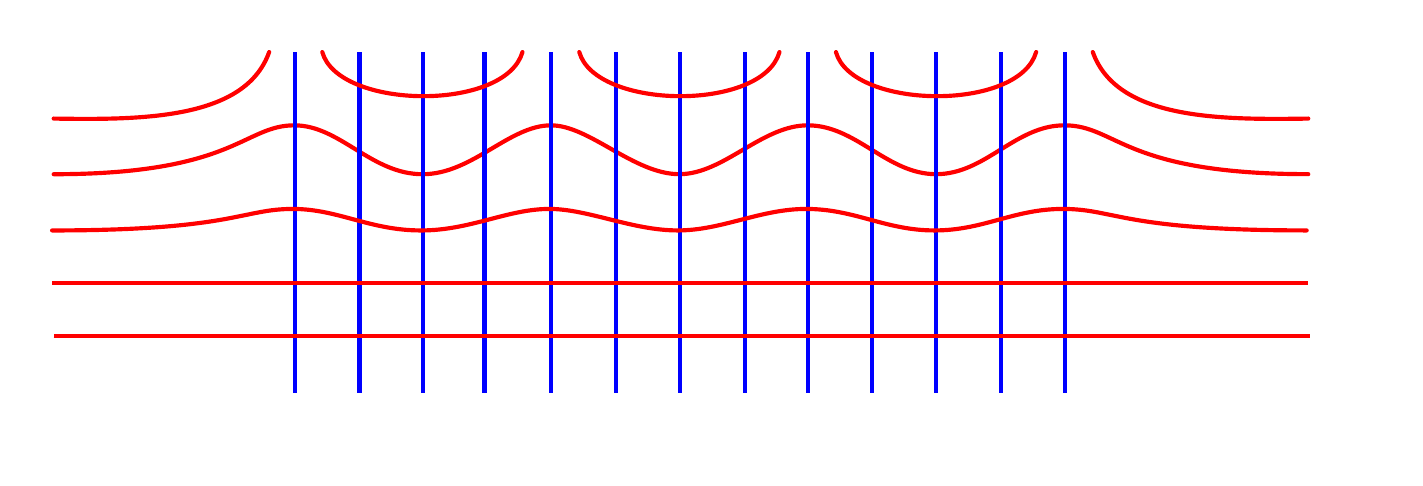}} }
\caption{The set $S$ separating the leaves $l$ and $l'$}
 \label{fig_linear_order} 
\end{figure}
 
 We leave the proof as an exercise in applying the definitions.

In a general bifoliated plane, one can realize any order type of a countable ordered set by pairwise nonseparated leaves.  As one example, let $C \subset \mathbb{R}$ be a Cantor set and consider $P = \mathbb{R}^2 \setminus ( C \times [0, \infty))$, equipped with (the restriction of) the horizontal and vertical foliations from $\mathbb{R}^2$.  Then the horizontal leaves $I \times \{0\}$ where $I$ is a connected component of $\mathbb{R} \setminus C$ are all pairwise nonseparated.  

\begin{exercise}
Generalize the example above to justify the assertion that any order type of countable ordered set can be realized as a set of pairwise nonseparated leaves of a bifoliated plane.  
\end{exercise} 

The next proposition says that, by contrast, for planes with Anosov-like actions, sets of pairwise nonseparated leaves must have the order type of subsets of $\bZ$ and these leaves are adjacent to each other along perfect fits.  It also gives an important strengthening of Axiom \ref{Axiom_nonseparated}. 

\begin{proposition} \label{prop:4weak4strong}
Suppose a group $G$ acts on $(P, \cF^+, \cF^-)$ satisfying \ref{Axiom_A1} and \ref{Axiom_nonseparated}.  Let $S$ be a maximal\footnote{Following the terminology introduced in \cite{CD03}, 
a maximal set of pairwise nonseparated leaves is sometimes called a \emph{cataclysm}.} set of pairwise nonseparated leaves in $\cF^+$ or $\cF^-$. 
Then, the leaves of $S$ can be indexed by a subset of $\bZ$, with consecutive leaves appearing on lozenges sharing a side.  Moreover,  
any $g\in G$ fixing all rays of any leaf $l$ of $S$ also fixes every leaf of $S$.  
\end{proposition}

The proof of Proposition \ref{prop:4weak4strong} is adapted from Theorem 4.3 in \cite{Fen98} where a related result is proved for the orbit spaces of Anosov flows.  A slightly weaker version is given in \cite{BBM24b}, with a similar proof.  We start with a useful lemma, illustrated in Figure \ref{fig_nonseparated_forces_lines}.

\begin{lemma}  \label{lem:find_next_lozenge}
Suppose a nontrivial $g \in G$ fixes both corners of a lozenge $L$.  Denote the sides of $L$ by $l^\pm$ and $f^\pm$, with $l^+ \cap l^-$ and $f^+ \cap f^-$ the two corners.  
Suppose there is a leaf $l' \in \cF^+$ nonseparated with $l^+$ such that $l'$ and $l^+$ lie in distinct connected components of $P \setminus f^-$, and some leaf $f'$ intersects $l'$ but not $f^+$.  Then there exists a lozenge adjacent to $L$ with corner $f^+ \cap f^-$.
\end{lemma}

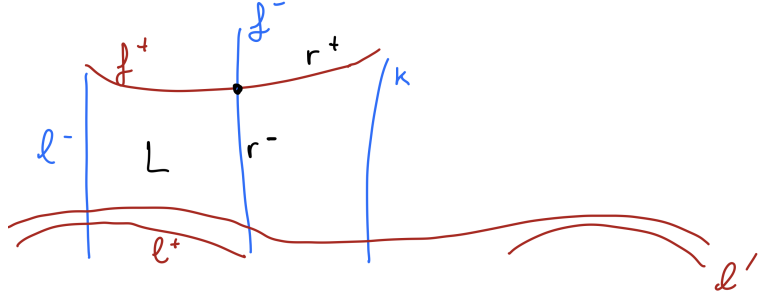
\begin{figure}[h]
\begin{center}
\labellist
\small \hair 2pt
\pinlabel {$l^-$} at 75.796 94.244
\pinlabel {$l^+$} at 40 41.211
\pinlabel {$f^-$} at 160.327 162.429
\pinlabel {$f^+$} at 120 132
\pinlabel {$r^+$} at 179.565 132
\pinlabel {$r^-$} at 162 93.661
\pinlabel {$k$} at 212.212 121.635
\pinlabel {$L$} at 118.936 97.741
\pinlabel {$l'$} at 376.027 44.708
\endlabellist
\includegraphics[width=8cm]{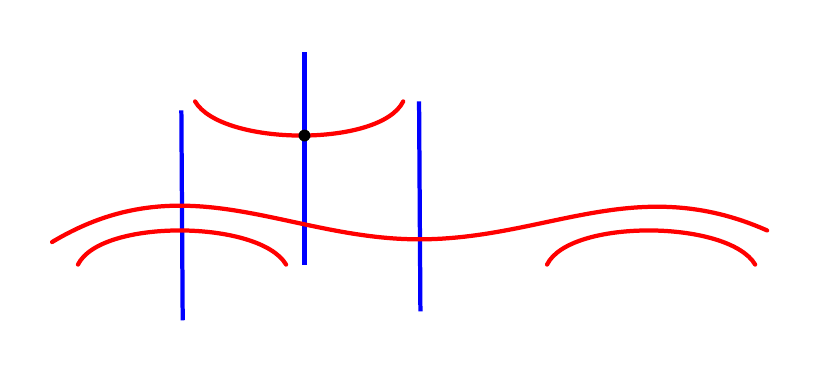}
\caption{Having a side of a lozenge nonseparated with another leaf creates an adjacent lozenge}
\label{fig_nonseparated_forces_lines}
\end{center}
\end{figure}

\begin{proof} Since $l^+$ and $l'$ are nonseparated, we can take a sequence of leaves 
 $l_n$ limiting to both $l^+$ and $l'$.  
Let $r^+$ denote the ray of $f^+$ that is \emph{not} a side of $L$, and $r^-$ the ray of $f^-$ that forms one side of $L$.  

Consider the saturation $\cF^-(r^+)$.  
This set is bounded on one side by $f^-$, and, since we assumed some leaf $f'$ intersects $l'$ but not $f^+$, it is bounded on the other side as well.  This boundary is either a single $g$-invariant leaf or a $g$-invariant union of leaves.   There is a unique leaf $k$ of this limit that either intersects $l'$, or separates $l'$ from $l^+$.   Since the set of leaves nonseparated with $l^+$ is a $g$-invariant set, we have $g(k) = k$.  
We will show $k$ makes a perfect fit with $r^+$. 

Half the definition is already satisfied because $k$ is in the boundary of $\cF^-(r^+)$.  For the other half, we use the fact that $\cF^+(k)$ is also $g$-invariant.   The union of leaves in $\cF^+(k)$ contains an infinite segment of $r^-$ because $k$ meets each leaf $l_n$.  Thus, by $g$-invariance, $\cF^+(k)$ contains the full interior of $r^-$, so its boundary contains $r^+$.  This is what we needed to show.  
We conclude that $k$ and $r^+$ make a perfect fit.  
Since $r^+$ is on a fixed leaf of $g$,
Lemma \ref{lem:fixed_is_corner} implies that $k$ and $r^+$ form two sides of a lozenge, necessarily with $r^-$ as another side. 
\end{proof}

\begin{proof}[Proof of Proposition \ref{prop:4weak4strong}] 
To fix notation, suppose $S$ is a set of $\cF^+$ leaves.  Let $l_0\in S$ and let $g \in G$ be a nontrivial element fixing $l_0$, as given by Axiom \ref{Axiom_nonseparated}.  
 By Axiom \ref{Axiom_A1}, there exists a unique $x_0\in l_0$ fixed by $g$. Up to taking a power, we further assume that $g$ fixes all the rays of $\cF^\pm(x_0)\smallsetminus \{x_0\}$.

If $S$ contains singular leaves, we instead consider only the {\em faces} of these singular leaves that are pairwise nonseparated with the other non-singular leaves or faces.  Abusing notation, we denote this set of leaves and faces by $S$ and call an element of $S$ a {\em leaf}, as this simplifies terminology and does not affect the proof.  

Let $l_0'\in S$ be another leaf.  
By Observation \ref{obs:nonseparated_implies_pf}, there exists $f_0\in \cF^-$ that makes a perfect fit with $l_0$ and separates $l_0$ from $l'_0$.  By Lemma \ref{lem:fixed_is_corner}, $f_0$ and $l_0$ are two sides of a lozenge $L_0$ with corners fixed by $g$. 

Similarly, reversing the roles of $l_0$ and $l'_0$, there exists $f'_0 \in \cF^-$ making a perfect fit with $l_0'$, separating it from $l_0$ in $P$, and forming two sides of a lozenge $L_0'$ fixed by some $g' \in G$.  Call $S_0$ the subset of $S$ between $f_0$ and $f_0'$.   We will show that $S_0$ is finite and consists of sides of adjacent lozenges.  
 To do this, we will apply Lemma \ref{lem:find_next_lozenge} to one of $L_0$ or $L_0'$ in order to find an adjacent lozenge to one of them that lies in between the two.  We will do this repeatedly, and show that the process terminates after finitely many steps, producing a finite chain of lozenges between $L_0$ and $L_0'$.  

Let $y_0$ be the fixed point of $g$ on $f_0$.
Let $r_0$ be the ray of $\cF^+(y_0)\smallsetminus \{y_0\}$ that is contained in the side of $f_0$ that contains $l'_0$, and $r_0'$ the ray of the corresponding side of $L_0'$.  
To apply Lemma \ref{lem:find_next_lozenge}, it suffices to show that either $l_0$ is not in the $\cF^-$ saturation of $r_0'$, or $l_0'$ is not in the $\cF^-$ saturation of $r_0$.  
We argue this as follows:  
Suppose that $l_0' \subset \cF^-(r_0)$ holds.  Let $y'$ be the fixed point of $g'$ on $f'_0$. Let $x'$ be the fixed point of $g'$ on $l'_0$.  
Since $g'$ has a unique fixed point on $\cF^+(y')$, by Axiom \ref{Axiom_A1}, the leaves $\cF^-(x')$ and  $\cF^+(y')$ cannot intersect.  Thus, $\cF^+(y_0)$ separates  $\cF^+(y')$ from $l'_0$ and $l_0$.   
Similarly, $\cF^-(x_0)$ cannot intersect $\cF^+(y_0)$, and therefore won't intersect $\cF^+(y')$.  It follows that the $\cF^-$-saturation of $r'_0$ does not contain $l_0$.   
See Figure \ref{fig:4weak4strong}.

\begin{figure}[h]

     \centerline{ \mbox{
\labellist
\small \hair 2pt
\pinlabel {$y'$} at 83.451 148.236
\pinlabel {$y_0$} at 148.277 253.020
\pinlabel {$l_0$} at 299 360.820
\pinlabel {$f_0$} at 299 237.190
\pinlabel {$f'_0$} at 299 161.051
\pinlabel {$l'_0$} at 299 92.451
\pinlabel {$r_0$} at 152 120
\pinlabel {$r_0'$} at 84 280
\endlabellist
\includegraphics[width=5cm]{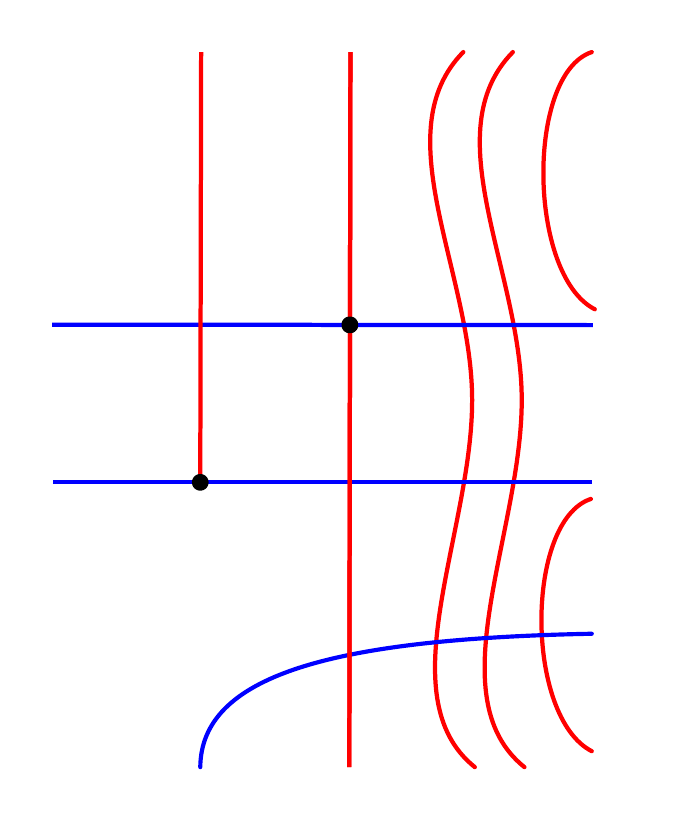} }}
\caption{If $\cF^-(r'_0)$ contained $l_0$ in this configuration, then $\cF^+(y)$ would have an additional fixed point for $g$ given by its intersection with the $\cF^-$ leaf through $x_0$, a contradiction.}
\label{fig:4weak4strong}
\end{figure}

Thus, for at least one of $L_0$ and $L_0'$, Lemma \ref{lem:find_next_lozenge} applies and we may find an adjacent lozenge which contains a side nonseparated with $l_0$ or $l'_0$ and lies in $S_0$.  Denote this lozenge by $L_1$ or $L_1'$ (depending on whether it shared a side with $L_0$ or $L'_0$), and the leaf nonseparated with $l_0$ by $l_1$ or $l_1'$.  
It is possible that $L_1 = L_0'$ (or $L_1' = L_0$ in the latter case)  but in this case we are already done.  Otherwise, working now with the new pair of nonseparated leaves $l_0$ and $l_1'$ (or $l_0'$ and $l_1$) we can start the argument again, finding leaves which make perfect fits and another adjacent lozenge between $L_0$ and $L_0'$.  
Iterating the process, we produce sequences of lozenges $L_1, L_2, \ldots$ and $L_1', L_2', \ldots$.  If for some $k, j$ we have $L_j = L_k'$, then we are done.  

It remains to argue that the process terminates.  If not, at least one sequence is infinite; for concreteness, suppose the sequence $\{L_i \}$ is infinite.  Let $\{f_i\}$ be distinct sides of $L_i$ in $\cF^-$, these are all fixed by $g$.  
The leaves $f_i$ lie in a bounded region of the leaf space, so must accumulate somewhere, i.e., limit onto a leaf or union of leaves. A unique leaf in that limit, call it $f_\infty$, separates the lozenges $L_i$ from the $L'_j$,  so $f_\infty$ is also fixed by $g$.  
This contradicts Axiom \ref{Axiom_A1} since the leaf of $\cF^+$ through the fixed point of $g$ on $f_\infty$ would intersect (infinitely many) of the fixed leaves $f_i$.  This contradiction concludes the proof.
\end{proof} 

The configuration of adjacent lozenges obtained in the proof above will play an important role, so we introduce the following terminology.  
\begin{definition}[Line of lozenges] \label{def_line_lozenges}
A {\em line of lozenges} is a finite or infinite set $\{L_i\}$ of lozenges, indexed by some set of consecutive integers, such that each consecutive pair $L_i$ and $L_{i+1}$ shares a common side, and some leaf of $\cF^\pm$ meets all lozenges simultaneously.\end{definition} 

Note that as a consequence of the definition, the shared sides of the lozenges are necessarily always in $\cF^+$ or always in $\cF^-$, and the union of lozenges in the line together with their shared sides gives a trivially foliated region.  

We record a useful lemma about infinite lines of lozenges: 
\begin{lemma}  \label{lem:half_scalloped} 
Suppose $G$ acts on $P$ satisfying Axioms \ref{Axiom_A1}, \ref{Axiom_nonseparated}, and suppose $L_0, L_1, L_2, \ldots$ is an infinite line of lozenges in $P$.  Let $f_i$ denote the leaf that is a common side of $L_i$ and $L_{i+1}$.  Either the sequence $f_i$ leaves all compact sets in $P$, or it limits onto a bi-infinite union of pairwise nonseparated leaves.
\end{lemma}

\begin{proof}  For concreteness, assume $f_i$ are leaves of $\cF^+$.   By Proposition \ref{prop:4weak4strong} there exists some nontrivial $g \in G$ fixing all the corners of all $L_i$, so $g$ fixes all the leaves $f_i$.  
Since the sequence $f_i$ is monotone along an embedded copy of $\R$ in $\Lambda(\cF^+)$, it either escapes all compact sets or converges to a limit.  This  limit is invariant under $g$ (because all $f_i$ are fixed by $g$) and consists of either a finite (possibly unique) or infinite union of nonseparated leaves. We will show that the union of leaves is in fact bi-infinite.  

To do that we prove that any power of $g$ must act freely on that union of leaves.  For contradiction, suppose not.  Then there exists a leaf $f_{\infty}$ in the limit that is fixed by $g^n$ for some $n\neq 0$.
By Axiom \ref{Axiom_A1}, $g^n$ fixes a point $x\in f_{\infty}$.  Then, $f_i$ intersects $\cF^-(x)$ for all large enough $i$ and each intersection point is fixed by $g^n$.  This gives infinitely many fixed points for $g^n$ on $\cF^-(x)$, a contradiction with the uniqueness of fixed point given by Axiom \ref{Axiom_A1}.
\end{proof}

The orbit space of the Bonatti--Langevin flow (introduced in Section \ref{sec_example_bonatti_langevin}) is one example where a sequence of leaves converges to a bi-infinite collection of leaves in the limit.  See Section \ref{BL_orbit_space} for a description of this orbit space.  The configuration appearing there is part of a larger structure called a {\em scalloped region}.  Informally, a scalloped region is a trivially foliated open set $U \subset P$ whose boundary $\partial U$ consists of four infinite families of pairwise nonseparated leaves, which can be realized as two distinct bi-infinite lines of lozenges: one sharing sides along leaves of $\cF^+$ and the other sharing $\cF^-$-sides. See Figure \ref{fig_scalloped}.  The formal definition, originally due to Fenley \cite{Fen98}, is given below.  

We will show in Proposition \ref{prop:infinite_lozenge_is_scalloped} that any infinite line of lozenges in a bifoliated plane with an Anosov-like action is in fact part of a scalloped region.  

\begin{definition}[Scalloped region] \label{def_scalloped}
  A \emph{scalloped region} is an open, unbounded set $U \subset P$ with the following properties:
  \begin{enumerate}[label=(\roman*)]
   \item The boundary $\partial U$ consists of the union of four families of leaves $l_k^{1,+}, l_k^{2,+}$ in $\cF^+$ and $l_k^{1,-},l_k^{2,-}$ in $\cF^-$, indexed by $k\in \bZ$.
   \item The leaves of each family $l_k^{i,\pm}$, $k \in \bZ$ are pairwise nonseparated.
   \item The boundary leaves are ordered so that there exists a (unique) leaf $f_k^{1,-}$ that makes a perfect fit with $l_k^{1,+}$ and $l_{k+1}^{1,+}$. Moreover, $f_k^{1,-}$ accumulates on the leaves $\cup_{i\in \bZ} l_i^{1,-}$ as $k \to \infty$, and  on $\cup_{i\in \bZ} l_i^{2,-}$ as $k \to -\infty$.  The analogous statement holds for leaves making perfect fits with the other families $l_k^{i,\pm}$.
   \item The bifoliation is trivial inside $U$, i.e., for all $x \neq y\in U$, $\cF^+(x)\cap \cF^-(y) \neq \emptyset $ and  $\cF^+(y)\cap \cF^-(x) \neq \emptyset $ and $U$ contains no singular points.
  \end{enumerate}
 \end{definition}

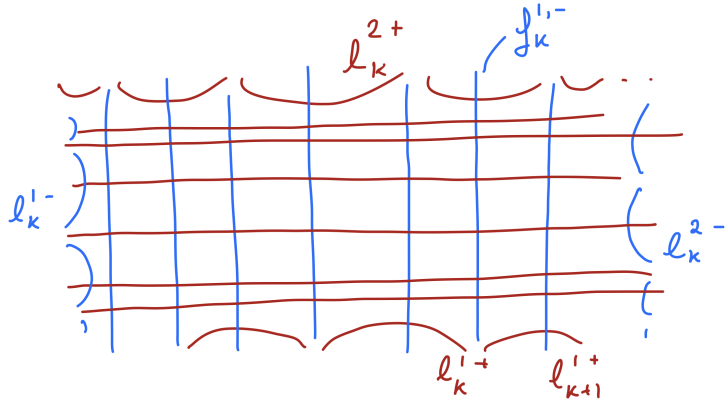
\begin{figure}[h]
     \centerline{ \mbox{
\labellist
\small \hair 2pt
\pinlabel {{\color[HTML]{0000FF}$l_k^{1-}$}} at 34.220 184
\pinlabel {{\color[HTML]{0000FF}$l_k^{2-}$}} at 531.954 161.319
\pinlabel {{\color[HTML]{FF0000}$l_k^{2+}$}} at 201.902 283.435
\pinlabel {{\color[HTML]{FF0000}$l_k^{1+}$}} at 355 75
\pinlabel {{\color[HTML]{FF0000}$l_{k+1}^{1+}$}} at 407 75
\pinlabel {{\color[HTML]{0000FF}$f_k^{1-}$}} at 402.968 283.435
\endlabellist
\includegraphics[width=11cm]{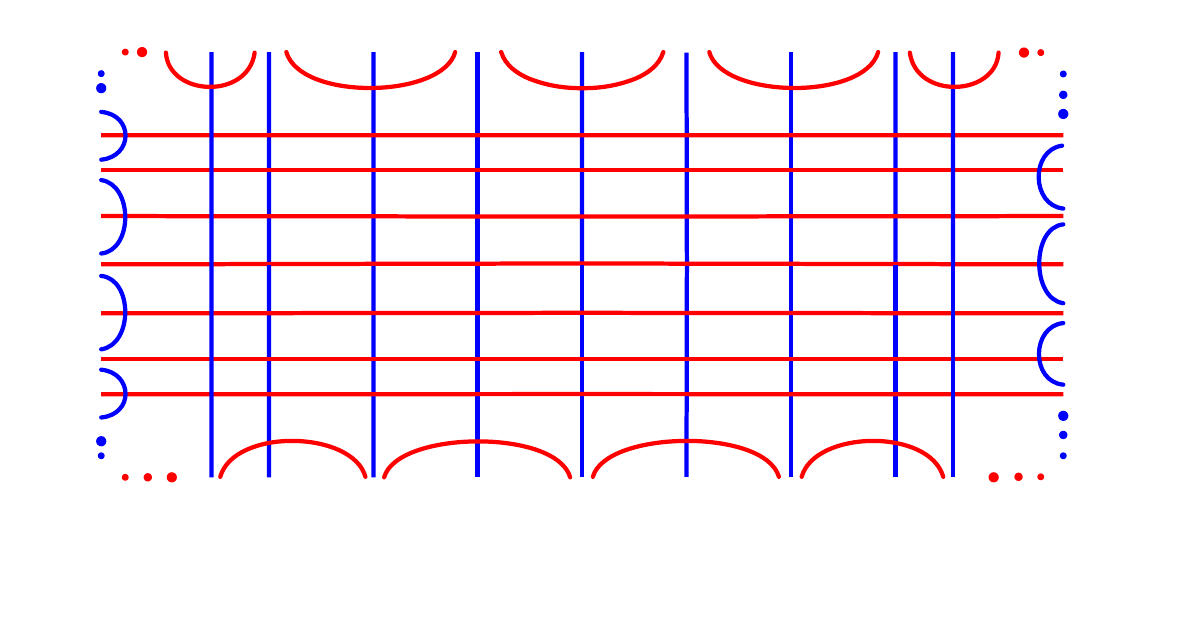} }}
\caption{A scalloped region}%
\label{fig_scalloped}
\end{figure}

Our next lemma shows that the presence of a scalloped region in a bifoliated plane $P$ puts strong algebraic constraints on any group that acts Anosov-like on $P$.  

\begin{lemma}\label{lem_action_of_stabilizer_scalloped}
Let $U$ be a scalloped region in $P$.  Then the stabilizer of $U$ under any Anosov-like group action of $G$ is virtually isomorphic to a subgroup of $\bZ^2$. More precisely, there is a $\bZ^2$-subgroup of the stabilizer, of index at most two, generated by elements $g,h$, with $g$ fixing one of the lines of lozenges making up the scalloped region and freely permuting lozenges in the other, and vice versa for $h$.
\end{lemma}

This was proved for orbit spaces of (pseudo)-Anosov flows by Barbot and Fenley, see e.g., \cite[Corollary 4.8]{Fen98}, and for Anosov-like actions in \cite[Lemma 2.33]{BFM22}.
\begin{proof}
Suppose $U$ is a scalloped region and let $H$ be the finite-index subgroup which preserves,
setwise, each of the four bi-infinite families of leaves forming the boundary of $U$.   We first discuss the action of $H$, and then will argue at the end of the proof that $H$ has index at most 2.  

 The action of $H$ on each family preserves the linear ordering of the indices so can be considered an integral translation group.  In other words, there is a homomorphism $\Phi\colon H \to \bZ \times \bZ$ by considering the translation actions of $H$ on the families $l_k^{1,+}$ and  $l_k^{1,-}$, where we use the notations from Definition \ref{def_scalloped}. 
If $h$ is in the kernel of this homomorphism, then $h$ preserves each leaf $f_k^{1,-}$ and $f_j^{1,+}$.  Each leaf $f_k^{1,-}$  intersects all of the infinitely many $f_j^{1,+}$; these intersection points must be fixed by $h$ and so by Axiom \ref{Axiom_A1}, $h$ is the identity.  

Thus, we have shown that $\Phi$ is injective.  
We now conclude the proof by showing that $\Phi$ is virtually surjective.  Since the leaves $l_k^{1,+}$ are nonseparated, Proposition \ref{prop:4weak4strong} says there is some nontrivial $g \in G$ fixing all leaves in this family simultaneously.  We may also take $g$ to fix all rays of these leaves, and hence all lozenges in one of the lines comprising $U$.  In particular, we have $g \in H$.  Since $g$ is nontrivial and $\Phi$ is injective, its image under $\Phi$ must be a nontrivial element of $\{0\} \times \bZ$.  A symmetric argument produces a nontrivial element of the image in $\bZ \times \{0\}$. Thus $\Phi$ is virtually surjective. 

Finally, to see $H$ has index at most 2, suppose that $G$ contains an element $g$ which nontrivially permutes two boundary families (say, for concreteness, those comprised of stable leaves).  We claim that, in this case, no element $g'$ can nontrivially permute the unstable boundary families:  indeed, if this were the case, then $gg' \neq \id $ would be forced to have a fixed point in $U$, so $(gg')^2$ has a fixed point and preserves all boundaries, hence is the identity, a contradiction.  
\end{proof}

A converse statement to Lemma \ref{lem_action_of_stabilizer_scalloped} will be shown in Proposition \ref{prop_stabilizer_scalloped}.  

\section{A trichotomy} 

The following theorem, in the case of flows, is a combination of results of Barbot \cite[Théorème 4.1]{Bar95a} and Fenley \cite[Theorem 3.4]{Fen94} and \cite[Main Theorem]{Fen95}. 
We give here an easier argument based on Axiom \ref{Axiom_nonseparated}, which was only discovered later.  

 \begin{theorem}[Trichotomy for Anosov-like actions] \label{thm:trichotomy}
 Let $(P,\cF^+, \cF^-)$ be a bifoliated plane with an Anosov-like action of a group $G$. Then exactly one of the following holds:
  \begin{enumerate}[label=(\roman*)]
   \item  $(P,\cF^+, \cF^-)$ is trivial.
   \item  $(P,\cF^+, \cF^-)$ is skew. 
   \item There is either a singular point in $P$, or the leaf spaces of $\cF^+$ and $\cF^-$ are both non-Hausdorff with two-sided branching. \footnote{Two-sided branching is a definition which applies to transversely orientable foliations. This is always the case when $\cF^\pm$ have no prongs.  We say that branching is two-sided if, for a fixed orientation, there exist an increasing sequence $l_n$ with more than one limit leaf, and a decreasing sequence $k_n$ with more than one limit.}
  \end{enumerate}
 \end{theorem}

We note that Axiom \ref{Axiom_prongs_are_fixed} is not used in the proof, since if there is a singular leaf, this theorem is trivially true.
 
\begin{proof}
Suppose that there exists a pair of nonseparated leaves in one foliation. Up to exchanging the labels, say that this is $\cF^+$.  We will show that $\cF^-$ has two sided branching.  Thus, reversing the roles of $\cF^+$ and $\cF^-$ in the argument, we can then conclude that $\cF^+$ has two sided branching as well.  
This argument is enough to prove the theorem because if neither foliation contains nonseparated leaves, then there is either a prong singularity in $P$, or both leaf spaces are Hausdorff manifolds, hence the plane is skew or trivial by Proposition \ref{prop:both_R}. 

Suppose that $\cF^+$ has nonseparated leaves.  
By Proposition \ref{prop:4weak4strong}, two such leaves are sides of adjacent lozenges.  Call these leaves $s$ and $s'$ and denote the respective lozenges by $L$ and $L'$.  Let $f$ be the common side of $L$ and $L'$.  
For concreteness, we fix an orientation so that $\cF^-$ leaves through $L \cup L'$ are horizontal, the lozenge $L'$ lies above $L$, and the leaves $s$ and $s'$ are on the right side of $L \cup L'$, as shown in Figure \ref{fig_proof_trichotomy}. 
\begin{figure}[h]
\labellist
\small \hair 2pt
\pinlabel {$f$} at 20.370 134.964
\pinlabel {$L$} at 135.200 105.215
\pinlabel {$L'$} at 135.200 162.149
\pinlabel {$s$} at 216 250
\pinlabel {$s'$} at 216 14
\endlabellist
\includegraphics[width=5cm]{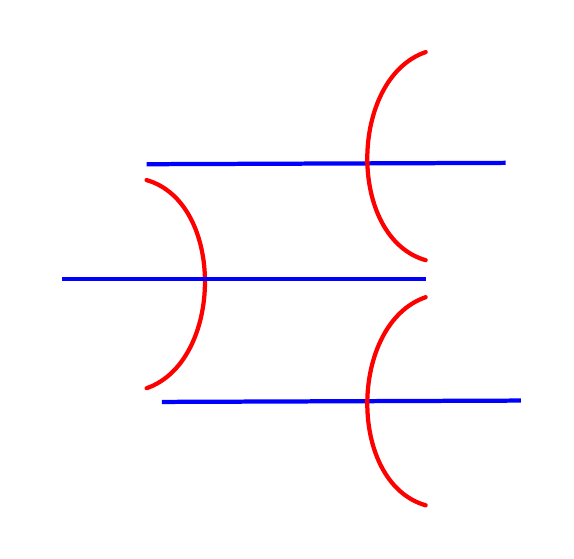}
\caption{Set-up towards finding nonseparated leaves in $\cF^-$}\label{fig_proof_trichotomy}
\end{figure}

Thus, with this orientation, we say a leaf $l$ of $\cF^-$ branches from below (respectively above) if there is a sequence $l_n$ approaching $l$ from below (resp.~above) with more than one leaf in its limit. 

We will find leaves branching from below.  Since the situation is symmetric by vertical reflection, the same argument gives the existence of leaves branching from above.  
We begin with a claim, illustrated in Figure \ref{fig_trichotomy_first_claim}. 

\begin{claim} \label{claim:fix_point_on_right}
Suppose $l_0$ is a leaf of $\cF^-$, $l_0 \cap L \neq \emptyset$, and $l_0$ contains a point $x$ to the right of $L$ fixed by a nontrivial element of $G$. 
Then, either some $\cF^-$-leaf intersecting $L$ branches from below, or $x$ is the corner of a lozenge with one side on $l_0$, the opposite side above $l_0$, and opposite corner to the left of $L$.  
\end{claim} 

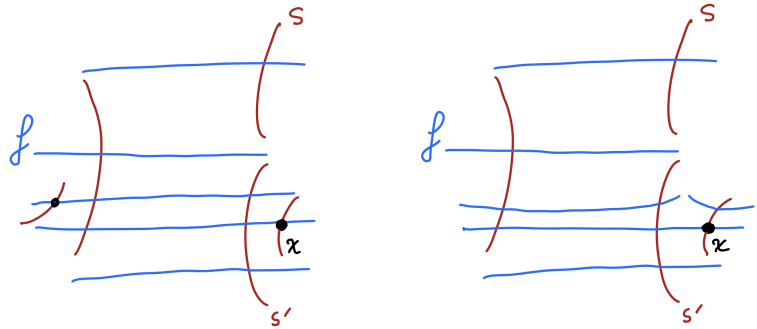
\begin{figure}[h]
\labellist
\small \hair 2pt
\pinlabel {$f$} at 34.332 134.095
\pinlabel {$s$} at 178.472 245.606
\pinlabel {$s'$} at 179.330 17.438
\pinlabel {$x$} at 193.058 81.771
\pinlabel {$f$} at 368.943 134.095
\pinlabel {$s$} at 494.207 245.606
\pinlabel {$s'$} at 494.207 17.438
\pinlabel {$x$} at 528.526 81.771
\endlabellist
\includegraphics[width=10cm]{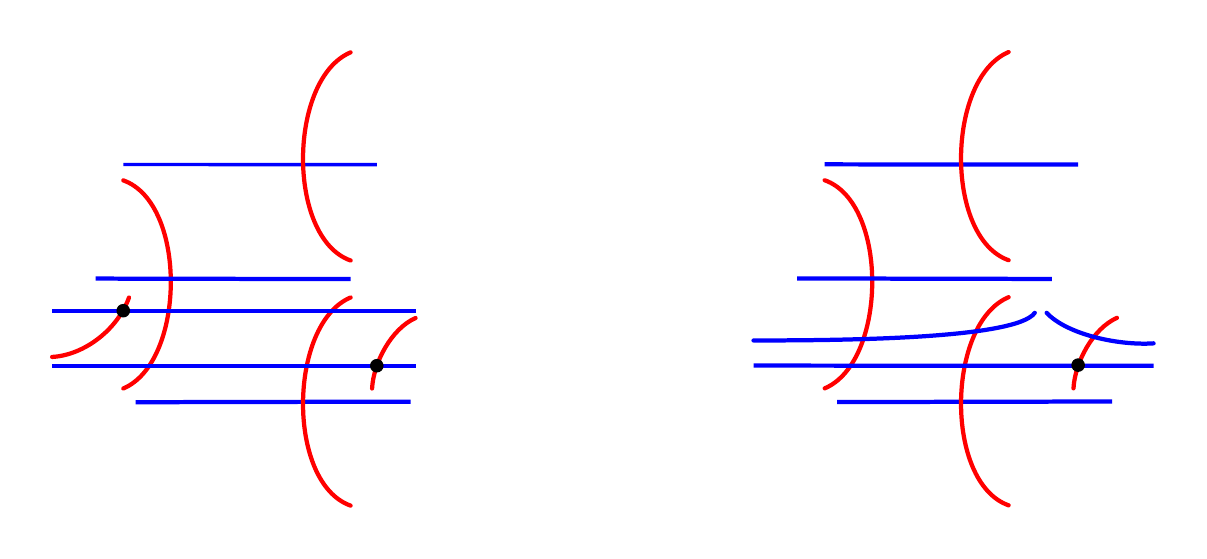}
\caption{The two possibilities in Claim \ref{claim:fix_point_on_right}.}\label{fig_trichotomy_first_claim}
\end{figure}

\begin{proof} 
Consider the set of leaves of $\cF^-$ which intersect both $\cF^+(x)$ and $s'$.  This is bounded above by $f$, so limits on to a leaf or union of nonseparated leaves.  Such a union cannot contain $f$: if it did, $x$ would be the corner of a lozenge adjacent to $L$, contradicting that $l_0$ intersects $L$.  If the limit is a union of nonseparated leaves, then we are done.  If the limit is a single leaf, it makes a perfect fit with $\cF^+(x)$.  Since $x$ is fixed by $g$, this is the corner of a lozenge with opposite corner on the opposite side of $L$.    
\end{proof} 

Similarly, we have: 
\begin{claim} \label{claim:fix_point_on_left}
Suppose $l_1$ is a leaf of $\cF^-$ that intersects $L$ and contains a point $y$ to the left of $L$ fixed by a nontrivial element of $G$. 

Then either some $\cF^-$-leaf meeting $L \cup L'$  branches from below, or $y$ is the corner of a lozenge with side $l_1$, opposite side above $l_1$, and opposite corner to the right of $L$.  
\end{claim}

\begin{proof} 
Suppose as a first case that $\cF^+(y) \cap f \neq \emptyset$.  Consider the set of leaves of $\cF^-$ which intersect both $\cF^+(y)$ and $s$.  This set is bounded, more precisely, it is bounded above by the top side of $L'$.   If its upper boundary consists of a nontrivial union of leaves, then we are done, having found an example of branching from below.  Otherwise, it has a unique boundary leaf, which must make a perfect fit with $\cF^+(y)$.  Since $y$ is fixed by a nontrivial element of $G$, Lemma \ref{lem:fixed_is_corner} says  $y$ is the corner of a lozenge, and the location of the perfect fit means this lozenge is in the upper right quadrant of $y$.  But this contradicts the non-corner criterion (Lemma \ref{lem_no_corner_criterion}).  We conclude the only possibility is that we have branching from below.  

If instead $\cF^+(y) \cap f = \emptyset$, then we may run the argument from Claim \ref{claim:fix_point_on_right}, and find either a branching leaf or a lozenge, as desired.  This concludes the proof of the claim.
\end{proof} 

Thus, as soon as we find a leaf of $\cF^-$ intersecting $L$ with a fixed point to the right (or left) of $L$, we can use these two claims to either find branching, or produce an infinite line of lozenges all contained in the $\cF^-$-saturation of $L$. 
In detail, suppose for concreteness that $l_0$ has a fixed point to the right.  Claim \ref{claim:fix_point_on_right} produces a lozenge above it with opposite side $l_1$ having a fixed point (the corner) on the left of $L$, Claim \ref{claim:fix_point_on_left} then produces an adjacent lozenge above that with opposite side $l_2$ having a fixed point on the right, etc.   Lemma \ref{lem:half_scalloped} says that the sequence $l_i$ limits to a union of branching leaves, and we have the desired branching from below. 

It remains only to find such a fixed point.  Consider any $\cF^-$ leaf $l$ intersecting $L$ fixed by some nontrivial $g \in G$ (which exists by Axiom \ref{Axiom_dense}).  If $l$ has a fixed point to the right or left, we are done.  If not, $g$ fixes a point in $L$. 
Since $L$ has a branching side, it is fixed by some nontrivial element $h \in G$. 
Up to replacing $g$ with $g^{-1}$, we can suppose $g$ is contracting on the $\cF^+$ leaf of its fixed point.  Thus $g(f)$ intersects $L$ and the image of the corner that lies on leaf $f$ is to the left of $L$.  The conjugate $ghg^{-1}$ fixes this corner, and we have the desired fixed point. 
\end{proof}

\section{Infinite product and infinite skew regions}  

At this point, we have built up enough structure theory for Anosov-like actions that we no longer restrict our use of the axioms and make the following standing assumption.  

\begin{convention} Going forward, $P$ always denotes a bifoliated plane with foliations $\cF^+, \cF^-$, equipped with an Anosov-like action of a group $G$.
\end{convention} 

Using the structure of nonseparated leaves, we next prove an important local-to-global type result.  It says that if $P$ contains an unbounded region resembling part of a trivial plane (called an {\em infinite product region}), then $P$ is in fact globally a trivial plane. 

 \begin{definition}[Infinite product region]
 An {\em infinite product region} is a subset of a bifoliated plane isomorphic via a proper homeomorphism to $I \times [0, \infty) \subset \bR^2$ with its product foliation; here $I$ may be any %
nontrivial interval, finite or infinite.    We say the interval $I$ is the {\em base} of the region.  
 \end{definition}

\begin{figure}[h]

	\centerline{ \mbox{
			\labellist
\small \hair 2pt
\pinlabel {$r_1^-$} at 264 235
\pinlabel {$r_2^-$} at 264 40
\pinlabel {$I^+$} at 90 144
\pinlabel {$\ldots$} at 447.352 138.954
\endlabellist
\includegraphics[width=6cm]{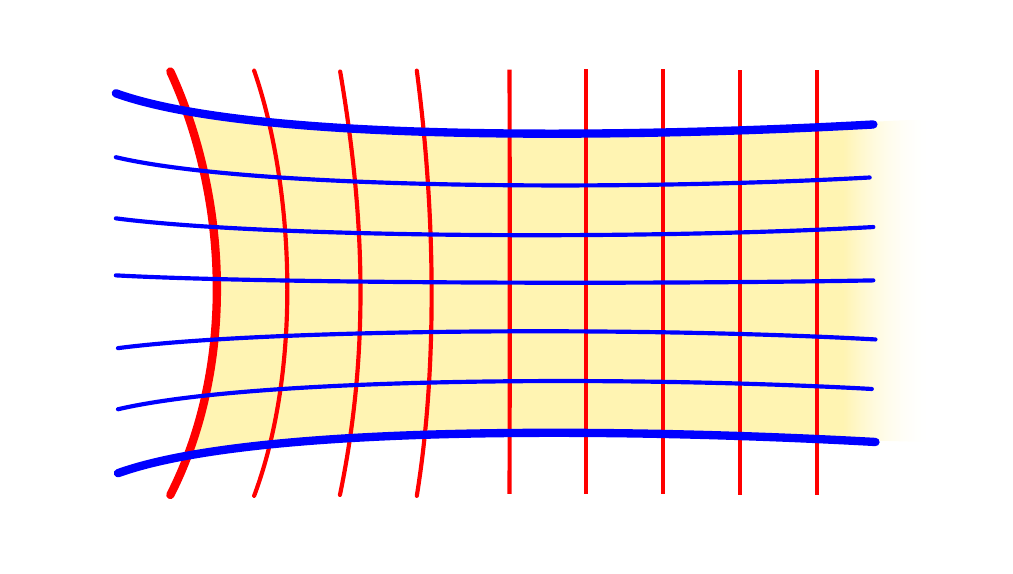}}}
	\caption{An infinite product region with base $I^+$} 
		\label{fig:product_region} 
\end{figure}
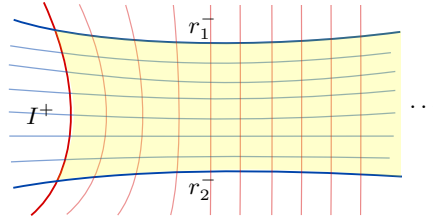

The following property was first proved by Fenley \cite{Fen98} for Anosov flows (and, previously, in \cite{Fen95} for transitive Anosov flows), and then in \cite{BFM22,BBM24b} for planes with Anosov-like actions.
 \begin{proposition}\label{prop:no_product_region}
 $P$ contains an infinite product region if and only if $P$ is trivial. 
 \end{proposition}
 
 \begin{proof} 
 Suppose $U$ is an infinite product region in $P$, based on an interval $I$. 
 Up to switching the labels $+$ and $-$, we assume $I$ lies in a leaf of $\cF^+$.  
 By axiom \ref{Axiom_dense} there is a leaf $l^-$ of $\cF^-$ intersecting $U$ and fixed by some nontrivial $g \in G$, which we assume also fixes all rays through its fixed point.    Let $x$ be the unique fixed point of $g$ on $l^-$.  Then the closure of $\bigcup_{k \in \mathbb{Z}} g^k(U)$ contains a $g$-invariant infinite product region based on $\cF^+(x)$; either equal to the closure of $\bigcup_{k \in \mathbb{Z}} g^k(U)$ if $x$ is outside of $u$, or (up to replacing $g$ with $g^{-1}$) equal to $\bigcap_{k \in \mathbb{N}} g^k(U)$ if $x$ is inside.  
 Let $V$ denote the interior of this region.  
Our first goal is to show that $V$ is a connected component of $P \setminus \cF^+(x)$.  (This is not automatic, for instance it would fail if $\cF^+(x)$ made a perfect fit with a leaf of $\cF^-$ on the side of $U$.) 
 
Let $y_i$ be a monotone sequence of points along  $\cF^+(x)$ which eventually leaves every compact set.  The $\cF^-$ leaves containing $y_i$ either escape to infinity or limit onto a finite or infinite union of leaves.  We claim that any leaf in the limit (if the limit exists) must be separated from $U$ by $\cF^+(x)$.   To see this, suppose first the limit is a finite collection.  Then any leaf $l_\infty$ of the limit is fixed by $g$, so contains a fixed point $z$ of $g$.   By Axiom \ref{Axiom_A1}, $\cF^+(z) \cap \cF^-(x) = \emptyset$, forcing $l_\infty$ to be on the opposite side of $\cF^+(x)$ from $U$.  
In the case where there is an infinite union of leaves in the limit, Proposition \ref{prop:4weak4strong} says these are sides of adjacent lozenges.  By Lemma \ref{lem:half_scalloped}, the parallel sides of these lozenges cannot accumulate anywhere in $U$ or onto $\cF^+(x)$, thus there is a first lozenge to the right of $\cF^+(x)$, which is therefore fixed by $g$. We conclude as in the previous case. 

We deduce that $V$ is a connected component of $P \setminus \cF^+(x)$.  Thus, the ray of $\cF^-(x)$ through $V$ is also the base of an infinite product region (on each side).   Consider a leaf $l$ of $\cF^+$ that is contained inside $V$ and has a fixed point $y$ (necessarily nonsingular) of some nontrivial element $g' \in G$.  Repeating the argument above shows that {\em both} connected components of $P \setminus \cF^-(y)$ are infinite product regions, and thus the plane is trivial.  
 \end{proof} 
 
 Using Proposition \ref{prop:no_product_region}, one can improve the non-corner criterion to the following characterization for fixed points; we leave the proof as an exercise.  
 
 \begin{exercise} \label{ex_noncorner_converse}
 Prove the following {\em non-corner characterization} for fixed points: 
 {\em  Suppose $P$ is not the trivial plane and $x\in P$ is fixed by some nontrivial element of $G$. 
 Then $x$ is the corner of a lozenge in a quadrant $Q$ if and only if the rays bounding this quadrant do not intersect a pair of leaves which make a perfect fit or are prongs of a singular point.}
\end{exercise}

 \begin{proof}[Hint:]
 Calling $r^+, r^-$ the two rays based at $x$ that bound $Q$, take a leaf $l^+$ intersecting $r^-$ and consider the boundary of the set $\cF^-(r^+)\cap\cF^-(l^+)$, using Proposition \ref{prop:no_product_region}.
 \end{proof}
 
Proposition \ref{prop:no_product_region} also allows us to improve Lemma \ref{lem:half_scalloped} to the following important statement.  
\begin{proposition}[Infinite lines are scalloped regions]  \label{prop:infinite_lozenge_is_scalloped}
If $L_0, L_1, L_2, \ldots$ is an infinite line of lozenges in $P$, then it belongs to a bi-infinite line of lozenges forming a scalloped region. 
\end{proposition}

\begin{proof}
Suppose $L_0, L_1, L_2, \ldots$ is an infinite line of lozenges, and let $f_i$ denote the side shared by $L_i$ and $L_{i+1}$.  Without loss of generality, suppose  $f_i$ are leaves of $\cF^+$.  

If the sequence $f_i$ does not converge, then the line is an infinite product region with base along a side of $L_0$.  By Proposition \ref{prop:no_product_region}, this is impossible.  
Thus, by Lemma \ref{lem:half_scalloped}, the $f_i$ must limit to a bi-infinite union of leaves.  Denote the leaves in the bi-infinite collection by $l_i^\infty$, indexed by $i \in \bZ$.  
By Proposition \ref{prop:4weak4strong}, the leaves $l_i^\infty$ form the sides of a line of adjacent lozenges.  Let $f_i^{-}$ denote the shared side of the adjacent lozenges bounded by leaves $l_i^\infty$ and $l_{i+1}^\infty$.  

Applying Lemma \ref{lem:half_scalloped} and Proposition \ref{prop:no_product_region}, we conclude that as $i \to \infty$  the sequence of leaves $f_i^{-}$ limits onto a bi-infinite union of leaves; and similarly as $i \to - \infty$.  Denote these two sets of limit leaves by $\{k^{+\infty}_j\}$ and $\{k^{-\infty}_j\}$, respectively.   These limits necessarily contain the $\cF^-$--sides of all the lozenges $L_i$ because the union of the lozenges $L_i$ is trivially foliated. 

By Proposition \ref{prop:4weak4strong}, the bi-infinite set $\{k^{+\infty}_j\}$ is the boundary of a bi-infinite line of lozenges containing the lozenges $L_i$, so we may extend our indexing from $\bN$ to include all $\bZ$.   Applying again the same reasoning as before, as $i \to - \infty$, the sides of $L_i$ limit to a bi-infinite union of leaves, say $\{l_i^{-\infty}\}$.  One can now verify easily that, together with the union of $\{l_i^\infty\}$ and $\{k^{\pm \infty}_j\}$, these leaves bound a product foliated region satisfying the definition of scalloped region.  
\end{proof} 

Using Proposition \ref{prop:4weak4strong}, which says that nonseparated leaves are associated to lines of lozenges, one can restate Proposition \ref{prop:infinite_lozenge_is_scalloped} as follows.  
\begin{corollary}\label{cor_infinite_non_sep_implies_scalloped}
Any infinite collection of pairwise nonseparated leaves is part of the boundary of a scalloped region.
\end{corollary}

To conclude this section, we give a variation on the ``local-to-global" statement that infinite product regions only occur in trivial planes.  Here, we show that a plane with Anosov-like group action which contains a ``large enough" subset isomorphic to a piece of the skew plane is necessarily skew.  
 
\begin{definition} \label{def:biinfinite_skew} 
A {\em bi-infinite skew region} is a subset $S$ of a bifoliated plane properly isomorphic to a subset of the skew plane of the form $\{(x, y) : c_1 < y < c_2 \}$ for some $c_1 < c_2$ (possibly with $c_1 = -\infty$ and/or $c_2 = \infty$).  
\end{definition} 

\begin{proposition}  \label{prop_biinfinite_skew} 
If $P$ contains a bi-infinite skew region, then $P$ is the skew plane.   
\end{proposition} 

\begin{proof} 
Suppose $P$ contains an bi-infinite skew region, and let $S$ be a {\em maximal} infinite bi-skew region in $P$.  
Up to relabeling we assume that in $S$, $\cF^+$ corresponds to the horizontal foliation $y=c$ of the skew plane.  
Let $l_0$ be a $\cF^+$ leaf in $S$ containing a fixed point $x$ of some nontrivial $g \in G$, fixing all rays through $x$.  If $\cF^-(x)$ is contained in $S$, then its endpoints make perfect fits, and so one obtains two lozenges sharing a corner in $S$, and the sides opposite $l_0$, say $l_{-1}$ and $l_1$, are also fixed by $g$ containing fixed points $x_1$ and $x_{-1}$.  Iterating this process, one obtains sequences of lozenges sharing corners, with sides $l_i \in \cF^+$, for $i = -k, -k+1, \ldots, 0, \ldots k-1, k$ and fixed points $x_i$.  

Either $S$ is unbounded in both directions and $P$ is isomorphic to the skew plane, or for some $i$, we have that $l_i$ contains $x$ but $\cF^-(x_i)$ is not contained in $S$.  But, in this latter case, either $g(S) \cup S$ or $g^{-1}(S) \cup S$ is a strictly larger infinite skew region, contradicting maximality.  
\end{proof} 

In fact, with a bit more work one can do slightly better.   Since this result is not used in the sequel, we leave it as a (challenging) exercise:

\begin{exercise}
Define an {\em infinite skew region} to be an open, nonempty subset of the skew plane of the form $\{(x, y) : c_1 < y < c_2 \text{ and } x > d \}$.   Show that if $P$ contains a subset properly isomorphic to an infinite skew region, then $P$ is skew.  
\end{exercise} 

\section{Chains of lozenges} 
We make the following definition to capture other possible configurations of lozenges. 

\begin{definition}[Chain of lozenges]\label{def_chain_lozenges}
A chain of lozenges is a union of closed lozenges that satisfies the following connectedness property: for any two lozenges $L$, $L'$ in the chain, there exist lozenges $L_0, L_1, \ldots L_k$ of the chain such that $L = L_0, L'= L_k$ and for all $i = 0, \ldots k-1$ the pair of lozenges $L_i$ and $L_{i+1}$ share a corner.
\end{definition}
Note that $L_i$ and $L_{i+1}$ may also (but are not required to) share a side.  The reason we take {\em closed lozenge} in the definition of chain is precisely so that a chain forms a connected subset of the plane.  It is also convenient when describing the global topology of a chain, for instance, if the lozenges in the chain form a line, then the interior of the chain is a trivially foliated region.  It also gives the following connectedness property on the level of leaves.   
\begin{exercise}  \label{ex_leaf_cap_chain_connected} 
If $\cC$ is a chain of lozenges and $l$ a leaf of $\cF^+$ or $\cF^-$, then $l \cap \cC$ is connected.\footnote{Hint: $\bR^2$ is simply connected.} 
\end{exercise}

The configuration of lozenges in a chain can be described combinatorially with a (fat) graph.  The following definition appears in \cite{BF13}; we will make use of it in the proof of Proposition \ref{prop_Z2_stabilize_minimal_chain} and in Section \ref{sec_periodic_seifert}. 
\begin{definition} \label{def:chain_tree} 
Let $\cC$ be a chain of lozenges.  The {\em tree of $\cC$} denoted $\cT(\cC)$ is the graph whose vertices are corners of lozenges in $\cC$, and two vertices are connected by an edge if they are corners of the same lozenge.  
\end{definition} 
\begin{exercise} Show that $\cT(\cC)$ is in fact a tree, and can also be endowed with a \emph{fatgraph} structure induced from $P$.\footnote{See \cite{BF13} for the solution.} 
\end{exercise}
Of particular use is that automorphisms of $P, \cF^\pm$ preserving $\cC$ induce automorphisms of $\cT(\cC)$. 
Note that $\cT(\cC)$ does not determine $\cC$.  For example, the tree associated to an infinite line of lozenges is a regular valence two graph, as is the tree associated to an infinite chain where lozenges share only vertices, not sides, called a {\em string}. 

\begin{definition} \label{def_string_lozenges}
A {\em string of lozenges} is a chain of lozenges where no two share a side.  That is, each adjacent pair share only a single corner.  
\end{definition} 

If two chains share a common lozenge, then their union is again a chain.  Thus, each chain of lozenges is contained in a maximal (with respect to inclusion) chain.   Connectedness of chains has the following useful consequence. 
\begin{observation} \label{obs:fixes_all_corners}
Suppose $g \in G$ fixes a corner $c$ of a lozenge $L$.  Then, $g$ preserves all rays of $c$ if and only if it fixes both corners of $L$.  Thus (inductively) $g$ fixes each corner of the maximal chain containing $L$.
\end{observation} 

Remarkably, Observation \ref{obs:fixes_all_corners} has a kind of converse.  This is due to Fenley \cite[Theorem 3.3]{Fen95b} for Anosov flows;  the statement and proof below for the case of Anosov-like actions is adapted from Fenley and appears in \cite[Proposition 2.24]{BFM22}. 

\begin{theorem} \label{thm:distinct_fix_is_chain}
If $g \neq \id$ fixes distinct points $x, y \in P$, then $x$ and $y$ are corners of a chain of lozenges.
\end{theorem} 

\begin{proof} 
Assume that $g$ fixes distinct points $x$ and $y$.
Let $Q$ be the quadrant of $x$ containing $y$, and consider the set $S$ of leaves of $\cF^{-}$ that intersect both 
$\cF^{+}(x)$ and $Q$.  
This set is $g$-invariant, indexed by points along the half-leaf of $\cF^{+}(x)$ bounding $Q$ and bounded on one side by $\cF^-(x)$. 
It must also have some additional boundary because $\cF^{\mp}(x) \cap \cF^{\pm}(y) = \emptyset$. 

There is a unique leaf $l$ of the boundary such that either $y\in l$, or $l$ separates $x$ from $y$.  Since $g$ fixes $x$ and $y$ and preserves $S$ (hence preserves its boundary), we also have $g(l) = l$.  By \ref{Axiom_A1} $g$ fixes some point $x_1$ on $l$.  The set $\{l^- \in S : \cF^+(x_1) \cap l^- \neq \emptyset \}$ is $g$-invariant and contains an open neighborhood of $l$ in $S$ (considered as a subset of the leaf space), so by the hyperbolicity of the action of $g$ on $\cF^{+}(x)$ given by  \ref{Axiom_A1}, the set is equal to $S$. 

Since there are no infinite product regions (Proposition \ref{prop:no_product_region}), $\cF^+(x_1)$ makes a perfect fit with either $\cF^-(x)$ or with a leaf nonseparated with $\cF^-(x)$. By Lemma \ref{lem:fixed_is_corner}, in the first case $x$ and $x_1$ are the two corners of a lozenge; in the latter case, by Proposition \ref{prop:4weak4strong}  they are corners of a line of lozenges. 

If $x_1=y$, then we are done. Otherwise we can iterate the process above and find a sequence $x_n$ of corners, all fixed by $g$, in a chain of lozenges starting at $x$ and disjoint from $\cF^{\pm}(y)$.   We need only show that this process must terminate.  
If not, one obtains an infinite sequence of leaves $\cF^+(x_n)$, disjoint from $\cF^+(y)$ and $\cF^-(y)$, and lying in a compact subset of the leaf space of $\cF^+$.   Consider a convergent subsequence limiting to a leaf or union of leaves.  As in the first part of the argument, 
exactly one of these limit leaves, call it $l'$, separates $x_n$ from $y$.  Since $\cF^+(x_n)$ and $\cF^+(y)$ are all fixed by $g$, the leaf $l'$ is as well, and thus contains a fixed point $x_\infty$ of $g$.  But then $\cF^-(x_\infty)$ will intersect the fixed leaf $\cF^+(x_n)$ for sufficiently large $n$, contradicting uniqueness of fixed points given by Axiom \ref{Axiom_A1}.
\end{proof} 

Combining this with Observation \ref{obs:fixes_all_corners} we immediately obtain:
\begin{corollary} \label{cor:maximal_fixed}
If $g \neq \id$ fixes all rays through a point $x$, then the set of fixed points of $g$ is either equal to the singleton $\{x\}$, or is the set of corners of the maximal chain of lozenges containing $x$. 

Furthermore, for any element $g$ preserving some lozenge in $P$, there is a {\em unique} maximal chain of lozenges invariant under $g$.  
\end{corollary}

\subsection{Pivots and finite branching}  \label{subsec_pivot}
We conclude this section with a discreteness result for lozenges sharing sides, which has as a consequence an important theorem of Fenley, \cite[Theorem F]{Fen98}.  Fenley's theorem says that, for any pseudo-Anosov flow on a compact 3-manifold, there are only finitely many branching leaves in the weak foliations $\cF^s$ and $\cF^u$.  

Following \cite{Fen98}, we call the common corner of two lozenges sharing a side a \emph{pivot} point.

\begin{proposition}\label{prop_pivots_discrete}
The set of pivot points is a closed, discrete subset of $P$.
\end{proposition}
\begin{proof}
Suppose that $x_n$ is a sequence of pivot points converging to $x$.   After passing to a subsequence, we can assume that $x_n$ lie in a single quadrant $Q$ of $x$, and passing to the tail end we assume that all are in a trivially foliated subset of this quadrant, in particular they are nonsingular. Up to taking a further subsequence (and potentially relabelling the foliations) we may assume that the leaves $\cF^+(x_n)$ are always the shared sides of two lozenges. 

For concreteness, fix a local orientation on $Q$ in a product-foliated neighborhood of $x$ so that $Q$ is the upper-left quadrant of $x$ and the local foliations are represented with $\cF^+$ as vertical and $\cF^-$ horizontal.  After passing to a further subsequence, we can assume that the pivots all have the same orientation, that is to say that the 
perfect fits of all the $\cF^+(x_n)$ leaves are either all above $\cF^-(x)$ or all below $\cF^-(x)$, as illustrated in Figure \ref{fig_pivot_discrete}. We will show that each case contradicts the non-corner criterion. 

Suppose first that the perfect fits are above.   Note that the fact that all $x_i$ are in a trivially foliated neighborhood of $x$ in $Q$ implies that each of $\cF^\pm(x_i)$ intersects $\cF^\mp(x)$ (respectively). 
Take $n$ large enough so that $x_n$ lies in the rectangle bounded by $\cF^\pm(x_1)$ and $\cF^\pm(x)$.  Then the perfect fit formed by $\cF^+(x_1)$ and the side of the rightmost lozenge with corner $x_1$ is in the top left quadrant of $x_n$, so, by Lemma \ref{lem_no_corner_criterion}, $x_n$ cannot have a lozenge there, a contradiction.  Similarly, if the perfect fits are all below, $x_n$ cannot have a lozenge in its lower left quadrant, again a contradiction.  

\begin{figure}[h]
\labellist
\small \hair 2pt
\pinlabel {$x_1$} at 93.596 174.393
\pinlabel {$x_n$} at 184.920 133.804
\pinlabel {$x$} at 237.910 83.067
\pinlabel {$x_1$} at 508.499 253.317
\pinlabel {$x_n$} at 595.313 209.345
\pinlabel {$x$} at 651.686 130.421
\endlabellist
\includegraphics[width=10cm]{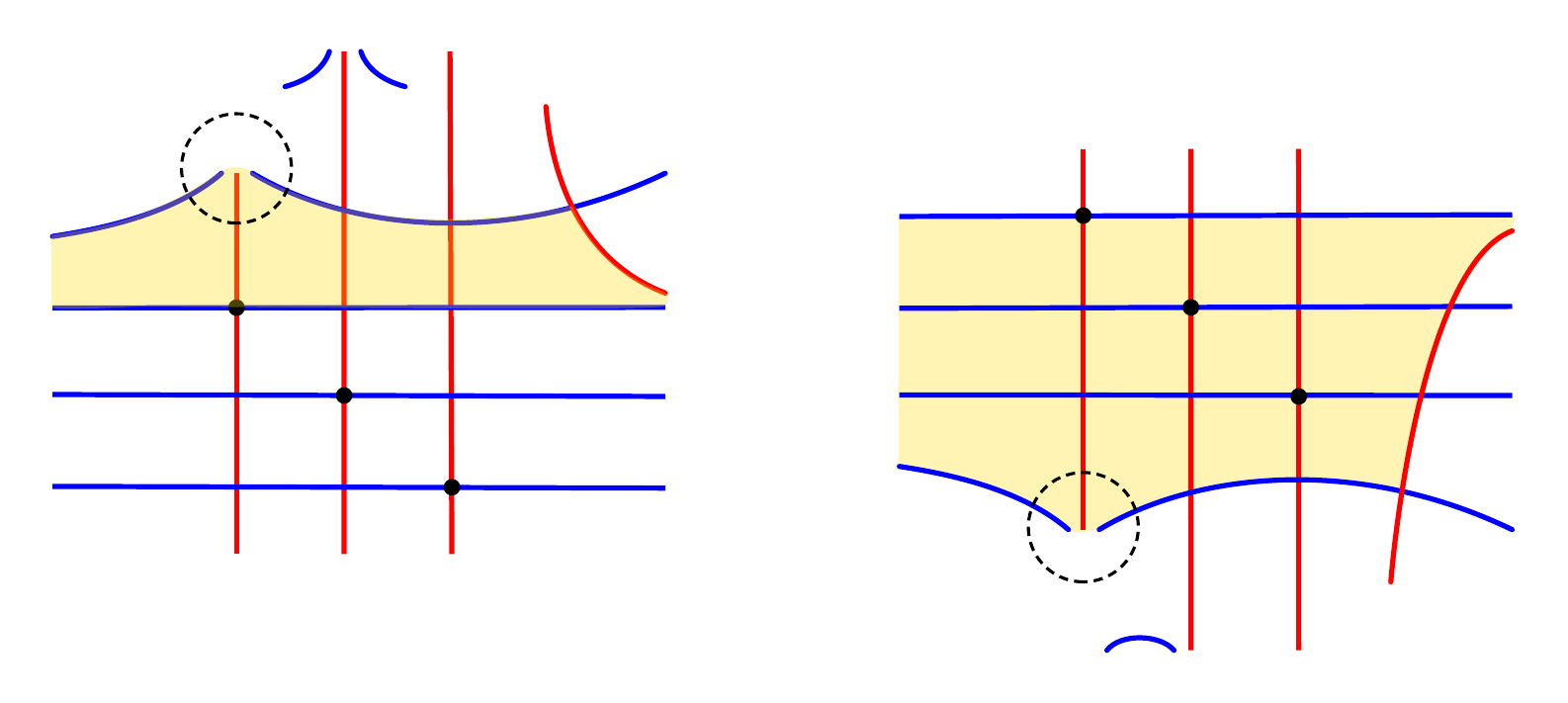} 
\caption{Two cases for the proof of Proposition \ref{prop_pivots_discrete}}
\label{fig_pivot_discrete}
\end{figure}
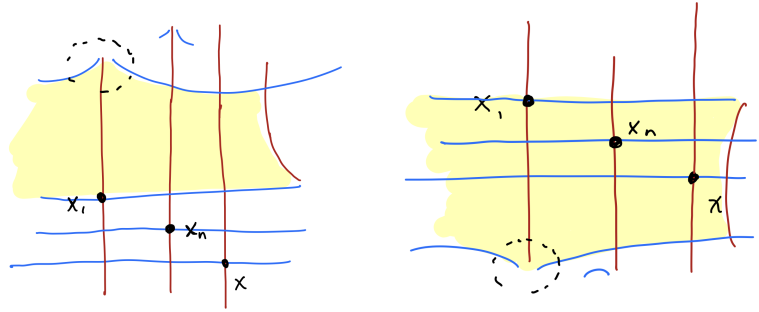

\end{proof}

A direct corollary of Proposition \ref{prop_pivots_discrete} is as follows.  
\begin{corollary}\label{cor_cocompactness_imply_finitely_many_branching}
Suppose $G$ acts cocompactly on $P$.  Then there are finitely many distinct orbits of $\mathrm{Pivot}$ under the action of $G$.
\end{corollary}
The cocompactness hypothesis always holds for the orbit space action induced from a pseudo-Anosov flow on a compact manifold.  Thus, we recover Fenley's finiteness result of \cite[Theorem F]{Fen98}.

\section{Transitivity} \label{sec_transitive}
A group action is called \emph{topologically transitive} if it admits a dense orbit.  This applies both to flows (actions of $\bR$), and to actions of discrete groups. In the setting of pseudo-Anosov flows on 3-manifolds and their induced actions on orbit spaces, these notions of transitivity are closely related.  

We start by recalling a classical result for smooth Anosov flows (see e.g., \cite[Theorem 5.3.50]{FH19}):
\begin{theorem} \label{thm_characterize_transitive_flows}
Let $\flow$ be a smooth Anosov flow on $M$. The following are equivalent: 
\begin{enumerate}[label=(\roman*)]
\item $\flow$ is transitive.
\item The set of periodic orbits of $\flow$ is dense.
\item Each leaf of $\fs$ (resp.~$\fu$) is dense in $M$.
\end{enumerate}
\end{theorem}
Translated to the orbit space, Theorem \ref{thm_characterize_transitive_flows} says that topological transitivity of the induced action of a smooth Anosov flow on its orbit space is equivalent to having a dense set of points in $\orb$ fixed by nontrivial elements of $\pi_1(M)$, and also equivalent to the density of the orbit of every leaf of $\bfs$ and of $\bfu$.  

In this section, we recover this characterization for general Anosov-like actions.  Translating this back to the flows setting, it shows that Theorem \ref{thm_characterize_transitive_flows} is true not only for smooth Anosov flows, but for topological pseudo-Anosov flows as well.  We prove the following.  

\begin{theorem}\label{thm_characterization_transitive}
Suppose $G$ acts Anosov-like on $P$.  The following are equivalent:
\begin{enumerate}[label=(\roman*)]
\item \label{item_charac_top_transitive} The action of $G$ is topologically transitive;
\item \label{item_charac_dense_fixed_points} The set of points fixed by nontrivial elements of $G$ is dense in $P$;
\item \label{item_charac_dense_leaves} For any leaf $l$ of $\cF^\pm$, $G\cdot l$ is dense in $P$.
\end{enumerate}
\end{theorem}
Here $G \cdot l$ denotes the set $\bigcup \{ g(l) : g \in G\} \subset P$. 

When $P$ is not skew, one can make condition \ref{item_charac_dense_fixed_points} even more specific:
\begin{proposition}[\cite{BFM22}]\label{prop_transitive_non_corner_dense}
Suppose $P$ is not a skew plane. An Anosov-like action of $G$ is topologically transitive if and only if the set of \emph{non-corner} points\footnote{Here and in what follows we say a point is {\em non-corner} if it is not the corner of any lozenge} fixed by nontrivial elements of $G$ is dense in $P$.  
\end{proposition}
This proposition is proved in Section \ref{sec_noncornerdense}, after the proof of Theorem \ref{thm_characterization_transitive} in Section \ref{sec_char_transtitive}.  

The implication \ref{item_charac_dense_fixed_points}$\Rightarrow$\ref{item_charac_dense_leaves} in Theorem \ref{thm_characterization_transitive} is a simple consequence of the connectedness of the plane: Given any leaf $l$, the closure of $G\cdot l$ is a $G$-invariant subset of $P$.    If the closure of $G\cdot l$ is not all of $P$, then we can find an open set $U$ in the complement such that $\cF^+(U) \cap (G\cdot l) \neq \emptyset$.   If also the set of fixed points of elements is dense in $P$, some nontrivial $g \in G$ has a fixed point in $U$, applying iterates of $g$ to $G \cdot l$, one sees that $G \cdot l$ must in fact accumulate on the fixed point of $g$, giving a contradiction.  

In order to prove the more difficult equivalence between \ref{item_charac_dense_fixed_points} and \ref{item_charac_top_transitive}, as well as the implication \ref{item_charac_dense_leaves}$\Rightarrow$\ref{item_charac_dense_fixed_points}, we will need to introduce the notion of {\em Smale classes} and the {\em Smale order}. 

\subsection{Smale order and Smale classes}

Smale \cite{Sma67} proved that the non-wandering set of a hyperbolic dynamical system can be decomposed into a family of \emph{basic sets} with a partial order, called the Smale order.   See \cite[Section 5.3]{FH19} for a description of this theory.  Here we describe the generalization of these structures to the setting of bifoliated planes with Anosov-like actions, which was introduced in \cite{BBM24b}. 
In the case of orbit space of flows, when $x$ and $y$ represent periodic orbits, then $x \leq y$ in the Smale order if there is an orbit which spirals onto $x$ in the past and $y$ in the future.  %
This is extended to points in the closure of such fixed points:

\begin{notation} 
Let $\Fix_G := \{x \in P : \exists g \neq \id \text{ with } g(x) =x\}$, and denote the closure of $\Fix_G$ in $P$ by $\Fixbar_G$.
\end{notation} 

Classically, the theory of Smale classes was only developed for smooth Anosov rather than (topological) pseudo-Anosov flows.  Certain unexpected pathologies can arise in the presence of prongs -- see Remark \ref{rem_nonwandering_vs_chain_recurrent} below.  In order to define Smale classes in a general setting, we need to restrict our attention to the points of $\Fixbar_G$ that are not on singular leaves.

\begin{definition} 
The {\em regular set}  $\cR_G$ is the subset of $\Fixbar_G$ consisting of points $x\in \Fixbar_G$ such that no leaf of $\cF^{\pm}(x)$ is a singular leaf. Equivalently, $\cR_G$ is obtained from $\Fixbar_G$ by removing all prongs and their leaves.
\end{definition}

\begin{rem} \label{rem_nonwandering_vs_chain_recurrent}
When $P$ is the orbit space of an Anosov flow on a 3-manifold $M$, and  $G =\pi_1(M)$ with the induced action, then 
 $\Fixbar_G = \cR_G$ and it is equal to the non-wandering set of the action of $G$. It is also equal to the projection to the orbit space of the non-wandering set of the flow, which is the chain-recurrent set.
 
However, in the presence of prong singularities (even for actions coming from pseudo-Anosov flows on compact manifolds)
the non-wandering set may be strictly smaller than the chain recurrent set. Some examples are given in \cite[Section 8]{BBM24b}.  In these cases, $\Fixbar_G$ corresponds to the projection of the non-wandering set. 
\end{rem}

For Anosov-like actions, the interesting dynamics happens in the set $\Fixbar_G$.  Working in the regular set $\cR_G$ is necessary to deal with the unusual behavior that arises in the presence of prongs (as noted in Remark \ref{rem_nonwandering_vs_chain_recurrent} above), but it may happen that the closure of $\cR_G$ misses parts of $\Fixbar_G$, and thus misses parts of the dynamics.  Thus, we introduce some terminology for points of $\Fixbar_G\smallsetminus \bar{\cR}_G$ and will keep track of them separately.  

\begin{definition}
A prong $p\in P$ is called an \emph{isolated prong singularity} if $p$ is an isolated point of $\Fixbar_G$. Equivalently, isolated prong singularities are the points in $\Fixbar_G\smallsetminus \bar{\cR}_G$.
\end{definition}

\begin{definition}[Smale order, Smale class] \label{def:smale_class}
For $x,y\in \cR_G$, we write $x\lG y$ if there exists $g\in G$ such that $\cF^+(x)\cap \cF^-(gy)\neq\emptyset$.
If $x\lG y$ and $y\lG x$, we say that $x,y\in \cR_G$ are in the same {\em Smale class}, and write $x\sim_G y$. 

We extend this relation to isolated prong singularities, by saying $x \sim_G y$ if $x = gy$ for some $g \in G$.  
\end{definition}

Notice that, by definition, if $x \lG y$, then $x \lG gy$ for all $g \in G$.  In particular, Smale classes are $G$-invariant sets.

\begin{lemma}\label{lem_lG_reflexive_transitive}
The relation $\lG$ is reflexive and transitive on $\cR_G$.  Consequently, the relation $\sim_G$ is an equivalence relation, and $\lG$ induces a partial order on the Smale classes that do not consist of the orbit of an isolated prong singularity.
\end{lemma}
\begin{proof}
Reflexivity is trivial, taking $g = id$, and symmetry of $\sim_G$ holds by definition. 
For transitivity, there is nothing to check on isolated prongs, so 
suppose $x,y,z\in \cR_G$ satisfy $\cF^+(x)\cap \cF^-(g_1y)\neq\emptyset$ and $\cF^+(y)\cap \cF^-(g_2 z)\neq\emptyset$, so also $\cF^+(g_1 y)\cap \cF^-(g_1 g_2 z)\neq\emptyset$.  Let $y' \in \cR_G$ be a point fixed by some nontrivial element $g \in G$ and sufficiently close to $g_1 y$ so that $\cF^+(x)\cap \cF^-(y') \neq \emptyset$ and $\cF^+(y')\cap \cF^-(g_1 g_2 z)\neq\emptyset$.  Then (up to replacing $g$ with $g^{-1}$) we will have
that $\cF^-(g^n g_1 g_2 z) \cap \cF^+(x) \neq \emptyset$ as desired. 
\end{proof}

What will be important for us is that the Smale classes of an Anosov-like action have a {\em product structure} in the same sense that classical hyperbolic basic sets do -- see e.g. \cite[Definition 6.2.6]{FH19}. 

\begin{proposition}[Product structure on Smale classes] \label{prop_basic_product_structure}
 Let $a,b\in \cR_G$ be in the same Smale class $\Lambda$. Assume that $\cF^+(a)\cap \cF^-(b) \neq \emptyset$. 
Then, for any neighborhood $U$ of $\cF^+(a)\cap \cF^-(b)$, we have $U\cap \Lambda \neq \emptyset$. 
In particular, the intersection point $x= \cF^+(a)\cap \cF^-(b)$ is in $\cR_G$ and in the same Smale class as $a$ and $b$.
\end{proposition}

A slightly more general version of this is proved in \cite[Section 3]{BBM24b}.

\begin{proof}
Suppose $a \sim_G b$ and $\cF^+(a)\cap \cF^-(b) \neq \emptyset$.  Let $U$ be a small product foliated neighborhood of $\cF^+(a)\cap \cF^-(b)$.   Since $a \sim_G b$, there exists $g \in G$ with $g( \cF^+(b)) \cap \cF^-(a) \neq \emptyset$.  

We will find a point of $\Fix_G$ in $U$.  Since $U$ is trivially foliated, this point will be nonsingular, and its leaves will intersect both $\cF^+(a)$ and $\cF^-(b)$, so it will be in the same Smale class as $a$ and $b$, proving the Proposition. 

First we reduce to the case where $a$ and $b$ are in $\Fix_G$.  If either is not in $\Fix_G$, we can replace them by arbitrarily nearby (nonsingular) points which are in $\Fix_G$, while preserving the fact that $\cF^+(a)\cap \cF^-(b)\neq \emptyset$ and $\cF^-(a)\cap g\cF^+(b)\neq\emptyset$. Thus, going forward we make the assumption that $a, b \in \Fix_G$.  
 
 Let $\alpha$ and $\beta$ be elements fixing $a$ and $b$, respectively.  Up to replacing with powers and inverses, we assume they fix all rays of $\cF^\pm(a)$ and $\cF^\pm(b)$, respectively, and are expanding on $\cF^+(a)$ and $\cF^+(b)$. 
 Let $I$ be a small interval of leaves of $\cF^+$ containing $\cF^+(a)$ in its interior, and choose this interval small enough so that  $I \subset \cF^+(U)$.  
 Then, for any large enough $n,m>0$, we will have that $g \beta^n(I)$ consists of the leaves passing through a small interval of $\cF^-(a)$.  Thus, $\alpha^m g \beta^n(I) \subset I$.  We conclude that  $h= \alpha^m g \beta^n$ fixes a leaf of $\cF^+$ in $I$. 
 
Similarly, if we consider an interval $J$ of leaves of $\cF^-$ that contains $\cF^-(b)$ in its interior and lies inside $\cF^-(U)$, we will have that (up to increasing $n,m$ if necessary) $h^{-1}(J) \subset J$. Therefore $h$ has a fixed point in $U$, which is what we needed to show. 
\end{proof}

\begin{figure}[h]
     \centerline{ \mbox{
\labellist
\small \hair 2pt
\pinlabel {$a$} at 113.050 208.289
\pinlabel {$b$} at 353 118
\pinlabel {$gb$} at 136.914 50.754
\pinlabel {$I$} at 384.974 196.992
\pinlabel {$\beta^n(I)$} at 383.718 161.217
\pinlabel {$g\beta^n(I)$} at 45 86
\endlabellist
\includegraphics[width=8cm]{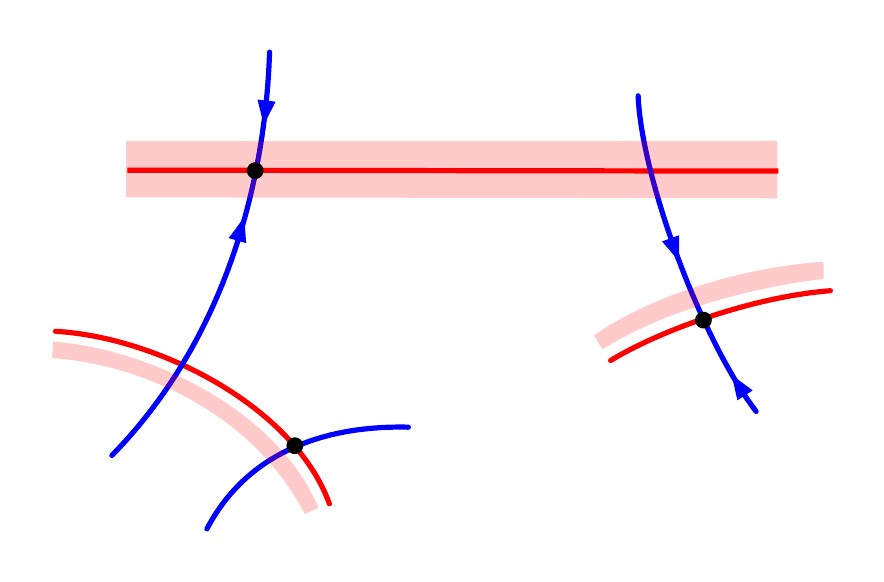} }}
\caption{Proof of local product structure}
\label{fig:product_structure}
\end{figure}
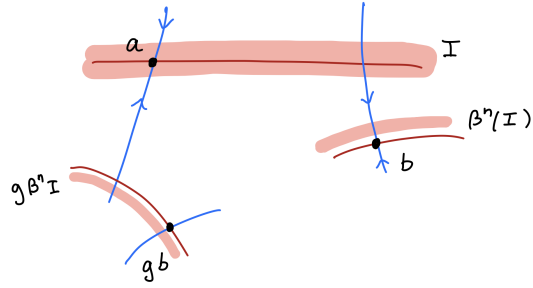

\subsection{Proof of Theorem \ref{thm_characterization_transitive}} \label{sec_char_transtitive}
With the tools above, we can now prove the remaining implications in the theorem.  We start by showing  \ref{item_charac_dense_fixed_points}$\Rightarrow$\ref{item_charac_top_transitive}, while also showing that density of fixed points is equivalent to having a unique Smale class.

\begin{theorem}\label{thm:transitive_on_Smale_class} 
If $G$ has an Anosov-like action on a bifoliated plane, then $G$ acts topologically transitively on each Smale class.   
\end{theorem}

In the proof we use the following standard characterization of topologically transitive actions. 
\begin{lemma} \label{lem_dense_baire}
Let $X$ be a second countable Baire space with an action of a group $G$ by homeomorphisms.  Suppose that for any nonempty open sets $U, V$ there exists $g \in G$ such that $gU \cap V \neq \emptyset$.  Then $G$ has a dense orbit in $X$. 
\end{lemma}
\begin{proof} 
Let $\{U_i\}$ be a countable basis for $X$, and let $W_i = GU_i$.  Since $G$ is transitive, each $W_i$ is an open dense set.  Since $X$ is a Baire space, the intersection of the $W_i$ is nonempty, and it is now straightforward to check each point in the intersection has dense orbit.  
\end{proof}

\begin{proof}[Proof of Theorem \ref{thm:transitive_on_Smale_class}]
Let $\Lambda$ be a Smale class.   If $\Lambda$ consists of isolated prong singularities, then it is a single orbit, and the action of $G$ is transitive, so we assume this is not the case.  
By definition, the set of singular points is a closed, discrete set of $P$ so there are only countably many singular leaves. 
Thus the complement of the singular leaves is a dense, $G_\delta$-subset.   $\Lambda$ is a closed subset of the complement of the singular leaves, hence a Baire space.  Thus, we can apply the criterion from Lemma \ref{lem_dense_baire}: to obtain a dense orbit, it suffices to show that for any open sets $U, V$ such that $U\cap \Lambda \neq \emptyset$ and $V\cap \Lambda \neq \emptyset$, there exists $g\in G$ such that $gU \cap V \cap \Lambda \neq \emptyset$.    We now do this.   

Let $U$ and $V$ be open subsets of the complement of the singular points in $P$.  Shrinking these if necessary we can assume both are product foliated.  
Let $a\in U \cap \Lambda$ and $b\in V\cap\Lambda$. Since $a\sim_G b$, there exists $h \in G$ such that $\cF^{+}(a) \cap \cF^-(hb) \neq \emptyset$. If $a = hb$, then we are done. Otherwise, Proposition \ref{prop_basic_product_structure} gives us the existence of a nonsingular point fixed by some element of $G$ in any neighborhood of $\cF^{+}(a) \cap \cF^-(hb)$.  In particular, we can find such a point $z$ with $\cF^+(z) \cap U \neq \emptyset$ and $\cF^-(z) \cap hV \neq\emptyset $.   If $g$ is the element fixing $z$, then for some $n$ sufficiently large (either positive or negative), we will have $g^n U \cap hV \neq\emptyset$. 
Since $U$ and $V$ are both product foliated, $g^n U \cap hV$ contains the intersection of $\cF^-(g^n a)$ with $\cF^+(hb)$.  Proposition \ref{prop_basic_product_structure} now implies that $g^n U \cap hV$ must contain points of $\Lambda$.  This shows $G$ acts topologically transitively on $\Lambda$.
\end{proof} 

Next we will prove that an action has a unique Smale class if and only if $\Fixbar_G = P$.  For this we need the following lemma to treat the case where the only fixed points are isolated prong singularities.  This behavior can in fact occur for Anosov-like actions, in examples constructed from pseudo-Anosov flows on non-compact 3-manifolds.  

\begin{lemma}  \label{lem_only_prongs}
If $\Fixbar_G$ consists only of isolated prong singularities, then these lie in at least two $G$-orbits.  Consequently, there are at least two Smale classes. 
\end{lemma} 

\begin{proof} 
Suppose that $\Fixbar_G$ consists only of isolated prong singularities.  Let $x$ be such an isolated prong, fixed by 
 $g \in G$.  We claim that each quadrant of $x$ must be a lozenge.  To see this, take two rays $r^+$ and $r^-$ at $x$ bounding a quadrant $Q$.  Let $r^+_1$ be another ray intersecting $Q$ and fixed by some element $g_1\in G$, which exists by Axiom \ref{Axiom_dense}.  The set  $\cF^-(r^+)\cap \cF^-(r^+_1) \cap Q$ is bounded on one side by $\cF^-(x)$.  For the other boundary, there are four possible cases, as shown in Figure \ref{fig:four_cases}: either it contains a set of non-separated leaves, or a singular leaf intersecting both $r^+$ and $r^+_1$, or a leaf making a perfect fit with $r_1^+$, or a leaf making a perfect fit with $r^+$.  In that last case, we deduce (from Lemma \ref{lem:fixed_is_corner}) that $x$ is a corner of a lozenge contained in $Q$ as claimed. So we need to eliminate each of the other options.

\begin{figure}[h]
\includegraphics[width=6cm]{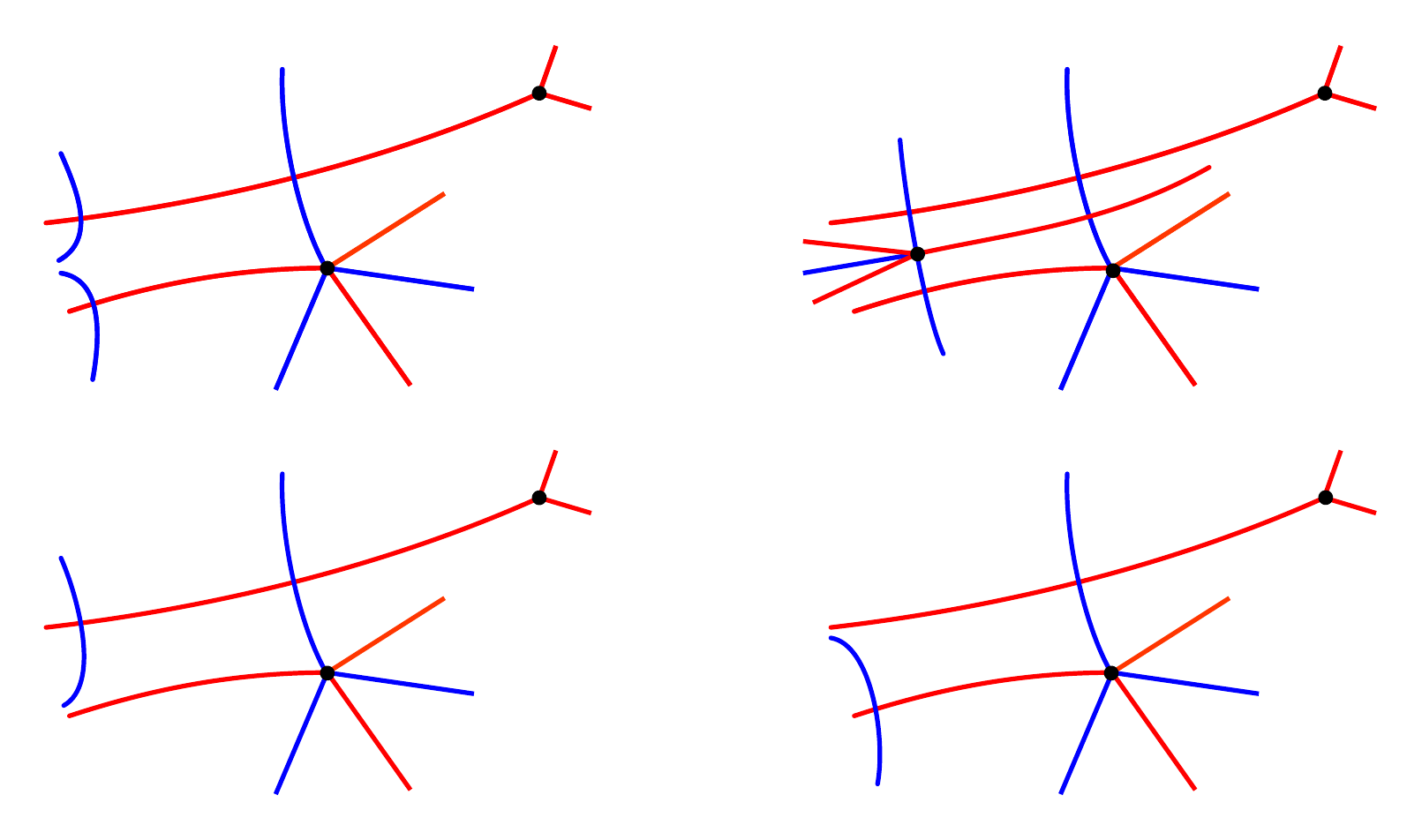}
\caption{The four possible cases}
\label{fig:four_cases}
\end{figure}

First suppose the boundary contains a set of nonseparated leaves.  Then each contains a fixed point of some $h \in G$.  Consider such a fixed point $y$ so that $\cF^+(y) \cap r^-(x) \neq \emptyset$ and $y$ is the closest such point to $x$.  Up to passing to powers, we can assume that $g$ preserves $Q$ and contracts $r^-(x)$, and $h$ preserves the quadrant containing $x$, contracting $\cF^-(y)$.  Then $gh$ contracts the set of $\cF^-$ leaves between $x$ and $r^- \cap \cF^+(y)$, so has a fixed leaf in the interior of this segment, and $h^{-1}g^{-1}$ similarly has a fixed leaf in $\cF^-(r^+)$.  These must meet at a fixed point for $gh$ in a trivially foliated region, contradicting the assumption that $\Fixbar_G$ consists only of isolated prong singularities. 

Similarly, if the boundary of $\cF^-(r^+)\cap \cF^-(r^+_1) \cap Q$ contains a singular leaf $l^-$ with fixed point $y$, that intersects both $r^+$ and $r^+_1$, then $\cF^\pm(y)$ and $\cF^\pm(x)$ bound a trivially foliated rectangle, and the argument from the previous case produces a fixed point in this rectangle.   Finally, if a leaf makes a perfect fit with $r^+_1$, %
this perfect fit is a part of a lozenge $L$ fixed by some $h \in G$. Since $x$ is a prong singularity, $x \notin L$ and thus one corner $c$ of $L$ must satisfy that $\cF^\pm(c)$ and $\cF^\pm(x)$ bound a trivially foliated rectangle, and the contradiction proceeds again as above.   We thus conclude that each singular point has a lozenge in each quadrant.

Proceeding now with the main argument of the proof, since each fixed point has a lozenge in all of its quadrants, each lozenge is included in a scalloped region in $P$.  Fix such a lozenge $L$ with corner $a$, in scalloped region $S$ and enumerate the quadrants of $a$ by $Q_1, Q_2, \ldots Q_k$, so that $L \subset Q_1$.   For concreteness, fix an orientation on this scalloped region $S$ so that $a$ is the upper right corner of $L$.  Since $S$ can be realized in two ways as a line of lozenges, there exists $L' \neq L$ in $S$ with $L \cap L' \neq \emptyset$, and also with a corner, say $b$ in the upper right, as shown in Figure \ref{fig_singular_point_orbit}. 

\begin{figure}[h]
\labellist
\small \hair 2pt
\pinlabel {$a$} at 294.366 240.815
\pinlabel {$b$} at 454.319 204.835
\pinlabel {$Q_1$} at 253.668 227.559
\pinlabel {$L$} at 253.668 80.798
\pinlabel {$L'$} at 200.666 170.748
\pinlabel {\tiny $L''$} at 527.197 151
\pinlabel {$S'$} at 577.360 164.121
\endlabellist
\includegraphics[width=10cm]{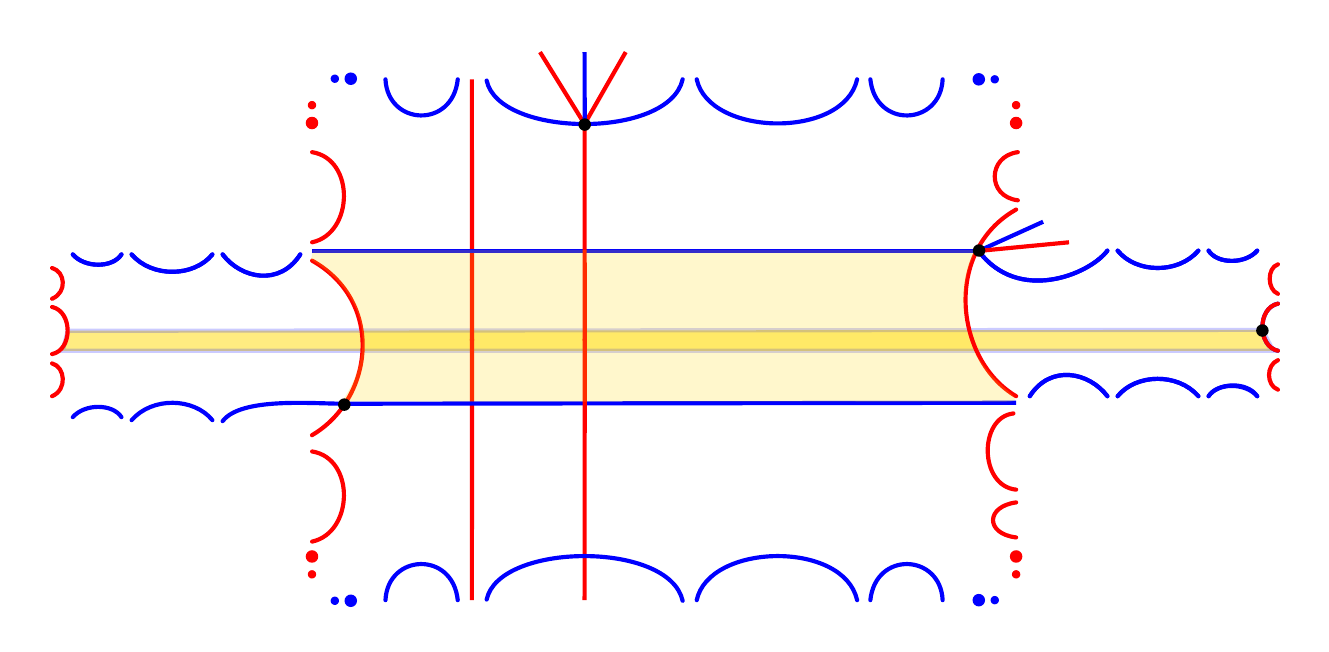}
\caption{Configuration of lozenges $L$ and $L'$}
\label{fig_singular_point_orbit}
\end{figure}
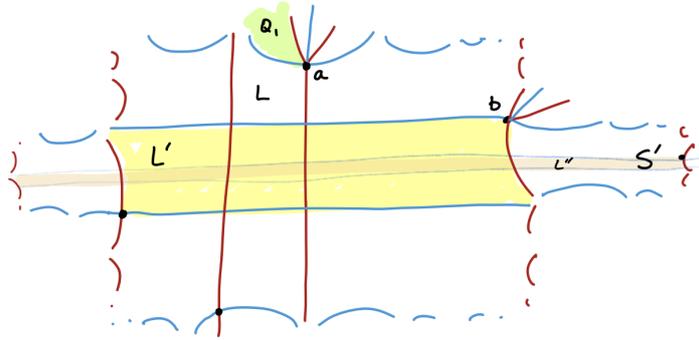

Assume now for contradiction that there is only one $G$-orbit of singular points.  Then $b = g_1(a)$ for some $g_1 \in G$.  If $g_1(L) =L'$ then we already have a contradiction: $g_1(L) \cap L$ would contain a (necessarily nonsingular) fixed point for $g_1$.  Thus, $L'$ is contained in $g_1(Q_j)$ for some $j \neq 1$.  We now repeat this process: $L'$ is contained in a scalloped region $S'$ and we may find a lozenge $L''$ in $S'$, with $L'' \cap L' \cap L \neq \emptyset$, and with corner $c$ on the upper right.  Then there is some $g_2$ such that $g_2(a) = c$.  If $g_2(Q_1)$ or $g_2(Q_j)$ contains $L''$, then we arrive at the same contradiction.  Otherwise, we may iterate the process again.  Since $a$ has only finitely many quadrants, this process eventually terminates by producing a nonsingular fixed point, giving the desired contradiction.  
\end{proof}

\begin{proposition}  \label{prop_unique_class_transitive}
An Anosov-like action of $G$ has a unique Smale class if and only if the set of points fixed by nontrivial group elements is dense in $P$.  In this case, the action of $G$ is topologically transitive.  
\end{proposition} 

\begin{proof} 
Assume there is a unique Smale class $\Lambda$.  By Lemma \ref{lem_only_prongs}, this means that there are {\em no} isolated 
prong singularities, since each $G$-orbit of such produces a Smale class. 
Thus,  $\Fixbar_G = \bar \cR_G = \bar \Lambda$. 
The fact that leaves containing fixed points are dense (Axiom \ref{Axiom_dense}) together with the product structure given by Proposition \ref{prop_basic_product_structure} implies that $\bar \Lambda = P$.  As shown above, $G$ acts topologically transitively on $\Lambda$, hence on $P$.  

Conversely, suppose $\Fix_G$ is dense in $P$.  As remarked above, there are only countably many singular leaves, so the complement of their union (namely, $\cR_G$) is dense in $P$. 
Any product neighborhood of any point intersects at most one Smale class.  Since the complement of the (discrete set of) singular points is connected, there can therefore only be one Smale class.    Transitivity is a consequence of Theorem \ref{thm:transitive_on_Smale_class}.  
\end{proof}

It is also easy to see from the definition that trivial and skew bifoliated planes can admit only one Smale class.  Thus, as a further consequence we have: 
\begin{corollary}[Hausdorff leaf space implies transitive]
\label{cor_skew_trivial_are_transitive}
Suppose $G$ acts Anosov-like on a plane $(P,\cF^+,\cF^-)$ such that one foliation has Hausdorff leaf space.  Then the action is topologically transitive and $\Fixbar_G=P$. 
\end{corollary} 
This was originally proved for Anosov flows in \cite[Th\'eor\`eme 2.5]{Bar95a}.

Our next proposition gives the implication \ref{item_charac_top_transitive}$\Rightarrow$\ref{item_charac_dense_fixed_points}. 
\begin{proposition} \label{prop_transitive_implies_fix_dense}
If $G$ acts topologically transitively on $P$, then there is a unique Smale class.  Consequently, $\Fixbar_G=P$.
\end{proposition}

In the case of orbit-space actions of (pseudo)-Anosov flows, this can be obtained as an easy consequence of the closing lemma (Proposition \ref{prop_pAclosing_lemma}) as follows: 
Suppose $G$ acts transitively.  Recall that the closing lemma, translated into the language of the orbit space (see Proposition \ref{prop_closing_translation}), says that for any open set $U$, there exists $V \subset U$ such that $gV \cap V \neq \emptyset$ implies that $g$ has a fixed point in $U$.  Transitivity gives the existence of some nontrivial $g$ with $gV \cap V \neq \emptyset$, hence $\Fixbar_G=P$.  

The proof below avoids use of the closing lemma by instead using a key result of \cite{BBM24b} about the structure of {\em boundary leaves} of Smale classes.   

\begin{lemma}[See Theorem 4.13 of \cite{BBM24b}]  \label{lem_wandering_lozenge}
Suppose $G$ has an Anosov like action on a bifoliated plane, with more than one Smale class.  Either $P$ contains a {\em wandering totally ideal quadrilateral}, or there exists a lozenge $L$ in $P$ with a corner in $\Fixbar_G$ and
with interior in $P \setminus \Fixbar_G$.  
\end{lemma} 
A {\em wandering totally ideal quadrilateral} is a trivially foliated region whose boundary consists of four leaves (one on each side) which pairwise make perfect fits, and which does not contain any point of $\Fixbar_G$. See \cite[Remark 4.19.]{BBM24b} and the discussion of totally ideal quadrilaterals below in Remark \ref{rem_no_ideal_quad}.  

While Lemma \ref{lem_wandering_lozenge} relies only on the axioms for Anosov-like actions, it is rather involved.  We omit it here, using the Lemma above as a black box.  

\begin{proof}[Proof of Proposition \ref{prop_transitive_implies_fix_dense}]
We prove the contrapositive: We assume that there is more than one Smale class and will deduce using Lemma \ref{lem_wandering_lozenge} that the action is not topologically transitive.

Assume that $G$ has more than one Smale class.  By Lemma \ref{lem_wandering_lozenge}, either $P$ contains a wandering totally ideal quadrilateral $Q$, or there is a lozenge $L$ in $P$ with a corner in $\Fixbar_G$ and
with interior in $P \setminus \Fixbar_G$.  

In the first case, the fact that $Q$ is ``wandering" (as defined above) in fact implies that all translates of $Q$ are disjoint: if a quadrilateral bounded by four leaves making perfect fits satisfied $gQ \cap Q \neq \emptyset$, then as in the proof of \ref{lem_markovian_or_corner}, we must either have $gQ \subset \cF^+(Q)$ or $g^{-1}Q \subset \cF^+(Q)$; and either way one obtains a fixed leaf for $g$ in $\cF^+(Q)$, replacing $g$ with $g^{-1}$ one finds a fixed $\cF^-$ leaf, so $Q$ could not be wandering.   Thus, $Q$ is wandering in the standard sense of the definition, and in particular the action of $G$ is not transitive.  

In the case where $L$ has a corner $c$ in  $\Fixbar_G$ and interior in $P \setminus \Fixbar_G$, then for any $g \in G$ we cannot have $g(c) \in L$.  But by Lemma  \ref{lem_markovian_or_corner}, we also cannot have $g(L) \cap L \neq \emptyset$ because such an intersection forces a point of $\Fixbar_G$ in $L$.   Thus, $L$ is wandering and so the action is not transitive.  
\end{proof}

Next we show a stronger version of the implication \ref{item_charac_dense_fixed_points}$\Rightarrow$\ref{item_charac_dense_leaves}:
\begin{lemma}\label{lem_intervals_have_dense_orbit}
Suppose $\Fixbar_G=P$ and assume that $I$ is any non-degenerate interval inside a leaf of $\cF^\pm$. Then $G\cdot I$ is dense in $P$.  
\end{lemma}

\begin{proof}
To fix notation, we assume that $I$ is a subset of a leaf $l^+$ of $\cF^+$.
Let $U$ be an open set in $P$.   We claim first that there exist nonsingular points $x_i$ fixed by nontrivial $g_i \in G$, and a path $\gamma$  from $I$ to $U$ consisting of a sequence of intervals $J_1^-, J_1^+, J_2^-, J_2^+, \ldots J_n^-, J_n^+$ where $J_i^-$ and $J_i^+$ are segments of $\cF^-(x_i)$ and $\cF^+(x_i)$, respectively.   
To see this, take an arbitrary path from a point in $I$ to a point of $U$, cover it with finitely many flow-box neighborhoods, and then use density of $\Fix_G$ to iteratively modify it on each flow-box neighborhood to a finite sequence of segments of this form.  

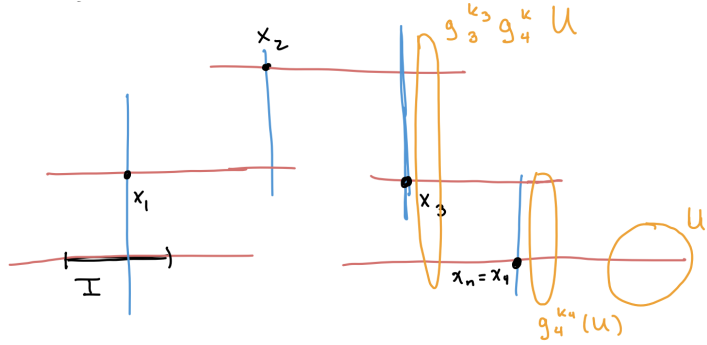
\begin{figure}[h]
\labellist
\small \hair 2pt
\pinlabel {$x_1$} at 146.921 175.041
\pinlabel {$x_2$} at 244.137 279
\pinlabel {$x_3$} at 400.762 148.014
\pinlabel {$x_n=x_4$} at 478.535 52.878
\pinlabel {$I$} at 120 83.149
\pinlabel {$U$} at 690.249 110
\pinlabel {$g_4^{k_4}(U)$} at 575.750 21.527
\pinlabel {$g_3^{k_3}g_4^{k_4}U$} at 439.648 341.529
\endlabellist
\includegraphics[width=10cm]{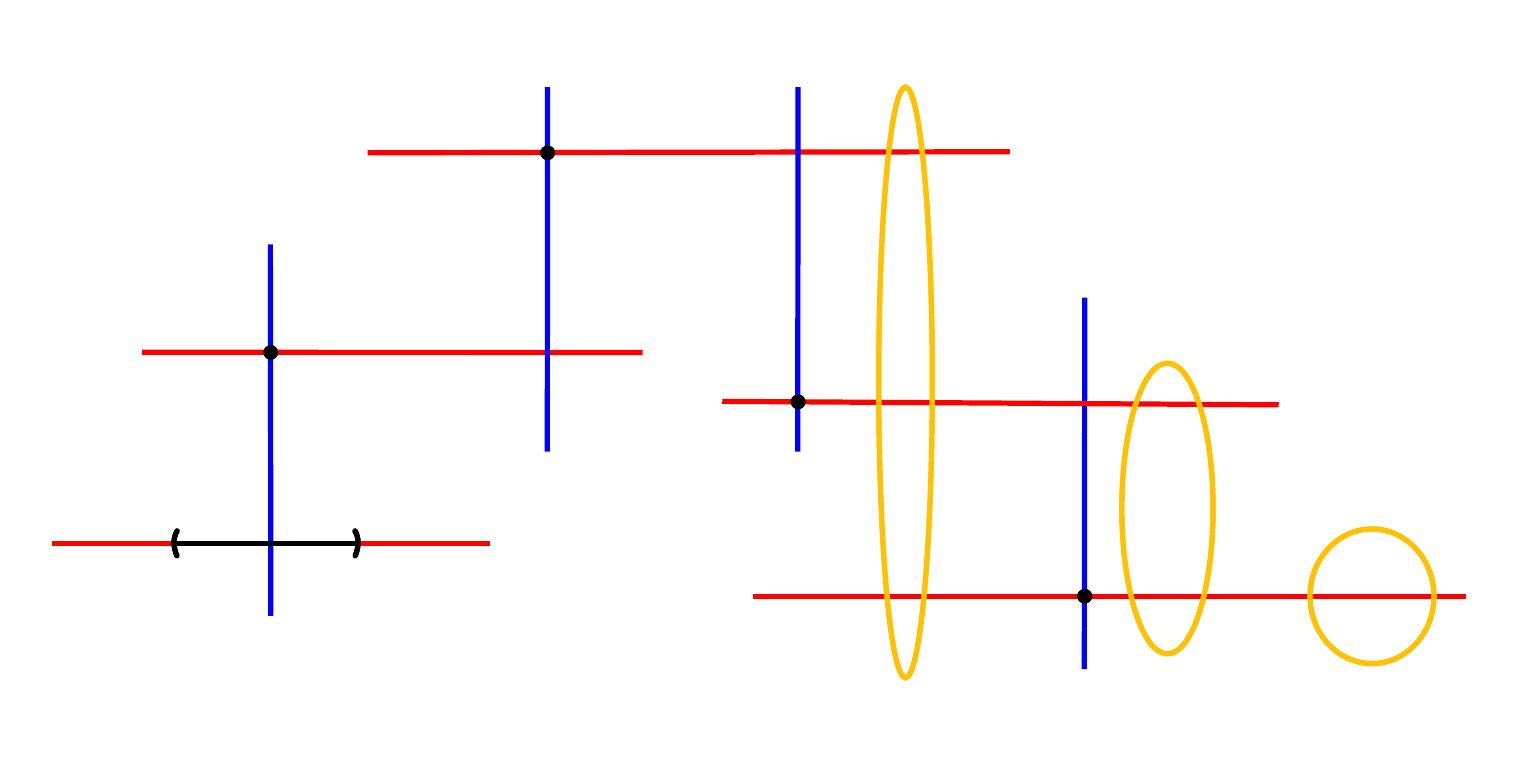}
\caption{Iterates of $U$ eventually will intersect $I$.} 
\label{fig_path_of_segments}
\end{figure} 

We have in particular that 
$\cF^-(x_1)\cap l^+\neq\emptyset$, and $\cF^+(x_n)\cap U\neq\emptyset$, and $\cF^-(x_{i+1})\cap \cF^+(x_{i})\neq \emptyset$ for each $i=1,\dots, n-1$.  See Figure \ref{fig_path_of_segments}.   Since $U \cap \cF^+(x_n) \neq \emptyset$ , there exists $k_n$ (in fact, for any sufficiently large $k_n$) such that  $g_n^{k_n}(U) \cap \cF^+(x_{n-1}) \neq \emptyset$.   Iteratively, we can find $k_{n-1}, \ldots k_1 \in \bZ$ such that $g_1^{k_1}\dots g_n^{k_n}U \cap I \neq\emptyset$;  equivalently
\[ g_n^{-k_n}\ldots g_1^{-k_1} I \cap U \neq \emptyset, \] 
which is what we needed to show.  
\end{proof}

Finally, to show that \ref{item_charac_dense_leaves}$\Rightarrow$\ref{item_charac_top_transitive}, we show the contrapositive:
\begin{lemma}
Suppose that $G$ is not topologically transitive. Then there exist leaves of $\cF^+$ and $\cF^-$ whose orbit under $G$ is not dense in $P$.
\end{lemma}

\begin{proof}
Since $G$ is not topologically transitive, by Proposition \ref{prop_transitive_implies_fix_dense}, there exist at least two distinct Smale classes $\Lambda_1$ and $\Lambda_2$.  Label these so that $\Lambda_1$ is not less than $\Lambda_2$ in the Smale order. In particular, this means that $\cF^+(\Lambda_2)$ does not intersect $\cF^-(\Lambda_1)$ -- otherwise $\Lambda_2$ would also be smaller than $\Lambda_1$ for the Smale order and thus equal, contradicting the assumption that they are distinct. Since $\Lambda_2$ is $G$-invariant, we deduce that no $\cF^+$-leaf of $\Lambda_2$ can have a dense orbit in $P$. Similarly, no $\cF^-$-leaf of $\Lambda_1$ has a dense orbit.
\end{proof}

This is the last required implication so concludes the proof of Theorem \ref{thm_characterization_transitive}.  

\subsection{Proof of Proposition \ref{prop_transitive_non_corner_dense}} \label{sec_noncornerdense}

Here we want to show that, if $P$ is not skew, then topological transitivity of an Anosov-like action of $G$ is equivalent to the density of non-corner points fixed by elements of $G$. Given Theorem \ref{thm_characterization_transitive}, all we have left to show is the following:
\begin{lemma}\label{lem_transitive_implies_dense_noncorner}
Assume $P$ is not skew. If an Anosov-like action of $G$ on $P$ is transitive, then there is a nonempty, open subset of $P$ which contains no corner fixed points.  Consequently, the non-corner fixed points are dense in $P$.
\end{lemma} 
The hypothesis that $P$ is not skew is obviously necessary -- in the skew plane every point is the corner of a lozenge!

\begin{proof}
We will do the proof in the case when the foliations of $P$ are non-singular. The general case is left as an exercise, and is also done in \cite[Lemma 2.30]{BFM22}.
Assuming $\cF^+$ and $\cF^-$ are nonsingular, Theorem \ref{thm:trichotomy} implies that either $P$ is trivial (in which case there are no corner points and we are done) or there exist nonseparated leaves. 
In the case of nonseparated leaves,  Proposition \ref{prop:4weak4strong} implies that there are at least two lozenges in $P$ that share a side.  Denote these lozenges by $L_1$ and $L_2$, and suppose for concreteness their shared side is a leaf of $\cF^-$. 

Since $G$ acts transitively, Theorem \ref{thm_characterization_transitive} implies that $\Fix_G$ is dense.  Thus,  we can pick a point $x$ in $L_1$ fixed by some nontrivial $g\in G$.  Up to replacing $g$ with its inverse, we may assume that $g$ contracts $\cF^-(x)$ and so $g(L_1)$ will intersect both $L_1$ and $L_2$, as shown in Figure \ref{fig:non-corner-dense}. 

If $p \in g (L_1) \cap L_2$ is a fixed point of some element, then each of its quadrants will contain a perfect fit formed by a leaf of a side of $L_1$ or $g(L_1)$ or $L_2$.  See Figure \ref{fig:non-corner-dense}.  
Thus, $p$ satisfies the conditions of Lemma \ref{lem_no_corner_criterion} in each of its quadrants, so cannot be a corner of a lozenge.  This shows that the open set  $g (L_1) \cap L_2$ contains no corners, as desired.  

The consequence that non-corner fixed points are dense in $P$ now follows directly from density of $\Fixbar_G$, hence density of non-corner fixed points in $g (L_1) \cap L_2$, and the fact that for transitive actions the orbit of any open set is dense.  
 \begin{figure}

     \centerline{ \mbox{
\labellist
\small \hair 2pt
\pinlabel {$L_1$} at 175.149 253.711
\pinlabel {$L_2$} at 336.344 253.711
\pinlabel {$g(L_1)$} at 65.363 150
\pinlabel {$p$} at 325.017 180.548
\endlabellist
\includegraphics[width=7cm]{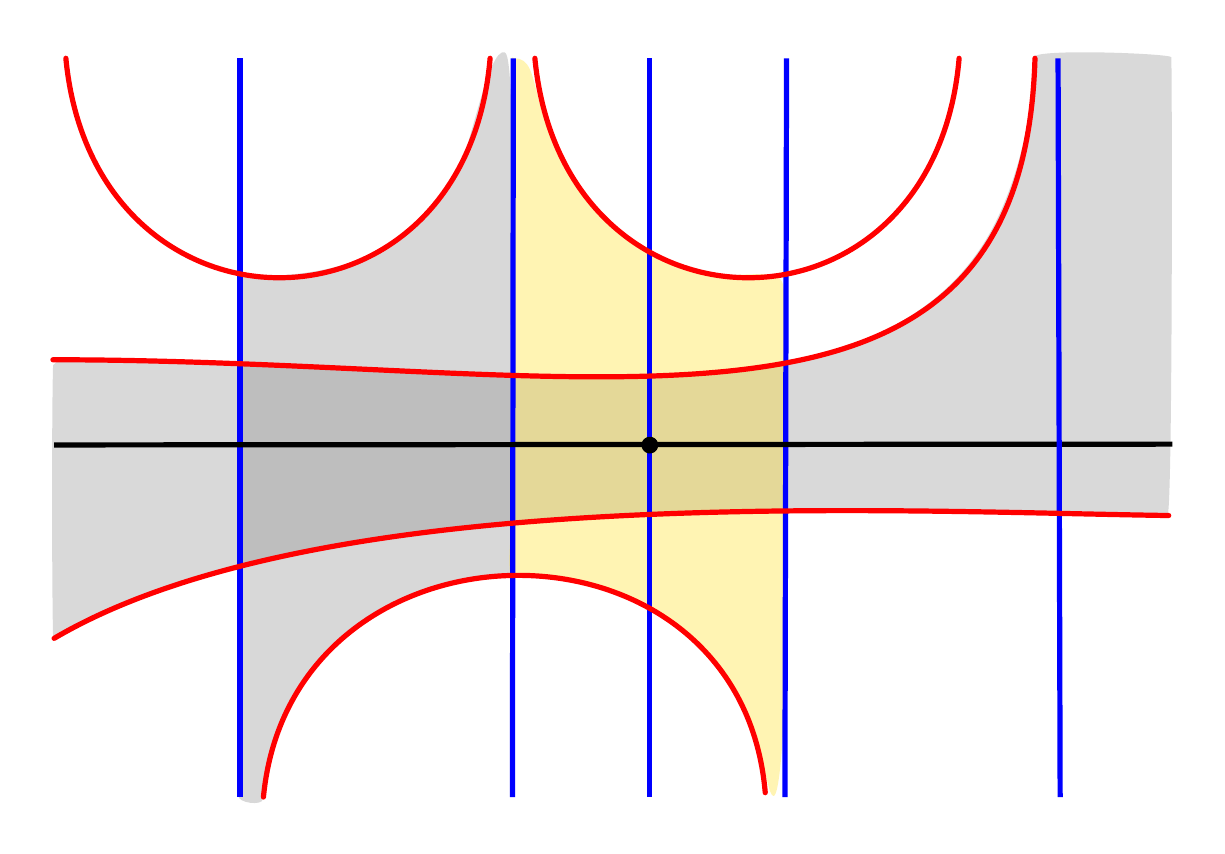}}}
\caption{$L_2 \cap g( L_1)$ contains no corner points}
 \label{fig:non-corner-dense} 
\end{figure}
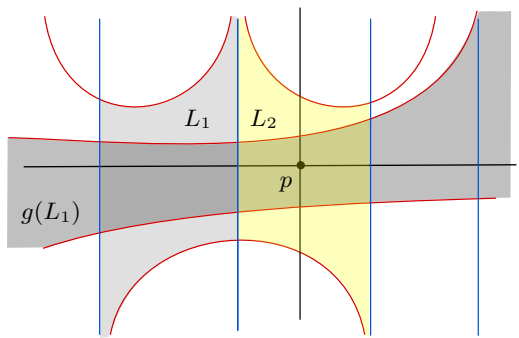

\end{proof}

\begin{exercise} \label{ex_noncorner} 
Finish the proof of Lemma \ref{lem_transitive_implies_dense_noncorner} by extending the argument to the case where $\cF^+$ and $\cF^-$ have prong singularities.  
\end{exercise}

\section{A generating set for Anosov-like groups}
Building on the argument of Lemma \ref{lem_intervals_have_dense_orbit}, we can show that a group with a transitive Anosov-like action is generated by the set of elements acting with fixed points.  
This result was proved by Toshiaki Adachi \cite{Ada87} for transitive Anosov flows (in any dimension). Thurston \cite[Proposition 3.5]{Thu97} gave a different argument in the case of Anosov flows on 3-manifolds whose orbit space is the skew plane.  We are not aware of any proof in the literature that obtains this for pseudo-Anosov flows, although it seems reasonable that the strategy of Adachi could extend to this case. 

\begin{theorem}\label{thm_generating_set}
Suppose that $G$ has a topologically transitive Anosov-like action. Then the set of elements in $G$ that fix points in $P$ is a generating set for $G$.
\end{theorem}

\begin{proof} 
Let $g \in G$ be given.  We will find $h_1, h_2 \ldots h_n$ such that the composition $h:= h_1 h_2 \ldots h_n g$ has a fixed point, thus $g = h_n^{-1}\ldots h_1^{-1} h$ is a composition of elements acting with fixed points.  

If $g$ has a fixed point, we are already done, so assume not.  Let $B$ be a small product-foliated box in $P$, chosen small enough so that $B$ and $g(B)$ are disjoint.   
Take a path between $B$ and $g(B)$ that avoids all singular points, and cover this path by finitely many trivially foliated boxes $B_1, \ldots B_{n}$, with $B = B_1$ and $g(B) = B_{n}$ and $B_i \cap B_{i+1} \neq \emptyset$.  Using density of fixed points, choose $x_i \in B_i$ fixed by some nontrivial $h_i \in G$ such that $\cF^-(x_i) \subset \cF^-(B_{i-1})$ for all $i \geq 2$.  Thus, $\cF^-(x_i) \cap \cF^+(x_{i-1}) \neq \emptyset$.  (This is why we are working with box neighborhoods).  
See Figure \ref{fig_boxes_path}, left for an illustration.  

 \begin{figure}[h]
     \centerline{ \mbox{
\labellist
\small \hair 2pt
\pinlabel {$B$} at 67.802 321.668
\pinlabel {$x_1$} at 99 235.680
\pinlabel {$x_2$} at 150.456 286.942
\pinlabel {$x_3$} at 261.213 363.008
\pinlabel {$x_4$} at 368.664 220.797
\pinlabel {$x_5$} at 458 143.077
\pinlabel {$x_6$} at 515 79
\pinlabel {$gB$} at 614.973 133.155
\pinlabel {$x_5$} at 796.812 278.674
\pinlabel {$h_6^N(gB)$} at 980 170
\pinlabel {$h_5^N h_6^N(gB)$} at 880 115
\endlabellist
\includegraphics[width=12cm]{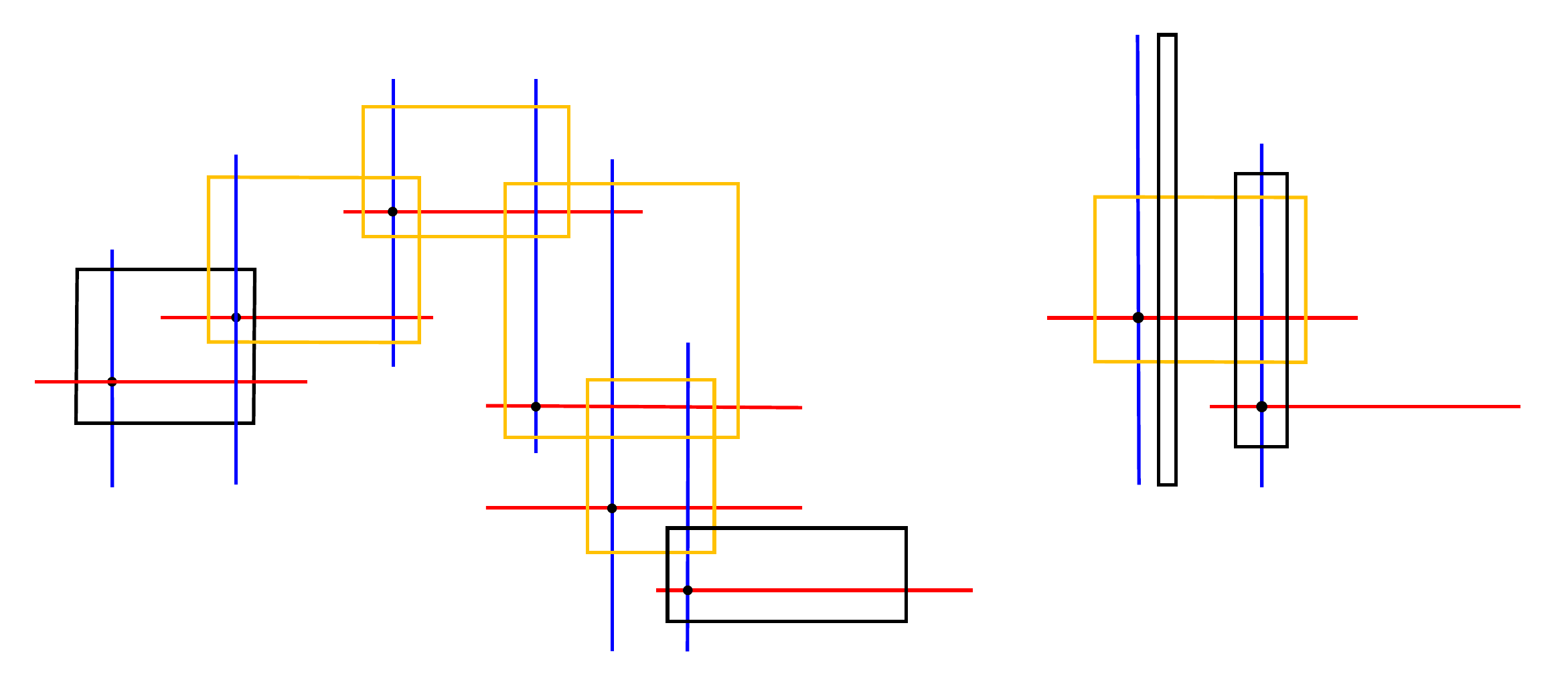}}}
\caption{Left: a path from $B$ to $g(B)$ along boxes containing elements with fixed points.  Right: the image of $g(B)$ under $h_5^N h_6^N$ has $\cF^+$-saturation covering $B_4$.}
 \label{fig_boxes_path} 
\end{figure}
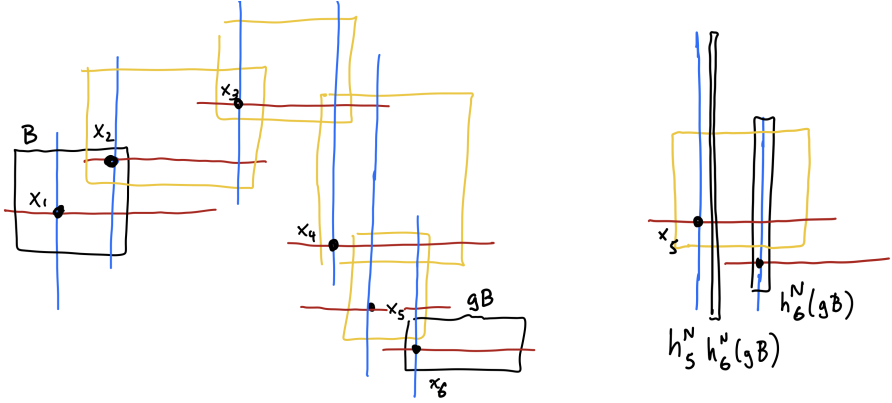

Up to replacing some $h_i$ with their inverses, we assume that all $h_i$ are contracting on $\cF^+(x_i)$ and expansive on $\cF^-(x_i)$.  Then for all $i$, for any sufficiently large $N$, we will have that $\cF^+(h_i^N B_i) \supset \cF^+(B_{i-1})$. Schematically, this is very similar to the argument as we made in Lemma \ref{lem_intervals_have_dense_orbit}.  The first iterations are shown in Figure \ref{fig_boxes_path}, right.  
It follows that  $\cF^+(h_1^N \ldots h_n^N g B) \supset \cF^+(B)$.  Thus, we have found an interval of $\Lambda(\cF^+)$ that is expanded by $h_1^N \ldots h_n^N g$, hence by the Brouwer fixed point theorem, there is a fixed leaf for $h_1^N \ldots h_n^N g$ in this interval, and by Axiom \ref{Axiom_A1} this element has a fixed point in $P$.  This is what we needed to show.  
\end{proof}

%% file: low_dimensional_vector.tex
A group acting on a bifoliated plane induces actions on three associated 1-dimensional spaces: the {\em leaf space} of each foliation, which was introduced in Section \ref{sec_bifoliated_intro}, and a compactification of the plane by a {\em circle at infinity}, which we describe here. 
Constraining the action to be Anosov-like puts constraints on these one-dimensional actions, which can be used to establish algebraic properties for the group, such as exponential growth (Theorem \ref{thm_exponential_growth} below) or rigidity of examples with nontrivial normal subgroups (Theorem \ref{thm_normal_abelian_implies_rigidity}), or show structural properties of the bifoliated plane, such as the existence of skew-like regions (Theorem \ref{thm_both_axes_R}).  
This chapter introduces a few of these tools and their applications.  

\section{The circle at infinity} \label{sec_circle}
The idea to compactify a (bi)-foliated plane by a circle has been independently discovered and re-discovered a number of times. 
Kaplan's classification of planar foliations in \cite{Kap41} gave a realization of leaves as level sets of pseudoharmonic functions in the closed disc, thus compactifying the plane (considered as the open disc) by a boundary circle. See \cite[Theorem 29]{Kap41}.  In the early 1980s, Mather \cite{Mat82} gave a purely topological approach to compactify a foliation of a surface minus a point at the missing point, which applies directly to a foliated or bifoliated plane by considering the plane as a sphere minus a point.  Fenley \cite{Fen12} gave a construction in the specific context of orbit spaces of pseudo-Anosov flows, and Frankel \cite{Fra13} gave a similar construction in the related context of ``sprig decompositions" (rather than bi-foliations) of the orbit spaces of quasi-geodesic flows. 
Bonatti \cite{Bon23} recently proved the following general result, using Mather's strategy of proof.  It also, through a uniqueness statement, shows that the constructions of Mather, Fenley and Frankel all give the same object in the context of bifoliated planes.\footnote{Kaplan's compactification is different as it will generally fail condition \ref{item_circle_uncountable_in_open}.} 

\begin{theorem}[See  \cite{Bon23}] \label{thm_compactification}
Let $P$ be a plane with a countable collection of foliations $\cF_i$, either pairwise transverse or, more generally, with any two distinct leaves intersecting along a compact set.  Then $P$ admits a compactification by a circle $\partial P$ such that $P \cup \partial P$ is homeomorphic to a closed disc and the following properties hold: 
\begin{enumerate}[label=(\arabic*)]
\item each ray limits to a unique point of $\partial P$ 
\item the set of (equivalence classes of) rays limiting to any point of $\partial P$ is at most countable, and 
\item\label{item_circle_uncountable_in_open} the set of (equivalence classes of) rays limiting to any nonempty open subset of $\partial P$ is uncountable. 
\end{enumerate} 
Moreover, there is a unique compactification with the three properties above and if a group $G$ acts on $P$ preserving the foliations $\cF_i$, then this action extends (uniquely) to the compactification. 
\end{theorem}  

Here, we describe the construction of $\partial P$, following Mather and Bonatti, in the special case of bifoliated planes with Anosov-like actions.  This specialization allows us to shorten some arguments by using work from Chapter \ref{chap_bifoliated_planes}.  The general proof follows the same outline, but requires a bit more care.  To start, we recall some basic definitions and describe the {\em circular order} on ends of leaves.   

Recall a {\em ray} of a leaf $l$ of a foliation is a properly embedded copy of $[0, \infty)$ into $l$.  An {\em end} of $l$ is an equivalence class of rays, where two rays are equivalent if their images coincide on a noncompact set.   Each nonsingular leaf has exactly two ends and a $k$-prong has $k$ ends.   

A {\em circular order} on a set $S$ is an abstract version of the way points are ordered on the circle, in the same way that a total order $<$ is a generalization of the way points are arranged on the real line.  Formally: 

\begin{definition} A {\em circular order} on a set $S$ is a map $c$ from distinct triples of points in $S$ to $\{+1, -1\}$ satisfying the following {\em cocycle relation} on 4-tuples: 
\[ c(s_1, s_2, s_3) - c(s_1, s_2, s_4) + c(s_1, s_3, s_4) - c(s_2, s_3, s_4) = 0\]
\end{definition} 
\noindent The reader can check as an exercise that the {\em orientation order} when $S = S^1$, assigning $1$ (resp. $-1$) to a triple that is oriented counter clockwise (resp.~clockwise) is an example.  

Since two inequivalent rays (of the same or different leaves, on any bifoliated plane) intersect along a compact set, the set of ends of leaves can be given a circular order.  To define this formally, fix an origin $O$ in the plane and let $C_\rho$ denote the boundary circle of the ball of radius $\rho$ about $O$.   Because rays are properly embedded, any ray $r_i$ intersects $C_\rho$, for all $\rho$ sufficiently large.   Since intersections occur along compact sets, given any three rays $r_1, r_2, r_3$, the points $C_\rho \cap r_1, C_\rho \cap r_2, C_\rho \cap r_3$, read in order, are oriented either clockwise or counter clockwise around $C_\rho$ and this orientation remains constant over all sufficiently large circles $C_\rho$.  Note that this order descends to ends of leaves since (up to increasing $\rho$) it is unchanged by replacing $r_i$ with any equivalent ray of the same leaf.  This gives the following: 

\begin{definition} \label{def_order_rays}
Let $(P, \cF^+, \cF^-)$ be a bifoliated plane.  
Define a circular order $c_P$ on ends of leaves by setting $c_P(e_1, e_2, e_3) = 1$ if for any rays $r_i$ representing ends $e_i$, and all $\rho$ sufficiently large,  the points
$C_\rho \cap r_1, C_\rho \cap r_2, C_\rho \cap r_3$ can be read in order counterclockwise along $C_\rho$, and $c_P(e_1, e_2, e_3) = -1$ otherwise.  
\end{definition} 
It is easy to verify that this satisfies the required cocycle relation, so defines a circular order.  

This order is the starting point to create the compactification of $P$ by a circle, but there are two difficulties. 
The first is a problem of embedding.  Consider first the more familiar case of total (linear) orders rather than circular orders.  Not every totally ordered set with the same cardinality as $\mathbb{R}$ can be embedded in $\mathbb{R}$ in an order-preserving way. For instance, any set with uncountably many distinct pairs which are {\em immediate successors} cannot be embedded.  (An element $b$ is an immediate successor of $a$ if $a<b$ and there exists no $d$ with $a<d<b$.)  
Similarly, we say that a pair of points $a, b$ in a circularly ordered set are immediate neighbors if there exists $d$ such that $c(a,b,d) = \pm1$ and there does not exist any point $e$ with $c(a,e,b) = c(a,b,d)$. The same argument as the linear case shows that uncountably many distinct immediate neighbors is an obstruction to embedding a circularly ordered set into $S^1$ in an order-preserving way. 
This is a problem already in the case of ends for the skew plane: each pair of rays $r, r'$ making a perfect fit defines a pair of immediate neighbors.  

The second difficulty is that, even in cases where an order-preserving embedding to $S^1$ exists, it may fail to be surjective.  For instance, in the case of the trivial plane the reader can check that the cyclic order on ends is isomorphic to that of four disjoint open intervals, so any embedding into the circle will necessarily miss at least four points.  

To resolve these issues, we will use a series of lemmas to first pass to a quotient of the set of ends which does admit an embedding to $S^1$, and then {\em complete} it to a topological circle, similarly to how one uses Dedekind cuts to complete the rationals (as a totally ordered set) to the real line.   Going forward, we assume $P$ is a bifoliated plane with Anosov-like action, and $c_P$ the circular order on ends of leaves.  

\begin{lemma} \label{lem_finite_or_uncountable}
For any inequivalent rays $r, r'$ either there are finitely many equivalence classes of rays $s$ such that $c_P(r, s, r') = 1$, or uncountably many equivalence classes of rays $s$ such that $c_P(r, s, r') = 1$.  Furthermore, in the finite case, all of the rays $r, r'$ and all $s$ satisfying $c_P(r, s, r') = 1$ are part of a sequence of sides $s_i$ of a chain of lozenges such that for all $i$, $s_i$ and $s_{i+1}$ make a perfect fit.  
\end{lemma}

\begin{proof} 
Let $r, r'$ be rays of leaves $l$ and $l'$ respectively. Without loss of generality we assume $l \in \cF^+$.  Fix $\rho$ large enough so that $l \cap l'$, if nonempty,  lies inside the open ball of radius $\rho$.  Since $l, l'$ are properly embedded, up to replacing $r$ and $r'$ with equivalent rays (respectively), we assume that they are based on $C_\rho$ and each intersect $C_\rho$ at a single point.  
The complement of $C_\rho \cup r \cup r'$ has two unbounded connected components, one containing rays ``between" $r$ and $r'$ in the clockwise sense (i.e. such that $c_P(r, s, r') = -1$ for any ray $s$ in this connected component) and one such that $c_P(r, s, r') = 1$ for all rays $s$ in the component.  Let $A$ denote this latter connected component.   

We will use the following elementary observation.  

\begin{observation} \label{obs_uncountable}
If $A$ contains a full leaf $l$ of $\cF^+$ or $\cF^-$, or if $A$ contains a ray $l'$ based at a point on $r$ or $r'$, 
then $A$ contains uncountably many distinct rays, consequently the set of equivalence classes of rays $s$ satisfying $c_P(r, s, r') = 1$ is uncountable. 
\end{observation} 
This uncountable family is obtained by considering rays of leaves based at points along $l$ or $l'$.  

Assume now that the number of equivalence classes of rays $s$ satisfying $c_P(r, s, r') = 1$ is countable.  Let $y = r \cap C_\rho$. Since $C_\rho$ is locally transverse to $\cF^+$ at $y$, each point near $y$ on the boundary of $A$ in $C_\rho$ defines a ray, no two of which are equivalent.  Thus, all but countably many of these rays must eventually exit $A$, so do not approach $r$.   Let $s_n$ be such a sequence of rays (starting in $A$ but exiting $A$), based at points $y_n$ approaching $y$.  Then $s_n$ limits to a set of nonseparated leaves, one of which, call it $l_1$ contains $r$.  In particular, $r$ is branching, so there is a leaf $f_1^-$ making a perfect fit with $r$ and another nonseparated leaf $l_2$.    See Figure \ref{fig_finite_or_uncountable}.

If either of $f_1^-$ or $l_2$ contains $r'$, then we are done.  Otherwise, 
we apply Observation \ref{obs_uncountable} to conclude that $l_2$ cannot be contained in $A$, nor can it intersect $r'$.  Thus, it exits $A$ through $C_\rho$, and we may repeat the argument above using a ray $r_2$ of $l_2$ in place of $r$ (perturbing $C_\rho$ if needed for our simplifying transversality assumption), producing leaves $f_2^-, l_3$. 

\begin{figure}[h]
\labellist
\small \hair 2pt
\pinlabel {$r$} at 202.694 409
\pinlabel {$r'$} at 341.123 87.627
\pinlabel {$l_2$} at 351.624 423.929
\pinlabel {$f_1^-$} at 300 400
\pinlabel {$y$} at 134.912 343.675
\pinlabel {$y_n$} at 155 300
\pinlabel {$A$} at 326.803 240.492
\endlabellist
\includegraphics[width=6cm]{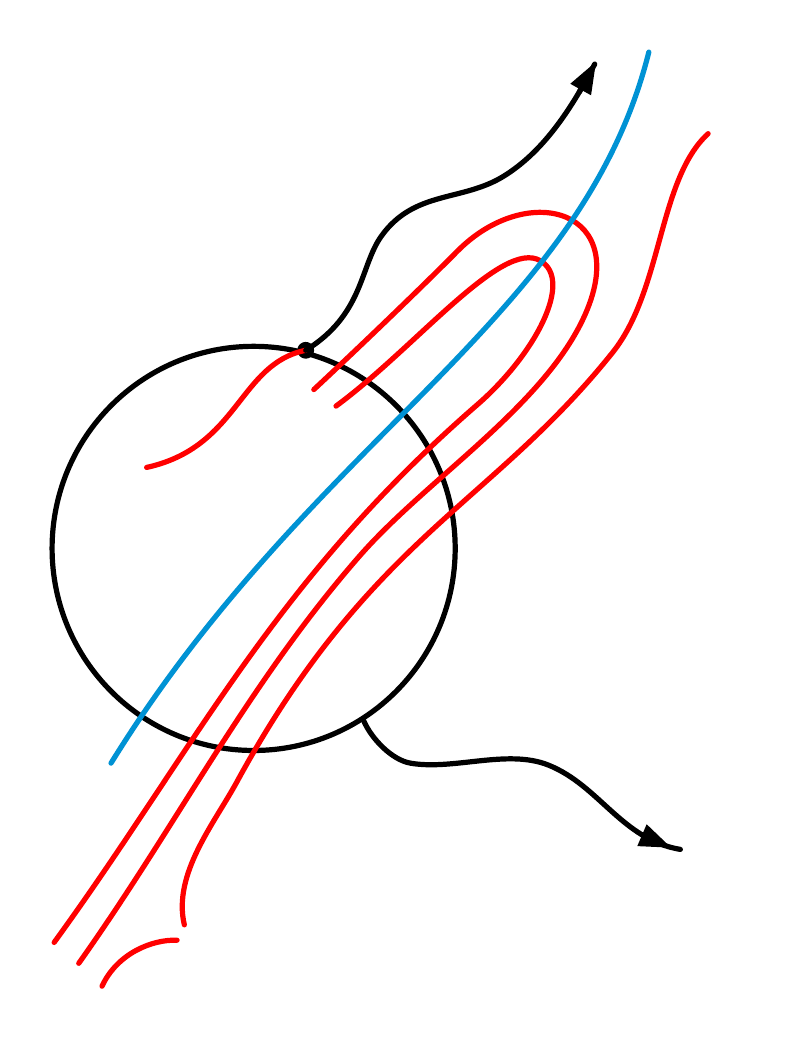}
\caption{The proof of Lemma \ref{lem_finite_or_uncountable}}
\label{fig_finite_or_uncountable}
\end{figure}

Iterating this argument, if it terminates in finitely many steps we have finished the proof.  If not, consider the sequences of leaves $l_n$ in $\cF^+$ and $f_n^-$ in $\cF^-$. These leaves are all pairwise disjoint due to the configuration of perfect fits, but all meet a compact set (of $C_\rho$ or a small perturbation) so must have an accumulation point.  This gives a contradiction with topological transversality of the foliations, proving the process must terminate.
\end{proof} 

As an easy consequence, we have: 
\begin{lemma}  \label{lem_no_successor}
Let $E'$ be the quotient of the set of ends of leaves by the equivalence relation generated by $e_1 \sim e_2$ if $e_1$ and $e_2$ are represented by rays of leaves making a perfect fit.  Then the circular order $c_P$ descends to an order $c'_P$ on $E'$, and on this quotient no two points are immediate successors. 
\end{lemma} 

\begin{proof}
Passing to the quotient collapses finite or countable families coming from chains of lozenges.  By definition, each equivalence class is at most countable so by Lemma \ref{lem_finite_or_uncountable}, in the induced order on the quotient there are in fact uncountably many elements between any two distinct equivalence classes, hence no immediate successors. 
\end{proof}

The next lemma is the key step needed for the completion process, showing that $E'$ has a countable set which is dense for the order topology. 
\begin{lemma}  \label{lem_countable_dense}
There is a countable set $X \subset E'$ such that for any distinct $e_1, e_2 \in E'$, there exists $x \in X$ such that $c'_P(e_1, x, e_2) = 1$.  
\end{lemma} 

\begin{proof}
Let $C_n$ be the circle of radius $n$.  Let $X_n$ be a dense, countable subset on $C_n$, and let $S_n$ denote the set of all rays of leaves through points of $X_n$.  Then $\bigcup_n S_n$ is dense for the order topology on rays, hence descends to a dense set for the order on $E'$.  
\end{proof} 

\begin{exercise} 
Show the following improved version of Lemma \ref{lem_countable_dense}, using the fact that there are no infinite product regions: For a non-trivial bifoliated plane with Anosov-like action, each of $\cF^+$ and $\cF^-$ contains a set of rays which are dense for the order topology in the sense above. 
\end{exercise} 

Given Lemma \ref{lem_no_successor} and \ref{lem_countable_dense}, the completion of $E'$ to a set homeomorphic to $S^1$ in the order topology is now a standard argument in topology, via a Dedekind-cut style construction.  See, e.g., \cite[Proposition 2.1]{Bon23} for details. 

We next define a topology on $P \cup \partial P$.

\begin{definition}
A basis of open sets of $P \cup \partial P$ is given by open sets of $P$, and for any two rays $r, r'$ and segment $\gamma$ connecting the basepoints of $r$ to $r'$ and intersecting them only at the basepoints, the sets $O^1_{r, r', \gamma}$ and $O^2_{r,r',\gamma}$ consisting of points of $P$ in one of the open regions bounded by $r \cup \gamma \cup r'$ together with the equivalence classes of ends of all rays contained in that region (and the order completion points defined by them). 
\end{definition}

\begin{exercise} 
Check from this definition that the topology restricts to the standard topology on $P$ and that induced from the completed circular order on $\partial P$, and that each end of a leaf limits onto a single point of $\partial P$.  
\end{exercise} 

Note also that, by construction, the set of rays limiting to any point of $\partial P$ is at most countable, and the set of rays limiting to any nonempty open subset is uncountable.
To show that $P \cup \partial P$ is homeomorphic to a disc, one can use a limiting argument, taking an exhaustion of $P$ as an increasing union of balls, and defining a sequence of homeomorphisms from these balls to concentric balls in the disc, so that the process converges to a homeomorphism from $P \cup \partial P$ to the closed disc.  An explicit realization of this is sketched in \cite[Section 2.3]{Bon23}, following the arguments of Mather from \cite[Section 12-13]{Mat82}.  We leave this as an exercise with these papers as reference. 
Given this, to complete the proof of Theorem \ref{thm_compactification} in our setting, we need only to show uniqueness and extension of the group action.  

\begin{lemma} 
A compactification of a bifoliated plane satisfying the conditions in Theorem \ref{thm_compactification} is unique, and any group of homeomorphisms of the plane preserving the foliations extends to act by homeomorphisms of $P \cup \partial P$. 
\end{lemma} 

\begin{proof} 
Suppose $Z$ is some abstract compactification by a circle satisfying the properties given in Theorem \ref{thm_compactification}.  Since rays limit to unique points and the topology on $P \cup Z$ is that of a disc, the order of limit points of rays on $Z$ must agree with the order $c_P$ defined above, whenever three rays have distinct limit points on $Z$.   Since the set of rays limiting to any point is countable, the only distinct equivalence classes of rays that can possibly limit to the same point are those making perfect fits as in Lemma \ref{lem_finite_or_uncountable}.   Since the set of rays limiting to any interval is uncountable, the quotient described in Lemma \ref{lem_no_successor} must be made.  Thus, there is a natural order-preserving bijection from the (dense) set of endpoints of rays in $\partial P$ to the same dense set in $Z$, thus these two compactifications agree.  

Density of endpoints of rays also gives the only possible candidate for the extension of an action of automorphisms of $P$ to $\partial P$: the action of an element $g$ is determined by its action on this dense set.  One verifies directly from the definition of the topology that this action is continuous.  
\end{proof}

\begin{exercise} 
As an exercise in using the definition, verify that the compactification of the skew plane is by the two boundary components of the standard realization as an infinite strip in $\mathbb{R}^2$, together with one additional point at each end.  
Describe also the compactification of the trivial plane as a union of four open intervals consisting of endpoints of leaves, together with four additional points not realized by endpoints of leaves, as hinted in Figure \ref{fig_cartoon_boundary}. 
\end{exercise}

\begin{figure}[h]
\includegraphics[width=8cm]{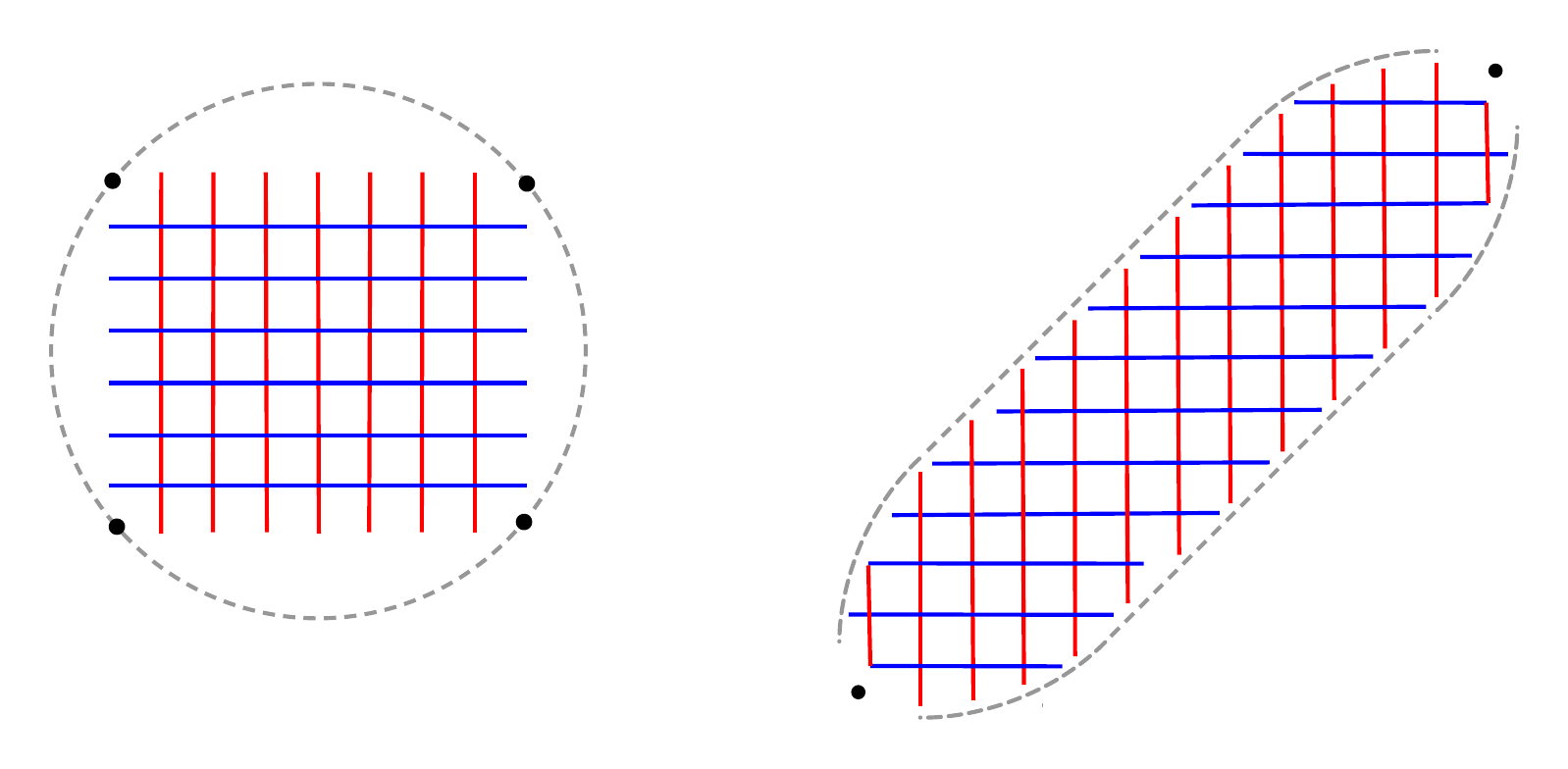}
\caption{A cartoon of the compactification of the trivial and skew planes}
\label{fig_cartoon_boundary}
\end{figure}

\subsection{Dynamics on $\partial P$} 
Since Fenley's work in \cite{Fen12}, the boundary $\partial P$ has become an important tool in the topological study of pseudo-Anosov flows.  More recently it was shown that, in parallel to Barbot's theorem \ref{thm_action_determines_OEflow}, the induced action on $\partial P$ (up to conjugacy) determines the flow up to orbit equivalence.  
\begin{theorem}[See Theorem 1.10 of \cite{BBM24b}] \label{thm_boundary_action_determines_flow}
Suppose $G$ acts Anosov-like on bifoliated planes $P_1$ and $P_2$, with compact fundamental domains.  
Let $\rho_i \colon G \to \Homeo(S^1)$, $i=1, 2$ be the induced actions on $\partial P_i$.  These actions are conjugate in $\Homeo(S^1)$ if and only if there is a conjugacy of the actions on $P_1$ and $P_2$, sending foliations to foliations. 

Consequently, in the special case of two actions induced from pseudo-Anosov flows on a 3-manifold with fundamental group $G$, the induced actions on $S^1$ are conjugate if and only if the flows are orbit equivalent via a map isotopic to the identity on $M$. 
\end{theorem} 

The proof of Theorem \ref{thm_boundary_action_determines_flow} proceeds by ``reconstructing" the bifoliation out of the dynamics of the action, using the results of \cite{BBM24a}.  This procedure is concrete enough that one can in fact rephrase the definition of Anosov-like action purely in terms of the dynamics of elements acting on the circle, although the statements become more cumbersome.  We do not pursue this here, but instead describe the dynamics of individual elements to give some idea of how to work with the boundary.  In the next section, we apply this framework to give a new proof of exponential growth.  

 \begin{proposition}[Boundary action of elements on a nontrivial bifoliated plane] \label{prop:boundary_action_general}
Let $(P,\cF^+, \cF^-)$ be a nontrivial bifoliated plane with Anosov-like action of $G$, and $g \neq 1 \in G$.  
\begin{enumerate} 
\item If $g$ fixes a non-corner point $x$, then the only points of $\Pbound$ fixed by $g$ are endpoints of $\cF^\pm(x)$. 

\item If $g$ fixes all corners in a maximal chain of lozenges $\cC$, then the set of fixed points of $g$ on $\Pbound$ is the closure of the set of endpoints of leaves of the sides of lozenges in $\cC$.  

\item If $g$ acts freely on $P$, then it either has one or two fixed points on $\Pbound$, or has exactly four fixed points and preserves a scalloped region.  \end{enumerate} 
\end{proposition} 

\begin{proof} 
Suppose $g$ fixes some point in $P$, so we are in the setting of cases 1 or 2 of the statement.  If this is a unique fixed point, we may pass to a power to assume that $g$ fixes all half-leaves through the point. (This is automatically satisfied in the setting of case 2.)  Note that this does not affect the conclusion in case 1, since the set of fixed points of an element is a subset of the fixed set of any of its powers.   Let $\partial(\cF_g)$ denote the set of endpoints of leaves fixed by $g$. 

To establish the conclusion in these cases, it suffices to prove the following. 

\begin{lemma} If $\xi, \eta \in \partial(\cF_g)$ are adjacent points, i.e. we have $(\xi, \eta) \cap  \partial(\cF_g) = \emptyset$, then the interval $(\xi, \eta)$ contains no points fixed by $g$.  
\end{lemma}
\begin{proof}
Let $\xi$ and $\eta$ be two such points.  Consider first the case where $g$ has a unique fixed point $x$ in $P$.  Then $\xi$ and $\eta$ are the endpoints of leaves of $\cF^\pm(x)$, necessarily bounding a quadrant.  Let $l^+ \in \cF^+$ be a leaf intersecting this quadrant and $\cF^-(x)$.
 If $g$ has a fixed point $\zeta$ between $\xi$ and $\eta$, then either as $n$ to $+\infty$ or $-\infty$, $g^n(l^+)$ will converge to some union of leaves with an endpoint equal to $\zeta$ or whose endpoints limit to $\zeta$.  This union of leaves is necessarily $g$-invariant, contradicting the fact that $g$ has a single fixed point in $P$.
Hence, there are no fixed points of $g$ in the interval $(\xi, \eta)$.

Now consider the case where $g$ does not have a unique fixed point, but instead preserves each corner of a maximal chain $\cC$ of lozenges.  
Then $\xi$ is the endpoint of a side $l_0$ of a lozenge.  It is possible that other (at most countably many) sides of lozenges also terminate at $\xi$.  Consider those in the connected component of $P \setminus  l_0$ whose boundary contains $\xi$, and 
 enumerate them in order by $l_0, l_1, \dots$.  See Figure \ref{fig_boundary_action}.
 Since these leaves share a common endpoint in $\partial P$, they cannot belong to a scalloped region (any two ends of leaves of a scalloped region have uncountably many ends between them, so cannot be identified), thus the set is finite, terminating with some $l_k$.  Let $c$ be the corner in $l_k$, and $f$ the other leaf through $c$.  Since $l_k$ is terminal, there is no lozenge in the quadrant of $c$ bounded by the ray of $l_k$ ending at $\xi$ and $f$, and thus a similar analysis to the case above shows that the endpoint of the ray of $f$ bounding this quadrant is $\eta$ and no points between $\xi$ and $\eta$ are fixed by $g$.  
 \begin{figure}[h]
\labellist
\small \hair 2pt
\pinlabel {$\xi$} at 94.709 204.169
\pinlabel {$\nu$} at 266.392 321.853
\pinlabel {$l_0$} at 73.315 136.170
\pinlabel {$l_1$} at 129.718 132.932
\pinlabel {$l_k$} at 170.392 241.853
\pinlabel {$c$} at 252 261
\endlabellist
\includegraphics[width=5cm]{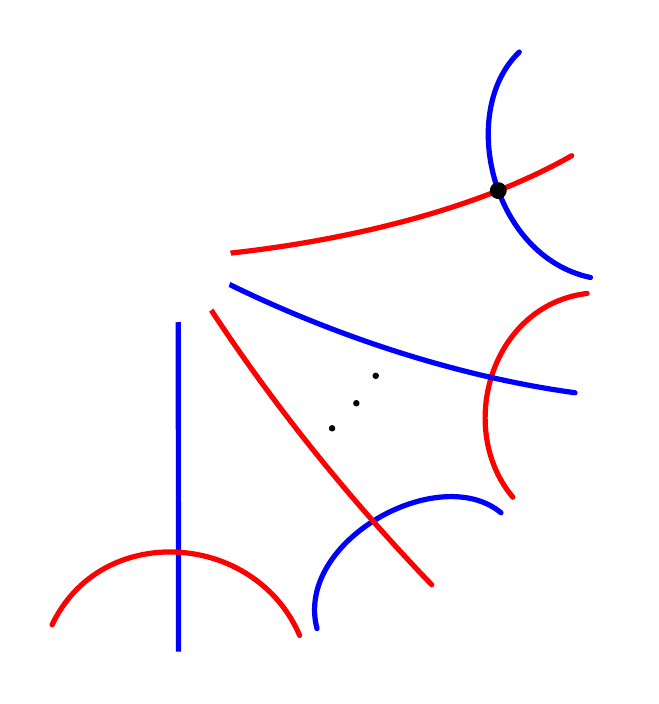}
\caption{Adjacent fixed points from the leaves of the sides of a lozenge}
\label{fig_boundary_action}
\end{figure}

\end{proof}

We return to prove the third statement of the proposition, where $g$ acts freely on $P$.  Then $g$ acts freely on the leaf spaces of $\cF^+$ and $\cF^-$ by \ref{Axiom_A1}, and thus has an {\em axis} $A$ consisting of leaves $l$ such that $l$ separates $g^{-1}(l)$ from $g(l)$ -- see Section \ref{sec_axes} for more on axes.  Let $l \in A$ be a nonsingular leaf.

 Let $\xi, \zeta$ be the endpoints of $l$.   If $g^n(\xi)$ and $g^n(\zeta)$ approach each other as $n \to \infty$, and also approach each other as $n \to -\infty$, then $g$ can have at most two fixed points on $\Pbound$. %
 (Note that it is possible that $g$ can have exactly one fixed point -- but this occurs only when it preserves an infinite chain of lozenges all making a common perfect fit, as occurs in the Franks--Williams example, illustrated in Figure \ref{fig_FW_torus_reeb}.  In particular, such elements always lie in some $\mathbb{Z}^2$ subgroup.) %
 So we can now assume that $g$ has more than two fixed points, and thus conclude that in at least one case (say, without loss of generality as $n \to \infty$) the endpoints of  $g^n(l)$ converge to distinct points.  Call  the two distinct limits $\xi_\infty,\zeta_\infty \in \Pbound$.   
Since $P$ is not trivial, the sequence $g^n(l)$ converges to a union of leaves, either finite or infinite.  
  If the union is finite, then a unique leaf, say $l_1$ has one of its endpoints equal to $\xi_\infty$, which is fixed by $g$. Thus $l_1$ must itself be fixed by $g$ and so $g$ has a fixed point in $P$, contrary to our assumption.  
Otherwise, the limiting union is infinite. In particular, $\xi_\infty$ is a fixed point of $g$ which is accumulated by endpoints of pairwise non-separated leaves and thus is the corner of a $g$-invariant scalloped region. 

It remains to note that $g$ cannot have additional fixed points other than the corners of the scalloped region.
For this, take $\eta \in \Pbound$ not a corner of the scalloped region $U$ preserved by $g$.  Then $\eta$ is separated from the scalloped region by a unique leaf $l'$ in the boundary of $U$.  Since $g$ acts freely on $P$, $g(l') \neq l'$, i.e. it is a different leaf in the boundary of $U$, and thus $g(\eta) \neq \eta$.  
 \end{proof}

\section{Application: exponential growth} \label{sec_exp_growth}
The {\em growth} of a finitely generated group $G$ is the asymptotic growth rate of the function $b_S(n)$ that counts the number of elements which can be written as words of length at most $n$ in the generating set $S$.  While $b_S(n)$ depends on the choice of generating set $S$, whether its growth is polynomial (and of what degree) or exponential, does not.   When $G$ is the fundamental group of a manifold $M$, one can replace $b_S(n)$ with the volume of the ball of radius $n$ about a basepoint $x$ in the universal cover $\wt M$.  The precise value of this volume function depends on the metric, but its growth rate does not, and agrees with the growth rate of $b_S$.  

An early result relating the dynamics of Anosov flows on 3-manifolds with the topology of the supporting manifold was obtained by Margulis (in the appendix of \cite{AS67}) and says that any manifold $M$ supporting an Anosov flow must have exponential growth of $\pi_1(M)$. 
Plante and Thurston gave a different proof, generalizing to codimension 1 Anosov flows in higher dimension \cite{PT72}, and Paternain \cite{Pat93} gave a proof for the pseudo-Anosov (expansive) $3$-dimensional case.  
Here, by using the action on $\partial P$, we give a different proof that all groups acting Anosov-like on bifoliated planes have exponential growth, and in fact in most cases contain nonabelian free subgroups (and in all cases free semi-groups).  

\begin{theorem} \label{thm_exponential_growth}
Suppose that $G$ is a finitely generated group that acts Anosov-like on a bifoliated plane $P$.  Then $G$ has exponential growth.  If $P$ is nontrivial, then $G$ in fact contains a free group on two generators. 
\end{theorem}

The proof of Theorem \ref{thm_exponential_growth} uses different arguments depending on whether the plane is skew, trivial, or not.  
The trivial case follows from the description of such actions as groups of affine transformations -- it is easy to see (by ping-pong) such groups contain free semi-groups.  (In fact, in general, solvable groups are either virtually nilpotent or contain free semi-groups.  But note that they cannot contain free groups because of solubility).  We now give the proof in the skew case, followed by the non $\mathbb{R}$-covered case.  

\begin{proof}[Proof of Theorem \ref{thm_exponential_growth}, skew case]
Assume $P$ is skew, $G$ acts Anosov-like, and let $g, h$ be elements preserving orientation and with disjoint fixed points in $P$.  Let $\eta$ be the one step up map.  Then $\eta$ induces a translation on the leaf space $\Lambda(\cF^+) \equiv \mathbb{R}$, and the action of 
$g$ and $h$ descends to the circle $\Lambda(\cF^+)/\langle \eta \rangle  \equiv S^1$. 
The induced actions of $g$ and $h$ each have disjoint fixed sets on $S^1$, each with one attractor and one repeller.  We now apply the following general result, which is a consequence of the classical Ping-Pong Lemma. 

\begin{proposition} \label{prop_ping_pong}
Suppose $g, h\in G$ with disjoint sets of fixed points in $S^1$.  Then, for any sufficiently large $k, m$ the powers $g^k$ and $h^m$ freely generate a free group.  
\end{proposition}

\begin{proof}
Since $g$ and $h$ have disjoint fixed sets, and these are closed subsets of $S^1$, we can find disjoint open sets $U$ and $V$, each consisting of finitely many intervals, such that $\Fix(g) \subset U$ and $\Fix(h) \subset V$.  For instance, one can take $U$ and $V$ to be $\epsilon$-neighborhoods of $\Fix(g)$ and $\Fix(h)$, respectively, where $\epsilon$ is less than half the minimal distance between the closed sets $\Fix(g)$ and $\Fix(h)$.  
If $I$ is a connected component of the complement of the fixed set of a homeomorphism of the circle, then that homeomorphism attracts all points of $I$ towards one of its endpoints.  Thus, by taking $k$ very large, we can ensure that both $g^k$ and $g^{-k}$ map $V$ into $U$.  Similarly, for large $m$, we have $(h^m(U) \cup h^{-m}(U)) \subset V$.  By the {\em ping-pong lemma} (see, e.g., \cite[Chap. II.B]{dlH00}), we conclude that the group generated by $g^k$ and $h^m$ is a free product $\langle g^k \rangle \ast \langle h^m \rangle$.  
\end{proof}
This shows that $G$ contains a nonabelian free group, hence has exponential growth. 
\end{proof} 

We now discuss the case of planes with prongs or branching.  One might hope to apply Proposition \ref{prop_ping_pong} to the action of elements on $\partial P$.  As soon as one can find two elements with disjoint fixed sets, the proposition applies and gives a free group.  
If the action of $G$ is transitive, the existence of such elements follows directly from Proposition \ref{prop_transitive_non_corner_dense}, which states that non-corner points are dense in $P$, together with the description of dynamics given by Proposition \ref{prop:boundary_action_general}.  However, 
if the action of $G$ is not transitive, it is much harder to establish the existence of any elements with non-corner fixed points.  These do in fact exist -- this is shown in \cite{BBM24b}, but to keep this work self-contained and to give an illustration of another tool in low-dimensional dynamics, we give a different proof of exponential growth, using Margulis' ``Tits alternative" for $\Homeo(S^1)$.  

\begin{theorem}[Margulis' Tits alternative \cite{Mar00}]  \label{thm_tits_alternative}
Let $G$ be a group of homeomorphisms of $S^1$.  Either $G$ preserves a probability measure on $S^1$, or $G$ contains a free group on two generators. 
\end{theorem} 
In the measure-preserving case, one can further describe the algebraic structure of $G$ and the invariant probability measure, depending on the dynamics of the action.  We summarize this here; see \cite{Ghy01} for a detailed description as well as a proof of Theorem \ref{thm_tits_alternative}.  First, if the action of $G$ is minimal, then (up to change of coordinates) the invariant measure is Lebesgue and $G$ acts by rigid motions (rotation and reflection).  If the minimal set for $G$ is a Cantor set, then there is a surjection $S^1 \to S^1$ semiconjugating the action to a minimal one, reducing to the previous case.  If the minimal set is finite, then $G$ has a finite index subgroup fixing this set, hence acting on the line.  

Before the main proof, we need first a definition and lemma about configurations of lozenges limiting to a common boundary point.   

\begin{definition} \label{def_intertwined}
Chains of lozenges $\cC$ and $\cC'$ are said to be {\em intertwined} if each contains a subchain given by an infinite ray 
$R$ in the associated tree $\cT(\cC)$ (respectively, $R' \subset \cT(\cC')$) such that the lozenges $L_1, L_2, ...$ and $L'_1, L'_2, ... $ of $R$ and $R'$ are arranged in one of the two following patterns:
\begin{itemize}
\item {\em Diagonal}: For all $i$, the corners of $L_i$ lie in $L'_{i}$ and $L'_{i+1}$. 
\item {\em Markovian}: $R$ and $R'$ are partitioned into finite lines of lozenges $\mathcal{L}_j$,  $\mathcal{L'}_j$ (allowing for the possibility of trivial lines consisting of a single lozenge) so that for all $j$, $\mathcal{L}_j$ and $\mathcal{L}_{j+1}$ share exactly one corner, $\mathcal{L'}_j$ and $\mathcal{L'}_{j+1}$ share exactly one corner, $\mathcal{L}_{2j} \subset \cF^\pm(\mathcal{L}'_{2j})$ and $\mathcal{L}_{2j+1} \subset \cF^{\mp} (\mathcal{L}'_{2j+1})$.
\end{itemize} 
See Figure \ref{fig_intertwined} for an illustration.
\end{definition}

\begin{figure}[h]
\includegraphics[width=12cm]{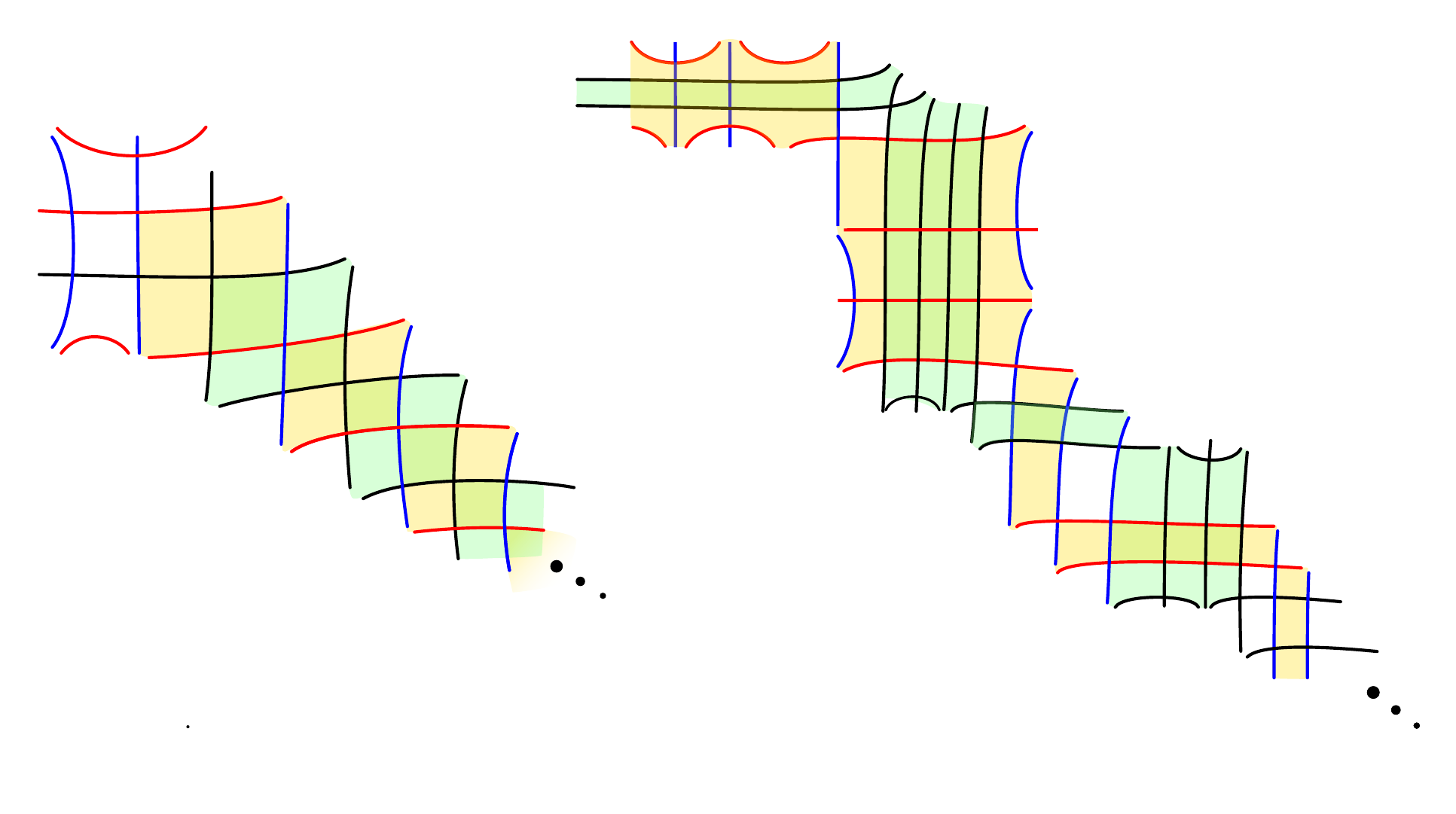}
\caption{Intertwined chains of lozenges, with a diagonal intersection on the left, and a Markovian intersection on the right.}
\label{fig_intertwined}
\end{figure}

\begin{lemma} \label{lem_intertwined} 
Suppose $G$ acts Anosov-like on a bifoliated plane $P$ and $g$ and $h$ are elements preserving distinct maximal chains of lozenges $\cC_g$ and $\cC_h$, respectively, containing lozenges $L_g \subset \cC_g$ and $L_h \subset \cC_h$ with $L_g \cap L_h \neq \emptyset$.  
If $g$ and $h$ fix a common point $\xi \in \partial P$, then either $\cC_g$ and $\cC_h$ contain a common scalloped region, realized as lozenges in two different ways, or $\cC_g$ and $\cC_h$ are intertwined and $\xi$ is a limit of corners of lozenges in the rays $R$ and $R'$ given by Definition \ref{def_intertwined}.  
\end{lemma}

\begin{proof}
Since $\cC_g \neq \cC_h$, and they are assumed maximal, they have no lozenge sides in common and so $\xi$ cannot be the endpoint of a side of some lozenge in the chains, but is a limit of endpoints of sides.  
Let $l_i$ be a sequence of leaves of $\cC_g$ whose endpoints limit to $\xi$.  If $l_i$ does not escape to infinity in $P$, then it limits to one side of a scalloped region fixed by $g$ and $\xi$ is a corner of this scalloped region.  The fact that $h$ also fixes $\xi$ implies that $h$ preserves this scalloped region as well, since two distinct scalloped regions share no common limit points of leaves.  Thus, $\cC_g$ and $\cC_h$ contain a common scalloped region.  Since we assumed $\cC_g$ and $\cC_h$ are distinct, these chains realize the scalloped region as a line of lozenges in two different ways. 

Thus, we can assume that the sequence of leaves $l_i$ (and by symmetry, any sequence of leaves of $\cC_h$ whose endpoints limit to $\xi$) escapes all compact sets.   
Let $\cC'_g \subset \cC_g$ and $\cC'_h \subset \cC_h$ be subchains containing $L_g$ and $L_h$ and each containing leaves limiting to $\xi$, chosen {\em minimal} with respect to inclusion among chains with this property.  We will show these define rays in $\cT(\cC_g), \cT(\cC_h)$ and represent intertwined chains.  

Consider first the case where some corner $c_g$ of $L_g$ is contained in the interior of $L_h$.  By Lemma \ref{lem_markovian_or_corner}, some corner $c_h$ of $L_h$ is contained in $L_g$.  By the non-corner criterion (Lemma \ref{lem_no_corner_criterion}), only one other quadrant of $c_g$ (that diagonal to $L_g$) can contain a lozenge of $\cC_g$.  Let $Q_g$ denote this diagonal quadrant. Similarly, only the diagonal quadrant $Q_h$ can contain another lozenge with corner $c_h$.  The boundary point $\xi$ must be either in $Q_g$ or $Q_h$; up to relabeling we assume it is in $Q_g$, which must then be a lozenge $L'_g$.  Repeating this argument with $L'_g$ and $L_h$, one obtains a diagonal lozenge $L'_h$, and iteratively an infinite pair of intertwined strings limiting to $\xi$. 

Now suppose that no corner of $L_g$ is contained in $L_h$, so by Lemma \ref{lem_markovian_or_corner} they have markovian intersection.  Let $\mathcal{L}_h$ be the maximal line of lozenges containing $L_h$, and $\mathcal{L}_g$ the maximal line containing $L_g$, which are necessarily finite and intersect markovianly. See Figure \ref{fig_finding_chain}.  Consider the extremal corners $c_g$ and $c_g'$ of $\mathcal{L}_g$ and $c_h$ and $c_h'$ of $\mathcal{L}_h$.  We now analyze which (non-adjacent, since we are considering maximal lines) quadrants of  these corners may contain the boundary point $\xi$:  Since $\xi$ is a limit of both $\cC_g$ and $\cC_h$, it must be in quadrants with nonempty intersection.  Up to choice of orientation, these must both have a common orientation (lower right, in the choice of the figure). Thus we can run the above argument again, with $L_g$ and $L_h$ replaced by the lozenges in these intersecting quadrants and their maximal lines.  Iterating as before, one obtains an intertwined chain with markovian intersection.  
\begin{figure}[h]
\labellist
\small \hair 2pt
\pinlabel {\scriptsize $L_g$} at 182.844 292
\pinlabel {\scriptsize $L_h$} at 229.581 245
\pinlabel {$L_g$} at 659.832 218
\pinlabel {$L_h$} at 657.083 260.588
\endlabellist
\includegraphics[width=12cm]{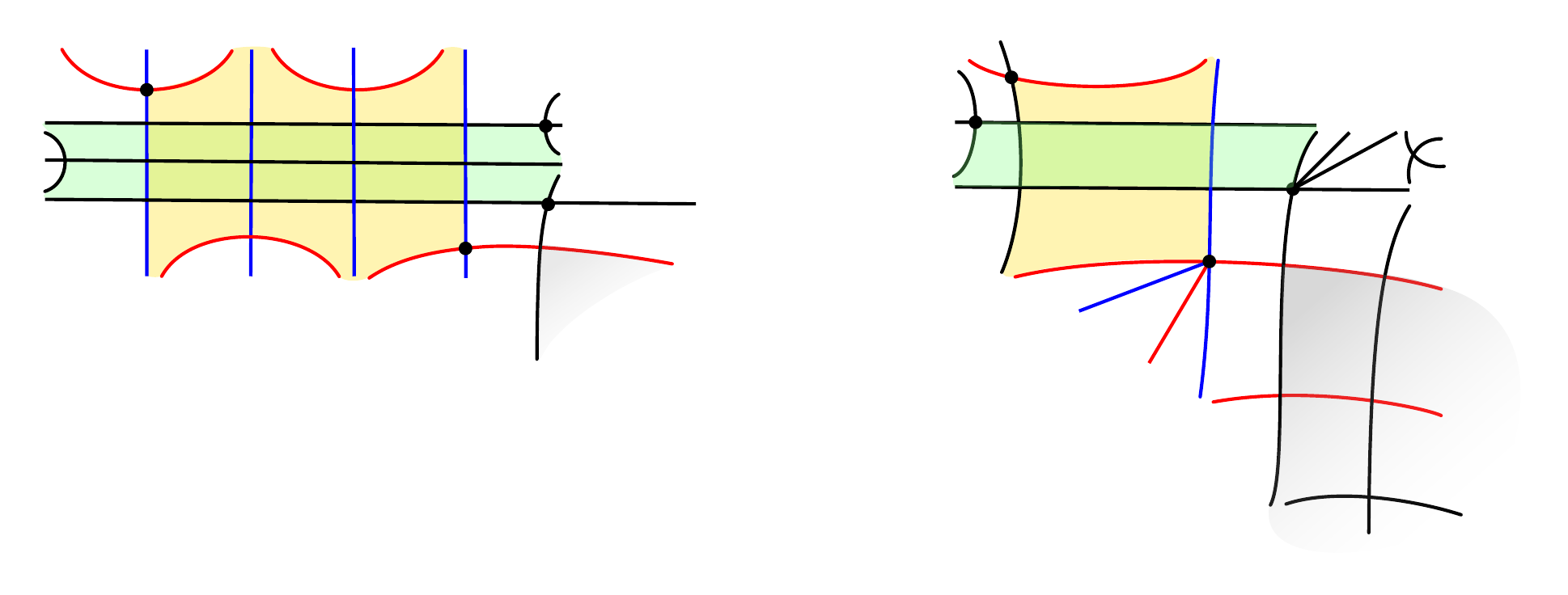}
\caption{Finding a chain of intertwined lozenges.  Note that, even if lozenges have singular corners as in the right side of the figure, there is still only one possible quadrant of overlap where the next diagonal lozenges can be.}
\label{fig_finding_chain}
\end{figure}
\end{proof} 
Now we can finish the proof of Theorem \ref{thm_exponential_growth}.  

\begin{proof}[Proof of Theorem \ref{thm_exponential_growth}, pronged or branching case]
The strategy in this case is to show that the action of $G$ has no invariant measure on $\partial P$, then apply Theorem \ref{thm_tits_alternative}.   Since $G$ is not abelian, it is not conjugate to an action by rotations.  
If $G$ were semiconjugate to an action by rotations on the Cantor set, then the kernel of this action, which is a subgroup of $G$ would still act Anosov-like: indeed, the only axiom that does not pass to subgroups is the density axiom, but this is immediate in this case, since any element fixing a leaf has a power fixing all rays of the leaf, and so in the kernel of the rotation action. Thus, we can replace $G$ with the subgroup acting trivially on the invariant Cantor set.  
 Similarly, if $G$ has a finite orbit on $S^1$, one can pass to a finite index subgroup to reduce to the case where $G$ has a global fixed point.  
Thus in order to apply the Tits' alternative and obtain a free subgroup, we need only prove the following general proposition: 

\begin{proposition} \label{prop_global_fixed_implies_R_covered}
Let $G$ be a group acting Anosov-like on a bifoliated plane.  If $G$ has a global fixed point on $\partial P$, then $P$ is skew or trivial.
\end{proposition}
\begin{proof} 
Most of the work of the proof has already been done in Lemma \ref{lem_intertwined}.  
Suppose $G$ acts with global fixed point $\xi$.   If $\xi$ is the endpoint of a leaf, then this leaf is fixed by $G$, which easily contradicts Axiom \ref{Axiom_dense} and Theorem \ref{thm:distinct_fix_is_chain}.   Similarly, $\xi$ cannot be the corner of a scalloped region, as that would also need to be preserved by $G$, which is absurd.  
Thus, by Proposition \ref{prop:boundary_action_general}, every element of $G$ must either act freely or preserve some infinite chain of lozenges, which contains a sequence of sides whose endpoints limit to $\xi$.   

Suppose for contradiction that $P$ is neither skew nor trivial.  As a first case, suppose there is some branching so there exists a pair of adjacent lozenges $L_1$ and $L_2$ whose corners are fixed by some nontrivial $g$ in $G$.  For concreteness assume their shared side $l$ lies in $\cF^+$, and let $c$ denote their shared corner.  Switching the labels of $L_1$ and $L_2$ if needed, we assume $\xi$ lies on the side of $l$ containing $L_2$, and choose a local orientation so that $L_2$ is below $L_1$, as in Figure \ref{fig_no_fix_point}.  By Axiom \ref{Axiom_dense}, we can find a leaf $f$ of $\cF^+$ passing through $L_1$, which contains a point fixed by a nontrivial element $h$ of $G$.  Since $h$ and $g$ both fix $\xi$, Lemma \ref{lem_intertwined} implies that they each preserve an infinite chain of lozenges, which are necessarily intertwined, and limit to $\xi$.  The position of $L_2$ prevents the possibility of diagonal lozenges, by the non-corner criterion (Lemma \ref{lem_no_corner_criterion}).   Thus, the intersection is markovian and $\xi$ is in the quadrant of $c$ diagonal from $L_1$.  (If $c$ is singular, $\xi$ is more precisely in the quadrant of $c$ not containing $L_2$ but adjacent to the quadrant adjacent to $L_1$.)

\begin{figure}[h]
\labellist
\small \hair 2pt
\pinlabel {$f'$} at 442.005 334.419
\pinlabel {$f$} at 368.611 255.793
\pinlabel {$L_1$} at 245.415 284.623
\pinlabel {$L_2$} at 245.415 178.478
\pinlabel {$c$} at 163 222
\pinlabel {$l$} at 240.172 220
\pinlabel {$\xi$} at 40 140
\pinlabel {$p$} at 755.240 213.859
\pinlabel {$x$} at 780 340
\pinlabel {$\xi$} at 672.672 22.536
\endlabellist
\includegraphics[width=12cm]{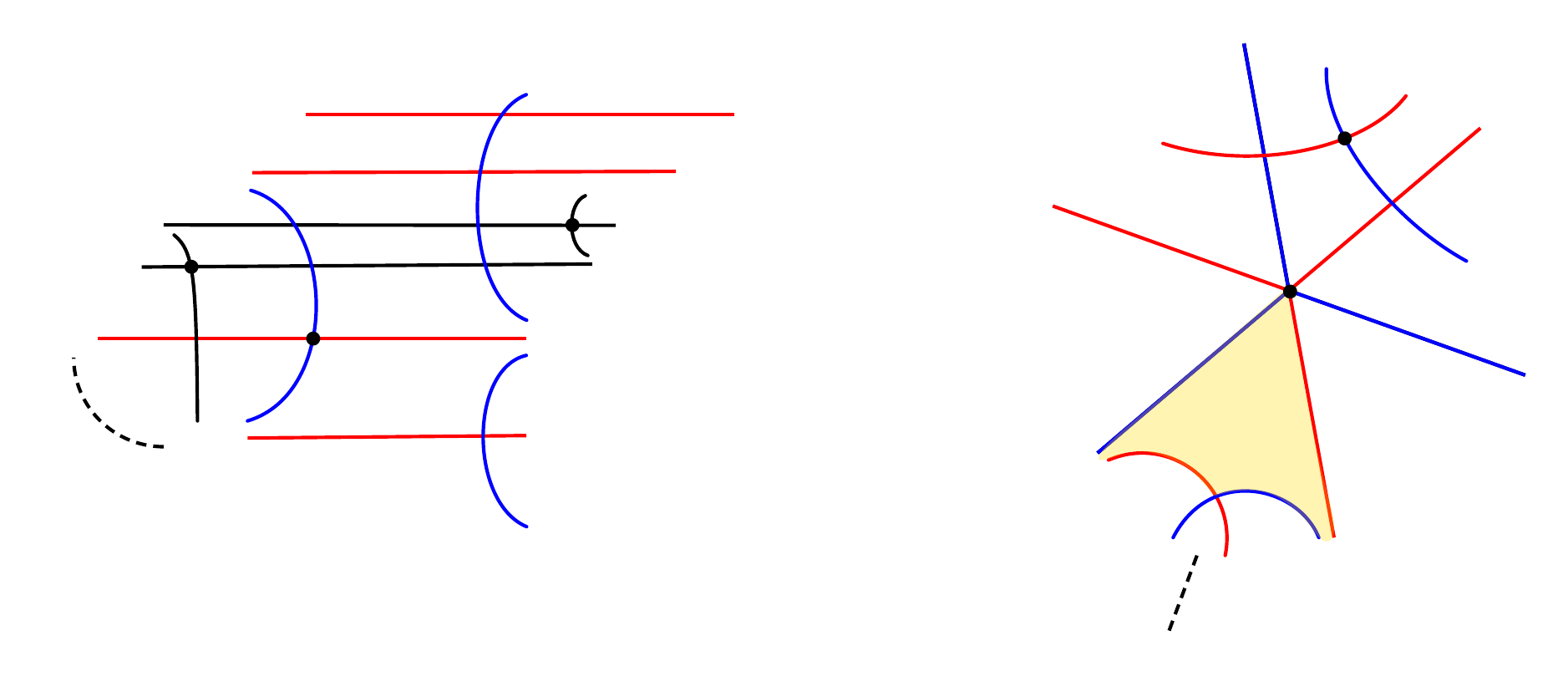}
\caption{The two cases to show no boundary point $\xi$ can be fixed, in the presence of branching or prongs}
\label{fig_no_fix_point}
\end{figure}

Let $f'$ be any leaf fixed by a nontrivial element of $G$ and above the top side of $L_1$.  Applying again Lemma \ref{lem_intertwined} and the same reasoning as above, we conclude that it must be part of an intertwined chain that intersects markovianly with the chain invariant by $g$.  But such a chain cannot limit to a point in the quadrant of $c$ containing $\xi$. See again Figure \ref{fig_no_fix_point}. This gives the desired contradiction.  

It remains to treat the case where $P$ has prong singularities.  Again, since $\xi$ is not the endpoint of a leaf, each prong (which is by Axiom \ref{Axiom_prongs_are_fixed} fixed by a nontrivial element of $G$) must be a corner of an infinite chain of lozenges which contains sides limiting to $\xi$.   Take such a prong $p$, and let $x$ be a point fixed by any nontrivial element and separated from $\xi$ by both $\cF^+(p)$ and by $\cF^-(p)$.  Such a point exists by Axiom \ref{Axiom_dense}.  Again, it is easy to verify by the non-corner criterion that $x$ cannot be part of a chain intertwined with the chain containing $p$ as a corner and limiting to $\xi$, giving a contradiction in this case.  

\end{proof} 
This concludes the proof of Proposition \ref{prop_global_fixed_implies_R_covered}, and thus of Theorem \ref{thm_exponential_growth}.
\end{proof} 

\section{Actions on leaf spaces and axes of elements}  \label{sec_axes}
We have already exploited the induced actions on leaf spaces in some of the results of Chapter \ref{chap_bifoliated_planes}.  Here we describe an important property for elements acting freely, namely, the existence and behavior of axes.  

If $g$ acts freely on a (bi)-foliated plane satisfying Axiom \ref{Axiom_A1}, then it acts freely on the leaf space(s) of the foliation(s).  
It is a classical result of  Tits \cite{Tit70} that a free automorphism of a (Hausdorff) tree has an {\em axis}: an invariant, embedded copy of $\bR$ along which it acts by translation.  Barbot \cite{Bar98b} generalized the notion of axis to automorphisms of simply connected non-Hausdorff 1-manifolds, and Roberts--Stein \cite{RS01} and Fenley \cite{Fen03} considered the case of non-Hausdorff trees, which include all leaf spaces of the bifoliated planes we consider here. 

\begin{definition}
For an element $g$ acting freely on a non-Hausdorff tree $T$, the \emph{axis} of $g$, denoted by $A(g)$, is defined to be the set of all points $x \in T$ such that $g(x)$ separates $x$ from $g^2(x)$.   
\end{definition}

\begin{theorem}[Theorem A of \cite{Fen03}] \label{thm_axes_exist}
If $g$ acts freely on a non-Hausdorff tree, then $A(g) \neq \emptyset$.  
\end{theorem} 
Fenley's proof of Theorem \ref{thm_axes_exist}
can be simplified somewhat in the situation of leaf spaces on bifoliated planes, where we know that $g$ acts freely if and only if all powers of $g$ act freely (this follows from the Brouwer plane translation theorem, see Proposition \ref{prop_primitive_stabilizer}), and we have more information on the structure of branching sets.  

When a tree is non-Hausdorff, axes need not be homeomorphic to $\bR$, but they do always have a linearly ordered structure, as follows. 

\begin{proposition}[See Lemma 3.5 and Proposition 3.10 of \cite{Fen03}] \label{prop:axis_structure}  
Suppose $g$ acts freely on a non-Hausdorff tree.  Then $A(g)$ is either an embedded (but not necessarily properly embedded) copy of $\R$, or has the following structure: 
there exists a collection of embedded closed intervals $[a_i, b_i]$ indexed by $i \in \bZ$ such that for all $i$, $a_i$ and $b_{i-1}$ are nonseparated points, and
\[A(g) = \bigcup_{i \in \bZ} [a_i, b_i].\]

In the first case, $g$ acts on the embedded copy of $\R$ by translations, and in the second case there exists $k \in \bZ$ such that $g([a_i, b_i]) = [a_{i+k}, b_{i+k}]$
\end{proposition} 

In the case of a free element of an Anosov-like action, one can upgrade the above proposition to separate the cases where the copy of $\bR$ is properly embedded versus embedded, see Lemma \ref{lem_commuting_axis_2}.

We note also the following observation, which is an immediate consequence of the definition. 

\begin{observation} \label{obs_conjugate_axis} 
Let $h, g$ be homeomorphisms of a non-Hausdorff tree $T$, with $g$ acting freely.  Then $hgh^{-1}$ acts freely and $A(hgh^{-1}) = hA(g)$. 
\end{observation} 
In particular, $A(g)$ is invariant under all elements commuting with $g$.   
Less obvious, though well known for Hausdorff trees, is the following: 
\begin{lemma} \label{lem_commute_same_axis}
If $h$ and $g$ commute and both act freely, then $A(h) = A(g)$  
\end{lemma}
We omit the proof here, but do prove a special case in Lemma \ref{lem_commuting_axis_2} below.  
A proof of the general statement can be found in \cite[Lemma 3.11]{Fen03}.

\section{Abelian normal subgroups and rigidity} \label{sec_abelian_normal}
As a first application of axes, we will prove the following rigidity result, which gives strong constraints on Anosov-like actions for any group with a nontrivial abelian normal subgroup. 

\begin{theorem}\label{thm_normal_abelian_implies_rigidity}
Suppose $G$ has a nontrivial abelian normal subgroup and acts Anosov-like on a bifoliated plane $(P, \cF^+, \cF^-)$.   Then exactly one of the following holds: 
\begin{enumerate}[label=(\roman*)]
\item $P$ is skew and the index at most $2$ subgroup of $G$ that preserves leafwise orientation has nontrivial center.  This center  is isomorphic to $\bZ$ and generated by a power of the one-step-up map. 
\item $P$ is trivial, $G$ is virtually solvable with trivial center, and its action is by affine transformations.
\end{enumerate} 
\end{theorem}

After the proof, we apply this to the setting where $G = \pi_1(M)$ and the action is that induced on the orbit space of a pseudo-Anosov flow.  We show that, with the hypotheses on $G$ above, the flow is either a suspension of a hyperbolic linear map of the torus, or a lift to a fiberwise cover of geodesic flow on a hyperbolic orbifold.  
A version of this result for Anosov flows is originally due to combined work of Plante and Ghys.  

\begin{proof}[Proof of Theorem \ref{thm_normal_abelian_implies_rigidity}]
Let $G$ be as in the statement of the theorem and let $H \subset G$ be a nontrivial abelian normal subgroup.  Consider the action of $H$ on $P$.  If some element of $H$ has fixed points, then this set of fixed points is preserved by $H$, so there is either a unique global fixed point for $H$ or an $H$-invariant chain of lozenges $\cC$. 

In the first case, since $H$ is normal, the global fixed point must also be invariant by $G$, which contradicts Axiom \ref{Axiom_dense}. 
In the second case, consider the maximal chain $\cC'$  containing $\cC$.  Then $\cC'$ is also $H$-invariant, and the {\em unique} $H$-invariant  {\em maximal} chain.  (Uniqueness follows from Corollary \ref{cor:maximal_fixed}.)
Uniqueness and the fact that $H$ is normal imply that $G$ preserves $\cC'$.  This is also easily seen to contradict Axiom \ref{Axiom_dense}: any element $g \in G$ which fixes a leaf that meets, but is not contained in, the interior of a lozenge of $\cC'$, cannot preserve the chain.  

Thus, we conclude that $H$ acts freely.   Since $H$ is abelian, Lemma \ref{lem_commute_same_axis} implies that all elements of $H$ share a common axis.  Denote this axis by $A$.   Since $H$ is normal, \ref{obs_conjugate_axis} implies that $A$ is invariant under $G$.  Thus, $A$ must be equal to $\Lambda(\cF^+)$.  By Proposition \ref{prop:axis_structure}, the axis has a linearly ordered structure; the fact that it is equal to $\Lambda(\cF^+)$ means it must be isomorphic to $\bR$ and thus the plane is either skew or trivial. We treat these two cases separately.

\vspace{.2cm}
\noindent {\bf Case 1: $P$ trivial.}
By the classification of Anosov-like actions on trivial planes (Proposition \ref{prop:trivial_affine}), $G$ is isomorphic to a subgroup of $\mathrm{Aff}(\bR)$.  In particular, it is solvable.  Moreover, any subgroup of $\mathrm{Aff}(\bR)$ with nontrivial center is either a group of translations or has a global fixed point.  This is ruled out by Axioms \ref{Axiom_A1} and \ref{Axiom_dense}, so we conclude $G$ has trivial center.

\vspace{.2cm}
\noindent {\bf Case 2: $P$ skew.}
If $P$ is skew, the action of $G$ on $\Lambda(\cF^+)$ realizes it as a group of homeomorphisms of $\bR$ that commute or anti-commute with translation by 1, depending on whether orientation-preserving or not.  
(See Exercise \ref{ex_charac_skew} and the preceding discussion.)
Let $G_0$ be the subgroup preserving orientation.  Since $H$ acts freely, we have $H \subset G_0$, and it is again a normal subgroup.  

We next show that $H$ is central in $G_0$, by using the {\em translation number}.  
It is a classical construction (used in the definition of the \emph{rotation number} of Poincaré), that homeomorphisms of $\bR$ commuting with integer translations have a well-defined, real-valued {\em translation number} given by $\tau(g):= \lim_{n \to \infty} g^n(0)/n$.  This number $\tau$ is conjugation invariant.  Moreover, restricted to any abelian group, $\tau$ is a homomorphism, and $\tau(g) \neq 0$ if $g$ acts freely.  See e.g., \cite[Proposition 5.3]{Ghy01} and \cite[Proposition 2.2.11]{Nav11}.  Thus, $\tau$ gives an isomorphism between $H$ and a subgroup of $\bR$.
As $H$ is normal in $G_0$, $G_0$ acts by conjugation on $H$, and since the translation number is conjugation invariant, the conjugation action of $G_0$ on $H$ must be trivial, i.e., $H$ is in the center of $G_0$.

It remains to show that the center of $G_0$ is cyclic and generated by a power of the one-step-up map.
Let $g$ be an element of the center. For each $h \in G_0$ acting with fixed points, since $g$ commutes with $h$, it must preserve the set of attracting fixed points of $h$.  In the standard model of the skew plane, such a set has the form $\{(x_0 + n, y_0 + n) : n \in \Z\}$, thus $g$ acts as translation by some integral pair $(k, k)$, when restricted to this set.  The set of all such attracting fixed points is dense in each coordinate by Axiom \ref{Axiom_dense},  so $g$ is globally an integral translation, i.e., a power of the one-step-up map. In particular, the center is cyclic and generated by a power of the one-step-up map, as claimed.
\end{proof}

Specializing Theorem \ref{thm_normal_abelian_implies_rigidity} to the setting of pseudo-Anosov flows, we obtain the following:  

\begin{theorem}\label{thm_Ghys_Plante_theorem} 
Let $M$ be a compact 3-manifold with pseudo-Anosov flow $\flow$. Suppose that $\pi_1(M)$ has a nontrivial abelian normal subgroup.
Then either $\flow$ is orbit equivalent to a suspension of an Anosov diffeomorphism or $M$ is a fiberwise cover of the unit tangent bundle of a hyperbolic surface or orbifold, and $\flow$ is orbit equivalent to a lift of the geodesic flow.  
\end{theorem} 

This statement combines a theorem due to Plante \cite{Pla81}\footnote{There is a caveat: one crucial step was not in fact proven in \cite{Pla81}, see \cite[Chapitre 3]{Bar05} for a discussion of the history of this result.} and one due to Ghys \cite{Ghy84} (both proved originally for Anosov flows).  
The perspective we give here uses the convergence group theorem as well as Theorem \ref{thm_extended_convergence}.  Barbot's {\em Mémoire d'habilitation} presents 
a similar result with a proof following the same strategy, see \cite[Théorème 3.1]{Bar05}.

\begin{proof}[Proof of Theorem \ref{thm_Ghys_Plante_theorem}]
Suppose $M$ has a pseudo-Anosov flow $\flow$ and  
let $G = \pi_1(M)$.  Suppose that $G$ contains a nontrivial normal abelian subgroup. 
By Theorem \ref{thm_normal_abelian_implies_rigidity}, the orbit space for $\flow$ is either trivial or skew.  
In the trivial case, Theorem \ref{thm_trivial_implies_suspension} gives the desired result: all such flows are orbit equivalent to suspensions of linear hyperbolic maps.  

In the skew case, Theorem \ref{thm_normal_abelian_implies_rigidity} further tells us that (up to taking an index at most $2$ subgroup of) $G$ has a cyclic center generated by a power of the one-step-up map.  Let $\eta$ denote the induced action of one-step-up map on $\Lambda(\cF^s) \cong \bR$. 

By Theorem \ref{thm_extended_convergence}, $G$ acts properly discontinuously on the space of triples 
\[ T:= \{(x, y, z) \in \Lambda(\cF^s)^3 : x < y < z < \eta(x) \}.\]  Fix a minimal $k>0$ so that $\eta^k \in G$.  
Since $\eta^k \in G$, the action of $G/\langle \eta^k \rangle$ on the quotient of $T$ by the diagonal action of $\eta$ is also properly discontinuous.  This quotient is precisely the space of distinct triples in $\Lambda(\cF^s) / \langle \eta \rangle  \cong S^1$.  Since $G$ acts faithfully on $\Lambda(\cF^s)$, the kernel of the action of $G$ is precisely the subgroup $\langle \eta^k \rangle$. 

The {\em convergence group theorem} of Pekka Tukia \cite{Tuk88}, David Gabai \cite{Gab92} and Andrew Casson and Douglas Jungreis \cite{CJ94} states that $G/\langle \eta^k \rangle$ is the fundamental group of some hyperbolic orbifold $\Sigma$ and the action on $S^1$ is conjugate to the standard action of this group on the boundary circle of $\bH^2 \cong \wt \Sigma$.   Lifting the action to $\Lambda(\cF^s) \cong \bR$ and using the description of actions on the skew plane,  it follows that the action of the group generated by $G$ and $\eta$ is isomorphic to the fundamental group of the unit tangent bundle of the orbifold $\Sigma$, and its action is  
conjugate to the standard action of $\pi_1(T^1 \Sigma)$ on the orbit space of geodesic flow.  
If $k = 1$ and $\eta = \eta^k \in G$, then $G = \pi_1(T^1\Sigma)$.  For general $k>0$, when 
 $\eta^k$ is the minimal power of $\eta$ in  $G$, we have instead that $G$ is the fundamental group of the $k$-fold fiberwise cover of $T^1 \Sigma$, and its action on the orbit space is conjugate to the lift of geodesic flow to this cover.  By Theorem \ref{thm_action_determines_OEflow}, $\flow$ is orbit equivalent to this lifted geodesic flow.  
\end{proof}

\section{Axes homeomorphic to $\bR$ and skew-like pieces}  \label{sec_skew-like}
If $g$ acts freely on $P$, then it acts freely on both $\Lambda(\cF^+)$ and $\Lambda(\cF^-)$, thanks to Axiom \ref{Axiom_A1}, so has axes in both leaf spaces.  
A particular case of interest  
is when both axes are homeomorphic to $\bR$, as it forces a surprising amount of additional structure.   We develop this structure now, proving a ``trichotomy" style result (Theorem \ref{thm_both_axes_R}), weaker versions of which for flows appeared in \cite{Fen03} and \cite{BF21}.

As a warm-up, we prove the following (extension of a) special case of Lemma \ref{lem_commute_same_axis} for commuting elements with axes homeomorphic to $\bR$.  

\begin{lemma}  \label{lem_commuting_axis_2}
Suppose $G$ acts Anosov-like on a bifoliated plane, and $h, g$ are two commuting elements, each acting freely.  Let $A(h), A(g)$ denote the axes of $h$ and $g$ respectively in $\Lambda(\cF^+)$. 
If $A(h)$ and $A(g)$ are both homeomorphic to $\bR$, then $A(h) = A(g)$, and are either properly embedded copies of $\bR$ or the union of leaves in a scalloped region.  
\end{lemma} 

\begin{proof}
Suppose $A(h)$ and $A(g)$ are both homeomorphic to $\bR$, so in particular they are connected.  
First consider the case where $A(h)$ is not properly embedded.  Then, in at least one direction, the leaves of $A(h)$ limit onto a leaf or nontrivial union of nonseparated leaves.   Since $A(h)$ is $g$-invariant, this limit is invariant under $g^2$.  

If this limit is a single leaf or finite union, then some leaf is fixed by $g^2$, contradicting that $g$ acts freely.  So the limit is an infinite union of nonseparated leaves. By Corollary \ref{cor_infinite_non_sep_implies_scalloped}, it is the boundary of a scalloped region, which therefore must also be invariant under $g^2$.  The description of the stabilizer of a scalloped region in Lemma \ref{lem_action_of_stabilizer_scalloped} and the fact that $g$ acts freely now imply that the union of leaves in $A(g)$ is equal to the $\cF^+$ saturation of this scalloped region.  

A symmetric argument applies if $A(g)$ is not properly embedded.  
Thus, we can now suppose that both axes are properly embedded copies of $\bR$.  We first show they intersect.  
If $A(h) \cap A(g) = \emptyset$, then $A(h)$ lies in a connected component of $\Lambda(\cF^+) \setminus A(g)$, which is impossible since $g$ translates along its axis, permuting connected components of the complement.   Thus, $A(h) \cap A(g) \neq \emptyset$.  This intersection forms a connected, $g$-invariant subset of $A(g)$ (because $g$ preserves $A(h)$), so it is necessarily equal to $A(g)$.  By symmetry we also have $A(h) \subset A(g)$, so they are equal.  
\end{proof}

Thus far we have only made use of one leaf space at a time.  However, when an element acts freely on $P$ it has axes in both $\Lambda(\cF^+)$ and $\Lambda(\cF^-)$.   We will show that, in the case where both axes are homeomorphic to $\bR$, their intersection is a nontrivial subset of $P$ and this subset satisfies the following generalization of Proposition \ref{prop:both_R}.  

\begin{theorem}  \label{thm_both_axes_R}
Suppose $g$ acts freely on $P$ with both axes $A^\pm$ of $g$ homeomorphic to $\bR$.  
Let $\Omega := \{p \in P : p = l^+ \cap l^- \text{ for some } l^+ \in A^+, l^- \in A^-\}$. Then $\Omega \neq\emptyset$ and either 
\begin{enumerate}
\item $\Omega = P$ and the plane is trivial, or 
\item $\Omega$ is a scalloped region, or 
\item The functions $s, i\colon A^+ \to A^- \cup \{-\infty, +\infty\} \cong \bR \cup \{-\infty, +\infty\}$ defined by 
\[s(l^+) = \sup\{l^- \in A^- : l^+ \cap l^- \neq \emptyset \}\]
\[i(l^+) = \inf\{l^- \in A^- : l^+ \cap l^- \neq \emptyset \} \] take only finite values and are weakly monotone, and their graphs (as seen in $A^+\times A^-$) bound the subset corresponding to $\Omega$.
Furthermore, one of (or equivalently both of) $s$ and $i$ is strictly monotone if and only if $P$ is skew. 
\end{enumerate} 
\end{theorem} 
Note the parallel between the third case and the construction of the skew plane in Proposition \ref{prop:both_R}.  We say in the third case that $\Omega$ is a {\em skew-like piece} of $P$.  

\begin{rem}
Having one axis homeomorphic to $\bR$ does not imply that the other axis is $\bR$. One can build examples of elements $g$ acting freely on a chain of lozenges and such that the topological types of the axes can be any combinations of $\bR$ and/or union of intervals, see Figure \ref{fig_different_axis}. 
\end{rem}
\begin{figure}[h]
\labellist
\small \hair 2pt
\pinlabel {$c$} at 40.924 368.656
\pinlabel {$gc$} at 290.430 368.656
\pinlabel {$c$} at 398 347
\pinlabel {$gc$} at 673 354
\pinlabel {$c$} at 83 49
\pinlabel {$gc$} at 262 180
\pinlabel {$c$} at 565 71
\pinlabel {$gc$} at 730.815 94.888
\endlabellist
\includegraphics[width=14cm]{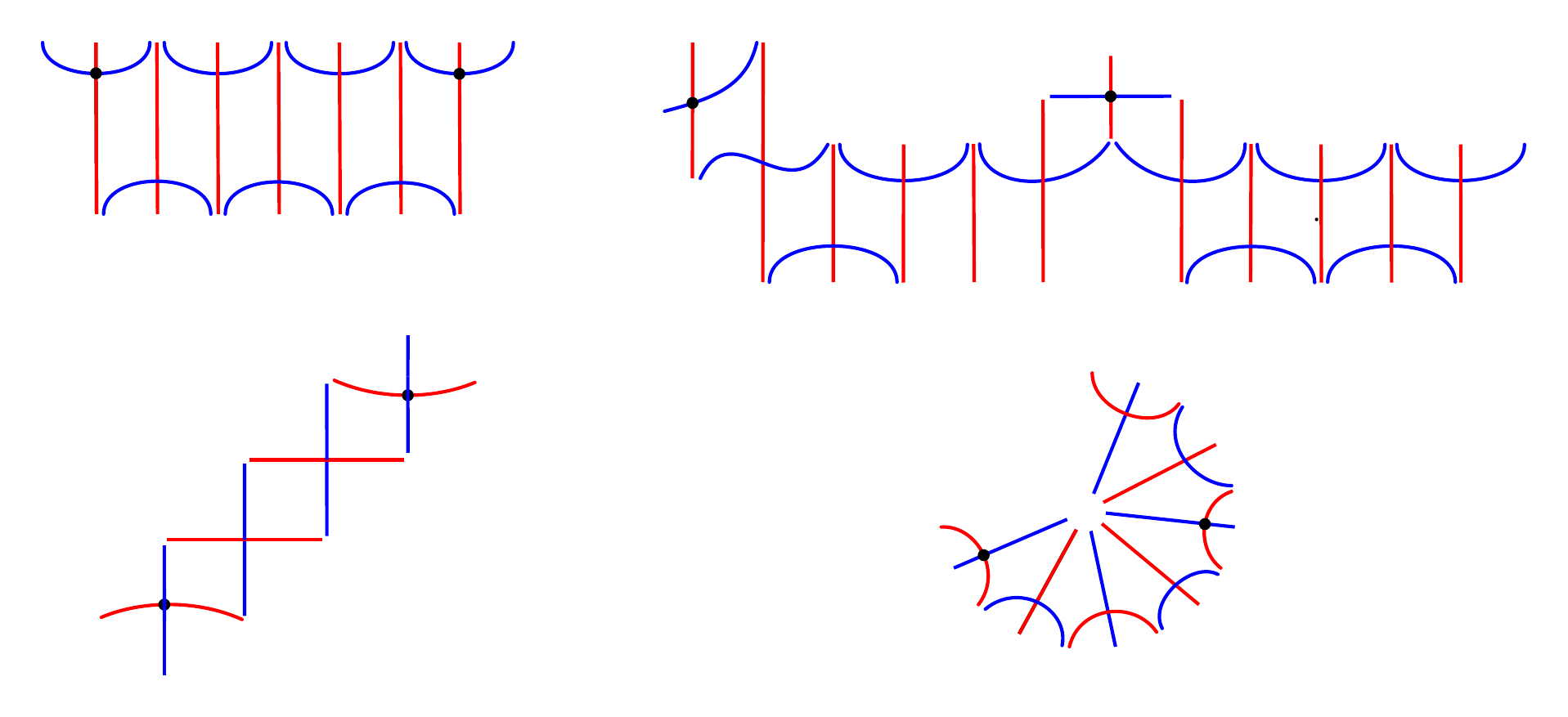}
\caption{An element $g$ acting freely on a chain of lozenges. The string and infinite line on the left gives $\bR$ axis for both leaf spaces. The top right chain gives $A^+\simeq \bR$ while $A^-$ is a union of disjoint single leaves. The bottom right chain give both axes as disjoint union of leaves.}
\label{fig_different_axis}
\end{figure}
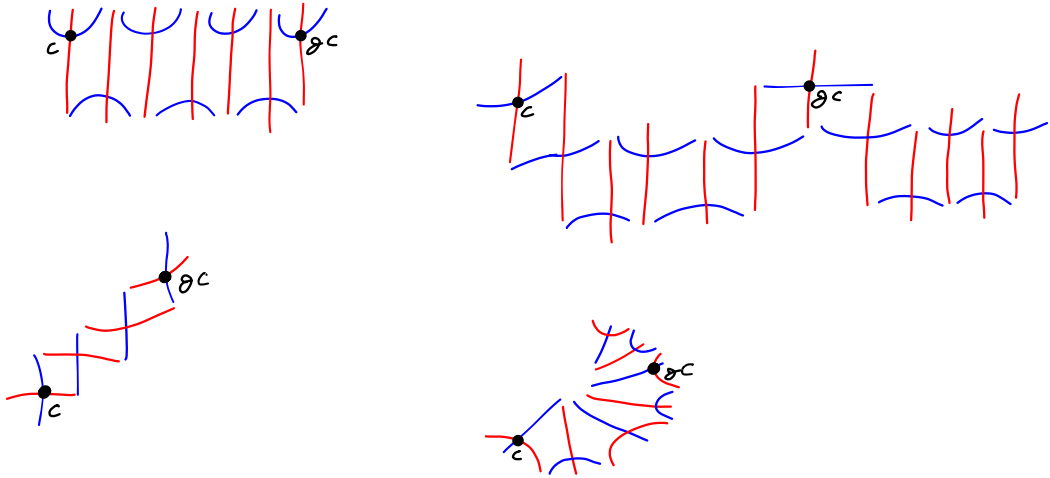

A version of  Theorem \ref{thm_both_axes_R} for pseudo-Anosov flows appears in Section 5 of \cite{BF21}.  It is a key step in describing the structure of orbits in the so-called {\em free Seifert pieces} of 3-manifolds admitting pseudo-Anosov flows.   We discuss these further in Section \ref{sec_free_seifert}.   Other versions of this result were also used previously by Fenley, for instance to prove the main result of \cite{Fen03}.  

The proof we give below is somewhat different and shorter than that of \cite{BF21}, using the leaf space rather than the plane as the primary object of study, and the properties of Anosov-like actions.   We begin with a series of Lemmas.  
As notation, we let $\cF^+(A^-)$ denote the set of leaves in $\Lambda(\cF^+)$ that contain points of leaves in $A^-$, and similarly with the roles of $-$ and $+$ reversed.  

\begin{lemma} \label{lem_intersections_axes}
Suppose $g \in G$ acts freely on $P$, and let $A^\pm$ denote the axis of $g$ in the leaf space $\Lambda(\cF^\pm)$.   If both $A^+$ and $A^-$ are homeomorphic to $\bR$, then $\cF^+(A^-) \supset A^+$ and $\cF^-(A^+) \supset A^-$.
\end{lemma} 

In particular, this lemma states that every leaf $l$ of $A^\pm$ meets a leaf of $A^\mp$.

\begin{proof} 
The statement is easily seen true if $P$ is skew or trivial, so we can assume $P$ is not.  
By Lemma \ref{lem_commuting_axis_2}, if $A^+$ is not a properly embedded copy of $\bR$, then it corresponds to a scalloped region $U$; the fact that $g$ acts freely implies that it translates both of the lines of lozenges forming $U$, and so $A^-$ corresponds to the same scalloped region, which gives the desired conclusion.   Thus, we can assume that the axes are properly embedded.  

The proof that $\cF^+(A^-) \supset A^+$ is very similar to the argument at the end of Lemma \ref{lem_commuting_axis_2}.  
Both $A^+$ and $\cF^+(A^-)$ are connected, $g$-invariant subsets of $\Lambda(\cF^+)$.   If $A^+ \cap \cF^+(A^-) = \emptyset$, then $\cF^+(A^-)$ would be contained in a single connected component of $\Lambda(\cF^+) \setminus A^+$.  However, since $g$ translates along the axis $A^+$, the connected components are permuted by $g$, contradicting $g$-invariance of $\cF^+(A^-)$.  This shows $A^+ \cap \cF^+(A^-) \neq \emptyset$.  Since $\Lambda(\cF^+)$ is a (non-Hausdorff) tree, $A^+ \cap \cF^+(A^-)$ is connected.  Finally, 
 $g$-invariance also means that $A^+ \cap \cF^+(A^-)$ is unbounded in $A^+$, in each direction.  It follows that $A^+ \cap \cF^+(A^-) = A^+$, which is what we needed to show.  

By symmetry, the same is true with the roles of $\cF^+$ and $\cF^-$ reversed, completing the proof. 
\end{proof}

Keeping the notation and assumptions from above, we prove the following.  
\begin{lemma} \label{lem_nondegenerate_interval}
Let $\cup{A^\pm} \subset P$ denote the union of leaves in $A^\pm$.  
If $l$ is a leaf of $A^\pm$, then $l \cap (\cup A^\mp)$ is a nondegenerate sub-interval of $l$. 
\end{lemma}

\begin{proof}
For concreteness, suppose $l$ is a leaf of $A^+$.  If $l_1^-$ and $l_2^-$ are leaves of $A^-$ that intersect $l$, then $l_i^-$ separates $g(l_i^-)$ from $g^{-1}(l_i^-)$, and so the same is true for any leaf through $l$ between them.  This shows $l \cap \cup A^-$ is connected.  
Lemma \ref{lem_intersections_axes} says that it is nonempty.  

Now suppose $l^- \in A^-$ is a leaf such that $l^- \cap l$ is nonempty. If $l^-$ is nonsingular, then a small segment of $l$ near the intersection point gives a local transversal to $\cF^-$ and so projects nontrivially to $A^-$, thus $l \cap (\cup A^-)$ is nondegenerate.
If instead $l^-$ is a singular leaf, consider a sequence $l^-_n$ of nonsingular leaves in $A^-$ that converge to $l^-$ in such a way that $l^-_{2n}$ and $l^-_{2n+1}$ are in distinct connected components of $A^-\smallsetminus \{l^-\}$ for all $n$. If, for some $n$, the leaf $l^-_n$ intersects $l$, then the first part of the argument lets us conclude. So we must have that for all $n$, $l^-_n$ does not intersect $l$. This implies that $l^-_{2n}$ and $l^-_{2n+1}$ converge to two faces $f_1,f_2$, respectively, of $l^-$ such that neither $f_1$ nor $f_2$ intersects $l$.
In particular, $f_1\cup f_2$ separates $\cup A^- \smallsetminus \{l^-\}$ from $l$, and $g$ maps $l^-$ to a leaf contained in one of the quadrants bounded by $f_1$ and $f_2$, and $g^{-1}$ maps $l^-$ into the other.  This implies that $l$ cannot separate $g l$ from $g^{-1} l$, contradicting that $l$ is in the axis of $g$.
\qedhere
\begin{figure}[h]
\labellist
\small \hair 2pt
\pinlabel {$l$} at 184.907 191.923
\pinlabel {$l'$} at 52.323 225.962
\pinlabel {$l$} at 482.916 187.060
\pinlabel {$l'$} at 426.964 229.609
\pinlabel {$l$} at 780.926 187.060
\pinlabel {$l'$} at 700 229.609
\endlabellist
\includegraphics[width=10cm]{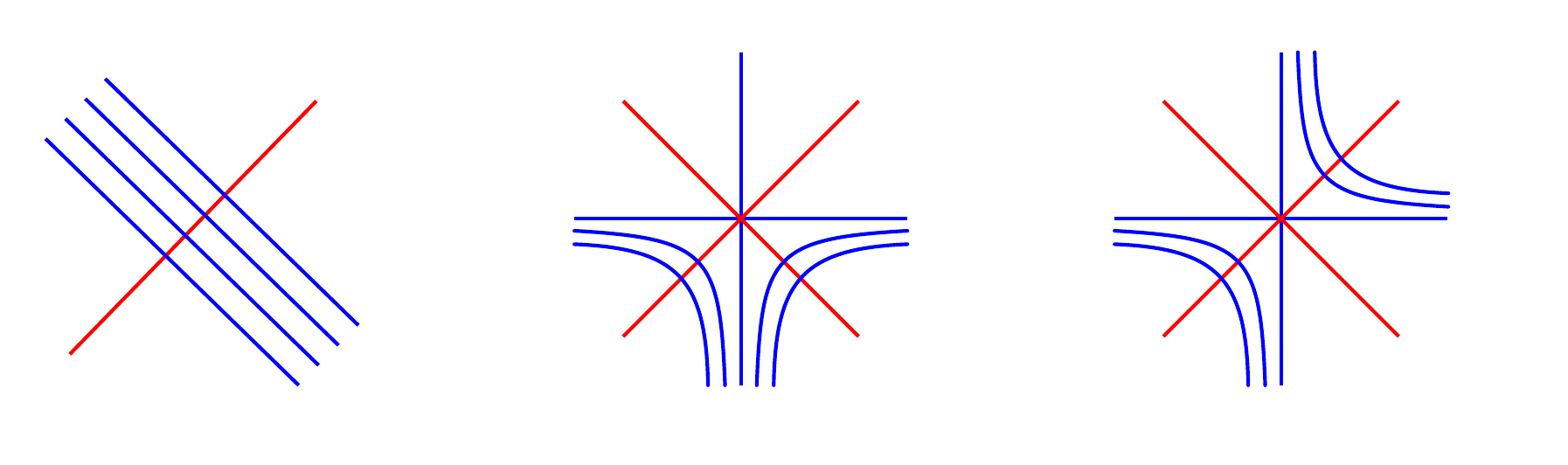}
\caption{The possible cases of intersection of $l$ with $\cup A^-$.}
\label{fig_nondegenerate_interval}
\end{figure}
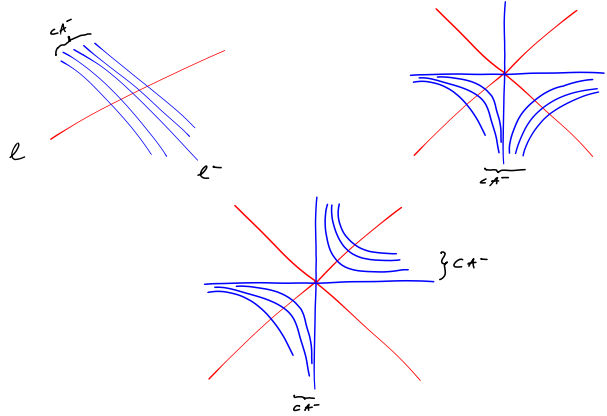
\end{proof}

To prove Theorem \ref{thm_both_axes_R}, we will show the following slightly more precise result. 
\begin{proposition} \label{prop_omega}
Suppose $g$ acts freely on $P$ with both axes $A^\pm$ of $g$ homeomorphic to $\bR$.  Let $\Omega$, $s$, and $i$ be as defined in Theorem \ref{thm_both_axes_R}.  
Exactly one of the following holds: 
\begin{enumerate} 
\item  $s$ and $i$ take only finite values and are strictly monotone, in which case $P$ is skew.
\item  $s$ and $i$ take only finite values and are weakly (not strictly) monotone, in which case $P$ is not skew and $\Omega$ is a proper subset of $P$. 
\item $s \equiv \infty$,  $i \equiv -\infty$, the plane is trivial and $\Omega = P$, or 
\item $s \equiv \infty$,  $i \equiv -\infty$, and $\Omega$ is a scalloped region.
\end{enumerate} 
\end{proposition} 

\begin{remark} 
One can adapt the proof of Proposition \ref{prop_omega} to show this statement is true not only for axes of an element acting freely, but when $A^+$, $A^-$ are {\em any} embedded copies of $\bR$ in $\Lambda(\cF^+)$ and $\Lambda(\cF^-)$ (respectively) which are invariant under some subgroup $H \subset G$ acting without global fixed points on $A^+$ and $A^-$.   (The statement for axes is simply the case $H \cong \bZ$.) 
See Exercise \ref{ex_general_omega}. 
\end{remark}

\begin{proof}[Proof of Proposition \ref{prop_omega}]
Allowing for the possibility that $s$ and $i$ take values $\pm \infty$, we will first show that they are weakly monotone. 
First, fix a leaf $a \in A^-$.  We show the restriction of $s$ to $\cF^+(a) \cap A^+$ is monotone.  
Suppose for contradiction that this is not the case. Then we can find $l_1, l_2, l_3 \in \cF^+(a) \cap A^+$ such that $l_2$ lies between $l_1$ and $l_3$, and $s(l_2) < \min \{s(l_1), s(l_3) \}$ or $s(l_2) > \max \{s(l_1), s(l_3) \}$. We treat the case $s(l_2) < \min \{s(l_1), s(l_3) \}$, the other being symmetric.  Choose $b \in A^- \cong \bR$ such that $s(l_2) < b < \min \{s(l_1), s(l_3) \}$.  By Lemma \ref{lem_nondegenerate_interval}, every leaf in the (closed) interval of leaves between $a$ and $b$ in $A^-$ meets $l_1$ and $l_3$, in particular $b$ does; which contradicts the fact that $b$ does not intersect $l_2$ and $l_2$ separates $l_1$ from $l_3$ in $P$. See Figure \ref{fig_prop_omega_monotone}.  

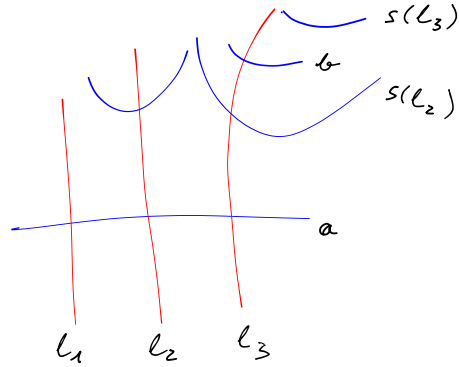
\begin{figure}[h]
\labellist
\small \hair 2pt
\pinlabel {$l_1$} at 53.768 17.252
\pinlabel {$l_2$} at 129.404 17.252
\pinlabel {$l_3$} at 206.648 17.252
\pinlabel {$a$} at 273.701 83.714
\pinlabel {$b$} at 307.495 215.567
\pinlabel {$s(l_2)$} at 314 169.472
\pinlabel {$s(l_3)$} at 330 257.910
\endlabellist
\includegraphics[width=6cm]{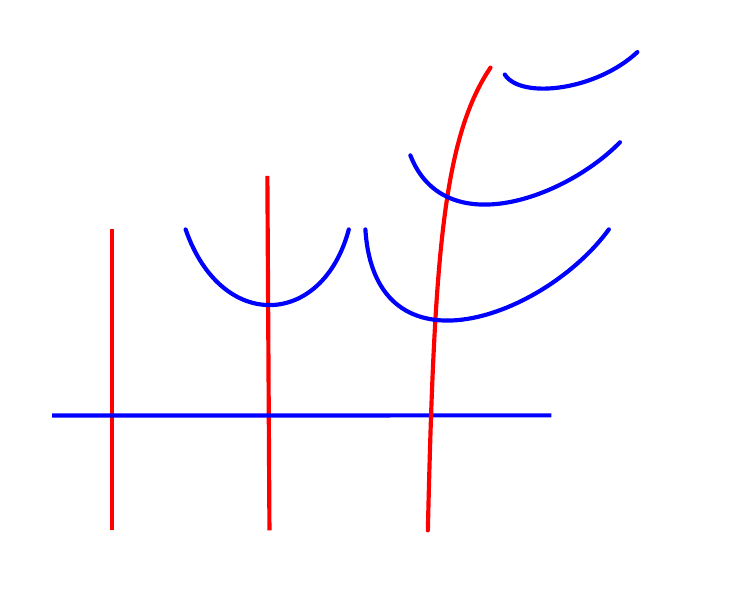}
\caption{The leaf $b$ cannot intersect $l_1$.}
\label{fig_prop_omega_monotone}
\end{figure}

This proves local monotonicity of $s$, which is sufficient since $A^+$ is connected.  The argument for monotonicity of $i$ is the same.  

Note also that the set of leaves of $A^+$ on which $s$ (respectively $i$) takes the value $+\infty$ (resp. $-\infty$) is invariant under $g$, so is either empty or unbounded in both directions.  Thus, if $s(l) =\infty$ for some $l$, then $s\equiv \infty$ by monotonicity.   If $A^+$ and $A^-$ are properly embedded and either $s \equiv \infty$ or $i \equiv - \infty$, one gets an infinite product region, so, by Proposition \ref{prop:no_product_region}, the plane is necessarily trivial (and thus both $s$ and $i$ are infinite).  

If instead $A^+$ is not properly embedded and either $s \equiv \infty$ or $i \equiv - \infty$, then by Lemma \ref{lem_commuting_axis_2} (taking $h=g$), the axis $A^+$ consists of the $\cF^+$ leaves of a scalloped region.  Thus, $g$ preserves the boundaries of this scalloped region, acting freely on each of its four ``sides" consisting of an infinite union of leaves.  It follows that $A^-$ is equal to the union of $\cF^-$ leaves intersecting the scalloped region; and thus $s$ and $i$ are both infinite. 

Finally, in the case where $s$ and $i$ are finite valued, if they are strictly monotone then by Lemma \ref{lem_s_pf}, each leaf of $A^+$ makes a perfect fit with a leaf of $A^-$ on each side, and (as in Proposition \ref{prop:both_R}) $\Omega$ is an open subset of $P$ properly homeomorphic to the skew plane, hence equal to $P$.  
\end{proof}

While we worked with an axis of a free element in this section, having this free element is not strictly necessary, as described in the following exercise:
\begin{exercise}  \label{ex_general_omega}
Suppose $A^+$, $A^-$ are embedded copies of $\bR$ in $\Lambda(\cF^+)$ and $\Lambda(\cF^-)$ invariant by a subgroup $H \subset G$ acting on $P$ without global fixed points. 

Define $\Omega = \{p\in P : p = l^+\cap l^- \text{ for some } l^+\in A^+, l^-\in A^-\}$ and show that the conclusions of Theorem \ref{thm_both_axes_R} and Proposition \ref{prop_omega} still hold.
\end{exercise} 
\begin{proof}[Hint]
To approach this, one can either redo the proof of Theorem \ref{thm_both_axes_R} by hand, not using the fact that there is a free element.  Alternatively, one can first show that 
given $f, g$ in $H$ and acting without common fixed points, the subgroup generated by 
$f$ and $g$ contains an element $h$ that acts freely.  Then apply Theorem \ref{thm_both_axes_R} to $h$ and conclude.  
\end{proof}

%% file: using_closing_lemma_vector.tex
Proposition \ref{prop_closing_translation} gave a translation of the statement of the closing lemma for pseudo-Anosov flows (Proposition \ref{prop_pAclosing_lemma}) into the language of the orbit space.  Around nonsingular points, it says that the action of $G$ on $\orb$ has the following property: 
{\em 
Each nonsingular point $x \in \orb$ has a pair of neighborhood bases $U_i$ and $V_i$ with $V_i \subset U_i$ for each $i$, such that  if $g \in G$ satisfies $g(V_i) \cap V_i \neq \emptyset$, then $g$ has a fixed point in $U_i$.}
\footnote{Proposition \ref{prop_closing_translation} also had a more technical statement for behavior around singular points, but we will not need to make use of that here.}
In this chapter, we add a weaker version of this closing property for nonsingular points to the axioms for Anosov-like actions, assuming there is a nearby fixed leaf but not necessarily a nearby fixed {\em point}.  
This weak statement ends up being sufficient to recover many essential algebraic-dynamical results, such as the fact that no $\bZ^2$ subgroup of an Anosov-like group on a nontrivial plane can act freely.  
These algebraic-dynamical results are used in an essential way in Barbot and Fenley's structure theory for essential tori in 3-manifolds with (pseudo-)Anosov flows, which we explain in Chapter \ref{chap_features_and_3manifold}.  

\begin{definition} 
An Anosov-like group action satisfies the {\em weak closing property} if the following axiom holds
\begin{enumerate}[label = (C)]
	\item\label{Axiom_closing} Each nonsingular $x \in P$ has a neighborhood basis $U_i$ 
	and smaller neighborhoods $V_i \subset U_i$  with the following property: 
	if $g \in G$ satisfies $gV_i \cap V_i \neq \emptyset$, then $g$ preserves a leaf of either $\cF^+(U_i)$ or $\cF^-(U_i)$.
\end{enumerate}
\end{definition}

\section{Cyclic point stabilizers and consequences}
We keep the standard conventions, using $(P, \cF^+, \cF^-)$ to denote a bifoliated plane with an Anosov-like action of a group $G$.  

In \cite[Theorem 1.4]{BBM24b}, it was shown that for any Anosov-like action on a nontrivial, non-skew plane, point stabilizers are virtually cyclic.  Here we give a shorter proof (applicable also to skew planes) by using weak closing. 

\begin{proposition}[Cyclic point stabilizers] \label{prop:cyclic_stabilizers}
If $G$ acts Anosov-like and satisfies Axiom \ref{Axiom_closing}, then the stabilizer of any point $x$ in $P$ is either trivial or virtually isomorphic to $\bZ$. More precisely, in the latter case, the finite index subgroup of the stabilizer fixing all rays of $x$ is isomorphic to $\bZ$. 
\end{proposition}

If one does not assume the weak closing property, then one can produce (using the characterizations given in Proposition \ref{prop:trivial_affine} and Exercise \ref{ex_charac_skew}) many examples of Anosov-like actions on skew and trivial planes for which Proposition \ref{prop:cyclic_stabilizers} fails.  For instance, many groups of affine transformations generated by hyperbolic diagonal matrices in $\SL(2,\bR)$ and translations will be Anosov-like, but if the subgroup of diagonal matrices is not virtually cyclic the point stabilizer of the origin can have arbitrary rank.  
Examples for skew planes can be constructed by taking subgroups of $\PSL(2, \bR)$ acting minimally on the circle with non-discrete point stabilizers (but no parabolic elements), and lifting them to the universal covering group $\wt{\PSL}(2, \bR) \subset \Homeo(\R)$.  Thinking of the leaf space as $\bR$ and the covering transformation of $\wt{\PSL}(2, \bR) \to \PSL(2,\bR)$ as given by the one-step-up maps, this induces an action on a skew plane, and such a subgroup without parabolics will induce an
Anosov-like action on the plane.  

\begin{proof}[Proof of Proposition \ref{prop:cyclic_stabilizers}]
Let $x \in P$ be a point with nontrivial stabilizer, and let $H$ denote the finite index subgroup of the stabilizer of $x$ fixing all rays through $x$.  We will first show that either $H$ is cyclic, or some (in fact any) point $y$ satisfying $\cF^+(y) \cap \cF^-(x) \neq \emptyset$ and 
$\cF^+(x) \cap \cF^-(y) \neq \emptyset$ has indiscrete orbit under $H$.  

Let $r^+$ and $r^-$ be rays of $\cF^+(x)$ and $\cF^-(x)$, respectively, bounding a quadrant $Q$.
Choose some nontrivial $h_0 \in H$, acting by expansion on $r^+$ and contraction on $r^-$.  
Fix an identification of each (open) half leaf $r^\pm$ with $\R$, oriented and scaled so that $h_0$ acts by $t \mapsto t+1$ on each.
With this choice, the proof of H\"older's theorem (see e.g., \cite[Section 2.2.4]{Nav11})  shows that there are canonical, well-defined, morphisms $\Phi^\pm\colon H \to \R$ given by the {\em translation number} of elements on each factor; so that $\Phi^+(h_0) = \Phi^-(h_0)=1$, and the actions of $H$ on $r^+$ and on $r^-$ are semiconjugate to the translation actions of $\Phi^+(H)$ and $\Phi^-(H)$, respectively.  

Our choice of normalization implies that an element $h\in H$ is a topological expansion on $r^+$ if and only if $\Phi^+(h) >0$, and a contraction on $r^-$ if and only if $\Phi^-(h) > 0$.  
Consider the map $\Phi^- \circ (\Phi^+)^{-1}$.  This is a homomorphism between subgroups of $\R$, and the hyperbolicity condition (Axiom \ref{Axiom_A1}) implies that $\Phi^- \circ (\Phi^+)^{-1}(t) > 0$ iff $t>0$. 
Thinking of $\R$ as a vector space over $\mathbb{Q}$, it is a pleasant exercise (stated as Exercise \ref{ex_pleasant} below) to show that this can only hold if $\Phi^- \circ (\Phi^+)^{-1}(r)$ is given by multiplication by some positive constant $c$.  Since $\Phi^+(h_0) = \Phi^-(h_0) = 1$, we also have $c = 1$.  
Thus, $\Phi^+ = \Phi^-$.  

If $H$ is not cyclic, the image of $\Phi=\Phi^+=\Phi^-$ is indiscrete, and so the orbit of any point $y$ satisfying $\cF^+(y) \cap \cF^-(x) \neq \emptyset$ and 
$\cF^+(x) \cap \cF^-(y) \neq \emptyset$ will have accumulation points.   
In particular, that means that no singular point satisfies this condition.  
Let $q$ be an accumulation point of the orbit of some $y$.  Let $U$ be a small trivially foliated neighborhood of $q$ small enough so that $\cF^\pm(x) \cap U = \emptyset$.    Let $V \subset U$ be a smaller neighborhood as in the conditions of Axiom \ref{Axiom_closing}.   %

Since $q$ is an accumulation point of an orbit, there exists some nontrivial $g \in H$ such that $gV \cap V \neq \emptyset$.  By the weak closing property, $g$ has a fixed leaf intersecting $U$, which by choice of $U$ meets either $\cF^+(x)$ or $\cF^-(x)$ at a point other than $x$.  This contradicts Axiom \ref{Axiom_A1} and concludes the proof.  
\end{proof} 

\begin{exercise} \label{ex_pleasant}
Suppose $\rho$ is a homomorphism between two subgroups of $\bR$ (as an additive group) which satisfies the property that $\rho(t)>0$ iff $t>0$.  Show that $\rho$ is given by multiplication by a positive constant.  
\end{exercise}

While the reader may have been following this text with examples of countable groups in mind, so far we have made no assumption that groups acting Anosov-like were countable.  However, this follows (in the non $\mathbb{R}$-covered case) from the property of having cyclic point stabilizers:
\begin{proposition}[See \cite{Cam25}, Prop. 2.3.9]
Let $G$ be a group acting Anosov-like on a non $\mathbb{R}$-covered bifoliated plane $P$ with virtually cyclic point stabilizers.  Then $G$ is countable. 
\end{proposition} 
\begin{proof} 
Since $P$ is not $\mathbb{R}$-covered, it has a leaf $l$ which is either singular or branching.  The set of all singular and branching leaves is countable, so $G \cdot l$ is countable. Since the stabilizer of $l$ is virtually cyclic, hence countable, $G$ is countable. 
\end{proof} 
Note that, as remarked above, by \cite[Theorem 1.4]{BBM24b} all Anosov-like actions on non $\mathbb{R}$-covered bifoliated planes have virtually cyclic point stabilizers, so this hypothesis is in fact redundant.

For Anosov-like actions satisfying the weak closing property (or more generally, with discrete point stabilizers), we can completely characterize the stabilizers of scalloped regions.  These correspond precisely to $\bZ^2$-subgroups generated by elements that act with fixed points.  

\begin{proposition} \label{prop_stabilizer_scalloped_withC}
Suppose $G$ acts Anosov-like on a bifoliated plane and satisfies Axiom \ref{Axiom_closing} or has discrete point stabilizers.   Then every subgroup isomorphic to $\bZ^2$ generated by elements with fixed points stabilizes a (unique) scalloped region.  Conversely, the stabilizer of any scalloped region is virtually a $\bZ^2$ subgroup generated by elements with fixed points. 
\end{proposition} 

This is an immediate consequence of the following more general statement, applicable to actions that do not necessarily have discrete point stabilizers. 

\begin{proposition} \label{prop_stabilizer_scalloped} 
Let $G$ act Anosov-like on a bifoliated plane.  
If $g_1, g_2$ have disjoint, nonempty sets of fixed points in $P$ and generate a subgroup isomorphic to $\bZ^2$, then they stabilize a (necessarily unique) scalloped region.  
Conversely, the stabilizer of any scalloped region is virtually isomorphic to a subgroup of $\bZ^2$, generated by elements with nonempty, disjoint sets of fixed points. 
\end{proposition}

The ``conversely..." statement was already proved in Lemma \ref{lem_action_of_stabilizer_scalloped}. What we show here is the first assertion.  
We begin with an easy lemma.

\begin{lemma}\label{lem_scalloped_off_of_themselves}
Let $U$ and $U'$ be two distinct scalloped regions in a bifoliated plane. Then there exists an element $g$ of the stabilizer of $U$ such that $gU'\cap U' =\emptyset$
\end{lemma}

\begin{proof}
If $U\cap U'=\emptyset$, then there exists a unique leaf $l$ in the boundary of $U$ that separates $U'$ from $U$. By Lemma \ref{lem_action_of_stabilizer_scalloped}, there exists an element $g$ of the stabilizer of $U$ that moves $l$ off itself, thus $gU'\cap U' =\emptyset$.

Similarly, if $U\cap U'\neq \emptyset$ (and since $U\neq U'$), there exist \emph{exactly} two leaves $l,l'$ in the boundary of $U$ that intersect $U'$ (see Figure \ref{fig_intersection_scalloped}). Again, an element of the stabilizer of $U$ taking $l$ off itself will take $U'$ off itself as well. \qedhere 
\begin{figure}[h]
     \centerline{ \mbox{
\includegraphics[width=8cm]{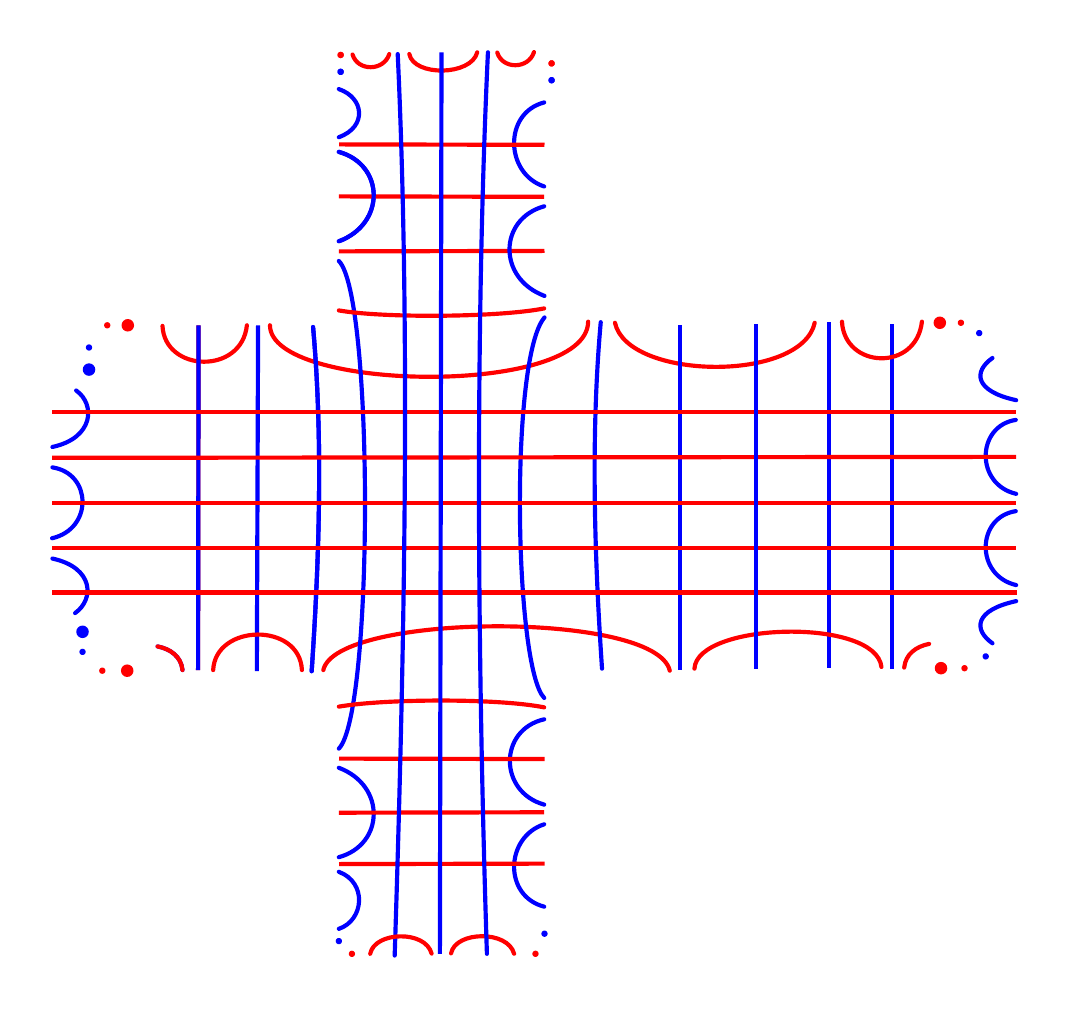}}}
\caption{The intersection of two scalloped regions.}
\label{fig_intersection_scalloped}
\end{figure}
\end{proof}

\begin{proof}[Proof of Proposition \ref{prop_stabilizer_scalloped}]
Suppose $g_1$ and $g_2$ have disjoint, nonempty sets of fixed points and generate a subgroup of $G$ isomorphic to $\bZ^2$. 
Let $\Fix(g_i)$ denote the set of points fixed by $g_i$.  
We assume first that $g_i$ each fix all rays through their fixed points, and will remove this assumption at the end. 
By Theorem \ref{thm:distinct_fix_is_chain}, $\Fix(g_i)$ consists of the corners of a maximal chain $\cC_i$ of lozenges.  Because $g_1$ and $g_2$ commute, $\Fix(g_i)$ is invariant by $g_j$.  
 
Assume first that $\cC_1$ and $\cC_2$ have nonempty intersection, so there exist lozenges $L_i$ of $\cC_i$ with $\overline{L_1} \cap \overline{L_2} \neq \emptyset$. (Recall a chain is, by definition, a union of \emph{closed} lozenges.)  If $L_1$ and $L_2$ meet along a side or corner, then $\cC_1 = \cC_2$ by maximality, and the $g_i$  would have a common fixed point, contradicting our assumption.  Thus, $L_1 \cap L_2 \neq \emptyset$.  By Lemma \ref{lem_markovian_or_corner}, up to relabeling, we have that $L_1$ is contained in the $\cF^+$ saturation of $L_2$ (and $L_2$ in the $\cF^-$ saturation of $L_1$).  
The set $\{ g_1^n(L_2) : n \in \bZ \}$ is an infinite collection of lozenges in $\cC_2$ meeting a common leaf of $\cF^+$.  Thus, it is an infinite line, so contained in a scalloped region $U$ contained in $\cC_2$.  Similarly, the lozenges $\{g_2^n(L_1) : n \in \bZ \}$ lie in a line, realizing this scalloped region as a line of lozenges in a different way.  Thus $U$ is invariant by both $g_1$ and $g_2$.
By Lemma \ref{lem_scalloped_off_of_themselves}, such a scalloped region is necessarily unique.

To conclude the proof, we need to show that the case where $\cC_1 \cap \cC_2 = \emptyset$ leads to a contradiction. This is done in several steps. 

Supposing that $\cC_1 \cap \cC_2 = \emptyset$, consider first the case where their interiors are separated by some leaf $l^+$ of $\cF^+$; meaning that the interiors lie in distinct connected components of $P\smallsetminus l^+$. Then there exists a \emph{unique} $\cF^+$-leaf, call it $l_0^+$, in the closure of the $\cF^+$-saturation of $\cC_1$ that separates the interior of $\cC_1$ and $\cC_2$. By invariance of each $\cC_i$ by both $g_1$ and $g_2$, we deduce that $l_0^+$ is stabilized by both $g_1$ and $g_2$. By Lemma \ref{lem:common_fixed_point}, this gives a common fixed point for the $g_i$, contradicting our assumption.  By the same argument, we deduce that no leaf of $\cF^-$ can separate the interiors of $\cC_1$ and $\cC_2$. 

As a second step, we argue that if $x$ is any corner in $\cC_1$, then either $\cC_2$ intersects a lozenge with corner $x$, or $\cC_2$ does not intersect the leaves $\cF^\pm(x)$. 
For contradiction, suppose that, say, $\cC_2\cap \cF^+(x)\neq \emptyset$ and that $\cC_2$ does not intersect any of the lozenges with corner $x$.  By $g_1$-invariance of $\cC_2$, either $x$ is contained in $\cC_2$, and we have a contradiction by the maximality of the chains, or $x$ is in the closure of $\cC_2$. In that latter case, since $\cC_2$ does not intersect any lozenges with corner $x$, we must have that $\cF^-(x)$ separates $\cC_2$ from (the interior of) $\cC_1$, which we ruled out above.
Thus we deduce that, if  $\cC_1 \cap \cC_2 = \emptyset$, then, for any corner $x$ of $\cC_1$, the chain $\cC_2$ lies in one quadrant of $x$.

Let $x_0$ be any corner of $\cC_1$, and call $Q_0$ the quadrant containing $\cC_2$. If $Q_0$ does not contain a lozenge of $\cC_1$, then $\cF^\pm(x_0)$ separates $\cC_1$ from $\cC_2$, which contradicts the above. So there exists a corner $x_1$ of $\cC_1$ in $Q_0$. By induction, we find an infinite sequence $x_n$ of corners of $\cC_1$ such that $\cF^+(x_n)$ separates $x_{n-1}$ from $\cC_2$. Hence $\cF^+(x_n)$ must accumulate, as $n\to \infty$, to a leaf or union of leaves, one of which separates the interior of $\cC_1$ from the interior of $\cC_2$. This final contradiction proves that we must have $\cC_1 \cap \cC_2 \neq \emptyset$. 

Finally, we remove the assumption that the $g_i$ fix all rays through their fixed points.  If not, we can pass to powers $g_1^k$ and $g_2^l$ which do; and apply the proof above to show that these stabilize a unique scalloped region.  Since $g_1$ and $g_2$ commute with their powers, they also stabilize this scalloped region.  In fact, more is true -- using the trivially foliated structure of the scalloped region, one can easily show that the property that $g_1$ and $g_2$ generate a $\bZ^2$ subgroup implies that $g_1$ and $g_2$ must also preserve local orientations, and so they themselves fix all rays through 
their fixed points.
\end{proof}

\section{No free $\bZ^2$} 
We now prove an important result linking the dynamics to the algebraic structure of an Anosov-like group. 

\begin{proposition} \label{prop:no_free_Z2}
Suppose $G$ acts Anosov like on a nontrivial bifoliated plane $P$, and satisfies Axiom \ref{Axiom_closing}.  If $G$ contains a subgroup isomorphic to $\bZ^2$, then some nontrivial element of this subgroup acts with a fixed point. 
\end{proposition} 

This has significant consequences for the theory of pseudo-Anosov flows on 3-manifolds.  If $M$ is either Seifert fibered, or has nontrivial JSJ decomposition, then its fundamental group contains $\bZ^2$ subgroups.  Proposition \ref{prop:no_free_Z2} has consequences in both cases.  For instance, it directly implies that some element of the fundamental group of any JSJ torus is represented by a periodic orbit.  We will discuss this, along with further structural consequences, in Chapter \ref{chap_features_and_3manifold}.

\begin{proof}[Proof of Proposition \ref{prop:no_free_Z2}]
Suppose $g$ and $h$ generate a subgroup $H \subset G$ isomorphic to $\bZ^2$, acting freely on $P$. 
By Lemma \ref{lem_commute_same_axis}, $g$ and $h$ have the same axis in each leaf space.  Let $A^+$ and $A^-$ denote their axes in $\Lambda(\cF^+)$ and $\Lambda(\cF^-)$, respectively.  

If $A^+$ is not an embedded copy of $\bR$, then by Proposition \ref{prop:axis_structure} it contains a countable, discrete set of points which come in non-separated pairs in $A^+$.  This set is permuted by elements of $H$, preserving an order indexed by $\bZ$.  This gives a morphism from $H$ to $\bZ$, which must have nontrivial kernel.  Thus, some nontrivial element of $H$ fixes each of these leaves, contradicting the assumption that $H$ acts freely.  Similarly, $A^-$ must be homeomorphic to $\bR$.  

We now are in the setting of Proposition \ref{prop_omega}: either $\Omega=\{p\in P \mid \cF^+(p)\in A^+ \text{ and } \cF^-(p)\in A^-\}$ is a scalloped region preserved by $H$ (contradicting that $H$ acts freely, by Lemma \ref{lem_action_of_stabilizer_scalloped}), or the plane is trivial (which we assumed not to be the case), or we have two monotone functions $s, i\colon A^+ \to A^- \cong \bR$.   Without loss of generality, up to changing the orientation on $\bR$, we assume $i$ is increasing.  

Since $H$ acts freely on $P$, it acts freely on each axis.  By H\"older's theorem (Theorem \ref{thm:holder}), the action is semi-conjugate to an action by an {\em indiscrete} group of translations.   In particular, it cannot have a discrete orbit so is either minimal, or has an invariant Cantor set on which the action is minimal.  Let $X$ denote a minimal set for the action, that is, $X$ is either $\bR$ in the first case or the invariant Cantor set in the second.  

Let $S \subset A^+$ denote the (possibly empty) set of leaves on which $i$ is locally constant.  This is an open, invariant set for the action of $H$, so is disjoint from $X$.   
Let $l_0$ be a leaf of $A^+$ that is a two-sided accumulation point in $X$, and let $l_1$ be another leaf that is a two-sided accumulation point, taken sufficiently close to $l_0$ so that some leaf $l^-$ of $A^-$ intersects both.   If $l_1$ is chosen on the appropriate side of $l_0$, then by monotonicity we have $i(l_1) > i(l_0)$ so $i(l_1) \cap l_0 \neq \emptyset$.  Let $x = i(l_1) \cap l_0$.  
Since $l_0$ and $l_1$ are both two-sided accumulation points, there exist elements $h \in H$ moving $l_0$ and $l_1$ (simultaneously) an arbitrarily small amount in $A^+$, to either side.  Choose such a sequence of elements $h_k$ so that $h_k(l_1) \to l_1$, on the right; i.e. so that  $i(h_k(l_1))> i(l_1)$.  Note that in fact $i(h_k(l_1))$ approaches $i(l_1)$, since intersecting $l_1$ is an open condition on $A^-$.  %
Thus, we have $i(h_k(l_1)) \to i(l_1)$, and $h_k(l_0) \to l_0$, hence $h_k(x) = h_k(l_0) \cap  i(h_k(l_1))$
 converges to $x$ in $P$. In particular, $x$ is not a singular point (since these do not accumulate), so, using Axiom \ref{Axiom_closing}, we find some nontrivial $h_k \in H$ that has a fixed point on $P$, contradicting the assumption that $H$ acts freely.  
\end{proof} 

As a consequence of Proposition \ref{prop:no_free_Z2}, we have: 

\begin{corollary}
If $G$ acts Anosov-like on a bifoliated plane and satisfies Axiom \ref{Axiom_closing}, then $G$ contains no subgroup isomorphic to $\bZ^3$.  
\end{corollary} 

\begin{proof} 
Suppose for contradiction that $G$ contains a subgroup isomorphic to $\bZ^3$, generated by $g_1, g_2, g_3$.  
By Proposition \ref{prop:no_free_Z2}, at least two of these generators act with fixed points; up to relabeling we assume these are $g_1$ and $g_2$.  By Proposition \ref{prop_stabilizer_scalloped},
$g_1$ and $g_2$ must stabilize a scalloped region $U$.  Furthermore, each of $g_1$ and $g_2$ acts by translation along one of the two lines of lozenges making up $U$.  
Since $g_3$ commutes with $g_1$ and $g_2$, either $g_3$ stabilizes $U$, or $g_3(U)$ is another scalloped region stabilized by $g_1$ and $g_2$. The first possibility contradicts Proposition \ref{prop_stabilizer_scalloped}, which says that stabilizers of scalloped regions are virtually $\bZ^2$.  The second contradicts the uniqueness assertion in Proposition \ref{prop_stabilizer_scalloped}.   Thus, we conclude that $G$ contains no $\bZ^3$ subgroup.  
\end{proof}

Another consequence of Proposition \ref{prop:no_free_Z2} is that $\bZ^2$ subgroups must preserve a \emph{minimal} (bi-infinite) chain of lozenges, where minimality is considered in the following sense:
\begin{lemma}\label{lem_charac_minimal_chain}
Let $\cC=\{L_i\}_{i\in \bZ}$ be a bi-infinite chain of lozenges. The following are equivalent
\begin{enumerate}[label=(\roman*)]
\item\label{item_charac_minimal_no_3_shared} For all $i$, $L_{i-1}$, $L_i$ and $L_{i+1}$ never share a corner;
\item $\cC$ does not contain any proper bi-infinite chain of lozenges.
\end{enumerate}
If either of these conditions is satisfied, then we say the chain $\cC$ is \emph{minimal}.
\end{lemma}
The proof of the above lemma is a simple exercise, left to the reader. See also Lemma 4.21 of \cite{BBM24b}.

\begin{corollary}\label{prop_Z2_stabilize_minimal_chain}
Suppose $G$ acts Anosov like on a bifoliated plane and satisfies Axiom \ref{Axiom_closing}.
Let $H$ be a subgroup of $G$ isomorphic to $\bZ^2$. Then $H$ preserves a minimal bi-infinite chain $\cC'$ of lozenges, and there exists a choice of generators $g, h$ with $g$ fixing each corner of $\cC'$ and $h$ acting freely. 

Moreover, $\cC'$ is unique except if $H$ preserves a scalloped region, in which case $H$ preserves exactly two minimal chains of lozenges.
\end{corollary}

\begin{proof}
Since $H$ does not act freely (by Proposition \ref{prop:no_free_Z2}), there exists an element $g\in H$ that fixes a point $x\in P$.  Choose such a $g$ that is not a power of another element (recall that $g$ fixes a point iff $g^k$ does), and let $h \in H$ be such that $g$ and $h$ generate $H$. %
Since $g$ and $h$ commute, the points $\{ h^nx : n \in \bZ\}$ are all fixed by $g$, and are distinct for different values of $n$,  because point stabilizers are cyclic. Thus, by Theorem \ref{thm:distinct_fix_is_chain}, $H$ preserves a maximal chain of lozenges $\cC$, whose corners are all fixed by $g$; moreover $h$ acts freely on $\cC$. 

Consider the tree $\cT(\cC)$ associated with $\cC$, as in Definition \ref{def:chain_tree}.  Since $h$ acts freely on $\cC$, it 
acts freely on $\cT(\cC)$.  Thus, $h$ has an invariant axis $\mathcal{A}$ of $\cT(\cC)$ which is an embedded line with every vertex having valence two.  This corresponds to a subchain $\cC'$ of $\cC$.  Since vertices have valence two in $\mathcal{A}$, Condition \ref{item_charac_minimal_no_3_shared} of Lemma \ref{lem_charac_minimal_chain} is satisfied by $\cC'$, and thus it is a minimal bi-infinite chain of lozenges that is left invariant by $H$, and we have found generators $g,h$ of $H$ such that $g$ fixes all corners of $\cC'$ and $h$ acts freely on it.

Finally, if $H$ has more than one invariant, minimal chain of lozenges, the above argument shows that they cannot be part of the same maximal chain of lozenges, as otherwise, they would both correspond to the same invariant axis of $h$. 
Thus, there must be two independent elements in $H$ fixing points in $P$. This characterizes stabilizers of scalloped regions by Proposition \ref{prop_stabilizer_scalloped}.
\end{proof}

\section{For flows, perfect fits give lozenges}
In Lemma \ref{lem:fixed_is_corner}, we saw that a leaf fixed by a nontrivial $g\in G$ and making a perfect fit with another leaf necessarily contains the side of a lozenge.  This in fact was a consequence only of Axiom \ref{Axiom_A1} for Anosov-like actions.  
For pseudo-Anosov flows, a much stronger result is true: {\em any} leaf making a perfect fit with another is necessarily the corner of a lozenge.  This is a result of Fenley \cite{Fen16} and uses the (standard, not weak) closing lemma in an essential way.  

Since Fenley's argument can be done largely in the orbit space, we include it here as a transition from general Anosov-like actions to the study of pseudo-Anosov flows on 3-manifolds which will be the focus of the last part of this work.

\begin{theorem}[Fenley, \cite{Fen16}, Proposition 5.5]  \label{thm_perfect_fit_implies_lozenge}
Let $\flow$ be a pseudo-Anosov flow on a compact 3-manifold and 
suppose that two leaves $l^s$ and $l^u$ of $\hfs, \hfu$ respectively make a perfect fit.  Then there is a lozenge $L$ in $\orb$ with this perfect fit as an ideal corner.

Moreover, either $L$ is fixed by some nontrivial element $g\in \pi_1(M)$ or there is another lozenge $L'$ fixed by $g$ and sharing one ideal corner with $L$.
\end{theorem}

\begin{remark}[Pseudo-Anosov flows without perfect fits]
The class of pseudo-Anosov flows without perfect fits became an important object recently due to their relationship with structures called \emph{veering triangulations} on hyperbolic 3-manifolds.
Theorem \ref{thm_perfect_fit_implies_lozenge} implies that a pseudo-Anosov flow has perfect fits if and only if lozenges appear in the orbit space, which (as we will see in Section \ref{sec_free_homotopy}) occurs if and only if there exists a pair of periodic orbits which are freely homotopic.  Thus, the pseudo-Anosov flows without perfect fits are exactly those where each orbit is unique in its free homotopy class.  See Section \ref{sec_free_homotopy} for further discussion. 
\end{remark}

To prove Theorem \ref{thm_perfect_fit_implies_lozenge}, we first prove an abstract lemma about Anosov-like actions, which shows the existence of a lozenge given certain dynamical behavior.   To prove the theorem, we will  use the closing lemma for flows (Proposition \ref{prop_pAclosing_lemma}) to find elements satisfying these dynamical conditions. 

\begin{lemma}  \label{lem_consequence_closing}
Assume $G$ acts Anosov-like on a bifoliated plane $P$.  Suppose leaves
$l^+$ and $l^-$ in $\cF^+, \cF^-$ respectively make a perfect fit, and some $g \in G$ has the following properties: 
\begin{itemize}
\item There exists $x$, fixed by $g$, with $\cF^-(x) \cap l^+ \neq \emptyset$,
\item $g$ contracts $\cF^+(x)$, and
\item $g(l^+) \cap l^- \neq \emptyset$.
\end{itemize}
Then $l^+, l^-$ are sides of a lozenge whose opposite ideal corner is fixed by $g$.  
\end{lemma} 

\begin{proof} 
Let $g, l^+$, and $l^-$ be as in the hypotheses of the lemma.  
Since $g(l^+) \cap l^- \neq \emptyset$, and $g$ expands the leaf $\cF^-(x)$ which intersects $l^+$, it follows that $x$ must be in the connected component of $P \setminus l^+$ which does not contain $l^-$.  

Consider the sequence of leaves $g^{-n}(l^-)$, as $n \to +\infty$.  We claim each such leaf intersects $g(l^+)$.  This is true for $n=0$ by hypothesis.  
Now observe that $g(l^+), l^-, l^+$ and $\cF^-(x)$ bound a trivially foliated region, so each leaf meeting $l^+$ on the side of $\cF^-(x)$ containing $l^-$ also meets $g(l^+)$.  Applying $g^{-k}$ to this picture, we see that $g^{-k}(l^-)$ intersects $g^{-k+1}(l^+)$ and thus inductively also $g(l^+)$.  

Since $P$ contains leaves that make a perfect fit, it is a nontrivial plane, and so there are no infinite product regions by Proposition \ref{prop:no_product_region}.  Thus, it follows that $g^{-n}(l^-)$ cannot escape compact sets, so limits to a leaf or $g$-invariant nontrivial union of leaves.  One of these leaves (call it $f^-$) must make a perfect fit with either $\cF^+(x)$, or a leaf nonseparated with it.  
In the latter case, Proposition \ref{prop:4weak4strong} implies that $\cF^+(x)$ also makes a perfect fit with some leaf.  Thus, in either case Lemma \ref{lem:fixed_is_corner} says that $x$ is the corner of a lozenge invariant by $g$ (as usual, up to replacing $g$ by a power).   Proposition \ref{prop:4weak4strong} also implies that this lozenge is contained in a $g$-invariant chain, and one lozenge in this chain must have one of its sides on $f^-$. Call that lozenge $L$, and 
let $c$ be the corner of $L$ on $f^-$.  

Then $\cF^+(c) \cap g^{-n}(l^-) \neq \emptyset$, for all $n$ sufficiently large.  The opposite corner of $L$ is either $x$ or a point $y$ such that $\cF^-(y)$ lies between $\cF^-(x)$ and $l^-$. 
See Figure \ref{fig_obtain_lozenge} for an illustration.  
 In either case, the $\cF^-$ leaf of this corner intersects $g^{-n}(l^+)$.  This gives a lozenge $L'$ with sides on $g^{-n}(l^-), g^{-n}(l^+)$ and two of the leaves forming sides of $L$.  Applying $g^n$ to $L'$, we see that $l^-$ and $l^+$ are two sides of a lozenge as well. 
 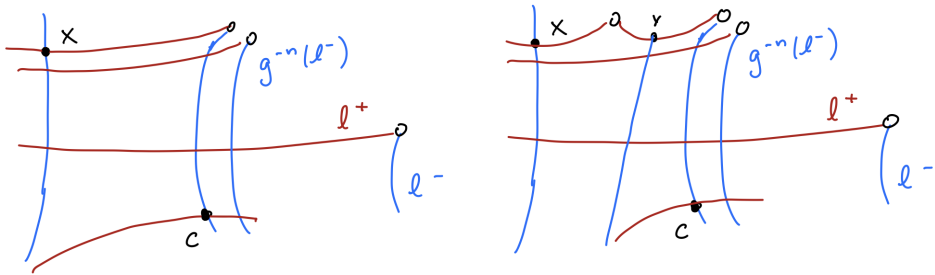
\begin{figure}[h]
\labellist
\small \hair 2pt
\pinlabel {$x$} at 69.882 221.279
\pinlabel {$c$} at 183.732 54.480
\pinlabel {$l^+$} at 342.862 143.698
\pinlabel {$l^-$} at 404.962 89.391
\pinlabel {$g^{-n}(l^-)$} at 270.412 179.902
\pinlabel {$x$} at 533.042 219.986
\pinlabel {$y$} at 628.779 221.279
\pinlabel {$c$} at 644.304 48
\pinlabel {$l^+$} at 791.791 134.647
\pinlabel {$l^-$} at 862.947 80.340
\pinlabel {$g^{-n}(l^-)$} at 729.691 177.316
\endlabellist
\includegraphics[width = 13cm]{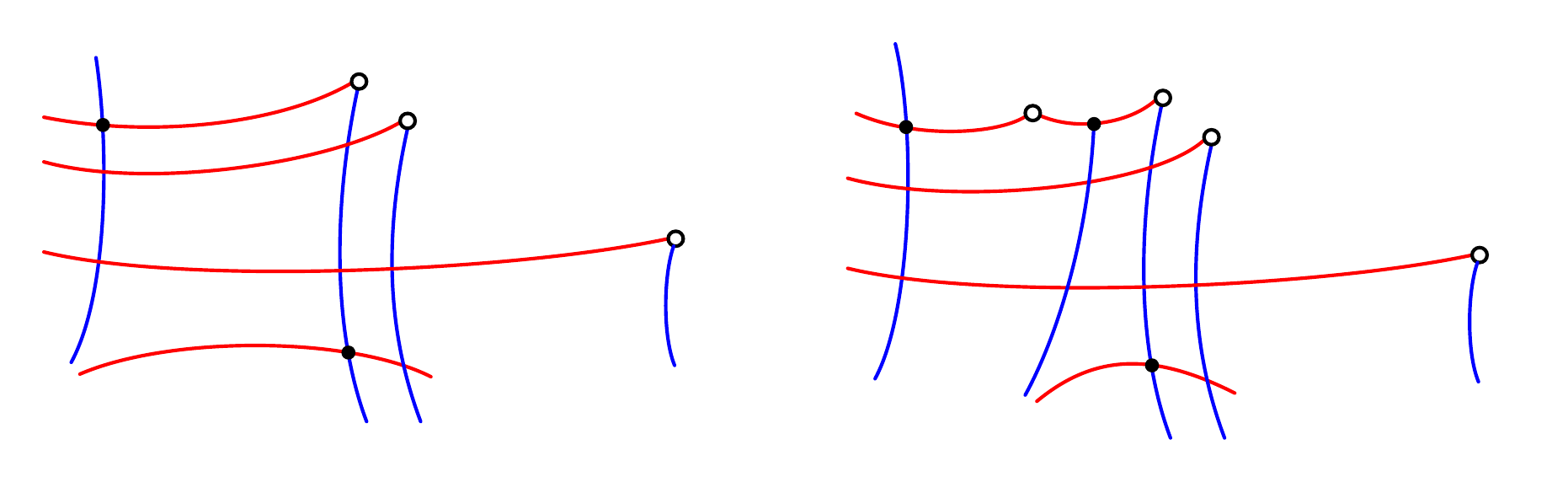}
\caption{Two possible configurations in the proof of Lemma \ref{lem_consequence_closing}}
\label{fig_obtain_lozenge} 
\end{figure}
\end{proof} 

We now use Lemma \ref{lem_consequence_closing} to prove that perfect fits give lozenges in pseudo-Anosov flows.  

\begin{proof}[Proof of Theorem \ref{thm_perfect_fit_implies_lozenge}]
Suppose that $l^s$, $l^u$ in $\orb$ make a perfect fit.  Since leaves of $\bfs,\bfu$ in $\orb$ are in bijective correspondence with those 
of $\hfs, \hfu$ in $\wt M$, we use $l^s$ and $l^u$ to denote the corresponding leaves in $\wt M$ as well.  
  If $l^u$ projects to a cylindrical leaf in $M$, then it is fixed by some nontrivial element of $\pi_1(M)$. Thus, by Lemma \ref{lem:fixed_is_corner}, $l^u$ contains the side of a lozenge, whose other side is on $l^s$, proving the statement.  

So we assume $l^u$ does not project to a cylindrical leaf.  
We need to find $g \in \pi_1(M)$ satisfying the conditions of Lemma \ref{lem_consequence_closing}, with $l^s, l^u$ playing the roles of $l^+, l^-$, respectively. 
Let $p\in M$ be a point on the projection of $l^u$ to $M$.  By compactness of $M$, $p$ has a non-empty $\alpha$-limit set; this (by definition) consists of all the points $x$ such that there exists a sequence $t_n\to -\infty$ such that $\flow^{t_n}(p) \to x$, and is a closed set.   
If the $\alpha$-limit set is a single periodic orbit, then $l^u$ is periodic, which we assumed not to be the case.  Thus, in particular, we can find some $p' \in \alpha(p)$ on a nonsingular orbit.  

Fix $\epsilon>0$ smaller than the injectivity radius of $M$, fix a compact neighborhood $N$ of $p'$ that does not contain any singular orbits, and let $\delta>0$ be the uniform constant  given by the pseudo-Anosov closing lemma (Proposition \ref{prop_pAclosing_lemma}) on $N$, with this choice of $\epsilon$.

We claim that $\alpha(p)$ intersects $N$ along infinitely many local unstable leaves.  To see this, first if $p'$ is in an orbit which itself intersects a local transversal infinitely many times, we are already done.  Else, we can replace $p'$ with a periodic orbit in $\alpha(p)$, so we need now only treat the case where $p'$ is periodic.  
Observe first that if the only intersections of the backward orbit of $p$ with $N$ (or any smaller neighborhood of $p'$) lie on the local unstable leaf of $p'$, then $p' \in l^u$ is a  periodic orbit, which we assumed not to exist.  Thus, we can assume without loss of generality that $\flow^t(p)$, for $t \in  (-\infty, 0]$ intersects $N$ along infinitely many distinct local unstable leaves, and shrink $N$ if needed so that $p'$ intersects it only a single time.  Considering iterates of the first return map to $N$ (in the past) following $p'$, which contracts the local unstable and expands the local stable, one obtains additional accumulation points of the backwards orbit of $p$ along the local unstable leaf of $p'$.  This gives additional $\alpha$-limit points, as desired. 

Thus, we can pick $t, s<0$ such that $q=\flow^s(p)$ and $\flow^{t}(q)$ are both in the $\delta/2$-neighborhood of $p'$ and on distinct local unstable leaves, and we can further choose $s$ and $t$ such that the local unstable leaf of $\flow^t(q)$ is on our preferred side of the local unstable leaf of $q$: That is, if we take a lift $\wt q$ of $q$ to $\wt M$ on $l^u$ (so that the local unstable leaf of $\wt q$ is $l^u$), then the corresponding lift $\wt l$ of the local unstable leaf of $\flow^t(q)$ is in the connected component of $\wt M \setminus l^u$ containing $l^s$. In particular, if $\delta$ is sufficiently small, then $\wt l$ will intersect $l^s$, by definition of perfect fit.  An illustration of this set-up, and the next step, is given in Figure \ref{fig_using_closing}.  

Applying the closing lemma, we obtain a closed orbit $\gamma$ in a $\epsilon$-flow-box neighborhood $\wt B$ of the segment from $q$ to $\flow^t(q)$, so that the deck transformation $g$ corresponding to $\gamma$ gives $g(\wt \flow^t(\wt q)) \in \wt l$.  
\begin{figure}[h]
\labellist
\small \hair 2pt
\pinlabel {$q$} at 125.466 178.552
\pinlabel {$\varphi^t(q)$} at 195 259.938
\pinlabel {$g(\varphi^t(q))$} at 100 150
\pinlabel {$\widetilde B$} at 162 150
\pinlabel {$l^u$} at 238.357 234.702
\pinlabel {$l^s$} at 325.390 107.259
\pinlabel {$\widetilde l$} at 256.646 175.397
\endlabellist
\includegraphics[width = 9cm]{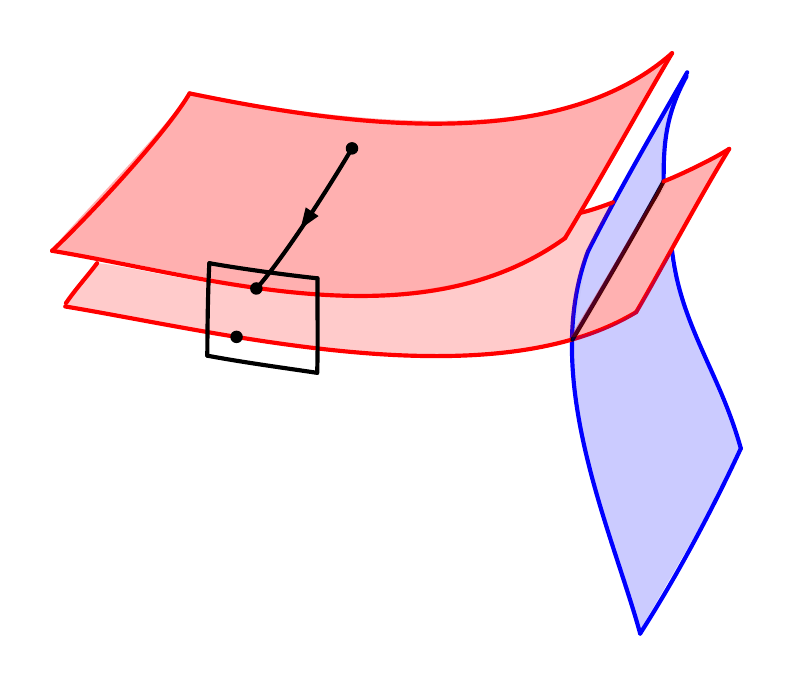}
\caption{Configuration of $q$ and $g(\flow^t(q))$ and their leaves in $\wt M$.}
\label{fig_using_closing} 
\end{figure}

We examine the action of $g$ on the orbit space and show it satisfies the conditions of Lemma \ref{lem_consequence_closing}.  
The first condition is due to $\gamma$ being in a flow-box neighborhood of $\wt q$, so the fixed point for $g$ will lie in this neighborhood.  Secondly, since $\gamma$ is oriented in the direction of the flow, $g$ {\em contracts} the unstable leaf of $\gamma$ (recall Proposition \ref{prop_hyperbolic}). 
Finally, by construction we have that $g(l^s) \cap l^u \neq \emptyset$. 
Thus, the conclusion of Lemma \ref{lem_consequence_closing} holds and we obtain the desired lozenge (whose opposite ideal corner is fixed by $g$).  
\end{proof} 

We now discuss an important consequence for the structure of the bifoliations in orbit spaces.  
We call a \emph{triple ideal quadrilateral} an open, trivially foliated region bounded by leaves $l_1^\pm$ and $l_2^\pm$ such that both pairs $l_i^+$ and $l_i^-$ make a perfect fit, and $l_1^+, l_2^-$ also make a perfect fit, but $l_1^-$ and $l_2^+$ intersect. 
Similarly a {\em totally ideal} quadrilateral is an open trivially foliated region as above, but where $l_1^-$ and $l_2^+$ make a perfect fit as well\footnote{In \cite{Fen16}, such regions were called $(3,1)$- and $(4,0)$-quadrilaterals respectively; we use the terminology as in \cite{BFM22} to avoid confusion between the asymmetrical labels $(1,3)$ and $(3,1)$.}. See Figure \ref{fig_13_04_quad}. 

\begin{figure}[h]
\includegraphics[width=5cm]{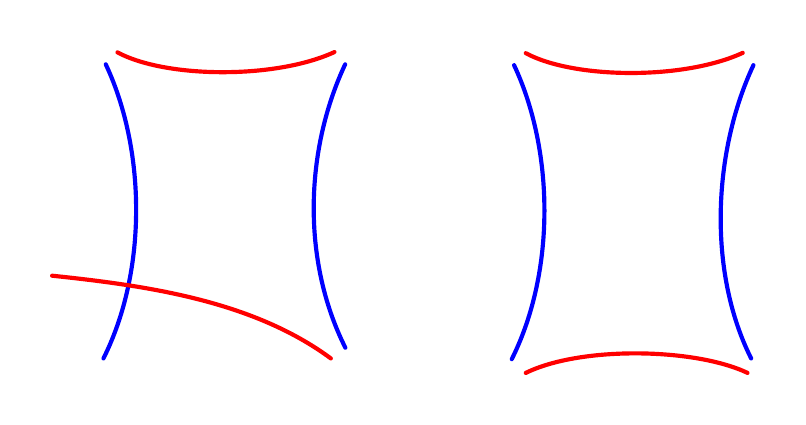}
\caption{Triple ideal and totally ideal quadrilaterals}
\label{fig_13_04_quad}
\end{figure}

\begin{corollary} \label{cor_no_ideal_quad}
The orbit space of a pseudo-Anosov flow has no triple ideal or totally ideal quadrilaterals. 
\end{corollary} 

\begin{proof} 
Suppose that $l_1^\pm$, $l_2^\pm$ bound a trivially foliated region, where both pairs $l_i^+$ and $l_i^-$ make a perfect fit, and the pair $l_1^+, l_2^-$ also makes a perfect fit.  By Theorem \ref{thm_perfect_fit_implies_lozenge}, $l_1^+$ and $l_2^-$ are part of a lozenge, $L$.  Let $f^+, f^-$ be the other two sides of this lozenge.  

Then $f^+$ and $f^-$ each separate $l_1^-$ from $l_2^+$. In fact, for any leaves $l^-$ and $l^+$ intersecting $l_1^+$ or $l_2^-$ (respectively) so that the point of intersection is {\em outside} of the closure of $L$, the sides $f^+$ and $f^-$ each separate $l^-$ and $l^+$ in $P$.  See Figure \ref{fig_no_04_quad}. It follows that $l_1^-$ and $l_2^+$ can neither intersect each other, nor make a perfect fit, so do not form a triple ideal or totally ideal quadrilateral.  
\begin{figure}[h]
\labellist
\small \hair 2pt
\pinlabel {$l_1^-$} at 32.833 167.287
\pinlabel {$l_1^+$} at 142.369 230.231
\pinlabel {$f^-$} at 113.265 141.897
\pinlabel {$f^+$} at 167.240 97.995
\pinlabel {$l_2^-$} at 253.493 119.153
\pinlabel {$l_2^+$} at 175.177 50.919
\endlabellist
\includegraphics[width = 5cm]{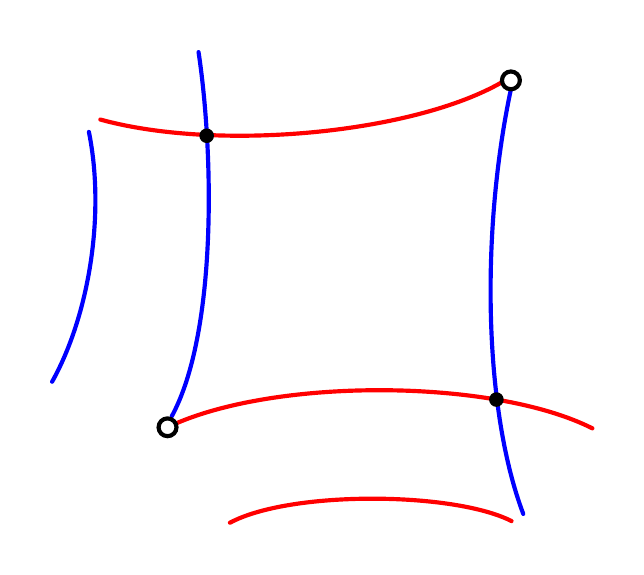}
\caption{A lozenge that prevents a triple ideal or totally ideal quadrilateral}
\label{fig_no_04_quad} 
\end{figure}
\end{proof} 

\begin{rem} \label{rem_no_ideal_quad}
The non-existence of totally ideal quadrilaterals was included as an axiom of Anosov-like actions in \cite{BFM22} and \cite{BBM24b}, but it was used only very sparsely in the development of the theory there, so we chose not to include it as an axiom here. 
\end{rem}

%% file: ExamplesOfOrbitSpaces_vector.tex
In this chapter we describe the orbit spaces of each of the (families of) flows exhibited in Section \ref{sec_flows_examples}.   %

\section{Suspension flows}  Let $f$ be a pseudo-Anosov homeomorphism of a surface $\Sigma$, and let $\flow$ denote the suspension flow on the mapping torus $M = \Sigma \times [0,1] / (x, 1) \sim (f(x), 0)$.  Then $\wt{M} \cong \wt{\Sigma} \times \R$, and the lift of $\flow$ is simply the straight line flow $\wt\flow^t(x, s) = (x, s+t)$.  The quotient map from $\wt M$ to the orbit space is projection to the first factor $\wt{\Sigma}$, and the weak stable and unstable foliations project to the (lifts of the) stable and unstable foliations for $f$ on $\wt\Sigma$.   In the case of an Anosov homeomorphism of $\bT^2$, these are simply the foliations by eigendirections (thus giving the trivial bifoliated plane), in the case of a pseudo-Anosov homeomorphism of a higher genus surface, these are the invariant singular foliations.

The fundamental group $\pi_1(M)$ has presentation $\langle \pi_1(\Sigma), s \mid s \gamma s^{-1} = f_\ast \gamma \rangle$.  Here $f_\ast$ is the map induced by $f$ by choosing a basepoint for $\pi_1(M)$ which is fixed by $f$ or an isotopic map.   
The action of $\pi_1(M)$ on $\orb$ restricts to the standard action of $\pi_1(\Sigma)$ on this subgroup, and the action of $s$ comes from lifting (rel.~basepoint) $f$ to $\wt \Sigma$.  
We note that all these orbit spaces are examples of bifoliated planes that have no perfect fits.

\section{Geodesic flow and skew planes} \label{ex_geodesic_flow_orbit_space}

Let $\Sigma$ be a surface with a hyperbolic metric, and let $\flow$ denote the geodesic flow on $M = T^1 \Sigma$.  Before lifting to the universal cover to construct the orbit space, we first consider instead the lift of $\flow$ to $T^1 \wt{\Sigma}$, the unit tangent bundle of the hyperbolic plane.  
This lift is simply the geodesic flow on $\bH^2$ and its orbits are the hyperbolic geodesics. 

An oriented geodesic in $\bH^2$ can be specified by its two endpoints on the ideal circle of the compactification of the Poincar\'e disc model of hyperbolic space.  Thus, the space of orbits of geodesic flow on $\bH^2$ is the complement of the diagonal $\Delta$ in $S^1 \times S^1$, with coordinates given by forward and backward endpoint of each geodesic.  The weak-stable (resp.~unstable) leaves of this flow are the sets of geodesics with a common forward (resp.~backwards) endpoint, thus $\cF^s$ and $\cF^u$ are the two coordinate foliations on $S^1 \times S^1 \setminus \Delta$. 

The group of deck transformations of this cover $T^1 \bH^2 \to T^1 \Sigma$ is equal to $\pi_1(\Sigma)$.  The action of $\pi_1(\Sigma)$ by isometries on $\bH^2$ induces an action on $\partial \bH^2 = S^1$ by M\"obius transformations.  Thus, the action of $\pi_1(\Sigma)$ on the space of geodesics $S^1 \times S^1 \setminus \Delta$ is simply the diagonal action of this action $g\cdot (a, b) = (ga, gb)$.  
The universal cover of $T^1 \Sigma$ is the infinite cyclic cover of $T^1 \bH^2$; thus the orbit space of $\phi$ is the infinite cyclic cover of $S^1 \times S^1 \setminus \Delta$, which can be naturally identified as the infinite diagonal strip given in the description of the {\em skew plane} in Section \ref{sec_trivial_skew}.  The 
action of $\pi_1(M)$, which is a central extension of $\pi_1(\Sigma)$ by $\bZ$, is the lift of the diagonal action. 

This perspective makes it easy to ``see" that varying the metric on $\Sigma$ gives orbit-equivalent flows.  Fixing a hyperbolic metric on $\Sigma$ gives a realization of $\pi_1(\Sigma)$ as a group of isometries of $\bH^2$, which act by M\"obius transformations on $\partial \bH^2$.  Different choices of metric give actions on $\partial \bH^2$ which are not conjugate in the group of M\"obius transformations, but {\em are} conjugate in $\Homeo(S^1)$.  Thus, the induced actions on the skew plane coming from the orbit space of different choices of geodesic flow are conjugate.  By Theorem \ref{thm_action_determines_OEflow}, the flows are orbit equivalent.  

\begin{remark}
This argument above also applies to geodesic flow on compact surfaces with metrics of strict negative curvature, using the visual or Gromov boundary of $\wt \Sigma$ which generalizes the compactification of the Poincar\'e disc.    
\end{remark}

\subsection{Skew planes and slithering}  \label{sec_slithering} 
In the case of geodesic flow on $M = T^1 \Sigma$, there is a nice representation of $\wt M$ as a foliated space which makes the quotient to the orbit space particularly nice to see.  This picture also generalizes to other flows with skew orbit space.   We begin with the geodesic flow example.  

Since $\bH^2$ is contractible, $T^1 \bH^2$ is a trivial circle bundle over $\bH^2$.  We choose an explicit trivialization $\tau\colon T^1 \bH^2 \cong \bH^2 \times S^1$ so that the leaves of the stable foliation $\cF^s$ are horizontal, defining $\tau(x)$ to be the pair $(p, \xi)$ where $p$ is the point at which $x$ is based, and $\xi$ is the forward endpoint of the geodesic tangent to $x$.  
In this trivialization, leaves of $\cF^u$ ``spiral" upwards, meeting each leaf of $\cF^s$ along a single geodesic. 
See Figure \ref{fig_slithering}, left.   For each fixed leaf $l^s$ of $\cF^s$, the set of $\cF^u$ leaves that intersect $l^s$ can be parametrized by the interval $S^1 \setminus \{\xi\}$ via their (non-$\xi$) endpoint.   
 
Lifting to the universal cover unwraps the $S^1$ direction, giving two foliations on $\bH^2 \times \bR$.  Leaves of $\hfs$ remain horizontal, and each leaf of $\hfu$ projects to fill one fundamental domain for the action of $\pi_1(S^1) \cong \bZ$ on $\bR$.   Projecting to the orbit space, we represent each leaf of $\hfs$ as an interval (identified with $S^1 \setminus \{\xi\}$, as $\xi$ varies continuously now over $\wt{S^1}$), and each leaf of $\hfu$ can be seen as a vertical segment through the set of $\hfs$ leaves it intersects.  In this way one recovers again the skew plane.  See Figure \ref{fig_slithering} for an illustration of one fundamental domain for the action of $\pi_1(S^1)$.  

\begin{figure}[h]
\labellist
\small \hair 2pt
\pinlabel {$\theta$} at 69.106 26.687
\pinlabel {$\xi$} at 262.813 223.138
\pinlabel {$\theta$} at 632.619 91.719
\pinlabel {$\xi$} at 746.405 215
\endlabellist
\includegraphics[width=10cm]{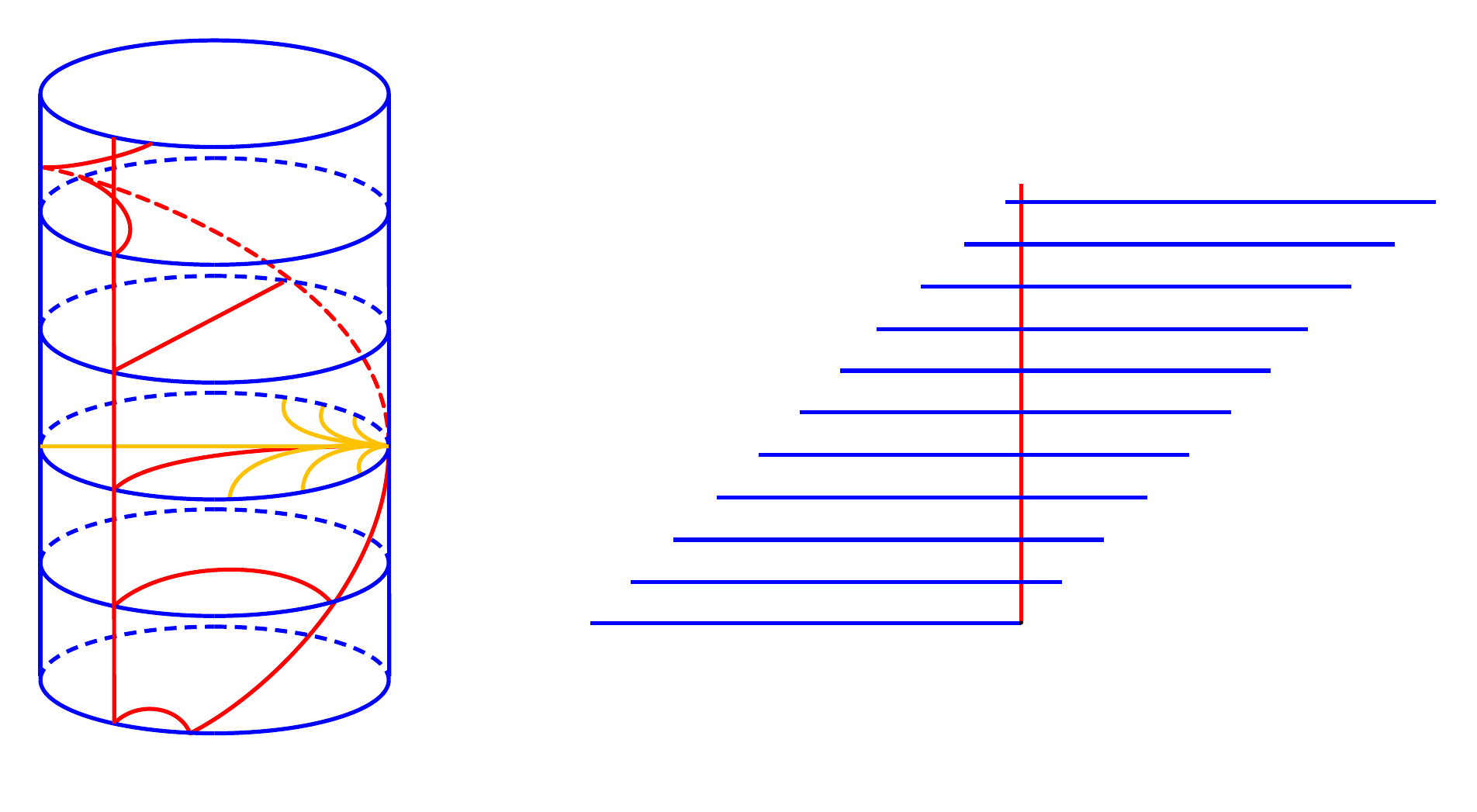}
\caption{The slithering of a skew Anosov flow}
 \label{fig_slithering}
\end{figure}

Topologically, the above description of $\R^3 = \bH^2 \times \bR$ as a foliated space agrees with the space of triples (described in Theorem \ref{thm_extended_convergence}) 
\[  \{(x, y, z) \in \R^3 : x < y < z < x+1 \}, \] 
with the foliations whose leaves are obtained by taking $z$ constant, and by taking $x$ constant.  

Thurston observed in \cite{Thu97} that this picture holds for all skew flows: given any Anosov flow $\flow$ on a manifold $M$ whose orbit space is the skew plane, the universal cover $\wt M$, as a foliated space, is exactly $\{(x, y, z) \in \R^3 : x < y < z < x+1 \}$ with the foliations described above.   
One can prove this by constructing the model $W^s_+$ given in the proof of Theorem \ref{thm_induced_action_model_flow}, using the $x$ and $z$ coordinates to specify a (unique) point $p$ in the skew plane, and the point $y$ a point on the positive side of the leaf $\bfs(p)$. 
Each of the projections to $S^1 = \bR/\bZ$ coming from the leaf spaces is an example of what Thurston calls a {\em slithering} of the manifold $M$ over $S^1$.

\section{The Handel--Thurston example and generalizations} \label{sec_Handel_Thurston_skew}
The Handel-Thurston construction, introduced in Section \ref{sec_example_Handel_Thurston}, is one example, or rather one family of examples, of a class of Anosov flows on graph manifolds.   This example was built out of two Seifert fibered 3-manifolds with torus boundary, glued along the boundary.    

Following \cite{BF21} we call a  pseudo-Anosov flow $\flow$ on a graph manifold {\em totally free} if in each Seifert fibered piece the free homotopy class of the fiber is not represented by a periodic orbit of $\flow$\footnote{In \cite{Bar96}, these are called ``generalized Handel-Thurston examples".}. We give a more detailed description of the structure of the restriction of a pseudo-Anosov flow to a Seifert piece in Sections \ref{sec_periodic_seifert} and \ref{sec_free_seifert}. 

The following lemma and theorem, while general, prove in particular that the Handel-Thurston example has the skew plane as its orbit space. 
For Anosov flows, Theorem \ref{thm_free_implies_skew} was originally proved (with a different argument) by Barbot in \cite[Théorème B]{Bar96}.

\begin{lemma} \label{lem_HT_free}
Any Anosov flow obtained by the Handel-Thurston construction is totally free.
\end{lemma}

\begin{theorem} \label{thm_free_implies_skew}
If $\flow$ is a totally free pseudo-Anosov flow on a graph manifold which is not a suspension, then $\orb$ is the skew plane.
\end{theorem}
Recall that a graph manifold is a closed, irreducible and $P^2$-irreducible $3$-manifold whose JSJ decomposition (see Theorem \ref{thm_geometric_JSJ_decomposition}) has only Seifert pieces. 
The suspension flow of a linear Anosov map does need to be ruled out, since it is a totally free flow supported on an exceptional graph manifold whose only Seifert piece is a trivial product.

Note that Theorem \ref{thm_free_implies_skew} implies, in particular, that a totally free pseudo-Anosov flow on a graph manifold is actually Anosov.

\begin{proof}[Proof of Lemma \ref{lem_HT_free}]
Following our previous notation, let $M = M_n$ be the manifold supporting the Handel-Thurston flow obtained by cut-and-paste surgery along a torus over a closed geodesic $c$ in $T^1\Sigma$.   There are two cases to consider, depending on whether $c$ is separating or not.  We treat the case where $c$ is separating (hence $M$ has two JSJ pieces) and leave the nonseparating case as an exercise; the outline is essentially the same. 

Let $P_1$ and $P_2$ denote the two JSJ pieces of $M$, glued along the torus $T$.    
We need to check that no periodic orbit is freely homotopic to a fiber.  Consider some periodic orbit $\gamma$.   By construction, if $\gamma \subset T$, then it is not a fiber of either piece.  Also, if $\gamma$ is contained in either $P_1$ or in $P_2$, then it agrees with an orbit of geodesic flow on a surface, so is not a fiber.  It remains to check the case where $\gamma$ crosses $T$.  In this case, $\gamma$ is a concatenation of orbit segments $\alpha_1, \beta_1, \alpha_2,\ldots$ where $\alpha_i \subset P_1$ and $\beta_i \subset P_2$ come from segments of geodesics in $T^1(\Sigma)$.  Consider the quotient space where the fibers of $P_1$ and $P_2$ are collapsed, and so $T$ is reduced to a point. This space is, topologically, the wedge of two surfaces (obtained by collapsing the curve in $\Sigma$ represented by the geodesic $c$ to a point), and its fundamental group is the free product $A \ast B$ of the fundamental group of these two surfaces.  Each segment $\alpha_i$ and $\beta_i$ projects to a nontrivial element of the respective factor of the free product.  Thus, in this quotient $\gamma$ represents a nontrivial word of the form $a_1 b_1 \ldots a_k b_k$ with $a_i \in A, b_i \in B$; it is necessarily cyclically reduced and so nontrivial, hence $\gamma$ was not freely homotopic to a fiber.  
\end{proof}

\begin{proof}[Proof of Theorem \ref{thm_free_implies_skew}]
Let $\flow$ be a totally free pseudo-Anosov flow on $M$ and let $P_1, \ldots P_k$ denote the Seifert fibered pieces of the JSJ decomposition of $M$.  
Let $\gamma_i$ denote the fiber of $P_i$.  By assumption, $\gamma_i$ acts freely on $\orb$, and thus has an axis $A_i$ in the stable leaf space $\Lambda(\bfs)$.  We first show, for each $i$, this axis is homeomorphic to $\bR$.  If not, then it has a structure as a union of intervals $\bigcup_j [a_j, b_j]$ between nonseparated points, as in Proposition \ref{prop:axis_structure}.   Since $\gamma_i$ is central in $\pi_1(P_i)$, each element of $\pi_1(P_i)$ acts on $A_i$ preserving this structure.  This gives a homomorphism $\pi_1(P_i) \to \bZ$.  The kernel of this homomorphism fixes each leaf $a_j$, so is trivial or virtually cyclic by Proposition \ref{prop_top_properties}.  We conclude that $\pi_1(P_i)$ is virtually $\bZ$ or a direct or semi-direct product $\bZ \rtimes \bZ$. Since $\pi_1(P_i)$ has a non-trivial center, we deduce that $\pi_1(P_i)$ is either $\bZ$ or a \emph{direct} product $\bZ \times \bZ$. A standard argument in 3-manifold topology (see for instance \cite[Prop 2.3]{Bar96}) shows that this is impossible.   %

Now, we claim that for any pieces $P_i, P_j$ we have $A_i = A_j$.  By connectedness, it suffices to prove this for adjacent pieces.  But adjacent pieces are glued along tori containing the fiber, and thus $\gamma_i$ and $\gamma_j$ commute.  By Lemma \ref{lem_commuting_axis_2}, this implies that $A_i = A_j$.  Let $A$ denote this common axis.

We next argue that $A$ is invariant by all of $\pi_1(M)$.  In the case where $\pi_1(M)$ is an amalgamated free product of $\pi_1(P_i)$, this is immediate, since the $\pi_1(P_i)$ generate.  However, in general $\pi_1(M)$ may be a graph of groups.  Loops in the graph give additional generators, but each such element normalizes $\pi_1(P_i)$. Since $A$ is the unique invariant embedded copy of $\bR$ in the leaf space for $\gamma_1$, hence the unique invariant  $\bR$ for $\pi_1(P_i)$, this shows that $A$ must be invariant under the normalizer as well.  

We conclude that $A$ is $\pi_1(M)$-invariant.  Thus, it follows from Axioms \ref{Axiom_A1} and \ref{Axiom_dense} that $A$ is equal to $\Lambda(\bfs)$. 
We have thus proved that $\flow$ is $\bR$-covered, and since $\flow$ is not a suspension, Theorem \ref{thm_trivial_implies_suspension} implies that $\orb$ is not the trivial plane. By Theorem \ref{thm:trichotomy}, $\orb$ is skew. 
\end{proof}

\section{The Franks--Williams example} 

Recall that the Franks--Williams construction, described in Example \ref{ex_Franks_Williams}, was obtained from gluing two manifolds $M_1$ and $M_2$.  Each $M_i$ was obtained by taking the mapping torus of a diffeomorphism of the torus obtained by doing an attracting DA on the fixed point of an Anosov diffeomorphism and removing a tubular neighborhood of the attracting periodic orbit, with boundary $T_i$ a torus transverse to the suspension flow.  $M_1$ was equipped with the suspension flow $\psi_1$, which is incoming on the boundary torus, and $M_2$ equipped with the time-reversal of the suspension flow (called $\psi_2$), so the boundary torus is outgoing. The two were then glued along the boundary by a rotation so that the stable foliation for $\psi_1$ on $T_1$ was transverse to the unstable foliation of $\psi_2$ on $T_2$. %
Let $\flow$ denote the glued flow on the glued manifold.  

Camargo showed that the orbit space of $\flow$ (as a bifoliated plane, that is, up to a foliation-preserving homeomorphism) is in fact sensitive to the hyperbolic linear map of the torus used in the construction.  By varying the map, one can produce countably many non-isomorphic examples, distinguishable by an invariant related to the continued fraction expansion of the slope of the eigenspace of the map.  See \cite[Theorem E, Corollary F]{Cam25} for a precise statement.  
However, the orbit spaces of all Franks--Williams type examples share some common features.  Our next goal is to describe these.  To do this, we start by describing the stable and unstable foliations of $\flow$ on $T$.  

Since stable (respectively, unstable) leaves are characterized by the asymptotic behavior of orbits in the future (respectively, past), the stable leaf of a point on $T$ is determined by the flow $\psi_1$ on $T_1$, hence its intersection with $M_1$ (and in particular its intersection with $T$) agrees with the stable leaf for $\psi_1$.  Similarly, the unstable foliation for $\flow$ on $T$ agrees with that for $\psi_2$.  
Thus, the intersections of the stable and unstable foliations for $\flow$ with $T$ appear as two transverse Reeb foliations of the torus, as shown in Figure \ref{fig_FW_torus_reeb} (left). 

\begin{figure}[h]
\labellist
\small \hair 2pt
\pinlabel {$a$} at 127.721 295
\pinlabel {$b$} at 263.573 205.869
\pinlabel {$T$} at 148.099 39
\pinlabel {$a$} at 460 295
\pinlabel {$b$} at 610 230.234
\pinlabel {$\widetilde T$} at 459.198 39
\pinlabel {$a$} at 850 40
\pinlabel {$b$} at 652 133
\pinlabel {$\overline T$} at 700 39
\endlabellist
\includegraphics[width=12cm]{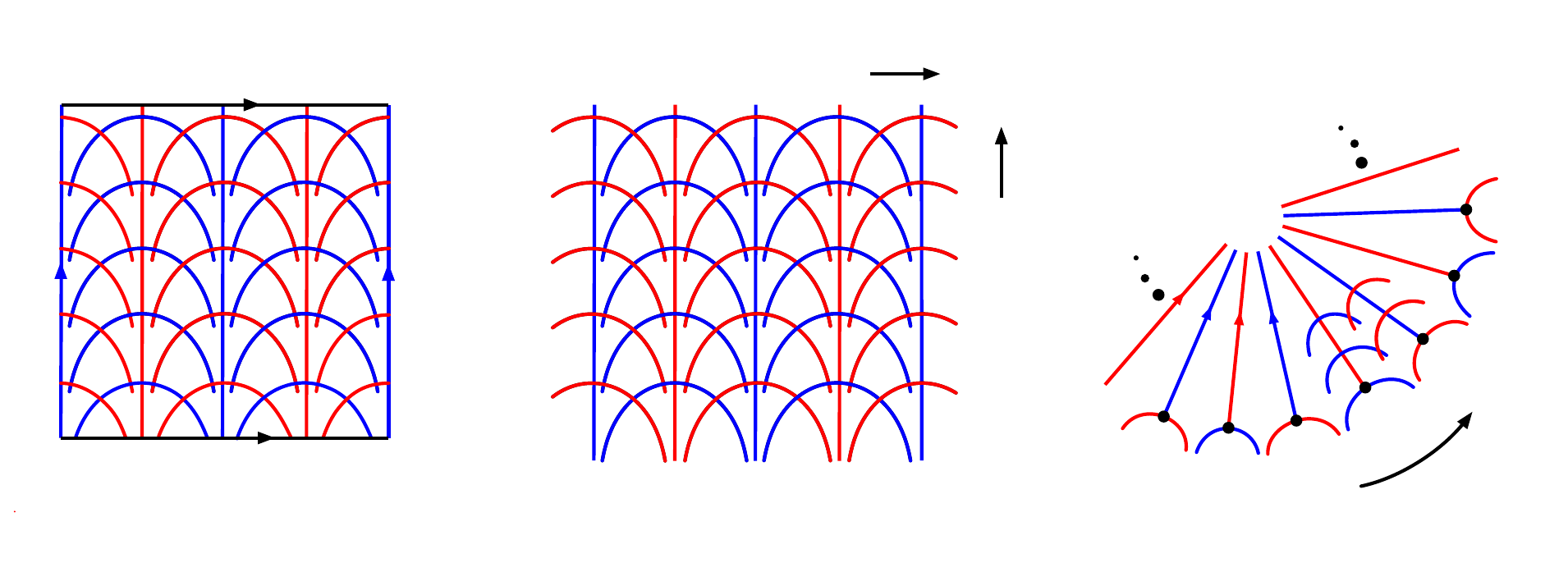}
\caption{The separating torus and its trace on the orbit space.}
 \label{fig_FW_torus_reeb}
\end{figure}

Let $\wt T$ denote a lift of $T$ to $\wt M$. The intersections of $\hfs$ and $\hfu$ with $\wt T$ are shown in Figure \ref{fig_FW_torus_reeb} (middle); and $\pi_1(T)$ acts by horizontal and vertical translations.  

Since $T$ is transverse, the projection of $\wt T$ to $\orb$ is injective (but not proper).  Let $\bar T$ denote its image.  We next describe the foliations $\bfs$ and $\bfu$ on $\bar T$.  The Reeb structure means that adjacent lifts of closed stable leaves are nonseparated, and similarly for unstable.  Moreover, there exists an element $b$ of $\pi_1(T)$ that preserves the lift to $\wt T$ of each closed leaf in $T$.  From this topological structure, it follows that $\bar T$ is (the interior of) a bi-infinite chain of lozenges, all sharing a single ideal corner, as shown in Figure \ref{fig_FW_torus_reeb} (right).  One of the two standard generators of $\pi_1(T)$ acts on $\bar T$ fixing all corners, and the other translates lozenges along the chain.   

Since no periodic orbit passes through $T$, the various lifts of $T$ to $\wt M$ project to pairwise disjoint subsets of $\orb$.   Thus, in $\orb$ we see countably many disjoint copies of this infinite chain, one for each lift of $T$ in $\wt M$.  The union of the lifts of $T$ to $\orb$ projects to a dense, open subset $U \subset \orb$. Indeed, for any open set $V \subset M_i$, $\flow(V)$ eventually intersects $T_i$.  %
Thus, no open set of orbits avoids all lifts of $T$, so $U$ is dense. Notice that $U$ corresponds to the projection of the \emph{wandering set}\footnote{Recall that a point $x$ is \emph{wandering} if there exists a neighborhood $V$ of $x$ and $T>0$ such that for all $t>T$, $\flow^t(V)\cap V =\emptyset$. See e.g., \cite{FH19} for more details.} of $\flow$.
   The complement of $U$ in $\orb$ is the closure of the boundaries of these regions (which corresponds to the orbits in the non-wandering set of $\flow$).  It has the local structure of a product of an interval and a Cantor set (see, e.g., \cite[Corollary 3.11]{BBM24b}); any leaf of $\bfs$ or $\bfu$ which passes through a lozenge in a lift of $T$ meets infinitely many lifts of $T$ along a set of intervals homeomorphic to the complement of a Cantor set.  See Figure \ref{fig_FW_orbitspace} for an illustration.  
   
   Finally, we note that the only lozenges in $\orb$ are part of a chain contained in one of the lifts $\bar T$.  If $x\in \orb$  is not one of the corners of one of the chains $g \bar T$ for some $g\in \pi_1(M)$, then either $x$ is an orbit in an attractor of $\flow$, in which case both rays of $\cF^s(x)$ meet $\pi_1(M) \cdot \bar T$, or $x$ is in a repeller, in which case rays of $\cF^u(x)$ meet $\pi_1(M) \cdot \bar T$, or $x$ is in the wandering set, where both are true.  
 In all cases, the non-corner criterion (Lemma \ref{lem_no_corner_criterion}) implies that $x$ is not a corner. 
A much more in depth treatment of the orbit space of these flows can be found in \cite[Chapter 4]{Cam25}.  

\begin{figure}[h]
\includegraphics[width=8cm]{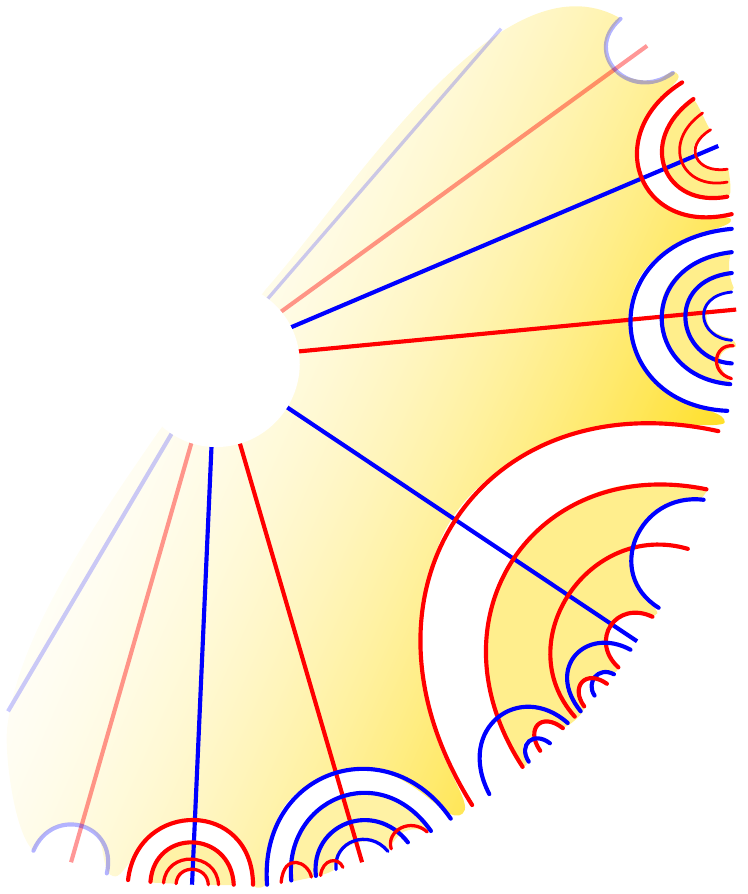}
\caption{The orbit space of a Franks-Williams flow, with a few lifts of $T$ highlighted.}
 \label{fig_FW_orbitspace}
\end{figure}

\subsection{A note on surgeries}

There is no universal procedure for describing the effect of Fried or Goodman's Dehn surgeries on the orbit space of a flow. 
However, in the case of flows with trivial or skew plane orbit space, one can sometimes show that a surgery of correct slope results in a flow whose orbit space is again the skew plane.  Fenley showed the following

\begin{theorem}[Theorem A of \cite{Fen94}]
If the orbit space of $\flow$ is the trivial or skew plane, and $\psi$ is obtained from $\flow$ by {\em coherently oriented Dehn surgery} then $\mathcal{O}_\psi$ is the skew plane. 
\end{theorem} 

``Coherently oriented" means, in the trivial case, that one may simultaneously do surgeries along many periodic orbits, but these must be all positive slope, or all of negative slope.   In the skew case, for a natural choice of orientation determined by the structure of the orbit space and orientation of the flow,  one may perform any Dehn surgery of one sign (what Fenley labels negative, but other conventions would choose as positive slope) but not the other sign.   
 The above result also follows from the characterization obtained by Asaoka, Bonatti and Marty in \cite[Theorem A]{ABM24} of skew Anosov flows as those that admit a coherently oriented Birkhoff section (see Lemma 14 and Proposition 15 of \cite{ABM24} for the description of the effect of Fried--Goodman surgery on Birkhoff sections). 

The effect of multiple surgeries on the orbit spaces of suspension flows was further studied in \cite{BI23}.
Note that there is no hope to have a simply stated universal theorem: The Fried--Ghys Conjecture \ref{ghys_fried_conjecture} (if true) implies that every transitive Anosov flow is obtained from a suspension flow by a sequence of Fried--Goodman surgeries, so every possible orbit space will appear.

\section{The Bonatti--Langevin example}  \label{BL_orbit_space}

We recall briefly the Bonatti--Langevin example described in Section \ref{sec_example_bonatti_langevin}.   The construction starts with a model neighborhood of a hyperbolic periodic orbit $\alpha$.  Annuli in the boundary of this neighborhood which are tangent to the (partially defined) flow are then glued together to produce an orientable manifold which is an $S^1$-bundle over $\Sigma$, where $\Sigma$ is the complement of two discs in $\mathbb{RP}^2$, with a flow that is transverse to the two torus boundary components, incoming on one and outgoing on the other.  
These boundary tori are then glued together so that the stable foliation on one is transverse to the unstable on the other.   Let $M$ denote the resulting (closed) manifold, $\flow$ the Anosov flow, and $T \subset M$ the transverse torus corresponding to the glued boundaries of  $\Sigma \mathbin{\widetilde{\times}} S^1$.  

To describe the orbit space, we begin (as in our discussion of the Handel--Thurston example) with the transverse torus $T$ and the structure of $\fs$ and $\fu$ on $T$.   Although there is some flexibility in the construction, for concreteness, we assume that $M$ was constructed as in the original paper of Bonatti and Langevin, so that the boundary tori of $\Sigma \mathbin{\widetilde{\times}} S^1$ are glued by exchanging the two coordinate directions specified by this product structure.  Thus, the periodic orbit $\alpha$ is freely homotopic to {\em both} standard generators of $\pi_1(T)$.   These can both be seen by pushing $\alpha$ along its local stable and unstable foliations:  following $\cF^s(\alpha)$ from $\alpha$ until it meets $T$ gives a closed curve on $T$, following $\cF^u(\alpha)$ approaches $T$ from the other side and gives a transverse closed curve.  
As in our discussion of the Handel--Thurston example, the fact that $\cF^s$ and $\cF^u$ leaves are characterized by future (respectively, past) behavior of orbits justifies our discussion above: the local stable and unstable leaves of $\alpha$ for the partial flow defined on $\Sigma \mathbin{\widetilde{\times}} S^1$ agree with those for $\flow$.   Thus, $\cF^s$ and $\cF^u$ meet $T$ in a pair of transverse, trivial foliations.  

Let $\wt T$ denote a lift of $T$ to $\wt M$ and $\bar{T}$ its projection to $\orb$.   This is an open subset of $\orb$ invariant under $\pi_1(T) \cong \bZ^2$, and our discussion above implies that it is trivially foliated.   Both (standard) generators of this $\bZ^2$ subgroup are freely homotopic to the periodic orbit $\alpha$, so act with fixed points on $\orb$.  Thus, by Proposition \ref{prop_stabilizer_scalloped}, $\pi_1(T)$ preserves a (unique) scalloped region $U$ in $\orb$. We will see in the next chapter (Proposition \ref{prop_trace_QT_tori}) that this region $U$ corresponds to $\bar{T}$. 

Having described the image of the torus, as a next step we consider its complement.  
Let $M_0$ denote $M \smallsetminus T \cong  \Sigma \mathbin{\widetilde{\times}} S^1$.  Its fundamental group is (redundantly) generated by four elements $a, b, c, d$, where $a$ is represented by $\alpha$ (the generator of the Seifert fiber), $b$ and $c$ each represent a boundary component of $\Sigma$, and $d$ represents the generator of $\pi_1(\mathbb{RP}^2)$, so $db = (dc)^{-1}$, as illustrated in Figure \ref{fig_BLpi1}.

\begin{figure}[h]
\labellist
\small \hair 2pt
\pinlabel {$b$} at 89.896 141.174
\pinlabel {$c$} at 194.309 141.174
\pinlabel {$d$} at 152 232.317
\pinlabel {$db = (dc)^{-1}$} at 187.701 75.129
\pinlabel {$b$} at 508 160
\pinlabel {$c$} at 500.942 236.280
\pinlabel {$d$} at 409.745 215
\pinlabel {$\mathrm{in}$} at 732.238 145.137
\pinlabel {$\mathrm{out}$} at 860.442 145.137
\pinlabel {{\color{blue} $\mathcal F^s(\alpha)$}} at 725.629 35
\pinlabel {{\color{red}$\mathcal F^u(\alpha)$}} at 882 49
\endlabellist
\includegraphics[width=12cm]{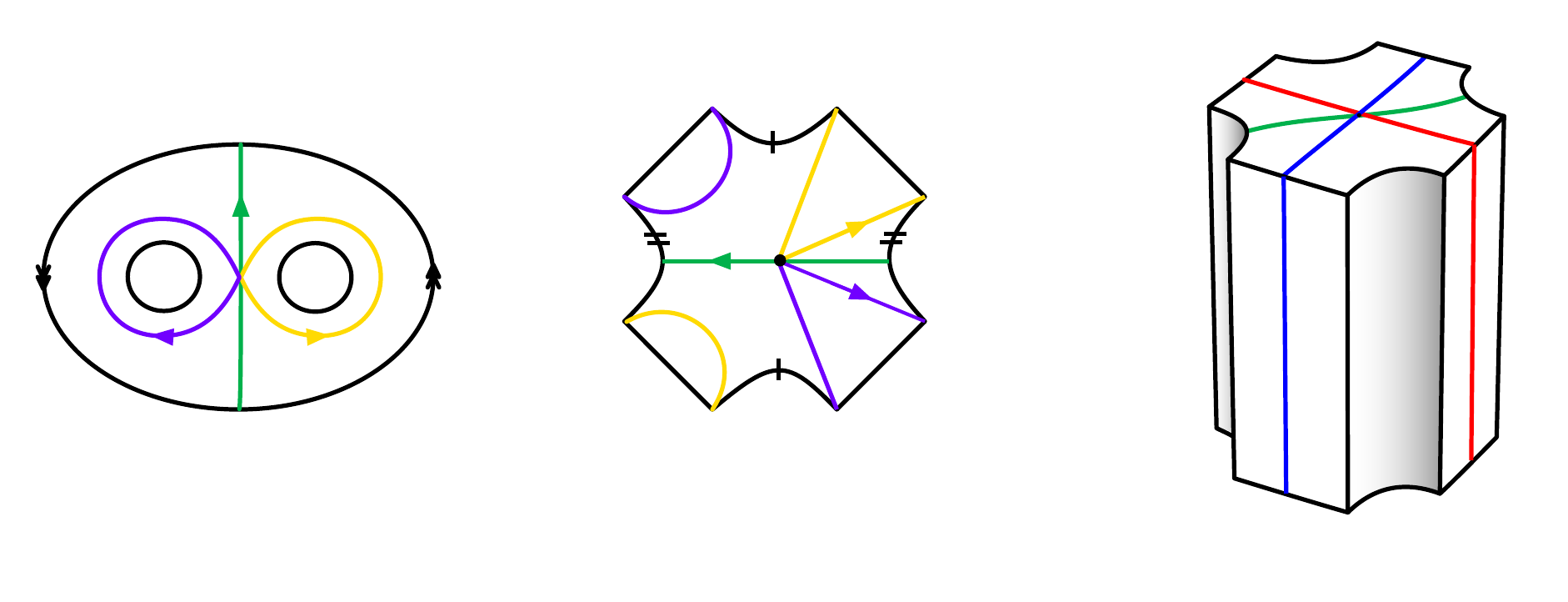}
\caption{The model block in Bonatti--Langevin's example}
 \label{fig_BLpi1}
\end{figure}
The pairs $\{a, b\}$ and $\{a, c\}$ each generate the fundamental group of a boundary torus.  
Following our discussion above, let $U$ denote the scalloped region invariant under the subgroup generated by $a$ and $b$.  The subgroup generated by $a$ and $c$ also stabilizes a scalloped region.  Because this is invariant under $a$, it must share a common lozenge $L$ with $U$, as shown in Figure \ref{fig_BL_scalloped}.

\begin{figure}[h]
\labellist
\small \hair 2pt
\pinlabel {$x$} at 167.406 281.674
\pinlabel {$dx$} at 245.752 182
\pinlabel {$L$} at 194.695 229.701
\endlabellist
\includegraphics[width=8cm]{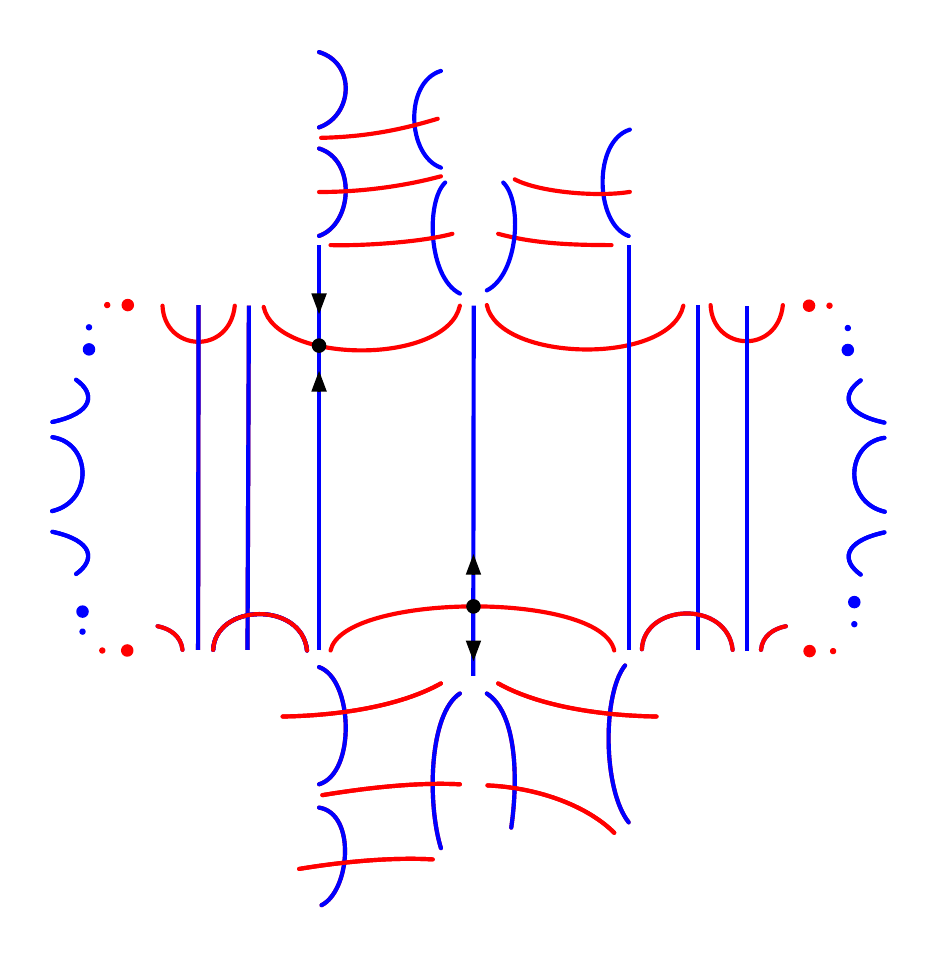}
\caption{Scalloped regions sharing a lozenge}
 \label{fig_BL_scalloped}
\end{figure}

Thus, we have described the projection to the orbit space of the two lifts of $T$, corresponding to neighboring lifts of the two boundary components of $M_0$.  Consider now the action of $d$.  Recall that tangent boundary annuli were glued with a rotation, so we have $dad^{-1} = a^{-1}$.  Thus, $d$ conjugates the action of $a$ on $\orb$ to the action of $a^{-1}$.  Observe also that $\alpha$ can be pushed off itself along the loop $d$ staying transverse to the flow.  The trace of this isotopy gives an embedded Klein bottle $K$ in $M \smallsetminus T$, which is transverse to the flow on the complement of $\alpha$.  
As we will explain in more detail in the next chapter, the annulus $K \smallsetminus \alpha$ lifted to $\wt{M}$ projects to a lozenge, whose two corners correspond to the two lifts of $\alpha$ differing by the action of $d$.  Considering the action of $d$ on the orbit space, we see now that the lozenge adjacent to $L$ is also part of two scalloped regions.  
Continuing this process iteratively, one can show that the projection of $\wt{M}_0$ to $\orb$ is what is called a {\em tree of scalloped regions} -- a chain of lozenges such that each lozenge is adjacent to another on each of its four sides.  

The universal cover $\wt{M}$ is obtained by gluing copies of $\wt{M}_0$ along boundary components, the lifts of the transverse torus.   As we have seen above, each lift of a boundary torus corresponds to a scalloped region inside a tree of scalloped regions.  Gluing another copy of $\wt{M}_0$ along this boundary torus corresponds to realizing this scalloped region as a line of lozenges in a different way, and including it into another tree of scalloped regions, as shown in figure \ref{fig_gluing_BL}.   The orbit space can thus be completely filled out by iteratively moving through copies of $\wt{M}_0$.  For a detailed description, applicable not only to this example but to the broader class of {\em totally periodic flows} that will be discussed in Section \ref{sec_periodic_seifert}, we refer the reader to \cite[Chapter 3]{Cam25}. 

\begin{figure}[h]
\includegraphics[width=10cm]{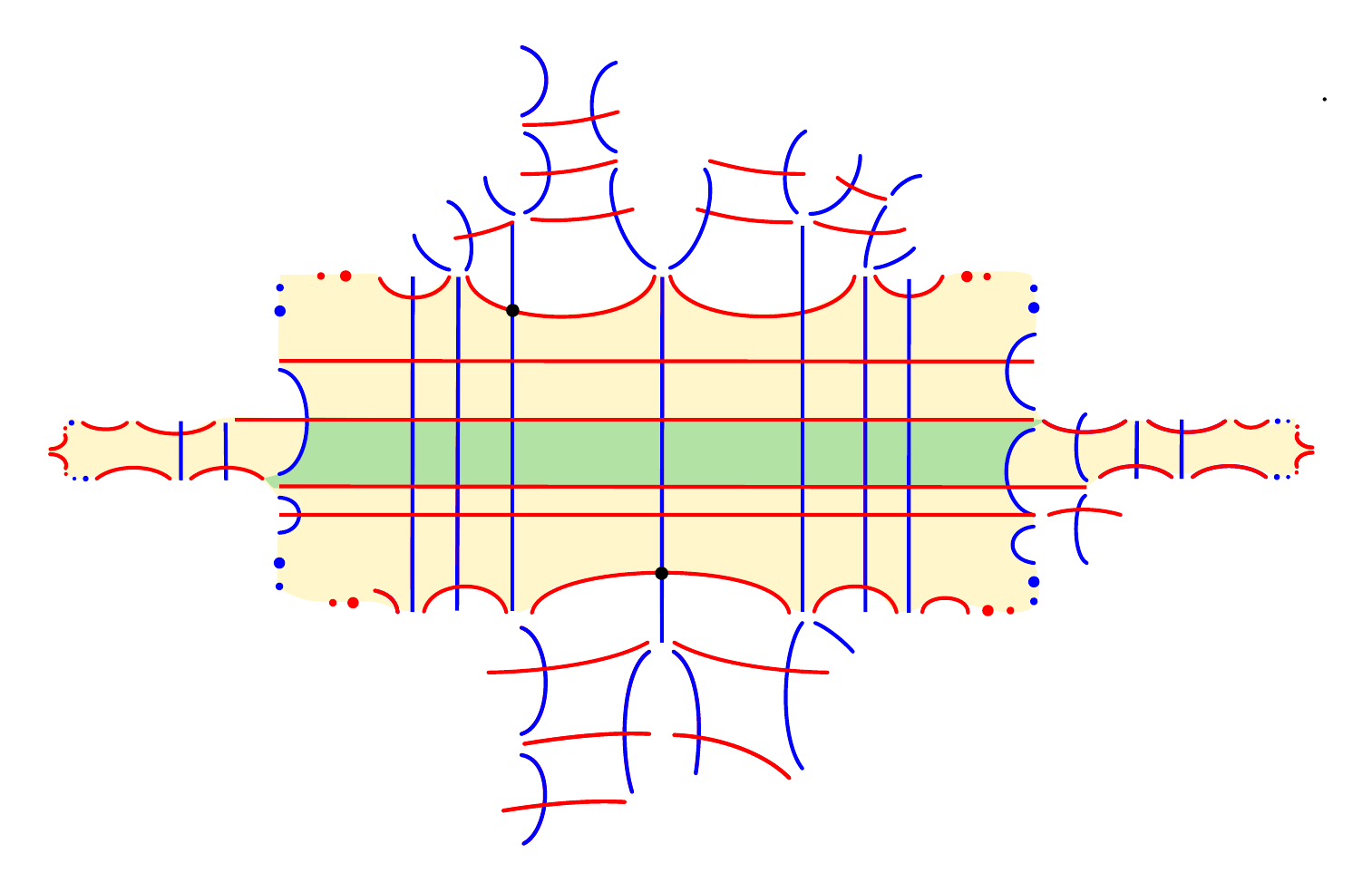}
\caption{A lozenge (highlighted in green) shared by two scalloped regions in a tree of scalloped regions.}
 \label{fig_gluing_BL}
\end{figure}

%% file: Plane3ManifoldDictionary_vector.tex
In this chapter, we describe a dictionary correspondence between topological features of a compact 3-manifold $M$ equipped with a pseudo-Anosov flow $\flow$, and the structures visible in its orbit space $\orb$.    

Our first goal is to describe {\em Birkhoff annuli}.  These are immersed annuli in $M$ which realize a free homotopy between periodic orbits of $\flow$, and they correspond to certain  chains of lozenges in $\orb$.    The structure of these chains also contains significant information about the annuli themselves, such as whether they can be made embedded after a homotopy.  Using this understanding of Birkhoff annuli, we then prove an important theorem of Barbot and Fenley which states that the tori or Klein bottles from the JSJ decomposition of $M$ can be realized as surfaces that are transverse to $\flow$ except possibly along finitely many periodic orbits.  These periodic orbits divide the JSJ surfaces into Birkhoff annuli, hence the need for their study.

We will also see that a realization of the JSJ surfaces in this ``quasi-transverse" position is essentially unique up to homotopy along orbits, 
so gives a canonical decomposition of $M$ into JSJ pieces which are well-adapted or in ``good position" with respect to the flow.  
This decomposition has become an important tool towards the classification of pseudo-Anosov flows on graph manifolds, culminating in the works \cite{BF13,BF15,BF21}.
In Sections \ref{sec_periodic_seifert} and \ref{sec_free_seifert} we describe some of this work, classifying the restriction of pseudo-Anosov flows to (well-adapted) Seifert fibered JSJ pieces, again making use of the interaction between orbit space and 3-manifold.   

Quasi-transverse surfaces also play an important role when constructing examples of Anosov flows.  Paulet's work in \cite{Pau24} gives an inverse procedure to the JSJ decomposition along quasi-transverse tori; following this, one may build many examples of pseudo-Anosov flows by gluing together partial flows on manifolds with quasi-transverse torus boundary components.

\section{Free homotopy of periodic orbits} \label{sec_free_homotopy}

We begin with an illustrative example of the relationship between the features of a flow on a 3-manifold and its orbit space.  
 Let $\flow$ be the geodesic flow on $M = T^1 \Sigma$, where $\Sigma$ is a hyperbolic surface. Let $c$ be a periodic orbit of $\flow$ (i.e., a closed geodesic), and let $c'$ denote $c$ with the opposite orientation.  Rotating the fibers of the bundle $T^1 \Sigma \to \Sigma$ gives a free homotopy from $c$ to $c'$.   

To understand this homotopy in the orbit space, we consider first the orbits of the flow lifted to the cover $T^1 \wt{\Sigma}$ of $T^1 \Sigma$, as in Section  \ref{ex_geodesic_flow_orbit_space}. 
Let $\wt c$ be a lift of $c$ to $T^1 \wt{\Sigma}$, and let $h_t$ be a lift of the homotopy with $h_0(\wt{c}) =  \wt c$ and $h_1(\wt{c})$ a lift of $c'$, as illustrated in Figure \ref{fig_free_homotopy_geo}, left. For values of $t$ between $0$ and $1$, the set of forward endpoints of the unit tangent vectors along $\wt c$ covers an interval on $\partial \bH^2$.  When $t$ is small, this interval lies near the forward endpoint of $c$.  It increases in size until $t = 1/2$ when half the boundary is covered, and then decreases again towards the other endpoint of $\wt c$, as shown in Figure \ref{fig_free_homotopy_geo}, left.
Using the explicit description of $\orb$ in Section \ref{ex_geodesic_flow_orbit_space}, one can easily see that the projection of this homotopy to $\orb$ traces out a single lozenge, whose corners are (the projections of) $\wt c$ and $\wt c'$.  See Figure \ref{fig_free_homotopy_geo}, right. 

\begin{figure}[h]
\labellist
\small \hair 2pt
\pinlabel {$y$} at 114.712 281.697
\pinlabel {$x$} at 114.712 77.985
\pinlabel {$\widetilde c$} at 144.379 127.429
\pinlabel {$h_{1/4}$} at 114.712 46.340
\pinlabel {$h_{1/2}$} at 323.369 46.340
\pinlabel {$h_{7/8}$} at 534.004 46.340
\pinlabel {$c$} at 756.506 102.707
\pinlabel {$h_1(\widetilde c)$} at 920 244
\pinlabel {$y$} at 895.941 117.540
\pinlabel {$x$} at 779.251 21.617
\endlabellist
\includegraphics[width=13.5cm]{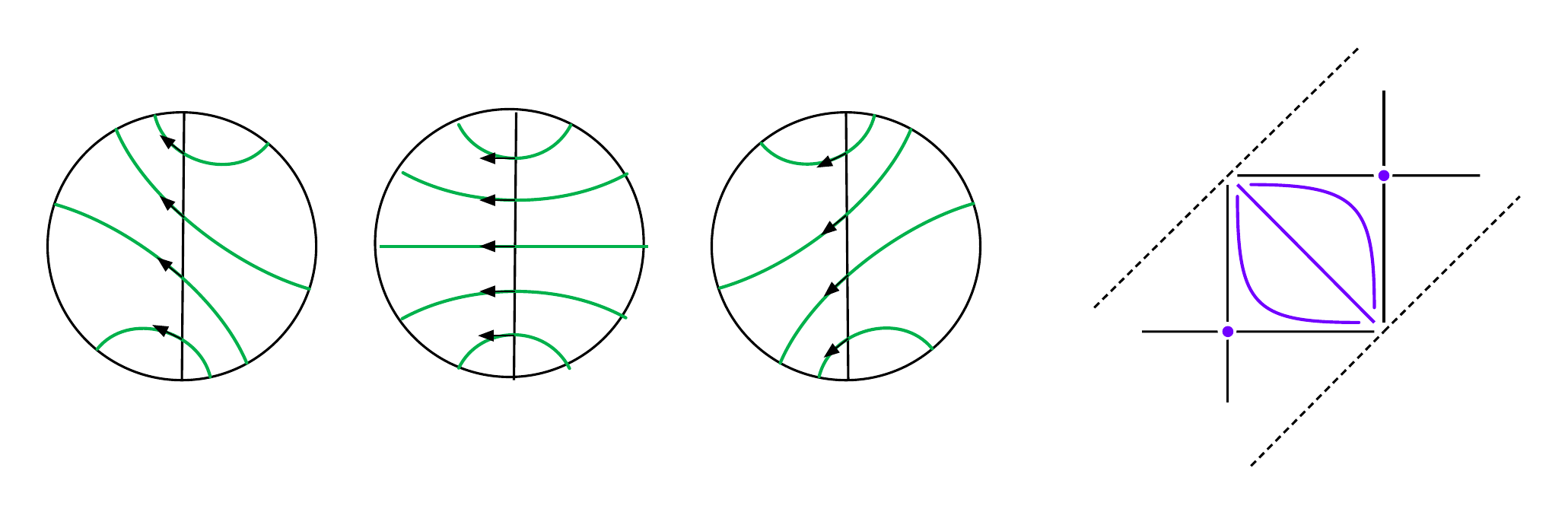}
\caption{The homotopy $h_t$ and its projection to $\orb$}\label{fig_free_homotopy_geo}
\end{figure}
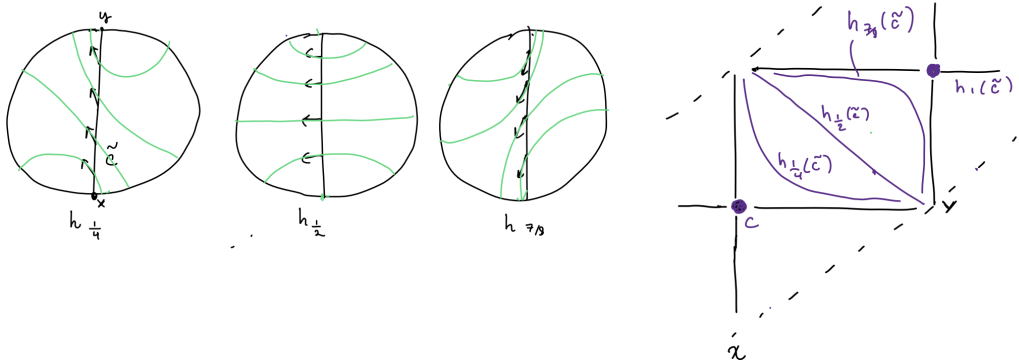

It turns out this picture is quite general. Our first goal in this section is to make precise the relationship between free homotopies of (unoriented) periodic orbits and chains of lozenges invariant by elements of $\pi_1(M)$.  

\begin{proposition}[Fenley, \cite{Fen95b}]\label{prop_free_homotopy}
Let $\alpha, \beta$ be periodic orbits of a pseudo-Anosov flow $\flow$ on $M$ and $g \in \pi_1(M)$.  Then $\alpha, \beta$ are freely homotopic to 
$g$ as \emph{unoriented} curves (and freely homotopic to each other by a {\em nontrivial} free homotopy if $\alpha=\beta$)  if and only if there exist lifts $\wt\alpha, \wt \beta$ in $\wt M$ such that their projections to $\orb$ are distinct corners of a chain of lozenges $\cC$ fixed by $g$. 
\end{proposition}
\begin{proof}
Let $\alpha$ and  $\beta$ be periodic orbits of $\flow$. We allow the possibility $\alpha = \beta$.  
By elementary covering space theory, if there exist distinct lifts $\wt\alpha, \wt \beta$ of these orbits to $\wt M$ that are both invariant by $g$, then $\alpha$ and $\beta$ are freely homotopic to a loop representing $g$ in $\pi_1(M)$ as unoriented curves, and freely homotopic to each other by a nontrivial homotopy.  
For the other direction, suppose we know that $\alpha, \beta$ are freely homotopic to $g$.  As in Proposition \ref{prop_closed_orbit_fixed_point}, we can pick a lift $\wt \alpha$ of $\alpha$ that is invariant by $g$ and projects to a fixed point for the action of $g$ on $\orb$.  Lifting a nontrivial free homotopy from $\alpha$ to $\beta$ will give a distinct lift $\wt \beta$, also invariant by $g$, which projects to a distinct fixed point for $g$ on $\orb$.  By Theorem \ref{thm:distinct_fix_is_chain}, these fixed points are corners of a $g$-invariant chain of lozenges. Since $g$ fixes two corners of the chain, it fixes every corner.  
\end{proof}

Provided $M$ is orientable and the foliations $\cF^s$, $\cF^u$ are transversely orientable, the free homotopy between two orbits which lift to corners of a {\em single} lozenge always reverses orientation of the curve.  More precisely, if $\alpha$ and $\beta$ are freely homotopic and lift to opposite corners of a single lozenge, then as {\em oriented} orbits this free homotopy is between $\alpha$ and $\beta^{-1}$, just as in the example of geodesic flow discussed above.  One can see this by considering the dynamics of the action of $\pi_1(M)$ at the corners of the lozenge: an element $g$ fixing each corner must contract the stable leaf of one, and expand the stable leaf of the other.  
Thus, Proposition \ref{prop_free_homotopy} has the following consequence: 
\begin{corollary} \label{cor_reverse_orientation} 
If there is a nontrivial free homotopy between any periodic orbits of $\flow$, then there exists $g \in \pi_1(M)$ where both $g$ and $g^{-1}$ are represented by (oriented) periodic orbits.  
\end{corollary}

Surprisingly, the existence of free homotopies between periodic orbits actually gives significant global information about the flow. 
As one example of this phenomenon, the following theorem characterizes suspension flows among all (nonsingular) topological Anosov flows in terms of free homotopies between orbits. 

\begin{theorem}  \label{thm_suspension_free_homotopy}
A topological Anosov flow is orbit equivalent to a suspension if and only if there is no nontrivial free homotopy between any two periodic orbits.  
\end{theorem} 
\begin{proof} 
The orbit space of suspension flows is trivial, and by the characterization of Anosov-like actions on trivial planes (Proposition \ref{prop:trivial_affine}) every element of $\pi_1(M)$ either acts freely or with a unique fixed point.  Thus, in the suspension case, Proposition \ref{prop_free_homotopy} gives that there are no nontrivial free homotopies between periodic orbits.  
For the converse, if an Anosov flow is not a suspension, then by Theorem \ref{thm:trichotomy} its orbit space is either skew or has branching.  In the skew case every point is the corner of a lozenge, and in the branching case one obtains lozenges with corners fixed by nontrivial elements of $\pi_1(M)$, thus Proposition \ref{prop_free_homotopy} gives nontrivial free homotopies between orbits.  
\end{proof} 

\begin{remark}[Parallel with Fried's theorem] 
Theorem \ref{thm_suspension_free_homotopy} can be seen as a homotopical version of Fried's homological characterization in \cite{Fri82}.  Fried's characterization can be rephrased as stating that an Anosov flow on $M$ is \emph{not} a suspension if and only if there exists an element $[g]\in H_1(M;\bZ)$ such that both $[g]$ and $-[g]$ (allowing for the possibility $[g]=0\in H_1(M;\bZ)$) are represented by periodic orbits of the flow.  
By Corollary \ref{cor_reverse_orientation}, the homotopical characterization in Theorem \ref{thm_suspension_free_homotopy} says that an Anosov flow is not a suspension if and only if there exists an element $g\in \pi_1(M)$ such that both $g$ and $g^{-1}$ are represented by periodic orbits. 
\end{remark} 

One can also rephrase the trichotomy from Theorem \ref{thm:trichotomy} in the Anosov case in terms of free homotopies between periodic orbits, as follows.  For simplicity we assume some orientability in the statement.  

\begin{theorem}  \label{thm_free_homotopy_trichotomy}
Let $\flow$ be a topological Anosov flow %
and assume $\cF^s$ and $\cF^u$ are transversely orientable.  Exactly one of the following holds: 
\begin{enumerate}[label=(\roman*)]
\item $\flow$ is a suspension and there are no nontrivial free homotopies between periodic orbits. 
\item $\flow$ has skew orbit space, and every periodic orbit admits a nontrivial free homotopy to another. 
\item  $\flow$ has branching leaves, and some periodic orbits are alone in their free homotopy class, while others admit nontrivial free homotopies. 
\end{enumerate} 
\end{theorem} 
The condition of transverse orientability only comes up in the characterization of skew flows: without this, there are examples of flows that have periodic orbits alone in their free homotopy class, but such that the square of the orbit (i.e., going twice around) is freely homotopic to other orbits.  %

In addition to Proposition \ref{prop_free_homotopy}, a key element in the proof is the following
\begin{theorem}[\cite{BFM22, BBM24b}] \label{thm_exist_noncorners}
Let $\flow$ be a pseudo-Anosov flow whose orbit space is not the skew plane.  Then there exists $g \in \pi_1(M)$ whose action on $\orb$ has a unique fixed point, which is not the corner of any lozenge.
\end{theorem} 

For transitive pseudo-Anosov flows, this was proved in Lemma \ref{lem_transitive_implies_dense_noncorner}. %
The proof is much more involved in the nontransitive case, even for topological Anosov flows. See \cite[Corollary 6.8]{BBM24b}.  
As in the statement of Proposition \ref{prop_transitive_non_corner_dense} or Lemma \ref{lem_transitive_implies_dense_noncorner}, a much stronger result is in fact true: Among the set of points of $\orb$ fixed by nontrivial elements of $\pi_1(M)$, a dense subset arises as unique fixed points.   

\begin{proof}[Proof of Theorem \ref{thm_free_homotopy_trichotomy}]
By Theorem \ref{thm:trichotomy}, the orbit space of an Anosov flow is either trivial, skew or has branching leaves.  The trivial case is already covered in Theorem \ref{thm_suspension_free_homotopy}.  
When the orbit space is skew, each point in the orbit space is the corner of a chain of lozenges.  Thus, if some element $g\in \pi_1(M)$ preserves the transverse orientations and fixes a point, it fixes all corners of the chain, hence represents periodic orbits which are related by nontrivial free homotopies. 

Thus, in order to prove the theorem, we have to show that when the plane is neither trivial nor skew, there exists some nontrivial element of $\pi_1(M)$ with more than one fixed point, and some other which preserves transverse orientations and has only one fixed point.  
The existence of elements with multiple fixed points is given by Proposition \ref{prop:4weak4strong}, which states that any pair of nonseparated leaves are fixed by a common nontrivial $g \in \pi_1(M)$, which will thus have a fixed point on each leaf.  Finally, the existence of elements preserving transverse orientations and with only one fixed point is Theorem \ref{thm_exist_noncorners}. 
\end{proof}

\begin{rem}
The characterization of suspensions given by Theorem \ref{thm_free_homotopy_trichotomy} no longer holds for \emph{pseudo}-Anosov flows, as there exist many pseudo-Anosov flows that are not suspensions, but nevertheless do not admit any lozenges and hence no nontrivial free homotopies between periodic orbits.  

These are known in the literature as ``pseudo-Anosov flows without perfect fits".  This is because, (by 
Theorem \ref{thm_perfect_fit_implies_lozenge})  the orbit space of a flow has a perfect fit if and only if it contains a lozenge.   The class of pseudo-Anosov flows without perfect fits, especially on hyperbolic 3-manifolds, has been of particular recent interest due to their relationship with veering triangulations, see for 
instance \cite{Ago11,Gue16,LMT23,LMT24,Mos92a,Mos92b}. 
\end{rem}

Theorem \ref{thm_free_homotopy_trichotomy} gives information about the orbit space of an Anosov flow based on {\em how many}  periodic orbits represent a free homotopy class. In \cite{BMB24,BFM22}, it was shown that knowing instead \emph{which} free homotopy classes are represented by at least one periodic orbit gives even more information.  In fact, in many cases, this data completely determines the flow.  To describe this precisely, we make a definition.  

\begin{definition}\label{def_free_hom_data}
Let $\flow$ be a flow on $M$. 
The \emph{free homotopy data} of $\flow$, denoted $\cP(\flow)$ is the set of free homotopy classes represented by unoriented periodic orbits of $\flow$. That is,
\[
\cP(\flow) :=\{ [g] \in [\pi_1(M)] \mid [g] \text{ or } [g^{-1}] \text{ is represented by a periodic orbit of } \flow\},
\]
where $[g]$ denotes the conjugacy class of an element of $\pi_1(M)$.
\end{definition}

The following result, a special case of the main theorem of \cite{BFM22}, gives a large class of examples (including all pseudo-Anosov flows on hyperbolic 3-manifolds) where the free homotopy data determines the flow.  

\begin{theorem}[\cite{BMB24,BFM22}]\label{thm_free_homotopy_data}
Let $\flow_1$, $\flow_2$ be two pseudo-Anosov flows on $M$. Assume that $\flow_1$ is transitive and 
has no transverse, incompressible tori or Klein bottles.
Then $\cP(\flow_1)=\cP(\flow_2)$ if and only if there is an orbit equivalence of $\flow_1$ and $\flow_2$ by a map isotopic to the identity.

More generally, $\cP(\flow_1)$ is the image of $\cP(\flow_2)$ under an automorphism of $\pi_1(M)$ if and only if $\flow_1$ and $\flow_2$ are orbit equivalent (by a map inducing this automorphism). 
\end{theorem}
The proof of this theorem is beyond the scope of this text; however the material presented up to this point should give the reader all the prerequisite tools needed to understand it.  We have stated only a special case of the results of \cite{BMB24,BFM22} for simplicity -- the general statements also address the case of flows with transverse tori or Klein bottles, where in certain situations an additional finite invariant is required to determine the flow up to orbit equivalence.  
As a preview of some of the results to come, we note also that the incompressibility hypothesis given in the statement above is redundant: Lemma \ref{lem_transversetorus} below shows that a torus or Klein bottle transverse to a pseudo-Anosov flow is automatically incompressible.

\section{Birkhoff Annuli}

Proposition \ref{prop_free_homotopy} showed that $g$-invariant lozenges in the orbit space correspond to freely homotopic periodic orbits in the manifold.  Such a free homotopy can be realized as a $\pi_1$-injectively immersed annulus bounded by the freely homotopic periodic orbits.   In this section we study a special class of annuli bounded by periodic orbits called {\em Birkhoff annuli} and describe how they correspond to $g$-invariant lozenges or chains.   This study was largely initiated by Barbot in \cite{Bar95b}.

\begin{definition}\label{def_Birkhoff_annulus} 
Let $\flow$ be a pseudo-Anosov flow on $M$.  A {\em Birkhoff annulus} (for $\flow$) is an immersed annulus $A$ such that 
the boundaries of $A$ are periodic orbits, and
$\flow$ is transverse to the interior of $A$.  
\end{definition}

Since the boundaries of a Birkhoff annulus are periodic orbits, Proposition \ref{prop_top_properties} implies such an annulus is automatically $\pi_1$-injectively immersed.  However, the two boundary orbits are permitted to coincide, in which case the image of $A$ in $M$ will (topologically) be an immersed torus or Klein bottle.   However, we still think of $A$ as an annulus.  In particular, we use the terminology the {\em core of $A$} to refer to a generator of  the image of $\pi_1(A)$ in $\pi_1(M)$.  

If $A$ is a Birkhoff annulus, the stable and unstable foliations of $\flow$ intersect its interior $\mathring A$ to give two transverse, 1-dimensional foliations of $\mathring A$.  These extend continuously to foliations on $A$ that include the boundary components of $A$ as leaves.  {\em A priori}, these foliations may be singular, but we will very soon prove they are not.  We call these the \emph{induced foliations} on $A$, denoted $\fs_A$ and $\fu_A$.  

\begin{proposition} \label{prop_hopf_nonsingular} 
Let $A$ be a Birkhoff annulus.  Then $\fs_A$ and $\fu_A$ are nonsingular, and closed leaves of $\fs_A$ in $\mathring A$ are disjoint from closed leaves of $\fu_A$.  
\end{proposition} 

To prove this we use the Poincar\'e--Hopf formula, which  relates the Euler characteristic of a surface to the indices around singular points of a vector field with isolated zeroes.  Hopf proved this also for line fields, where index is defined similarly to that for vector fields but may take values in $\frac{1}{2} \bZ$. With this definition, a $k$-prong singularity has index $(2-k)/2$.  
\begin{lemma}[Poincar\'e--Hopf, see \cite{Hop83} Chapter 3.1]
Let $S$ be a closed surface and $\ell$ a continuous line field on $S$ with isolated singularities.  Then the sum of the indices of $\ell$ over all singular points is equal to $\chi(S)$.
\end{lemma}

\begin{proof}[Proof of Proposition \ref{prop_hopf_nonsingular}]
First we prove the foliations are each nonsingular.  Fix one of the induced foliations and perturb it to be smooth.  Let  $\ell$ be a line field tangent to the smoothed foliation and vanishing only at the prong points (if any).  Double $A$ along the boundary to produce a torus with a line field.  Since $A$ is transverse to $\flow$ except at the boundary, the only possible singularities of the doubled line field are $k$-prongs, where $k\geq 3$.  Such singularities have negative index, and the torus has Euler characteristic zero, so we conclude by Poincar\'e--Hopf that $\ell$ is nonsingular.   
The statement about non-boundary closed leaves of $\fs_A$ and $\fu_A$ being disjoint is now a simple consequence of transversality of these foliations in $\mathring A$.  
\end{proof}

\subsection{From Birkhoff annuli to chains of lozenges} 

Our next goal is to describe the features of $\orb$ corresponding to Birkhoff annuli in $M$, more precisely, their {\em trace}.  

\begin{definition}[Trace of a Birkhoff annulus]
If $A$ is a Birkhoff annulus for $\flow$ with core represented by $g \in \pi_1(M)$, the {\em trace} of $A$ is the subset of $\orb$ obtained by projecting the $g$-invariant lift $\wt A \subset \wt M$ to $\orb$.   
\end{definition} 
The definition of trace depends on $g$, hence on a choice of basepoint for $\pi_1(M)$.  However, a change of basepoint for $\pi_1(M)$ (or equivalently, taking a different lift of $A$ to $\wt M$) modifies the trace only by the action of some element of $\pi_1(M)$ on $\orb$.  Thus, the trace is well defined {\em up to the action of $\pi_1(M)$}, and any property which is well defined up to automorphism of $\orb$, such as being contained in a single lozenge, is a meaningful notion.   We typically use the notation $\hat A$ for the trace of a Birkhoff annulus $A$.   The following result allows one to determine the trace from the closed leaves of
$\fs_A$ and $\fu_A$.  

\begin{theorem}\label{thm_birkhoff_to_chain}
Let $A$ be a Birkhoff annulus with core represented by $g$ and trace $\hat A$.  
Suppose that the induced foliations of $\fs_A$ and $\fu_A$ have a total of $n-1$ closed leaves in $\mathring A$. Then the following hold: 
 \begin{enumerate}[label=(\roman*)]
\item \label{item_trace_vs_chain} $\hat A$ is contained in a chain of $g$-invariant lozenges $\cC = \{L_1, \dots, L_n\}$, and consists of the union of the 
interiors of the lozenges, all sides that are shared between two lozenges, and the corners of $L_1$ and $L_n$ (respectively) not shared by other lozenges of $\cC$.  These two corners are the projections of the boundary components of the $g$-invariant lift $\wt A$.  
\item \label{item_Birkhoff_turn_ideal_corners} Consecutive lozenges in $\cC$ share a side, and no three consecutive lozenges share a common corner. 
\item \label{item_closd_leaves} The closed leaves of $\fs_A$ (resp.~$\fu_A$) project to the shared stable (resp.~unstable) sides of consecutive lozenges in $\cC$.
\end{enumerate}
\end{theorem}
This combinatorial description and correspondence were first described in \cite{Bar95b} and then formalized and used in \cite{BBY17,Pau24} under the name of \emph{combinatorial type}, but without the explicit relation to the associated chain of lozenges.  The proof 
we give of Theorem \ref{thm_birkhoff_to_chain} here is different than the treatment in these works; we use the dynamics on the orbit space as the primary tool, rather than topological considerations in the 3-manifold.  

Before giving the proof, we make a few remarks on the conditions in the statement of Theorem \ref{thm_birkhoff_to_chain}.  
First, item \ref{item_trace_vs_chain} says that the chain $\cC = \{L_1, \dots, L_n\}$, while not equal to the trace, uniquely determines the trace.  We will call $\cC$ the {\em chain associated with $\hat A$} (or with $\wt A$). 
Items \ref{item_Birkhoff_turn_ideal_corners} and \ref{item_closd_leaves} impose significant constraints on the possible chains associated with Birkhoff annuli: the chain is a line if and only if the closed leaves of the induced foliations on $A$ are either all in $\cF^s$ or all in $\cF^u$; and, in general, the chain associated to $A$ cannot contain any ``diagonal" lozenges sharing a corner but not a side.  These conditions also give an easy way to recover the combinatorial configuration of the lozenges in $\cC$ from the finite sequence of closed leaves of $\cF^s_A$ and $\cF^u_A$ that appear in $A$.  
See Figure %
\ref{fig_general_birkhoff} for an illustration of one example.

\begin{figure}[h]
\labellist
\small \hair 2pt
\pinlabel {$\alpha$} at 31.903 153.112
\pinlabel {$\beta$} at 427.076 153.112
\pinlabel {$A_1$} at 72.279 152.253
\pinlabel {$A_2$} at 124.253 152.253
\pinlabel {$A_3$} at 176.227 152.253
\pinlabel {$A_4$} at 228.201 152.253
\pinlabel {$A_5$} at 280.174 152.253
\pinlabel {$A_6$} at 332.148 152.253
\pinlabel {$A_7$} at 384.122 152.253
\pinlabel {$L_1$} at 606.622 81.809
\pinlabel {$L_2$} at 676.207 75.796
\pinlabel {$L_3$} at 742.355 64.628
\pinlabel {$L_4$} at 753.523 115.313
\pinlabel {$L_5$} at 761.255 164.280
\pinlabel {$L_6$} at 774.141 210.670
\pinlabel {$L_7$} at 699.402 228.711
\endlabellist
\includegraphics[width = 11cm]{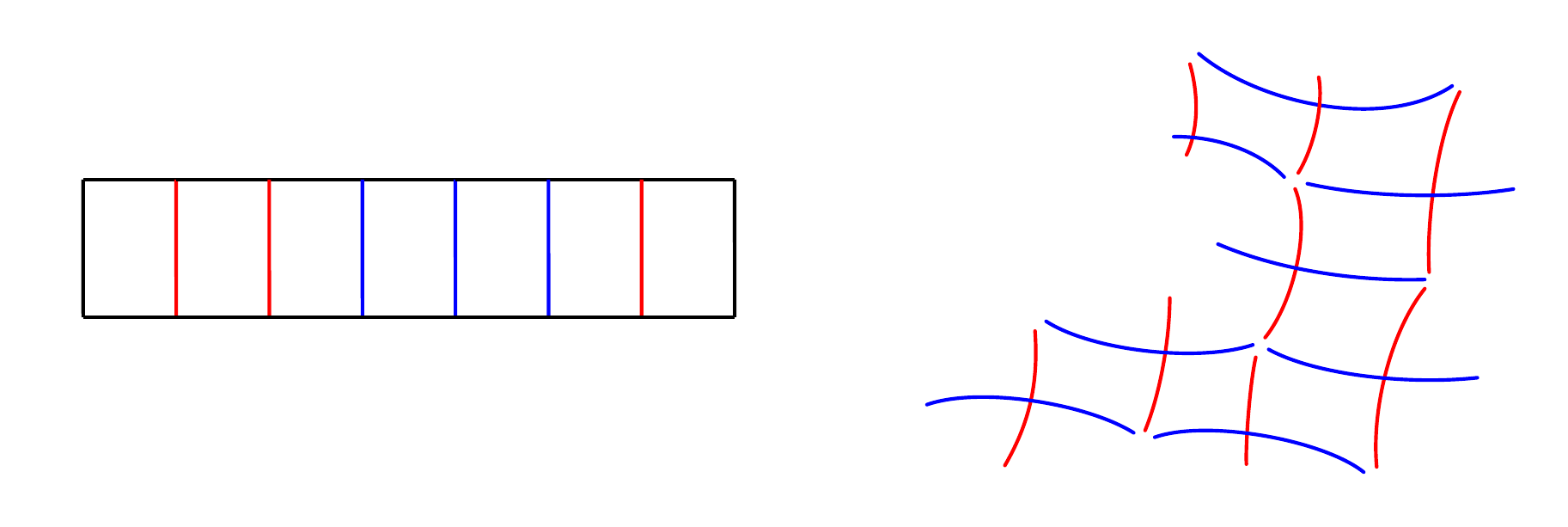}
\caption{An example of a Birkhoff annulus $A$, and its projection to $\orb$. The number and arrangement of the closed leaves in $A$ uniquely determine the configuration of the associated chain of lozenges.}
\label{fig_general_birkhoff}
\end{figure}
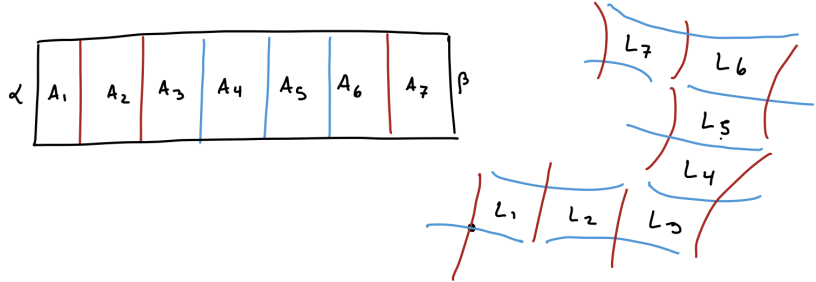

\begin{exercise} \label{ex_birkhoff_to_chain}
Draw a series of closed leaves on a Birkhoff annulus, giving some variation on the picture in Figure \ref{fig_general_birkhoff}, and use it to construct a chain of lozenges. Then do the reverse procedure, drawing a finite chain of lozenges satisfying item \ref{item_Birkhoff_turn_ideal_corners} of Theorem \ref{thm_birkhoff_to_chain} and show the corresponding arrangement of closed leaves on a Birkhoff annulus that would give rise to this chain.  
\end{exercise} 

\begin{exercise}
Show that item \ref{item_Birkhoff_turn_ideal_corners} is equivalent to the statement that $\cC$ is a minimal chain between its endpoints (with respect to inclusion), in the sense that any proper sub-collection of lozenges including $L_1$ and $L_n$ is disconnected.   Compare with Lemma \ref{lem_charac_minimal_chain}.  
\end{exercise} 

\begin{proof}[Proof of Theorem \ref{thm_birkhoff_to_chain}]

First, we label the relevant objects and set up notation for good lifts and projections.  
Let $\gamma \subset A$ be a simple arc connecting the two boundary components of $A$ and intersecting each closed leaf of $\cF^s_A$ and $\cF^u_A$ exactly once.  
Let $\wt A$ be the $g$-invariant lift of $A$ to $\wt M$, and let $\wt \gamma$ be a lift of $\gamma$ to $\wt A$.  Then $g \wt\gamma \cap \wt \gamma = \emptyset$ and the region bounded by these two curves and segments of the boundary of $\wt A$ between them gives a fundamental domain $\wt D$ for the action of $g$ on $\wt A$, and the interior of $\wt D$ is an open disc transverse to $\hflow$.  Let $\hat D$ denote its projection to $\orb$. 
Finally, let $a$ and $b$ be the projections to $\orb$ of the two boundary components $\alpha$ and  $\beta$ of $\wt A$.  
Thus, the region $\hat D$ is bounded by the union of the set $\{ a, b\}$ with the projections $\hat \gamma$ and $g \hat \gamma$ of $\wt \gamma$ and $g \wt \gamma$, respectively.   

Let $c_0 = \alpha, c_1, \ldots c_n = \beta$ be the closed leaves of $\cF^s_A \cup \cF^u_A$, enumerated in order and let $A_i$ denote the open sub-annulus bounded by $c_i$ and $c_{i-1}$, with $\wt A_i$ the lift of $A_i$ to $\wt A$ and $\wt D_i := \wt A_i \cap \wt D$. 
Since $\cF^s_A$ and $\cF^u_A$ have no closed leaves in $A_i$, the induced foliation on each strip $\wt A_i$ is trivial.   Thus, as in the proof of Theorem \ref{thm_orb_is_plane}, no orbit of $\hflow$ may intersect such a region twice, otherwise one could produce a closed transversal to $\hfs$ or $\hfu$.  In particular, each set $\wt D_i$ projects injectively to $\orb$.  Let $\hat D_i$ denote its projection and $\hat \gamma_i$ the projection of $\wt \gamma \cap \wt A_i$.    

We can now proceed with the main argument of the proof: we will describe exactly what each $\hat D_i$ looks like and then use this to show that each $A_i$ projects to a single lozenge, arranged in a chain satisfying the properties in the statement of the theorem.   This is done through several claims.   As usual, we use $\bfs$ and $\bfu$ to denote the induced stable and unstable foliations in $\orb$.  

\begin{claim}  \label{claim_D_in_one_quadrant}
The closure of $\hat \gamma_1$ meets $\bfs(a)$ and $\bfu(a)$ only at the point $a$.  Consequently, the closure of $\hat D_1$ lies in the closure of a single quadrant $Q$ of $a$.  
\end{claim}

\begin{proof} 
Suppose for contradiction that this is not the case.  
For concreteness, we suppose the first intersection point $x$ of the closure of $\hat \gamma_1$ with $\bfs(a) \cup \bfu(a)$ lies on $\bfu(a)$ and that $g$ contracts $\bfu(a)$, the other cases are symmetric (possibly by considering $g \hat \gamma$ in place of $\hat \gamma$ and the action of $g^{-1}$ instead).  Thus, the closure of $g \hat\gamma_1$ first intersects $\bfu(a)$ at $g x$. 

Parameterize $\gamma_1$ as a curve $s(t)$, $t \in [0,1]$ with $s(0)=a$.  For $t>0$ small, we have $\bfu(s(t)) \cap \bfs(a) \neq \emptyset$.  Let $y = s(t) \in \hat \gamma$ either be the first point after $s(0)$ such that $\bfu(y) \cap \bfs(a) = \emptyset$ as in Figure \ref{fig_D_one_quadrant} (right) or, if this intersection is always nonempty, then the projection of $s(t)$ to the leaf space $\Lambda(\bfu)$ is an interval, and we take $y$ so that $\bfu(y)$ is a boundary point of this interval furthest from $\bfu(a)$ as in Figure \ref{fig_D_one_quadrant} (left). 

\begin{figure}[h]
\labellist
\small \hair 2pt
\pinlabel {$a$} at 74.601 97.826
\pinlabel {$gy$} at 224.621 215.875
\pinlabel {$y$} at 284.466 160.950
\pinlabel {$gx$} at 239.378 61.755
\pinlabel {$x$} at 341.031 61.755
\pinlabel {$\hat\gamma$} at 381.200 128.978
\pinlabel {$a$} at 503.349 97.826
\pinlabel {$gy$} at 632 229.812
\pinlabel {$y$} at 721.412 156.031
\pinlabel {$gx$} at 668.945 60.936
\pinlabel {$x$} at 769.779 60.936
\pinlabel {$\hat\gamma$} at 785.355 140
\endlabellist
\includegraphics[width = 12cm]{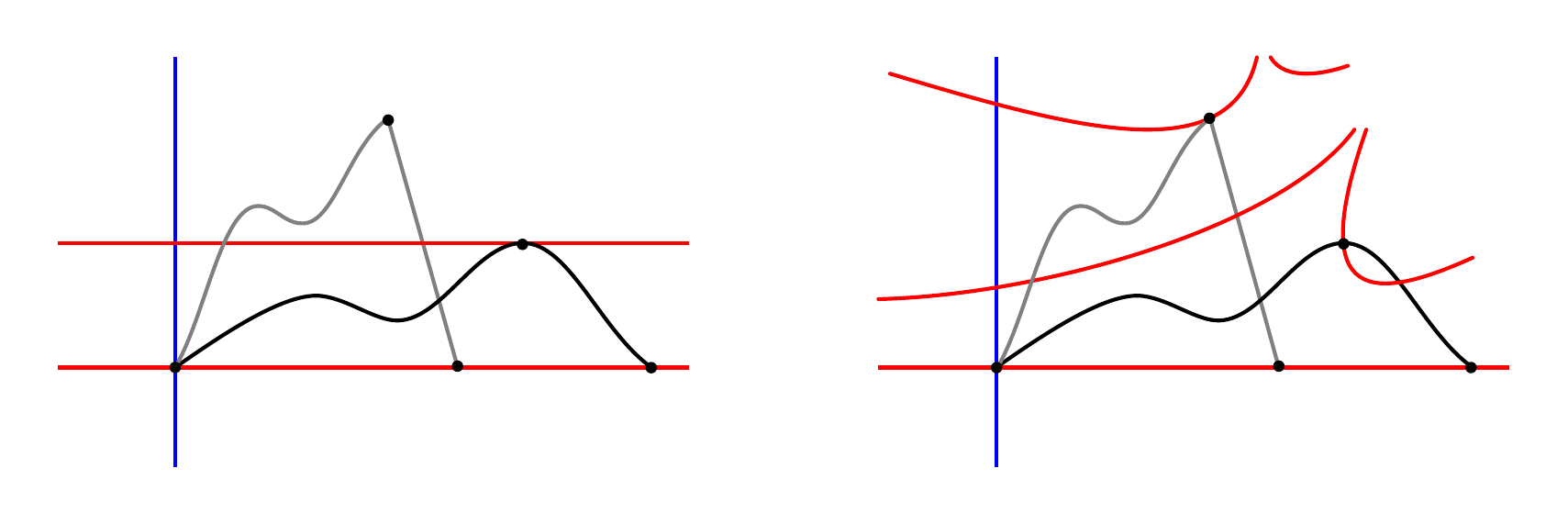}
\caption{If $\hat D$ does not lie only in one quadrant of $a$, then $\hat \gamma$ and $g \hat \gamma$ are forced to intersect.}
\label{fig_D_one_quadrant}
\end{figure}
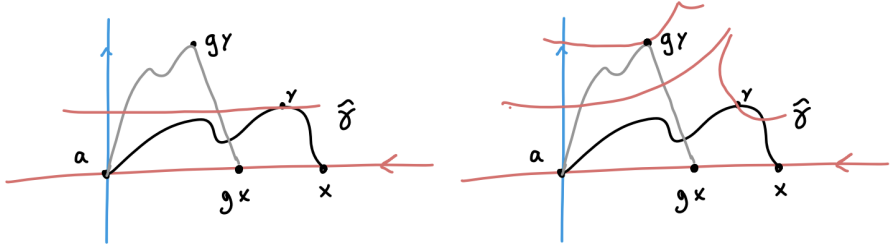

Since $g$ is contracting along $\bfu(a)$ and expanding along $\bfs(a)$, the points $gy$ and $gx$ are forced to lie in different connected components of $Q \setminus \hat \gamma$ as shown in Figure \ref{fig_D_one_quadrant}.  Thus, we have $g \hat\gamma \cap \hat\gamma \neq \emptyset$, which is a contradiction. 
\end{proof}

Keeping the notation from above, we next show: 
\begin{claim} \label{claim_next_lozenge}
$Q$ contains a lozenge $L_1$ with corner $a$.
\end{claim}

\begin{proof}
Since the closure of $\hat D$ contains $b$, and $b$ is connected to $a$ by a chain of $g$-invariant lozenges, the quadrant $Q$ containing $\hat D_1$ must contain a lozenge in the chain.  
\end{proof}

If the opposite corner of $L_1$ is $b$, Claim \ref{claim_D_in_one_quadrant} applied to $b$ instead of $a$ implies that $\hat \gamma \subset L_1$, and is a simple arc from $a$ to $b$.  Thus, $\hat \gamma$ crosses no other $g$-invariant leaf, $\hat D = \hat D_1$ is contained in a single lozenge (and is a fundamental domain for the action of $g$ on that lozenge) and we have proved what we wanted to show.  
So, going forward, we assume the opposite corner of $L_1$ is some point $x_1 \neq b$, and thus the closure of $\hat \gamma_1$ intersects $(\bfs(x_1) \cup \bfu(x_1))$ at some point $y_1$ projected from $\wt c_1$.  
We now make an argument along the same lines as that in Claim \ref{claim_D_in_one_quadrant} to show that the next segment $\hat \gamma_2$ of $\hat \gamma$ stays in one quadrant of $x_1$.  

\begin{claim} \label{claim_one_intersection}
The closure of $\hat \gamma_2 $ intersects $(\bfs(x_1) \cup \bfu(x_1))$ only at $y_1$. 
\end{claim}
\begin{proof} 
For concreteness, suppose $y_1 \in \bfs(x_1)$ and $g(y_1)$ separates $y_1$ from $x_1$ on $\bfs(x_1)$, so $g$ is contracting along the stable leaf of $x_1$.   As before, the other cases are proved by an identical argument.  Let $Q_2$ denote the quadrant of $x_1$ that $\hat \gamma$ enters after passing from $L_1$ through $y_1$, as shown in Figure \ref{fig_one_intersection}. 

\begin{figure}[h]

\labellist
\small \hair 2pt
\pinlabel {$a$} at 47.770 129.946
\pinlabel {$L_1$} at 94 245
\pinlabel {$x_1$} at 189.404 276.608
\pinlabel {$gy_1$} at 191.918 222
\pinlabel {$y_1$} at 190 158
\pinlabel {$z$} at 159 90
\pinlabel {$\hat\gamma$} at 94.702 147
\pinlabel {$g\hat\gamma$} at 110 210
\pinlabel {$\hat D_1$} at 130.739 181.906
\pinlabel {$\hat D_2$} at 259.802 156.764
\pinlabel {$Q_2$} at 283.268 212.915
\pinlabel {$a$} at 459.263 121.565
\pinlabel {$L_1$} at 502.843 238
\pinlabel {$x_1$} at 599.221 264.037
\pinlabel {$gy_1$} at 606.763 209.562
\pinlabel {$y_1$} at 600.897 150
\pinlabel {$z$} at 748.397 286.665
\pinlabel {$gz$} at 797.006 287.503
\pinlabel {$\hat\gamma$} at 512.061 135
\pinlabel {$g\hat\gamma$} at 519.604 200
\pinlabel {$\hat D_1$} at 548.936 167.659
\pinlabel {$\hat D_2$} at 724.931 191.125
\pinlabel {$Q_2$} at 714.875 98.099
\endlabellist
\includegraphics[width = 11cm]{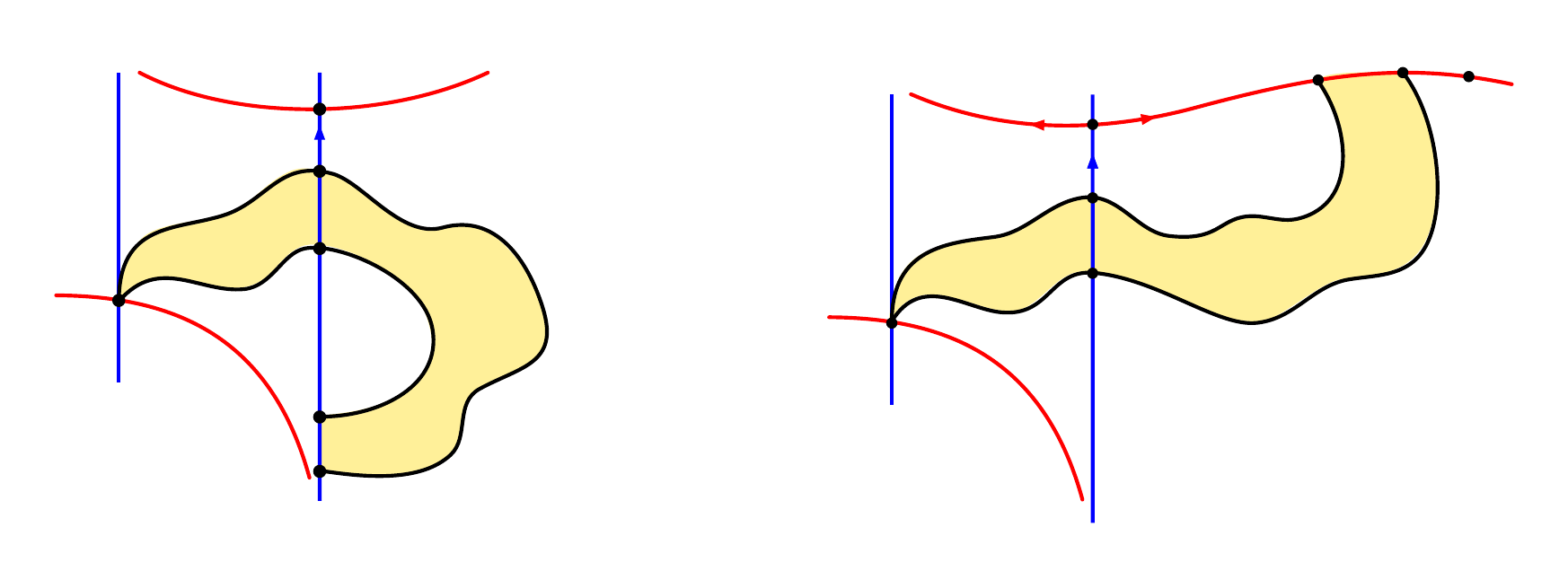}
\caption{Proof of Claim \ref{claim_one_intersection}: if $\hat \gamma$ intersects a stable/unstable leaf of $x_1$ a second time, then embeddedness of $\hat D_2$ shown here is incompatible with the dynamics of the action of $g$.}
\label{fig_one_intersection}
\end{figure}
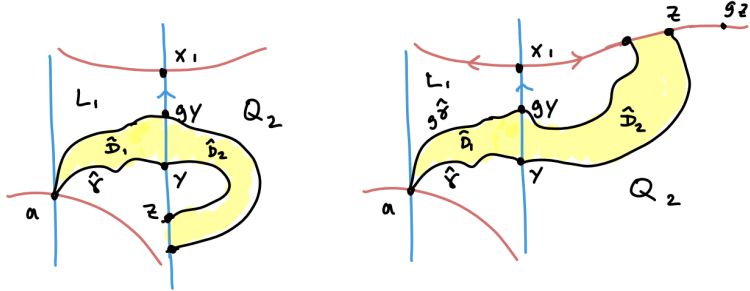

Suppose for contradiction that the closure of $\hat \gamma_2$ exits the quadrant $Q_2$, and let $z$ be the first point along $\hat \gamma$ where it exits, which is either on $\bfs(x_1)$ or on $\bfu(x_1)$.  If $z \in \bfs(x_1)$, then the fact that $\hat D_2$ is embedded implies that $z$ separates $g(z)$ from $x_1$, which contradicts the fact that $g$ contracts stable leaves.  See again Figure \ref{fig_one_intersection} for an illustration.  
If instead $z \in \bfu(x_1)$, again that $\hat D_2$ is embedded implies that $g(z)$ separates $z$ from $x_1$ on $\bfu(x_1)$, which contradicts the fact that $g$ expands the unstable leaf.  This proves the claim. 
\end{proof} 

Now, the same argument from 
Claim \ref{claim_next_lozenge} 
shows that $Q_2$ contains a lozenge with corner $x_1$.  Let $x_2$ denote its opposite corner.  Iteratively applying Claims \ref{claim_next_lozenge} and \ref{claim_one_intersection}, we produce a chain of lozenges $L_i$, with corners $x_{i-1}, x_{i}$ which contains $\hat{\gamma}$, and with each $L_i$ intersecting $\hat{\gamma}$ along the arc $\hat \gamma_i$.  Moreover, these have the property that $\hat\gamma$ meets only one quadrant of any corner of any lozenge.  Equivalently, this property means that no three consecutive lozenges in the chain share a corner, and it implies that the projection of $\wt D$ to $\hat D$ is injective.  Since $\hat{\gamma}$ is compact, this process terminates after a finite sequence of lozenges, ending at the corner $b$.  

We can now conclude the proof of the theorem. 
By construction, for each lozenge $L_i$, the set $\hat D \cap L_i$ is a fundamental domain for the action of $g$ on the interior of $L_i$, bounded on one side by an arc $\hat \gamma \cap L_i$ and on the other by $g(\hat \gamma \cap L_i)$.  Thus, the full projection of $\wt A$ to $\orb$ is the union of the interior of these lozenges, their shared sides and the corners $a$ and $b$, with each $\wt A_i$ projecting to a single lozenge, and the lifts of the closed leaves $c_i$ (for $i \neq 0, n$) projecting to shared sides of adjacent lozenges, which is what we needed to show. 
\end{proof} 
 
 \begin{remark}  \label{rem_stay_in_quadrant}
The fact that no three lozenges of $\cC$ share a corner implies that $A$ does not ``spiral" around its boundary component (as often happens for other surfaces called {\em Birkhoff sections}, see Remark \ref{rem_birkhoff_section}) like the local picture in Figure \ref{fig_spiraling}, left, but instead stays in a single quadrant, as in Figure \ref{fig_spiraling}, right.  
 \end{remark} 
 
 \begin{figure}[h]
\includegraphics[width = 10cm]{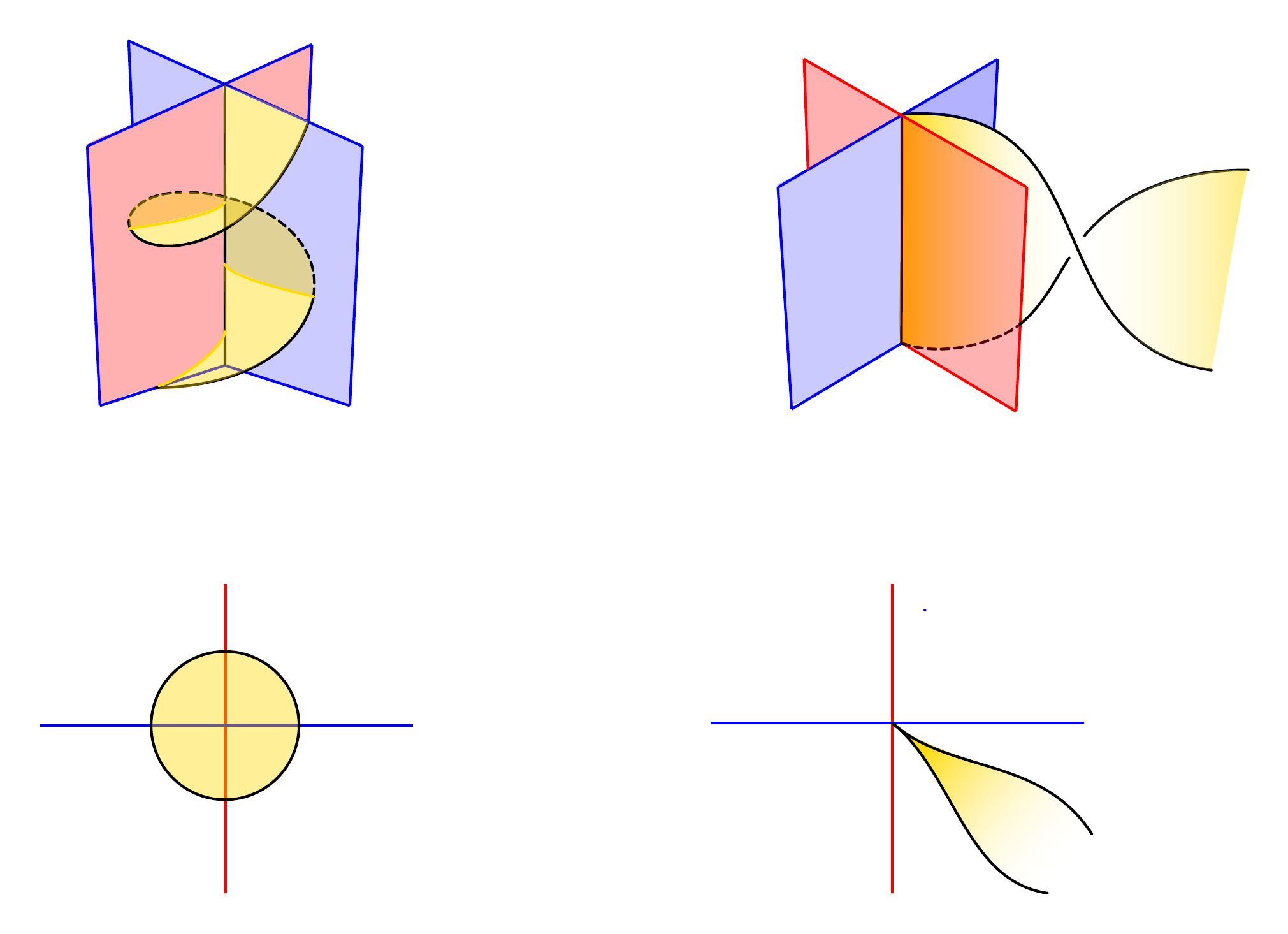}
\caption{A Birkhoff annulus may ``twist" relative to the direction of the flow, as shown on the right, but always stays in one quadrant of a boundary orbit, unlike the local transverse surface on the left.}
\label{fig_spiraling}
\end{figure}
 
We have already discussed the correspondence between the closed leaves of induced foliations on a Birkhoff annulus and the associated chain (recall for instance Exercise \ref{ex_birkhoff_to_chain}).  In fact, one can recover not only the closed leaves but the topology of the induced foliations by paying attention to the foliations on the trace in $\orb$.   Some examples are given in Figure \ref{fig_Reeb_vs_not}, and the following exercise develops a precise version of this correspondence.  

\begin{exercise}\label{ex_induced_foliations}
Show that, from the structure of the associated chain $\cC$, one can completely determine the induced foliations $\cF^u_A$ and $\cF^s_A$ up to homeomorphism. 

Concretely, 
between two consecutive closed leaves $c,c'$ of the \emph{same} foliation, say $\cF_A^s$, the foliation $\cF^s_A$ either has a Reeb component annulus or does not. The Reeb annulus case happens if and only if there exists an \emph{odd} number of closed leaves of $\cF^u_A$ between $c$ and $c'$. Equivalently, calling $\cC = \{L_1, \dots, L_n\}$ the chain associated to $A$, then if $c$ corresponds to the shared side between $L_{i-1}$ and $L_{i}$ and $c'$ corresponds to the shared side between $L_{i+k-1}$ and $L_{i+k}$, then $\{L_i, \dots, L_{i+k-1}\}$ is a line of lozenges (intersecting a common stable leaf) and $\cF^s_A$ has a Reeb component if and only if $k$ is even. See Figure \ref{fig_Reeb_vs_not}.  
\end{exercise} 
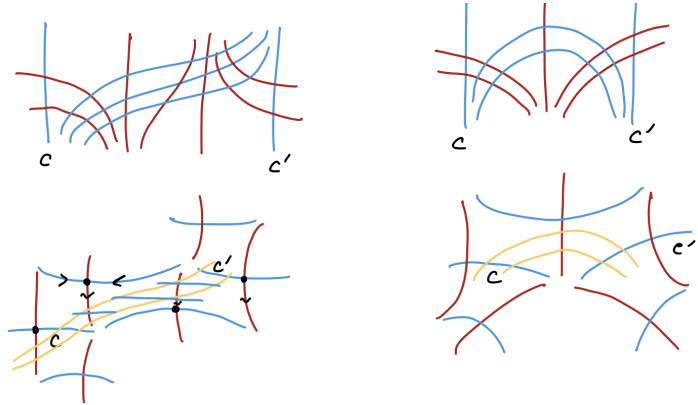
\begin{figure}[h]
\labellist
\small \hair 2pt
\pinlabel {$c$} at 63.107 306.758
\pinlabel {$c'$} at 274.602 306.758
\pinlabel {$c$} at 565.406 306.758
\pinlabel {$c'$} at 706.971 306.758
\pinlabel {$c$} at 86.986 125.112
\pinlabel {$c'$} at 229.403 202.717
\pinlabel {$c$} at 580 148.138
\pinlabel {$c'$} at 738 178
\endlabellist
\includegraphics[width=10cm]{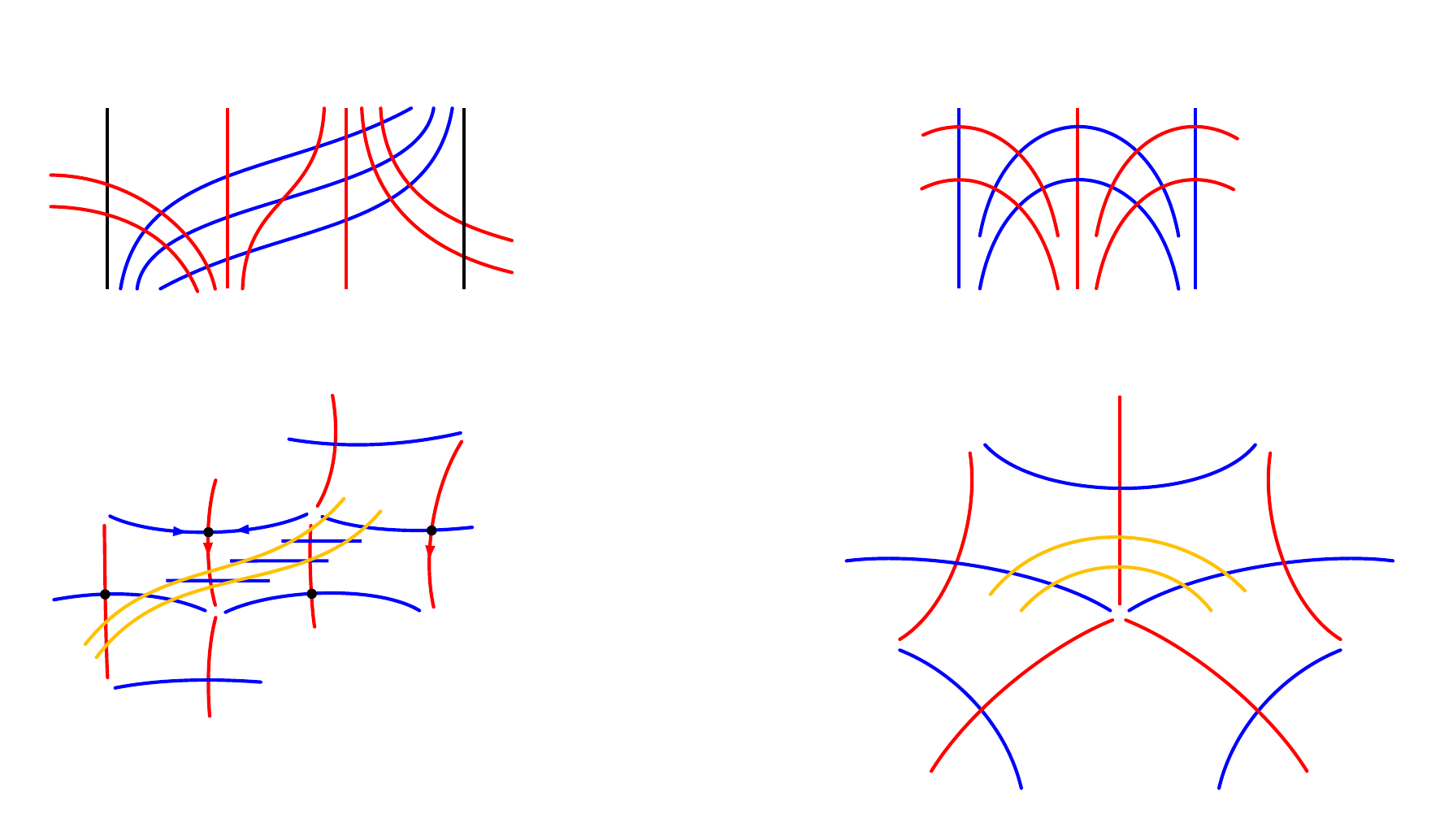}
\caption{The local picture of the induced foliation on a Birkhoff annulus between two closed stable leaves and the associated part of its trace.}\label{fig_Reeb_vs_not}
\end{figure}

\subsection{From chains of lozenges to Birkhoff annuli} 

Our next goal is to prove a converse to Theorem \ref{thm_birkhoff_to_chain}.  We show that any chain of lozenges with corners fixed by some nontrivial $g \in \pi_1(M)$, and satisfying the property that consecutive lozenges share a side and no three share a corner, does in fact correspond to the trace of a Birkhoff annulus with core represented by $g$.  This is Theorem \ref{thm_chain_to_anulus} below. 

We begin with the case of a single lozenge.  

\begin{lemma}  \label{lem_lozenge_to_annulus}
Let $L$ be a lozenge with corners fixed by some nontrivial $g \in \pi_1(M)$.  Then there is a Birkhoff annulus in $M$ whose trace is $L$. \footnote{Technically, the trace is the union of the open lozenge $L$ and its corners.  Going forward, we will often be slightly sloppy this way and simply say ``trace $L$".}
\end{lemma} 

\begin{proof} Denote the corners of $L$ in $\orb$ by $a$ and $b$ and 
let $\wt \alpha$ and $\wt \beta$ be the orbits of $\hflow$ in $\wt M$ projecting to $a$ and $b$ respectively. Consider a path $c$ in $L$ between $a$ and $b$, transverse to both foliations and such that $c$ and $gc$ do not intersect in their interior, i.e, $c \cap gc = \{a,b\}$.  
Lift it to a path $\wt c$ in $\wt M$ between points $x\in \wt \alpha$ and $y\in \wt \beta$.  Since $\wt M \to \orb$ is a trivial fibration, we can choose $\wt c$ to be transverse to $\hflow$.  

Similarly, we may lift the domain $D$ bounded by $c \cup gc$ to a domain $\wt D$ transverse to the flow on its interior, and bounded by the union of $\wt c$, $g\wt c$, the segments of $\wt \alpha$ between $x$ and $g(x)$, and the segment of $\wt \beta$ between $y$ and $g(y)$. 

Call $\wt A$ the union of $\wt D$ with all its translates by powers of $g$.  Topologically $\wt A$ is a bi-infinite strip, bounded by $\wt \alpha$ and $\wt\beta$, and transverse to the flow in its interior. It is also $g$-invariant by construction.   Thus, the projection of $\wt A$ to $M$ is a Birkhoff annulus bounded by the projections $\alpha$ and $\beta$ of $\wt \alpha$ and $\wt \beta$.  It is of course possible that $\alpha = \beta$.   The trace of this annulus is, by construction, $L\cup \{a,b\}$. 
\end{proof} 

Previously, we saw that the induced foliations on a Birkhoff annulus can be read off of the orbit space.  By paying attention to curves $c$ and $g(c)$ bounding a (local) fundamental domain  as in the proof above, one can also read the transverse orientation of the flow from the orbit space: 

\begin{lemma}  \label{lem_direction_of_twisting} 
Let $A_1, A_2$ be two Birkhoff annuli sharing a common boundary component $\alpha$.  Lift $A_1 \cup A_2$ to a strip containing a lift $\wt \alpha$ of $\alpha$, and let $\hat A_i$ denote the projection of this lift of $A_i$ to $\orb$; which share a corner $a$.  The transverse orientation of $\flow$ agrees on $A_1$ and $A_2$ if and only if $\hat{A}_1$ and $\hat {A}_2$ have an even number of quadrants between them. 
\end{lemma}

The proof relies on the fact that the direction a Birkhoff annulus ``twists" with respect to the flow depends on the dynamics on the orbit space, as shown in Figure \ref{fig_direction_of_twisting}.  When $A_1$ and $A_2$ are separated by an odd number of quadrants, a strip representing a fundamental domain for the annulus $A_1 \cup A_2$ will make a full twist between the two boundary components of $A_1$ and $A_2$. When $A_1$ and $A_2$ are separated by an even number of quadrants, the annulus will make a half-twist across $A_1$ and then twist back, keeping the same transverse orientation, across $A_2$.  

As with many results of this nature, the proof is obvious from the picture, but notationally heavy to write.  We refer the reader to Figure \ref{fig_direction_of_twisting} for the visual proof, and leave the formalization as an exercise.   See also \cite[Lemma 3]{Mar24}.
\begin{figure}[h]

\labellist
\small \hair 2pt
\pinlabel {$y$} at 31.050 261.524
\pinlabel {$gy$} at 29.562 202.717
\pinlabel {$gx$} at 176.951 267.479
\pinlabel {$x$} at 173.974 201.228
\pinlabel {$z$} at 286.377 262.268
\pinlabel {$gz$} at 290.099 201.228
\pinlabel {$g\widetilde c$} at 93.579 257.057
\pinlabel {$\widetilde c$} at 128.566 200.483
\pinlabel {$\widetilde c'$} at 214.915 201.228
\pinlabel {$y$} at 434.511 260.779
\pinlabel {$gy$} at 430.045 199.739
\pinlabel {$gx$} at 579.668 266.734
\pinlabel {$x$} at 578.179 199.739
\pinlabel {$z$} at 690.583 264.501
\pinlabel {$gz$} at 694.304 201.228
\pinlabel {$gc$} at 104.001 99.990
\pinlabel {$c$} at 115 74
\pinlabel {$c'$} at 188 76
\pinlabel {$gc'$} at 240 115
\pinlabel {$gc$} at 483 110
\pinlabel {$c$} at 484.386 70
\pinlabel {$gc'$} at 645 73.937
\pinlabel {$c'$} at 645 114.878
\endlabellist
\includegraphics[width=12cm]{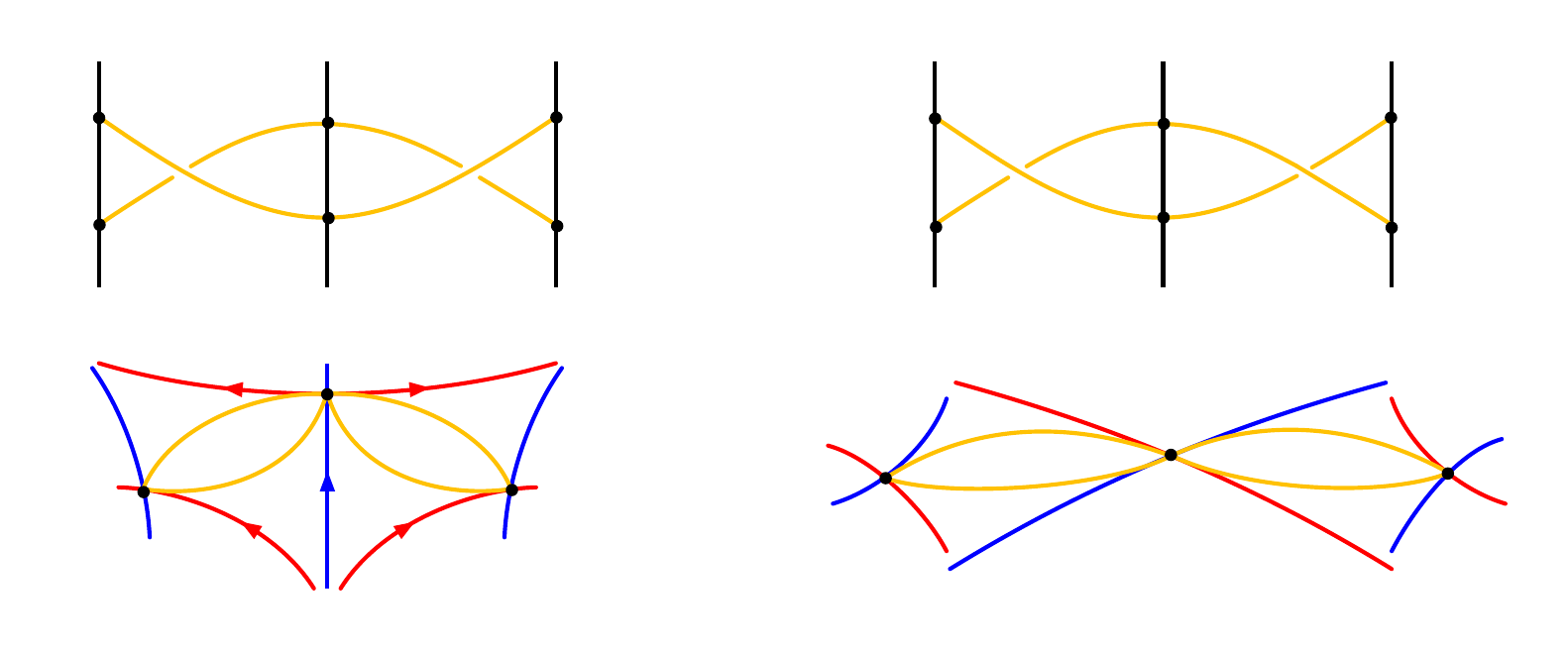}
\caption{The dynamics of $g$ on the orbit space and the quadrants of $A_1 \cup A_2$ determine the transverse orientation of the flow.}
\label{fig_direction_of_twisting}
\end{figure}
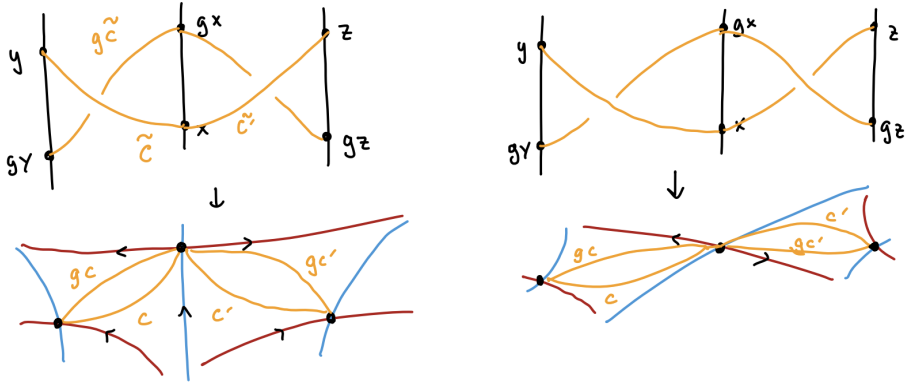

Birkhoff annuli whose trace is a single lozenge have a special name. 
 \begin{definition}[Barbot \cite{Bar95b}] \label{def_elementary_birkhoff}
 A Birkhoff annulus $A$ is called \emph{elementary} if the induced foliations do not have any closed leaves in $\mathring A$; equivalently, if its trace is a single lozenge.  
 \end{definition}

Given a finite chain of lozenges $\cC$ with corners fixed by some nontrivial $g \in \pi_1(M)$, one can apply the procedure of Lemma \ref{lem_lozenge_to_annulus} to each lozenge to produce a union of elementary Birkhoff annuli, glued along their boundaries, projecting to $\cC$.   We now show that if the chain 
is of the form specified in Theorem \ref{thm_birkhoff_to_chain}, this union can be perturbed to a single Birkhoff annulus with the interior of $\cC$ as its trace.  To motivate this, we begin with the inverse procedure.  

\begin{proposition}[Homotoping to elementary annuli, \cite{Bar95b} Corollaire 5.6]\label{prop_Birkhoff_to_union_elementary}
Any Birkhoff annulus is homotopic along flow lines to an arbitrarily small perturbation of a union of elementary Birkhoff annuli. 
\end{proposition}

\begin{proof}
Let $A$ be a Birkhoff annulus which lifts to $\wt A$, and let $\cC=\{L_1,\dots, L_n\}$ be the associated chain in $\orb$. Call $x_i$ the shared corner of $L_{i}$ and $L_{i+1}$ and $\alpha_i$ the periodic orbit in $M$ whose lift projects to $x_i$. As we showed in the previous section, each orbit $\alpha_i$ lies on a $\cF^s$ or $\cF^u$ leaf of a closed leaf $c_i$ of one of the induced foliations on $A$. 
As before, we let $A_i$ denote the subannulus of $A$ bounded by $c_{i-1}$ and $c_i$.

By pushing $A$ along the orbits of the flow, we may assume that each $c_i$ lies in an arbitrarily small neighborhood of the closed orbit $\alpha_i$.  Having done this, we can make a deformation of $A_{i}$ and $A_{i+1}$ in that small neighborhood, as shown in Figure \ref{fig_mod_Birkhoff_annulus} to obtain two new subannuli $A'_i$ and $A'_{i+1}$ which are transverse to the flow, agree with (the pushed versions of) $A_i$ and $A_{i+1}$ outside of the neighborhood, and which have $\alpha_i$ as their common boundary\footnote{Barbot \cite{Bar95b} called this modification ``sliding'' (or ``coulissage'' in the original French), since it amounts to sliding the annulus $A$ along the leaf of $c_i$ until you reach $\alpha_i$.}. 
Doing this local procedure simultaneously for each of the $\alpha_i$ gives us a union of elementary Birkhoff annuli. 
\end{proof}
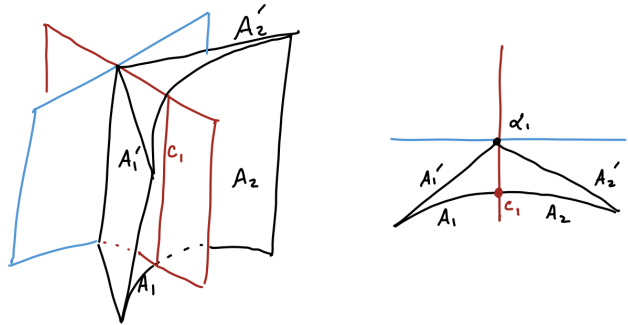
\begin{figure}[h]
\labellist
\small \hair 2pt
\pinlabel {$A_1'$} at 116.754 179.732
\pinlabel {$A_1$} at 178.656 90
\pinlabel {$A_2$} at 243.694 189.918
\pinlabel {$A_2'$} at 242.910 323
\pinlabel {$c_1$} at 210.783 218.127
\pinlabel {$\alpha_1$} at 647.238 256.523
\pinlabel {$A_1'$} at 550.074 191
\pinlabel {$A_1$} at 568.880 138.202
\pinlabel {$c_1$} at 645 152
\pinlabel {$A_2$} at 684.850 131.933
\pinlabel {$A_2'$} at 696.603 192
\endlabellist
\includegraphics[width=9cm]{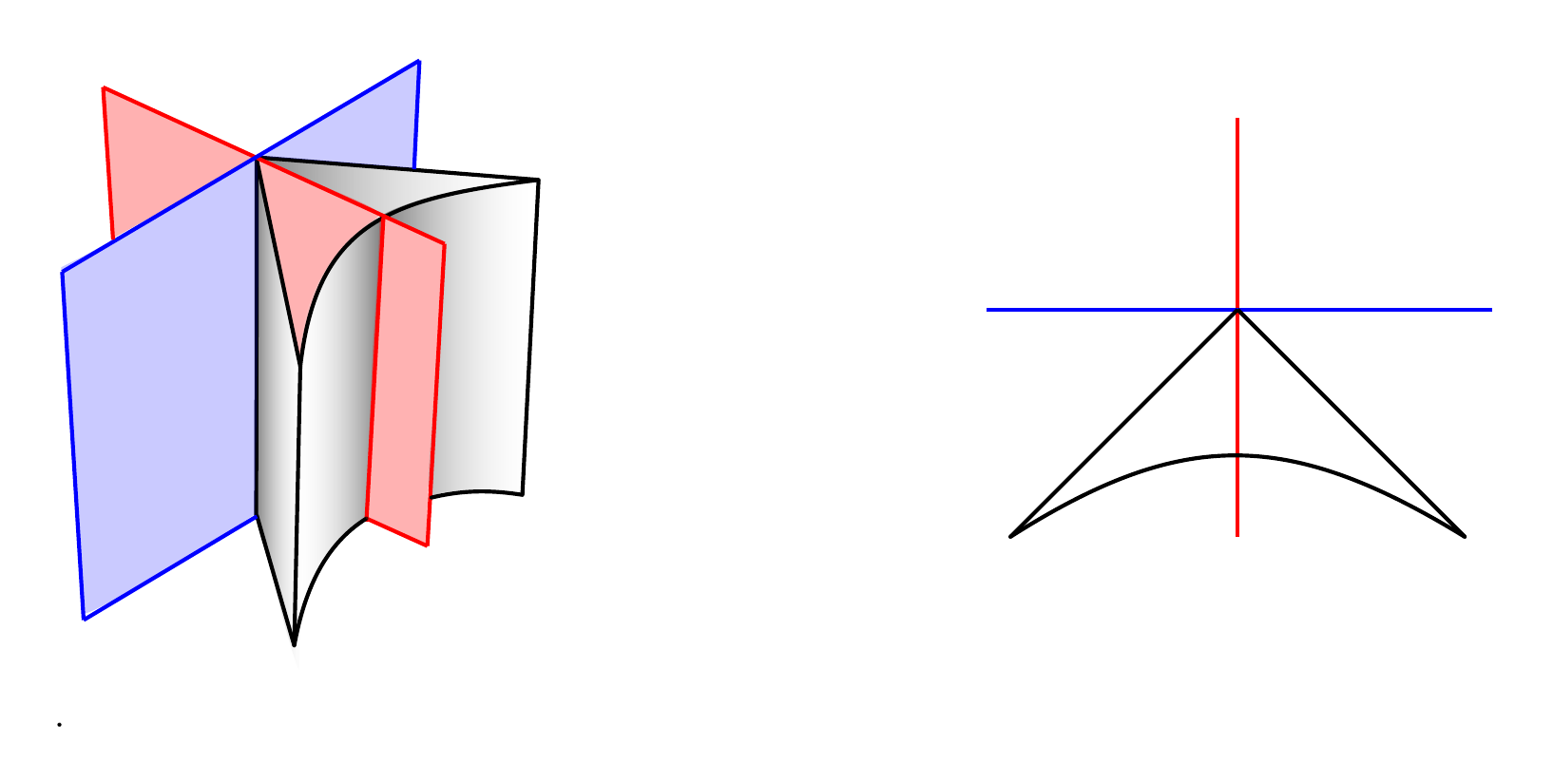}
\caption{Modifying a Birkhoff annulus near a periodic orbit}
\label{fig_mod_Birkhoff_annulus}
\end{figure}

The next proposition reverses the strategy above, creating a Birkhoff annulus from two elementary Birkhoff annuli, provided that their traces lie in adjacent quadrants.

\begin{proposition}[Pushing off periodic orbits]\label{prop_union_elementary_to_Birkhoff}
Let $A_1'$, $A_2'$ be two elementary Birkhoff annuli with a common boundary orbit $\alpha$, and let $\wt \alpha$ be a lift of $\alpha$.  
Suppose the lifts of $A_i'$ along $\wt \alpha$ project to lozenges $L_1$ and $L_2$ which share a side. Then 
there is an arbitrarily small deformation of $A'_1\cup A'_2$ to a Birkhoff annulus which admits a lift whose associated chain is $\{L_1, L_2\}$.  
\end{proposition}

\begin{proof}
The proof consists essentially in reversing the procedure from that of Proposition \ref{prop_Birkhoff_to_union_elementary}.  
A local picture of the lifts of $A'_1 \cup A'_2$ is given by Figure \ref{fig_mod_Birkhoff_annulus}. The assumption that $L_1$ and $L_2$ share a side implies that their lifts are in adjacent quadrants of $\wt \alpha$.  
Thus, one can push a small neighborhood of $\alpha$ in $A'_1 \cup A'_2$ away from $\alpha$, in the direction transverse to $\cF^s$ if their traces share an unstable side (as in the figure) or transverse to $\cF^u$ if they share a stable side.  The resulting annulus will be transverse to $\flow$, and hence give the desired Birkhoff annulus $A$. 
\end{proof}

Since this procedure was purely local, we easily obtain the following:  
\begin{theorem}  \label{thm_chain_to_anulus}
Let $\cC$ be a finite chain of lozenges so that adjacent lozenges share sides, and no three share a corner, which is invariant by some nontrivial $g \in \pi_1(M)$.  Then there is a Birkhoff annulus with trace the interior of $\cC$.  
\end{theorem} 

\begin{proof} 
Up to replacing $g$ with a power, we can assume that each corner of $\cC$ is fixed by $g$.  Now apply the push-off procedure one by one to each periodic orbit represented by a non-extremal corner of $\cC$.   
\end{proof}

\subsection{Weakly embedded versus immersed Birkhoff annuli}

Birkhoff annuli are not required to be embedded, but simply $\pi_1$-injectively immersed. In this section we discuss results on which self-intersections of a given Birkhoff annulus (or union of annuli) are essential, and which can be removed by an isotopy.  This can be seen from the structure of their associated chains.  We begin with some definitions. 

\begin{definition}
A Birkhoff annulus $A$ is \emph{weakly embedded} if its interior is embedded in $M$.
\end{definition}

One obvious reason that a Birkhoff annulus may fail to be weakly embedded is when the boundary orbits also intersect the interior.  Stated in terms of the orbit space, this corresponds to a corner of a lozenge in the associated chain being taken to the interior of the chain by some element of $\pi_1(M)$.  For convenience we make the following definition.  

\begin{definition}\label{def_simple_chain}
A chain of lozenges $\cC$ is called \emph{simple} if for every corner $c$ in $\cC$ and every $g\in \pi_1(M)$, $gc$ is \emph{not} in the interior of $\cC$.  If such a chain $\cC$ is a single lozenge, we say it is a {\em simple lozenge}.  
\end{definition}

 In  Theorem \ref{thm_simple_chain_gives_weakly_embedded}, we will see that failure of simplicity is the only irresolvable obstruction to replacing a Birkhoff annulus with a weakly embedded one with the same trace.  Conversely, Theorem \ref{thm_weakly_embedded_Birkhoff_to_simple} says that weakly embedded Birkhoff annuli always have simple trace.  For this, we need the following
 general lemma on simple chains.  
 
\begin{lemma} \label{lem_simple_chains}
Suppose $\cC =\{L_1,\dots, L_n\}$ is a chain of lozenges such that, for all $i$, $L_i$ and $L_{i+1}$ share a side, and $L_i$, $L_{i+1}$ and $L_{i+2}$ never share a corner.  Let $a, b$ be the corners (of $L_1$ and $L_n$) not shared by two lozenges.  
Then $\cC$ is simple if and only if $g(\{a, b\}) \cap (\mathring L_1 \cup \mathring L_n) = \emptyset$ for every $g\in \pi_1(M)$.  
\end{lemma}
Note that, if $\cC$ is a single lozenge, the statement is just the definition of simple.  

\begin{proof}
One direction is immediate: if $ga$ or $gb$ are in the interior of $L_1$ or $L_n$ for some $g\in G$, then $\cC$ is not simple by definition.
To prove the converse, suppose that $\cC$ is not simple. 
Call $x_i$ the common corner of $L_i$ and $L_{i+1}$; with $x_0 = a$ and $x_{n}=b$.  Since $\cC$ is not simple, there exists $x_i$ and some $g$ such that $g(x_i)$ is in the interior of $\cC$; which means $g(L_i) \cap L_j \neq \emptyset$ for some $i, j$.  By Lemma \ref{lem_markovian_or_corner}, we have also that one of the corners of $L_j$ lies in $g(L_i)$.  Let $c$ denote this corner.  

\begin{figure}[h]
\labellist
\small \hair 2pt
\pinlabel {$x_i$} at 112.610 132.562
\pinlabel {$c$} at 149.871 164.026
\pinlabel {$gL_i$} at 215.836 168.718
\pinlabel {$L_j$} at 139.383 106.065
\endlabellist
\includegraphics[width=5cm]{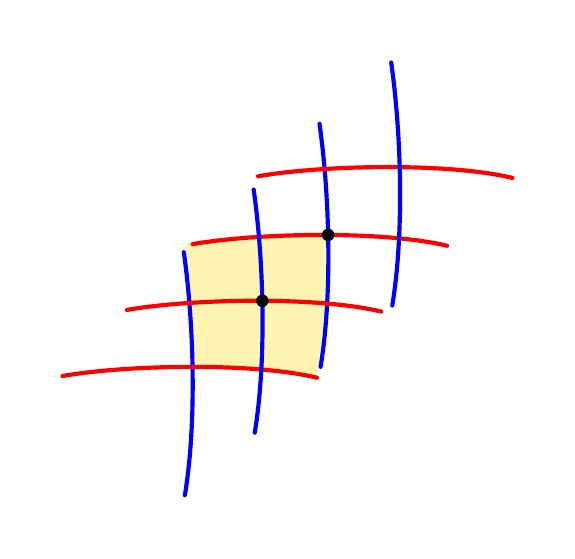}
\caption{If $\cC$ is not simple, lozenges that intersect cannot have members of the chain in adjacent quadrants.}
\label{fig_not_simple}
\end{figure}
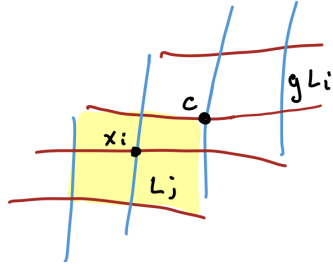

By the non-corner criterion (Lemma \ref{lem_no_corner_criterion}) applied to the perfect fits from the lozenge $g(L_i)$, no lozenge can share a side with $L_j$ at the corner $c$.  We conclude that $j \in \{1, n\}$ and $c \in \{a, b\}$.  Similarly, no lozenge can share a side with $L_i$ at the corner $x_i$, so we also have $i \in \{1, n\}$ and $x_i \in \{a, b\}$.  
\end{proof} 

Using this, we show that weakly embedded Birkhoff annuli correspond to simple chains. 

\begin{theorem}\label{thm_weakly_embedded_Birkhoff_to_simple}
Let $A$ be a Birkhoff annulus, and $\cC$ the chain of lozenges associated to its trace. If $A$ is weakly embedded, then $\cC$ is simple. 
\end{theorem}

\begin{proof}[Proof of Theorem \ref{thm_weakly_embedded_Birkhoff_to_simple}]
Let $A$ be a Birkhoff annulus, with boundary orbits $\alpha$ and $\beta$. Let $\cC$ be its associated chain of lozenges and let $a$ and $b$ denote the two extremal corners of $\cC$; these are the projections of the lifts of $\alpha$ and $\beta$. 
By Theorem \ref{thm_birkhoff_to_chain}, $\cC$ satisfies the first condition in Lemma \ref{lem_simple_chains}.  Thus, if $\cC$ is not simple, Lemma \ref{lem_simple_chains} states that either $ga$ or $gb$ intersects the interior of $\cC$.  
Translating this back to the 3-manifold, this means that $\alpha$ or $\beta$ intersects the interior of $A$, and so $A$ is not weakly embedded. 
\end{proof}

The converse statement is also true: 

\begin{theorem}\label{thm_simple_chain_gives_weakly_embedded}
Let  $\cC=\{L_1,\dots L_n\}$ be a simple chain of lozenges with corners fixed by a nontrivial $g \in \pi_1(M)$, satisfying the property that consecutive lozenges share a side, no three successive lozenges share a corner, and no lozenge $L_i$ of $\cC$ is the image of another lozenge $L_j$ under some element of $\pi_1(M)$.  Then
there exists a weakly embedded Birkhoff annulus $A$ such that its trace is the interior of $\cC$.   
\end{theorem}

This theorem is proved in Barbot \cite{Bar95b}, first for the case where $L$ is a simple lozenge and then for the case where $\cC$ is invariant under a $\mathbb{Z}^2$ subgroup (so corresponds to a Birkhoff annulus which is topologically a torus).  However, the $\bZ^2$ invariance does not play an important role in the proof, and the techniques of proof give this more general statement.  Barbot's proof uses two different tools to remove self-intersection points of surfaces, one called {\em Fried desingularization}, and the second involving pushing loci of double points along orbits. %

We note also that removing the assumption that $\cC$ is simple in Theorem \ref{thm_simple_chain_gives_weakly_embedded} but applying the same proof gives the existence of an {\em immersed} Birkhoff annulus with trace $\cC$.  See also \cite[Theorem 24]{ABM24} for a stronger version of this result in the case where $\cC$ is a single lozenge.

\begin{proof}[Proof of Theorem \ref{thm_simple_chain_gives_weakly_embedded}]
First consider the special case where $\cC = \{L\}$ is a single lozenge.   
 By Theorem \ref{thm_chain_to_anulus}, there is a Birkhoff annulus $A$ in $M$ whose trace is $L$.  However, $A$ may not necessarily be weakly embedded in $M$. 
To fix this, one can modify $A$ by a process known as  \emph{Fried desingularization}.  This process was described by Fried on p.301 of \cite{Fri83}; a precise account can be found in \cite{ABM24}.  Since this is purely topological, we give only the idea of the proof and refer the reader to these references for details.  

 The idea of Fried desingularization is as follows: Since $\cC$ is assumed to be simple, the boundaries  $\alpha$ and $\beta$ of $A$ do not intersect the interior of $A$.  If $A$ has any self-intersections, these can be put in general position by doing a small perturbation of the interior of $A$, thus keeping $A$ transverse to the flow, and keeping the boundary disjoint from the interior.  After this, there are only two types of self-intersections which remain:  triple points (of which there can only be finitely many) and double points (which lie along finitely many curves).   The transverse orientation to $A$ given by the flow allows one to cut and paste $A$ along these intersections to obtain a non self-intersecting surface $B$, which remains transverse to the flow.  This cut and paste procedure is illustrated in Figure \ref{fig_Fried_cutandpaste}.
\begin{figure}[h]
\includegraphics[width=9cm]{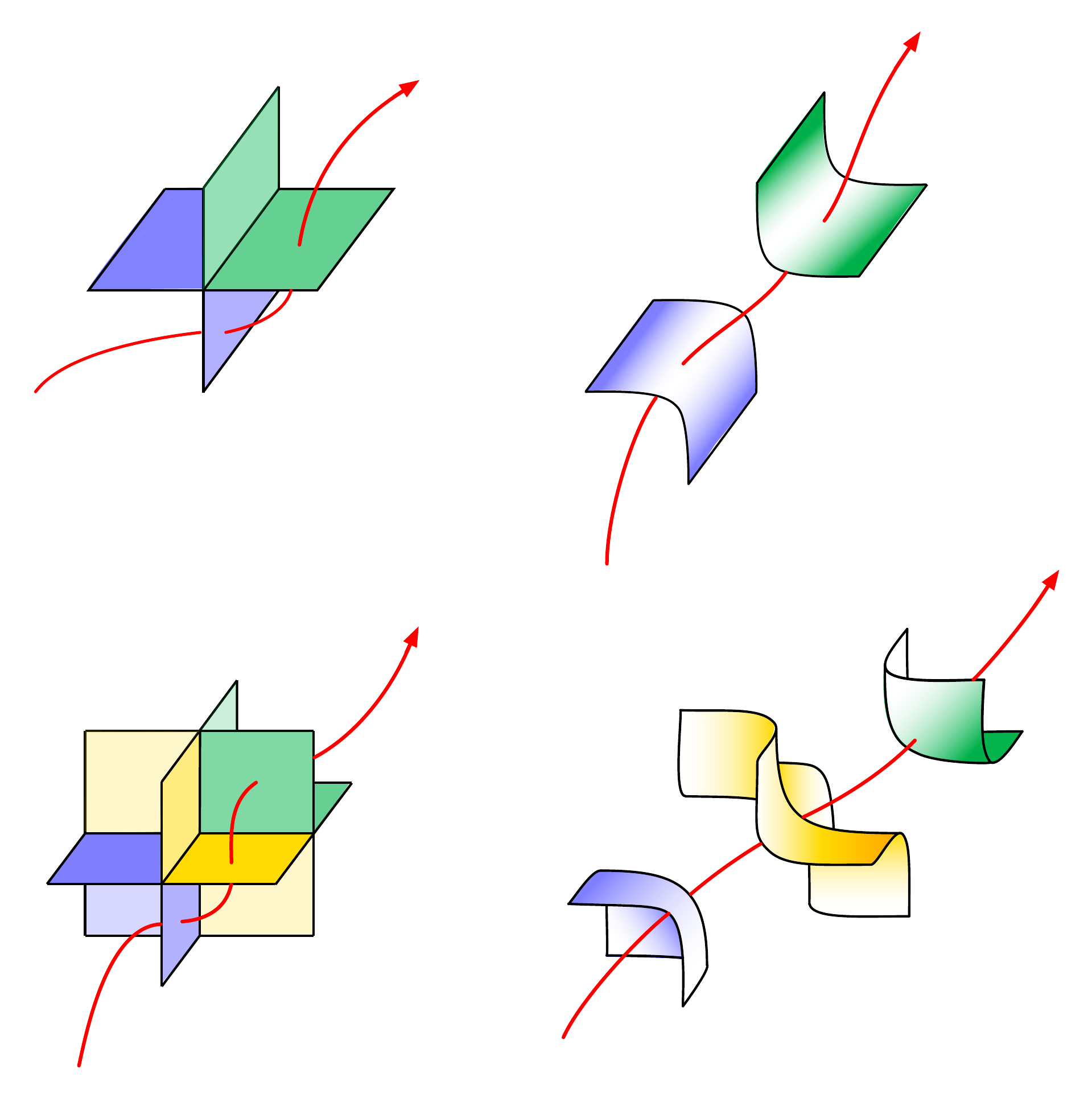}
\caption{The cut and paste procedure near curves of double intersections and points of triple intersection. The arrows show the orientation of the flowlines.}
\label{fig_Fried_cutandpaste}
\end{figure}

Applying the cut and paste procedure may result in a disconnected surface $B$. However, since all the modifications were done in the interior of $A$, the boundary of $B$ is still equal to $\alpha\cup\beta$, and by transversality of the flow, the stable and unstable foliations induce two transverse, non-singular, foliations on $B$.  Thus, every connected component of $B$ has Euler characteristic $0$, and $B$ has two boundary components.  In particular, there is a unique connected component which is an annulus; hence a weakly embedded Birkhoff annulus. Abusing notation slightly, we denote this component by $B$.  

What is left to show is that, perhaps after a further modification, the trace of $B$ in $\orb$ still corresponds to $L$.   Since no modifications were done to $B$ near the boundary, we can take a lift $\wt B$ of $B$ to $\wt M$ which is bounded by $\wt \alpha$ and $\wt\beta$ and with trace containing an open neighborhood of the corner $a$ in $L$ (and thus contains $L$).   If the trace of $B$ is a chain $\cC'$ strictly bigger than $L$, then we can apply Proposition \ref{prop_Birkhoff_to_union_elementary}, to turn $B$ into a union of elementary Birkhoff annuli or Klein bottles. The first of them will be a weakly embedded Birkhoff annulus with trace $L$, as claimed. 

For the general case where $\cC=\{L_1,\dots L_n\}$ contains more than one lozenge, we first apply our work in the case of a single lozenge to find weakly embedded Birkhoff annuli $A_1, \ldots A_n$ such that the trace of $A_i$ is $L_i$.  We can also arrange that adjacent $A_i$ agree along their boundary.  Thus, Proposition \ref{prop_union_elementary_to_Birkhoff} implies that the union of the $A_i$ has an (arbitrarily small) deformation to a Birkhoff annulus $A'$.  However, this union may fail to be weakly embedded if $h(L_i) \cap L_j \neq \emptyset$ for some $L_i \neq L_j$ in $\cC$ and $h \in \pi_1(M)$.  To treat this case, we analyze the points of self-intersection of $A'$, and show that these can be removed by an isotopy. 
As with the procedure for Fried desingularization, we give here the essential steps and refer the reader to \cite[Section 7]{Bar95b} for the details, which are again purely topological.  

First, as in the lozenge case, after a small deformation we may assume that $A'$ is in general position, so self-intersections lie along loci of double or triple points, which form a union of embedded circles in $A'$, disjoint from the boundary.  (Note we are using the notation $A'$ for the annulus itself, not for its immersed image in $M$.)   Each such circle is either a connected component of the set of double points, or formed by a union of open intervals of double points bounded on either side by triple points.   We next show no such circle is essential in $A'$:  for contradiction, assume some circle $c$ is essential.  Then it represents the generator of $\pi_1(A')$ which, under the induced map to $\pi_1(M)$, corresponds to an element $g$ fixing all corners of the lozenge.  Since $c$ is a locus of multiplicity, there is another circle $c'$ with the same image in $M$ as $c$, hence also representing $g$.   We may choose lifts $\wt{c}, \wt{c'}$ to curves in $\wt M$ which are invariant by $g$ and project each to a path between the ideal corners of some lozenges $L_i, L_j$ of $\cC$.  Since these have the same image in $M$, it follows that $L_i = h L_j$ for some $i \neq j$ and $h \in \pi_1(M)$, which was excluded by assumption.  

Thus, the singular loci are circles bounding discs in $A'$, and these discs are glued in pairs along their boundaries by the immersion of $A'$ into $M$.  Since $M$ is aspherical, each of these paired discs bounds a ball.  Now a standard ``innermost discs" argument allows one to iteratively push one disc past another to remove circles of self-intersection, starting with an innermost circle on $A'$.   See Figure \ref{fig_innermost_discs}. Moreover, since the image of $A'$ is 
transverse to the flow away from the boundary, this ``push" can be done by an isotopy along flowlines.  
\end{proof}

\begin{figure}[h]
\includegraphics[width=10cm]{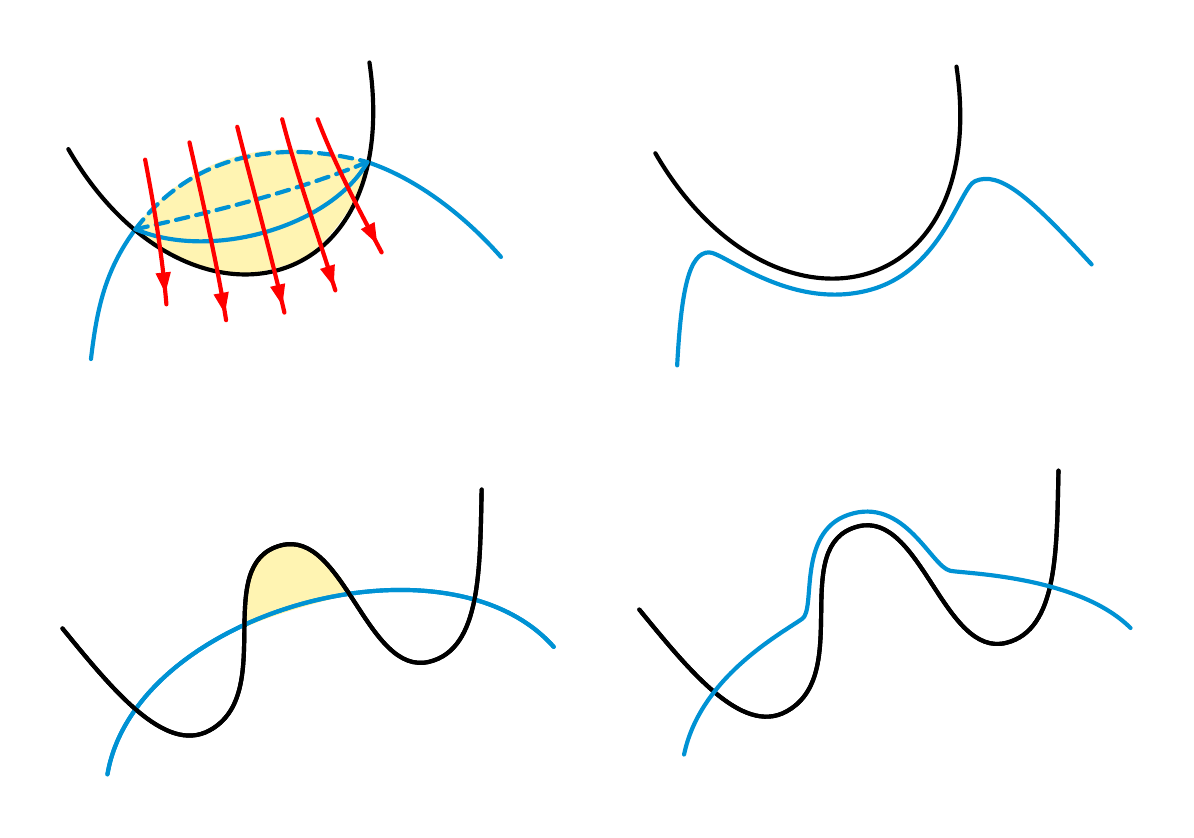}
\caption{When annuli intersect along a circle, one can push along flowlines across a ball to remove the intersection.  In the case of multiple intersections, one iterates this process starting with an innermost disc component of the complement of the circles of multiplicity in $A$ as shown in the lower half of the figure.}
\label{fig_innermost_discs}
\end{figure}

\section{Quasi-transverse surfaces and the JSJ decomposition}\label{sec_quasi_transverse_surface}
As described in the introduction to this chapter, Barbot and Fenley showed that for any 3-manifold $M$ with a pseudo-Anosov flow $\flow$, the JSJ surfaces of $M$ can be realized in a good position with respect to the orbits of $\flow$, called ``quasi-transverse".  In the other direction, Paulet showed that flows can in fact be composed, rather than {\em de}composed by gluing partial flows along such quasi-transverse surfaces. 
 Here we describe the theory of such surfaces, and make precise the definition of  ``quasi-transverse". 
 
Throughout this section $M$ is a closed three manifold and $\flow$ a pseudo-Anosov flow.  

\begin{definition}\label{def_quasi_transverse}
A  surface $S$
 in $M$ is said to be \emph{quasi-transverse}\footnote{Slight variations on this definition have appeared in the literature, sometimes assuming additional properties as in Definition \ref{def_quasi_transverse_subclasses} below.  We have chosen to use the terminology ``quasi-transverse" for this general class of surfaces because it seems natural: these are precisely the surfaces which are transverse except along included periodic orbits.  Definition \ref{def_quasi_transverse_subclasses}  disambiguates other possibly desirable additional properties.} for $\flow$ if it is immersed, incompressible and:
\begin{enumerate}[label=(\roman*)]
\item There are finitely many (possibly zero) periodic orbits $\alpha_1, \dots, \alpha_n$ of $\flow$ contained in $S$;
\item The surface $S \smallsetminus \{\alpha_1,\dots,\alpha_n\}$ is transverse to $\flow$.
\item If $S$ has boundary, then its boundary is a union of periodic orbits.  %
\end{enumerate}
\end{definition}

Generalizing the description for Birkhoff annuli, we make the following definition. 
\begin{definition} 
A quasi-transverse surface $S$ is {\em weakly embedded} if $S \smallsetminus \{\alpha_1,\dots,\alpha_n\}$ is embedded. 
\end{definition}

\begin{remark}[Birkhoff sections] \label{rem_birkhoff_section}
One class of examples of weakly embedded quasi-transverse surfaces is {\em Birkhoff sections}.  A Birkhoff section for $\flow$ is an immersed surface (with boundary) in $M$ that is transverse to $\flow$ and embedded on its interior, tangent to the flow on its boundary, and has the property that every orbit meets the surface in uniformly bounded time, both in the past and the future.   
These play an important role especially in the theory of  transitive (pseudo)-Anosov flows, since they allow us to see the flow as a suspension outside of the finitely many boundary orbits.
However, the interplay between these objects and the orbit space is not well understood, so we will not discuss them further here.  
\end{remark} 

There are many known examples of higher genus quasi-transverse surfaces, such as the Birkhoff sections mentioned above.  Since our focus in this chapter is on JSJ decompositions, we will consider only quasi-transverse annuli, tori and Klein bottles.

When a quasi-transverse surface is formed by a union of Birkhoff annuli, there may be different ways to put it (up to homotopy or isotopy) into a quasi-transverse position.   Using Propositions \ref{prop_Birkhoff_to_union_elementary} and \ref{prop_union_elementary_to_Birkhoff}, one may push the surface away from periodic orbits (when possible), or conversely push it {\em on} to periodic orbits.   To describe the extreme cases of this, and disambiguate previous definitions of quasi-transverse in the literature, we make the following definition: 

\begin{definition}\label{def_quasi_transverse_subclasses}
Let $T$ be a quasi-transverse torus or Klein bottle, and $\alpha_1, \dots, \alpha_n$\footnote{Here we allow also for the possibility that $n = 0$, i.e., that this collection is empty.} be the periodic orbits of $\flow$ on $T$.  Let $A_i$ denote the open annuli bounded by $\alpha_{i-1}$ and $\alpha_i$. We say $T$ is: 
\begin{enumerate}
\item \emph{Alternating} if $n$ is even, and for all $i$ the transverse orientation of $\flow$ switches between $A_i$ and $A_{i+1}$.
\item \emph{Maximally transverse}  if $n=0$ or if $A_i$ and $A_{i+1}$ are in non-adjacent quadrants of $\alpha_i$, for all $i$. 
\item \emph{Maximally periodic} if the only closed leaves of the induced foliations on $T$ are $\alpha_1, \dots, \alpha_n$.
\end{enumerate}
\end{definition}

Barbot's original definition of quasi-transverse in \cite{Bar95b} corresponds to those quasi-transverse surfaces that are at the same time maximally periodic and alternating.  Paulet \cite{Pau24} takes alternating as a requirement in the definition of quasi-transverse. As we will see below, any alternating quasi-transverse surface is automatically maximally transverse, and the converse also holds \emph{except} sometimes when some of the $\alpha_i$ are singular orbits.   Maximally periodic quasi-transverse surfaces are what Barbot and Fenley call ``Birkhoff surfaces" in \cite{BF13}, and this is the notion of JSJ tori in ``good position" that is used in that paper.  We have avoided the terminology {\em Birkhoff surface} since it is easily confused with the quite different notion of {\em Birkhoff section} described in Remark \ref{rem_birkhoff_section}.   

The three cases of Definition \ref{def_quasi_transverse_subclasses} 
can be rephrased in terms of their projection to the orbit space, according to the following description.  This description also shows that maximally transverse surfaces do indeed contain the fewest possible periodic orbits among such surfaces associated to a fixed chain of lozenges (so are as transverse as possible), and maximally periodic contain the most periodic orbits.  

\begin{proposition}[Trace of a quasi-transverse torus] \label{prop_trace_QT_tori}
Let $T$ be a quasi-transverse torus or Klein bottle, $\wt T$ a lift of $T$ to $\wt M$, and $\hat T$ its projection to 
$\orb$.  Then, the closure of $\hat T$ is a %
chain of lozenges $\cC$ with the property that consecutive lozenges share corners (and possibly sides) and no three share a corner.  Furthermore, for any two ``diagonal" lozenges of $\cC$ sharing a corner but no side, the shared corner is in $\hat{T}$, and: 
\begin{enumerate}[label=(\roman*)]
\item\label{item_trace_QT_maximally_transverse} $T$ is maximally transverse if and only if the only corners in $\hat T$ are shared by diagonal lozenges. 
\item \label{item_trace_QT_alternating} $T$ is alternating if and only if $T$ is maximally transverse and, for each corner $c$ contained in $\hat T$, there is an \emph{odd} number of quadrants separating the two lozenges sharing $c$. 
\item\label{item_trace_QT_maximally_periodic} $T$ is maximally periodic if and only if each corner of $\cC$ is in $\hat T$.
\end{enumerate}
\end{proposition}

\begin{rem}\label{rem_transverse_torus_no_diagonal}
In particular, a quasi-transverse torus $T$ is actually transverse if and only if $\hat T$ does not contain any corners, which can happen only if there are no diagonal lozenges in $\cC$.
\end{rem}

\begin{proof}
Let $A \subset T$ be a maximal transverse subsurface, so $A = T$ if $T$ is transverse. Otherwise, $A$ is an open annulus bounded by periodic orbits. 
 Let $c_1,\dots, c_n$ be the closed leaves of the induced foliations on $\mathring A$, and further decompose $A$ into subannuli $A_i$ bounded by $c_{i-1}$ and $c_i$ (with $c_0$ and $c_{n+1}$ being boundaries of $A$).  Let $\wt A$ be a lift of $A$ contained in $\wt T$, or (to streamline the proof) in the case where $A =T$ we take $\wt A$ to be a band bounded by two consecutive lifts of $c_n$ in $\wt T$.   Let $\wt A_i$ be the lift of $A_i$ in $\wt A$. Then the projections of $\wt A_i$ to $\orb$ are open lozenges (without their sides or corners), whose sides are the projections of $\wt c_{i-1}$ and $\wt c_i$.   
This is true for each lift of each maximal transverse subsurface of $T$, and we may apply this procedure inductively to 
produce a chain of lozenges $\{L_i\}_{i\in \bZ}$ associated to the projection of $\wt T$.  Between two adjacent maximal transverse subsurfaces, the shared periodic orbit in their boundary projects to a corner of the chain shared by two, but not three, lozenges.   

When two consecutive lozenges $L_{i-1}$ and $L_i$ share a corner but not a side, connectedness of $\hat{T}$ forces this corner to be in the projection.  Item \ref{item_trace_QT_maximally_transverse} now follows from the definition of maximally transverse.  
Similarly, $\hat{T}$ contains the shared side of two lozenges $L_{i-1}$ and $L_i$ if and only if this side corresponds to a closed leaf $c_i$ in $T$.  These closed leaves correspond to periodic orbits precisely when the corner shared by $L_{i-1}$ and $L_i$ lies in $\hat{T}$. This proves item \ref{item_trace_QT_maximally_periodic}. 

Finally, Lemma \ref{lem_direction_of_twisting} states that 
direction of transversality of the flow through $\wt T$ changes only when passing through a periodic orbit $\wt \alpha_i$ such that the two sides of $\wt T\smallsetminus \{\wt \alpha_i\}$ have an odd number of quadrants of $\wt\alpha_i$ between them. Equivalently, this says there are an odd number of quadrants between $L_{i-1}$ and $L_i$, proving item \ref{item_trace_QT_alternating}.   
\end{proof} 

Combined with the description above, Propositions \ref{prop_Birkhoff_to_union_elementary} and \ref{prop_union_elementary_to_Birkhoff} have the following immediate consequence. 
\begin{corollary} \label{cor_max_transverse_or_periodic}
Let $T$ be an embedded, incompressible, torus or Klein bottle that is homotopic to a weakly embedded quasi-transverse surface. Then it is also homotopic to a weakly embedded, maximally transverse quasi-transverse surface and to a weakly embedded, maximally periodic quasi-transverse surface.
\end{corollary}
Thus, ``maximally transverse" and ``maximally periodic" can be thought of as two (typically different) normal forms for a quasi-transverse surface.   We also obtain the following consequence: 

\begin{corollary} \label{thm_weakly_to_simple_general}
If $T$ is a weakly embedded quasi-transverse annulus, torus or Klein bottle, then its associated chain in $\orb$ is simple.
\end{corollary}

\begin{proof} 
By corollary \ref{cor_max_transverse_or_periodic}, $T$ can be made maximally periodic, so there is at least one periodic orbit in the image of $T$.  By cutting $T$ along this orbit, we can view it as a Birkhoff annulus, and apply Theorem \ref{thm_weakly_embedded_Birkhoff_to_simple}.  
\end{proof} 

 \subsection{The JSJ decomposition} 
 The {\em JSJ decomposition} %
is the following classical statement giving a canonical decomposition of any irreducible 3-manifold.  
For a proof which treats also the case of non-orientable 3-manifolds, see \cite{NS97}.  

\begin{theorem}[Geometric JSJ decomposition \cite{JS79, Joh79, NS97} ]\label{thm_geometric_JSJ_decomposition}
Let $M$ be a closed, irreducible and $P^2$-irreducible\footnote{Recall that a $3$-manifold is called $P^2$-irreducible if it is irreducible and does not admit a $2$-sided embedding of $\mathbb{RP}^2$, see, e.g., \cite{Hat23}.} 3-manifold.  There exists a collection $\mathcal J$ of disjoint incompressible tori and Klein bottles in $M$ such that each component of $M \setminus \mathcal J$ is either Seifert fibered or atoroidal.   Moreover, a minimal such collection is unique up to isotopy.
\end{theorem}
Note that, as a consequence of minimality, for a fixed collection $\mathcal J$, any other embedded incompressible torus or Klein bottle in $M$ can be isotoped either to an element of $\cJ$ or to lie in a single Seifert fibered piece. 

We first show that JSJ surfaces can, individually, be made quasi-transverse.  

\begin{theorem}[Barbot--Fenley \cite{BF13}, Theorem 6.10]\label{thm_embedded_tori_to_QT}
Let $T$ be an embedded incompressible torus or Klein bottle in $M$. 
Then $T$ is homotopic to a weakly embedded quasi-transverse surface. 
\end{theorem}

The statement of Theorem 6.10 of \cite{BF13} is not exactly the same as we give here. In 
\cite{BF13}, it is proved in many cases that $T$ is homotopic (in fact isotopic) to an \emph{embedded} quasi-transverse surface, and not just a weakly embedded one. However, there are two cases where one cannot turn $T$ into a truly embedded surface: either when $T$ is inside, or bounds, a Seifert piece which is \emph{periodic} (see Definition \ref{def_periodic_Seifert}), or when $T$ is homotopic to the boundary of a one-sided Klein bottle.  See Remark \ref{rem_1_sided_klein} below.  
When $T$ is homotopic to the boundary of a one-sided Klein bottle \cite[Theorem 6.10]{BF13} shows that the Klein bottle can indeed be made embedded and quasi-transverse, but does not mention the fact that $T$ itself is still homotopic to a weakly embedded quasi-transverse torus.

We present this simple version here since it is all that is needed for the JSJ decomposition.

\begin{rem} \label{rem_1_sided_klein}
Quasi-transverse one-sided Klein bottles do appear in examples of Anosov flows: We saw one instance already in the example of Bonatti-Langevin (see section \ref{sec_example_bonatti_langevin}) where a Klein bottle can be realized as a Birkhoff annulus starting and ending at the special periodic orbit.
Another example appears in the case of a geodesic flow on a non-orientable surface: The fiber over a one-sided geodesic gives a quasi-transverse one-sided Klein bottle. By doing an appropriate surgery on such a geodesic flow, one can even obtain a neighborhood of a quasi-transverse, one-sided Klein bottle as a piece of the JSJ decomposition.
\end{rem}

Given all we have so far, the strategy of the proof of Theorem \ref{thm_embedded_tori_to_QT} is relatively straightforward: $\pi_1(T)$ is (up to passing to an index 2 subgroup) a subgroup of $\pi_1(M)$ isomorphic to $\bZ^2$. 
In the case where $\flow$ is a suspension of an Anosov diffeomorphism of $\mathbb{T}^2$, any incompressible torus is homotopic to the fiber torus $\mathbb{T}^2$, which is transverse to the flow.  Otherwise, Proposition \ref{prop:no_free_Z2} says that no $\bZ^2$-subgroup acts freely on $\orb$, so the fundamental group of $T$ must preserve a chain of lozenges $\cC$.  Using our work in the previous section, we will show that $\cC$ is simple and thus we can build a weakly embedded quasi-transverse surface whose trace is $\cC$. Then, we will apply results in $3$-manifold topology to show that that quasi-transverse surface is homotopic (in fact, often isotopic) to $T$.  
We now give the details: 

\begin{proof}[Proof of Theorem \ref{thm_embedded_tori_to_QT}]
Let $\wt T$ be a lift of $T$ to the universal cover $\wt M$, which is stabilized by $\pi_1(T)$.  Let $H \cong \bZ^2$ be either equal to $\pi_1(T)$, or in the case where $T$ is a Klein bottle, we take $H$ to be an index two subgroup isomorphic to $\bZ^2$.  Proposition \ref{prop:no_free_Z2} implies that $H$ does not act freely on $\orb$, so by Corollary \ref{prop_Z2_stabilize_minimal_chain}, $H$ leaves invariant a 
minimal bi-infinite chain $\cC$ of lozenges, with one generator fixing all corners.  This chain is either unique, or one of two invariant minimal chains forming a scalloped region.  In either case, it is easy to see that $\pi_1(T)$ preserves $\cC$.

Let $g, h \in H$ be generators such that $g$ fixes all corners and $h$ acts freely. Such generators exist by Corollary \ref{prop_Z2_stabilize_minimal_chain}. 

Since $\cC$ is minimal, no three lozenges share a corner.  
Thus, we can take a subchain $\mathcal{D} \subset \cC$ that is a fundamental domain for the action of $h$ on $\cC$, and use Theorem \ref{thm_chain_to_anulus} to construct a Birkhoff annulus $A$ whose trace is $\mathcal{D}$, whose core is represented by $g$, and whose boundary orbits are corners of the form $x$ and $h(x)$ in $\cC$, so are identified under projection to $M$.   This means that the image of $A$ is topologically a $\pi_1$-injective immersed torus or Klein bottle.   To clarify notation, write $S$ for the surface obtained by identifying boundary orbits of $A$, so $\pi_1(S)$ agrees with either $\pi_1(T)$ or $H$
as (conjugacy classes of) subgroups of $\pi_1(M)$.  Thus, $\wt S$ is the union of the translates of $\wt A$ under powers of $h$, and has trace $\cC$.  

To show that $S$ can be made weakly embedded, we prove the following.  

\begin{claim}\label{claim_chain_simple}
The chain $\cC$ is simple.
\end{claim}

\begin{proof}
Suppose for contradiction that $\cC$ is not simple.
Let $f\in \pi_1(M)$ be an element that sends one corner of $\cC$ to the interior of a lozenge in $\cC$. Then, $f$ is not in $H' = \Stab(\cC)$, so, in particular, not in $\Stab (\wt T)$.  
Seen as a curve in $M$, $f$ is freely homotopic to a curve $c$ contained in the image of $S$ in $M$. Let $\wt S$ denote a lift of $S$ to $\wt M$, so there exists a lift $\wt c$ of $c$, contained in $\wt S$ and invariant by $f$.

Since $T$ is embedded, any element of $\pi_1(M)$ either stabilizes $\wt T$ or moves it disjoint from itself.  Thus, we have $f^n \wt T\cap \wt T= \emptyset$ for all $n\neq 0$.  Since $f$ acts by isometries on $\wt M$ and $\wt T$ is a properly embedded plane, the minimal distance between $f^n(\wt T)$ and $\wt T$ gets arbitrarily large as $|n| \to \infty$.  
However, $T$ and $S$ are homotopic in $M$ , and that homotopy moves points a finite distance, so the lifts $f^n \wt T$ and $f^n \wt S$ are at uniformly bounded Hausdorff distance apart, for all $n$, and thus $f^n \wt S$ can be made arbitrarily far from $\wt S$ by taking $n$ large.  This contradicts the fact that 
$f^n\wt S$ intersects $\wt c$ for all $n$, and concludes the proof that $\cC$ is simple. 
\end{proof}

Now that we know that $\cC$ is simple, we can 
choose a fundamental domain of $\cC$ for the action of 
$\pi_1(T)$ on $\cC$ and apply Theorem \ref{thm_simple_chain_gives_weakly_embedded} to obtain a surface $K$ that is a union of weakly embedded elementary annuli, and with $\pi_1(K)$ conjugate to $\pi_1(T)$. 

If $T$ is not the boundary torus of a neighborhood of a one-sided Klein bottle, then the  classification of essential singular tori in 3-manifolds and uniqueness of JSJ decomposition imply that $K$ is homotopic to $T$. (A concise summary statement of this result and references to proofs can be found in Section 3 and Theorem 3.5 of \cite{Bon02}.) Else, it is possible that 
$K$ is homotopic to a one-sided Klein bottle. In this case, $T$ is homotopic to a double cover of $K$ which we can consider, after the homotopy, as an immersion of a torus $T'$ where all points are double points. Pushing off one of the ``sheets'' of $T'$ by the flow outside of the periodic orbits then gives a weakly embedded torus $T''$, homotopic to $T'$. %
\end{proof}

We also point out the following consequence of the proof of Theorem \ref{thm_embedded_tori_to_QT}, phrased in algebraic language.  

\begin{proposition}
Suppose $\flow$ is not the suspension of an Anosov diffeomorphism.  
If $H$ is a subgroup of $\pi_1(M)$ isomorphic to $\bZ^2$ (and $\orb$ not trivial), then $H$ is conjugate to the $\pi_1$ of an immersed quasi-transverse torus.
\end{proposition}

Theorem \ref{thm_embedded_tori_to_QT} allows us to put JSJ tori and Klein bottles in quasi-transverse position.  However, this union of quasi-transverse surfaces may not give the JSJ decomposition since the different JSJ surfaces may intersect each other.   This is indeed unavoidable in certain cases (see, for instance, the examples built in \cite[Figures 5 and 6, p.~1948-1949]{BF13})  
but the failure is very mild: one can always take such surfaces to never cross one another, and intersect only along periodic orbits, as follows.  

\begin{theorem}[The Barbot--Fenley modified JSJ decomposition] \label{thm_well_adapted_JSJ}  
Let $M$ be a $3$-manifold with pseudo-Anosov flow $\flow$.  
Let $\cJ = \{T_1,\dots, T_n\}$ be a collection of tori and Klein bottles giving the JSJ decomposition of $M$. Then there exists a collection $\cJ_\flow=\{T'_1, \dots, T'_n\}$ of weakly embedded quasi-transverse surfaces such that 
\begin{itemize} 
\item for each $i$, $T_i'$ and $T_i$ are homotopic, 
\item the union of the $T_i'$ is weakly embedded, i.e. it is embedded on the complement of the union of periodic orbits contained in the $T_i'$, and
\item each connected component of $M\smallsetminus \cJ_{\flow}$ is either a manifold representing a piece of the JSJ decomposition, or it has an arbitrarily small neighborhood which is a piece of the JSJ decomposition. 
\end{itemize} 
\end{theorem}

\begin{remark} \label{rem_choose_tori}
In this statement, the $T_i'$ can be chosen to be any type of quasi-transverse surfaces, from maximally transverse to maximally periodic as well as anything in between.  Choosing the $T_i'$ to be maximally transverse minimizes the intersection of the surfaces, but does not necessarily eliminate all of them.
 \end{remark}

 \begin{proof} 
By Theorem \ref{thm_embedded_tori_to_QT} we can put each $T_i$, individually, in weakly embedded position.  
We now claim that the {\em union} of the chains giving the traces of the $T_i$ is simple, in the sense that no vertex of any chain gets mapped to the interior of a lozenge of any other. This is proved by the same argument as in Claim \ref{claim_chain_simple}:  Suppose for contradiction that there is a corner $c$ of one chain, say the chain associated to $T_i$, such that $f(c)$ is in the interior of a lozenge of another chain, say that associated to $T_j$.  Let $\wt \alpha$ be the orbit corresponding to $f(c)$ and call $h\in \pi_1(M)$ the element fixing $\wt \alpha$. Note that $\wt \alpha$ intersects a lift $\wt T_j$ of $T_j$ and that $h$ does not preserves $T_j$.  Let $\alpha$ be the projection of $\wt \alpha$ to $M$.   By the JSJ decomposition theorem, we have that $\alpha$ is freely homotopic in $M$ to a curve disjoint from $T_j$.  Thus, we can lift this homotopy and make $\wt \alpha$ disjoint from $\wt T_j$.  But this is impossible since the free homotopy moves points a bounded distance, $\wt T_j$ is a properly embedded plane not invariant under $h$, and $\wt\alpha=h^n(\wt\alpha)$ intersects $h^n(\wt T_j)$ for all $n$.  This shows simplicity.  

Now, just as in the proof of Theorem \ref{thm_simple_chain_gives_weakly_embedded}, one can first  put the union of the $T_i$ in general position, then analyze the loci of double and triple points showing they lie along circles in the interior of annuli.  Finally, one applies the innermost discs argument to resolve these intersections by  a homotopy along orbits.  This gives embedded surfaces $T'_i$, the union of which is embedded except along periodic orbits.  

Finally, the only way for a connected component of $M \setminus \cJ_\flow$ to not be an actual piece of the JSJ decomposition must arise from the failure of the collection $\cJ_\flow$ to be embedded: if all the surfaces are pairwise disjoint and embedded, they decompose $M$ into submanifolds giving the JSJ decomposition. If this fails because surfaces coincide along periodic orbits, one must pass to a small neighborhood to obtain a true piece of the JSJ decomposition, as claimed.  
 \end{proof}

 \section{Transverse surfaces} \label{sec_transverse} 
In this section, we show that the existence of truly transverse (rather than merely quasi-transverse) surfaces can be recognized directly from the structure of the orbit space. 
This provides an important certificate to show a flow is transitive, since nontransitive flows always admit transverse tori.  For Anosov flows, this was originally obtained by Brunella \cite{Brunella}, and is attributed to Mosher in the pseudo-Anosov case.  See \cite{BBM24b} for further discussion of the pseudo-Anosov and Anosov-like setting.    

\begin{theorem}\label{thm_transverse_torus_special_chains}
Suppose that $\flow$ is not a suspension of an Anosov diffeomorphism. Then $\flow$ admits a transverse torus or Klein bottle if and only if there exists a minimal bi-infinite chain of lozenges $\cC=\{L_i\}_{i\in \bZ}$ such that for all $i$, $L_{i}$ and $L_{i+1}$ always share a side.
\end{theorem}
Aside from a mention in a remark in \cite[Remark 5.6]{BBM24b}, this Theorem has not yet appeared in the literature; however, much of the proof follows from existing work of Barbot and Fenley \cite{Bar95b,BF13}.

Recall that the definition of quasi-transverse surface assumed that the surface was incompressible, i.e. $\pi_1$-injectively immersed. 
To show that transverse tori or Klein bottles give chains of the form described in Theorem \ref{thm_transverse_torus_special_chains}, we will show first that they are incompressible, a fact proved by Fenley for Anosov flows in \cite[Corollary 2.2]{Fen95b} (see also \cite[Lemma 1]{Brunella} and \cite[Proof of Proposition 2.7]{Mos92a} for a special case).  The proof we give here is quite different than Fenley's.  

\begin{lemma}\label{lem_transversetorus}
Let $T$ be a torus or Klein bottle transverse to a pseudo-Anosov flow $\flow$. Then $T$ is incompressible in $M$. 
\end{lemma}

\begin{proof}
Suppose $T$ is transverse to $\flow$.  Let $\wt T$ be a lift of $T$ to $\wt M$, so $\wt T$ is transverse to $\hflow$.  
Since $T$ has Euler characteristic $0$, the induced foliations on $T$ have no prong singularities, so $\wt T$ is a set transverse to $\hflow$ whose intersections with $\hfs$ and $\hfu$ give a trivial foliation.  The argument from Theorem \ref{thm_orb_is_plane} (which used only the fact that we were working with a trivially foliated transversal) shows that each orbit of $\hflow$ can intersect  $\wt T$ at most once. In particular, $\wt T$ cannot be closed, for otherwise there would be recurrent orbits of $\hflow$, and we showed in Proposition \ref{prop_top_properties} that orbits are properly embedded lines.  

Thus, $\wt T$ is topologically an annulus or a plane, since the transverse direction of $\hflow$ induces an orientation so it cannot be a topological M\"obius band.  If it is a plane, we are done; so assume for contradiction that $\wt T$ is an annulus.  
Let $g\in \pi_1(M)$ be a nontrivial element that leaves $\wt T$ invariant so that $\wt T / \langle g \rangle$ is a torus, and let $A \subset \wt T$ be an annulus fundamental domain for $g$.   
Since orbits of $\hflow$ intersect $\wt T$ at most once, $A$ projects injectively to an annulus $\hat A$ in $\orb$.  Let $c$ be the ``inner'' boundary (i.e., the boundary of $\hat A$ that bounds a topological disc $D$ disjoint from $\hat A$). Since $A$ is a fundamental domain for $g$, the curve $gc$ must be the outer boundary of $\hat A$, and thus $g$ has a fixed point $x$ inside $D$.  The leaves $\wt\cF^{s,u}(x)$ are properly embedded, so intersect $c$.  However, the fact that $g$ is topologically expanding along one leaf of $\wt\cF^{s,u}(x)$ and contracting along the other now implies that  
$gc\cap c \neq \emptyset$. This contradiction concludes the proof.  %
\end{proof}

For the other direction of Theorem \ref{thm_transverse_torus_special_chains}, we  need the following lemma.

\begin{lemma}\label{lem_merging_transverse_chains}
Let $\cW:=\{L_i, i \in I\}$ and $\cW':=\{L'_i,i\in I'\}$ be two chains of lozenges, finite or infinite, with index sets $I$ and $I'$ some consecutive subsets of integers containing $0$.  Suppose $\cW$ and $\cW'$ both satisfy the following two conditions:
\begin{enumerate}[label=(\roman*)]
\item \label{item_chain_no_triple_shared_corners} three consecutive lozenges never share a corner; and
\item\label{item_chain_share_side} two consecutive lozenges always share a side.
\end{enumerate} 
If there exists $k \geq 0$ such that, for all $0\leq i\leq k$, $L_i = L'_i$ (if $k=0$ we suppose moreover that the shared corner of $L_{-1}$ and $L_0$ is the same as the shared corner of $L'_{-1}$ and $L'_0= L_0$), 
then  $\cW'' := \{L_i,i\leq k\}\cup \{L'_i,i\geq k+1\}$
also satisfies conditions \ref{item_chain_no_triple_shared_corners} and \ref{item_chain_share_side}. 
\end{lemma}
\begin{proof}

We start with condition \ref{item_chain_share_side}, i.e., we need to show that consecutive lozenges in $\cW''$ share a side. By assumption on $\cW$ and $\cW'$, this is true for any pair of consecutive lozenges in $\{L_i,i\leq k\}$ or in $\{L'_i,i\geq k+1\}$. So we only need to check that it also holds for $L_k$ and $L_{k+1}'$. But this directly follows from the fact that $L_k= L'_k$.

Similarly, to prove condition \ref{item_chain_no_triple_shared_corners}, i.e., that no triple of lozenges in $\cW''$ share a common corner, it is enough to prove this at the joining of the two chains, that is, for the triples $L_{k-1}, L_k, L'_{k+1}$ and $L_{k}, L'_{k+1}, L'_{k+2}$.
Since $L_k=L'_k$, the three lozenges $L_{k}, L'_{k+1}, L'_{k+2}$ cannot share a common corner.
For the triple $L_{k-1}, L_k, L'_{k+1}$, we separate the case $k=0$ from $k>0$: If $k>0$, then $L_{k-1}=L'_{k-1}$ and $L_k=L'_k$, so the three do not share corners thanks to our assumption on $\cW'$.
 
 If $k=0$, call $c_i$, resp.~$c'_i$, the shared corner of $L_{i-1}$ and $L_i$, resp.~$L'_{i-1}$ and $L'_i$. The  triple $L_{-1}, L_0, L'_{1}$ share a common corner if and only if the corner shared by $L_{-1}, L_0$, which is $c_0$, is the same as the corner shared by $L_0=L'_0$ and $L'_1$, which is $c_1'$. But 
 by assumption, we have that $c_0=c'_0$ and $L_0=L_0'$, thus $c_1=c'_1$ and in particular $c_0\neq c_1'$.  We deduce that $\cW''$ satisfies \ref{item_chain_no_triple_shared_corners}. This finishes the proof.
\end{proof}

\begin{proof}[Proof of Theorem \ref{thm_transverse_torus_special_chains}]
First, suppose that $T$ is a transverse torus or Klein bottle. By Lemma \ref{lem_transversetorus}, $T$ is incompressible, so it is a quasi-transverse surface.  By Proposition \ref{prop_trace_QT_tori}, the trace of $T$ in $\orb$ is contained in a bi-infinite chain of lozenges $\cC$. Since $T$ is transverse, its trace does not contain any corner of $\cC$. Hence, any two consecutive lozenges in $\cC$ share a side, as claimed.

We now prove the other direction: assume that there exists a chain of lozenges $\cC=\{L_i\}_{i\in \bZ}$ such that for all $i$, $L_{i}$ and $L_{i+1}$ always share a side, and $L_{i-1}$, $L_{i}$ and $L_{i+1}$ never share a corner. Our goal is to show that there exists possibly another chain $\cC'$ satisfying the same properties, i.e., with consecutive lozenges sharing a side and no three sharing a corner, and left invariant by a subgroup $H$ of $\pi_1(M)$, isomorphic to $\bZ^2$. By Theorem \ref{thm_simple_chain_gives_weakly_embedded} and Corollary \ref{cor_max_transverse_or_periodic}, $H$ is then the $\pi_1$ of a maximally transverse quasi-transverse torus $T'$ whose associated chain of lozenges is $\cC'$. Since no two lozenges in $\cC'$ are diagonal, Remark \ref{rem_transverse_torus_no_diagonal} tells us that $T'$ is in fact transverse.

Thus, what remains is to construct $\cC'$. 
Call $c_i$ the common corner of $L_i$ and $L_{i+1}$. By condition \ref{item_chain_share_side}, every such corner $c_i$ is a pivot -- recall from Section \ref{subsec_pivot} that a {\em pivot} is by definition a common corner of two lozenges sharing a side.   By Corollary \ref{cor_cocompactness_imply_finitely_many_branching}, there exists $g\in \pi_1(M)$ such that, for some $j\in \bZ$ and $k>0$, we have $g c_j = c_{j+k}$. Moreover, we can assume that $g L_j = L_{j+k}$. Using again that each corner  is a pivot, there exists $h\in \pi_1(M)$ that fixes all the corners of $\cC$.  Take such an $h$ that is a generator of the stabilizer of all the corners in $\cC$.  

For simplicity, we re-index the chain so that $j=0$.  
By Lemma \ref{lem_merging_transverse_chains}, the chain $\{L_0, L_{1}, \ldots, L_{k}= g(L_0), g(L_{1}), \ldots, g(L_k)\}$ satisfies 
conditions \ref{item_chain_no_triple_shared_corners} and \ref{item_chain_share_side}.  Now define 
$\cC'$ to be the chain consisting of lozenges $g^n(L_i)$ for $n\in \bZ$ and $i \in [0, k]$. %
  Iterated applications of Lemma \ref{lem_merging_transverse_chains} show that this is a bi-infinite chain satisfying \ref{item_chain_no_triple_shared_corners} and \ref{item_chain_share_side}, and by construction $g$ preserves $\cC'$ and acts on it freely.
Hence $g$ and $h$ commute and so generate a $\bZ^2$-subgroup of $\pi_1(M)$ that fixes $\cC'$, as claimed.
\end{proof}

The special case of a transverse torus (or Klein bottle) whose trace is a line of lozenges can be detected from the free homotopy data of periodic orbits.  

\begin{proposition}\label{prop_scalloped_surface}
Suppose $T$ is a torus or Klein bottle in $M$ transverse to a pseudo-Anosov flow $\flow$.  The trace of $T$ is a scalloped region if and only if $\pi_1(T)$ contains two independent elements which are freely homotopic to periodic orbits of $\flow$.  
Moreover, any scalloped region in $\orb$ can be realized as the trace of a torus or Klein bottle transverse to $\flow$.  
\end{proposition}

This has important applications in the analysis of flows obtained by gluing constructions.  As we will see in the next section, many examples can be constructed by gluing together partial flows on manifolds with transverse, or sometimes quasi-transverse, boundary tori.  However, understanding the orbit space of the resulting example is typically quite difficult, even if the building blocks are simple and well understood.  Proposition \ref{prop_scalloped_surface} gives one tool to detect features in the orbit space of a glued flow.  

\begin{proof}[Proof of Proposition \ref{prop_scalloped_surface}]
Let $T$ be a transverse torus or Klein bottle and $\hat T$ its trace.   Suppose that $\hat T$ is a scalloped region $S$, so $\pi_1(T)$ is a subgroup of the stabilizer of $S$. Lemma \ref{lem_action_of_stabilizer_scalloped} says that this stabilizer is either $\bZ^2$ or has an index two subgroup isomorphic to $\bZ^2$, and in either case this $\bZ^2$ subgroup is generated by two elements which fix the corners of the two lines of lozenges comprising this region.   Thus, $\pi_1(T)$ must also contain independent elements which each fix the corners of the two different lines of lozenges.  Moreover, these elements are not conjugate in $\pi_1(M)$, since one preserves each leaf of the $\cF^s$ boundaries of the (unique, invariant) scalloped region while permuting the $\cF^u$ boundary leaves, while the other preserves each $\cF^u$ boundary leaf.  
We conclude that $\pi_1(T)$ contains two elements represented by non-freely-homotopic periodic orbits.  

Conversely, if the associated chain of lozenges, $\cC$, of $\hat T$ does not form a scalloped region, then a unique cyclic subgroup of $\pi_1(T)$ fixes all the corners of $\cC$, and any element whose square is not in this subgroup fixes no corners. %

For the second statement, suppose $S$ is a scalloped region, and let $g, h$ be generators of the stabilizer of $S$ where $h$ fixes all corners of one of the lines of lozenges making up $S$, and $g$ permutes these lozenges nontrivially.  Let $\cC \subset S$ be a sub-chain of lozenges acting as fundamental domain for the action of $g$.  The scalloped region $S$ is a simple chain (this is an easy exercise, or see Figure \ref{fig_intersection_scalloped} and the proof of Lemma \ref{lem_scalloped_off_of_themselves} for a description of the only possible intersections of scalloped regions), so it follows that $\cC$ is simple as well.  Thus, we can 
apply Theorem \ref{thm_simple_chain_gives_weakly_embedded} to build a quasi-transverse Birkhoff annulus $A$ with trace $\cC$; and by construction this has its two boundary components identified by $g$, so is topologically a torus or Klein bottle, depending on whether $g$ preserves or reverses orientation.  Finally, the fact that $S$ is scalloped means no three lozenges share a corner and each consecutive pair shares a side, so, as above, Corollary \ref{cor_max_transverse_or_periodic} and Remark \ref{rem_transverse_torus_no_diagonal} tell us that $A$ can be made transverse. %
\end{proof}

\section{Periodic Seifert fibered pieces} \label{sec_periodic_seifert} 
By putting JSJ tori in quasi-transverse position, one obtains a (nearly canonical) decomposition of a pseudo-Anosov flow $\flow$ into partial flows on Seifert fibered or hyperbolic manifolds with torus boundary components.  Thus, one approach to classifying pseudo-Anosov flows is to first classify such partial flows on JSJ pieces.   For Seifert fibered pieces, and more generally for graph manifolds (those comprised of only Seifert pieces)  this has been carried out to great success in \cite{BF13,BF15,BF21}, although the program is not yet entirely complete\footnote{By contrast, the restriction of pseudo-Anosov flows to hyperbolic pieces is much less understood.  However, this is not very  surprising given what we know for manifolds with trivial JSJ decomposition: Theorem \ref{thm_Ghys_Plante_theorem} gives a complete classification for Seifert-fibered manifolds, and this result has no known counterpart for hyperbolic 3-manifolds.}.  Our goal in this, and the following, section is to describe the main results of this work.   As usual, we take as standing assumption that $\flow$ is a pseudo-Anosov flow on a compact 3-manifold $M$.

Seifert fibered pieces of $M$ fall into two broad classes, depending on whether the fiber is represented by a periodic orbit of $\flow$.   
\begin{definition} \label{def_periodic_Seifert}
Let $P$ be a Seifert fibered JSJ piece of $M$.  The piece $P$ is called {\em periodic} (for $\flow$) if a regular fiber of some Seifert fibration of $P$ is freely homotopic to a periodic orbit of $\flow$, and is called a {\em free piece} otherwise.  
\end{definition} 

Most Seifert manifolds admit only one Seifert fibration up to isotopy, but a short list of them (5 types for orientable manifolds, see e.g.~\cite{Hat23}) may admit several, which is why one needs the ``...of some Seifert fibration...'' clause in the definition.  In the rest of this text, when we talk about fibers of a periodic Seifert piece, we adopt the convention to always take the fibration where the fiber is  represented by a periodic orbit.  This is unambiguous thanks to Theorem \ref{thm_spines_exist}, which implies that when one fiber is represented by a periodic orbit, then any periodic orbit in $P$ is freely homotopic to a (power of a) fiber.  
This theorem also says much more, describing a very specific structure of Seifert fibered pieces:  

\begin{theorem}[\cite{BF13} Theorem B, and \cite{BFeM23} Theorem 3.1] \label{thm_spines_exist}
Let $P$ be a periodic Seifert fibered piece for a pseudo-Anosov flow on $M$.  Then there exists a connected, finite union of elementary Birkhoff annuli in $M$, embedded along the complement of their periodic orbit boundaries,  such that any sufficiently small neighborhood of this union is a representative for the JSJ piece $P$. 
\end{theorem} 

Following \cite{BF13}, this union of Birkhoff annuli is called the {\em spine} of $P$. 

Theorem \ref{thm_spines_exist} was proved assuming $M$ orientable in \cite{BF13}, and generalized (using essentially no new ingredients) to the non-orientable case in \cite{BFeM23}.  Here we give a self-contained simplified exposition of the proof.  
For this, we need an elementary lemma for abstract groups acting on trees, which we leave as an exercise. (For a solution, see \cite[Section 7]{BF13}.)

\begin{exercise} \label{ex_pruned_tree_connected}
Let $T$ be a tree, and $g, h$ automorphisms of $T$ acting freely, with axes $A(g)$ and $A(h)$.  If $A(g)$ and $A(h)$ are disjoint, then the (unique) path between them is contained in the axis either of $gh$ or $gh^{-1}$. 
\end{exercise} 

\begin{proof}[Proof of Theorem \ref{thm_spines_exist}]
Let $P$ be a periodic Seifert piece and $g \in \pi_1(M)$ a representative of a regular fiber periodic orbit.   
If $g$ acts on $\orb$ with a unique fixed point, then the fact that $g$ is central (up to passing to an index 2 subgroup preserving orientation, which does not affect the proof.) implies that this point is globally preserved by $\pi_1(P)$.  Since point stabilizers are virtually cyclic, and 
$P$ is a piece of a JSJ decomposition, $\pi_1(P)$ is not virtually cyclic and this case is impossible.  

Thus, $g$ preserves a nontrivial chain of lozenges. Let $\cC$ be the maximal chain invariant under $g$.  Again, since $g$ is central in $\pi_1(P)$, the chain $\cC$  is  $\pi_1(P)$ invariant.  

Recall $\cT(\cC)$, the {\em tree associated to $\cC$}, from Definition \ref{def:chain_tree}. This is a graph with vertices given by corners of lozenges in $\cC$, and two vertices connected by an edge if they are corners of the same lozenge.  Since point stabilizers are cyclic, any non-central element of $\pi_1(P)$ acts freely on $\cT(\cC)$ so has an axis.  Let $\mathcal{T}' \subset \cT(\cC)$ be the union of all axes of elements of $\pi_1(P)$ which act freely on $\cT(\cC)$.  By Exercise \ref{ex_pruned_tree_connected}, $\mathcal{T}'$ is connected, so corresponds to a subchain $\cC' \subset \cC$ of lozenges.  Furthermore, since the image of an axis under an automorphism is the axis of a conjugate element, $\mathcal{T}'$ is invariant under the action of $\pi_1(P)$.  

The remainder of the proof consists of showing that $\cC'$ is simple, so can be realized as a weakly embedded union of Birkhoff annuli, and then that a small thickening of this union of annuli is a representative for $P$.  

Assume for a contradiction that $\cC'$ is not simple, so there exists some $f \in \pi_1(M)$ and corner $c_0 \in \cC'$ such that $f(c_0) \in L$, where $L$ is a lozenge of $\cC'$.
We will start by showing that, in this case, $\cC'$ must be a string:  Since $\mathcal{T}' = \cT(\cC')$ is a union of axes, each corner of $\cC'$ is the corner of at least two lozenges.  Since $f(c_0) \in L$, it is a nonsingular point so has four quadrants, two diagonal ones containing corners of $L$ and the other two containing the perfect fits of $L$.  By the non-corner criterion (Lemma \ref{lem_no_corner_criterion}), there cannot be a lozenge with corner $f(c_0)$ in the quadrants containing perfect fits, so $f(c_0)$ has lozenges in the other two quadrants, each containing a corner of $L$. Call these lozenges $f(L_0)$ and $f(L_1)$.  An illustration of this configuration is given in Figure \ref{fig_non_simple_string}. 

Applying the same argument to each of the corners of $L$ in place of $f(c_0)$, we conclude that each of them has another lozenge in the quadrant diagonal to $L$ (and not in the other quadrants), and these diagonal lozenges contain the other corners, say $f(c_1)$ and $f(c_{-1})$ of $f(L_1)$ and $f(L_0)$ respectively.   Iterating this argument shows that $\cC'$ must be equal to a string of lozenges and so $\cT(\cC')$ is a graph where each vertex has valence 2.  

\begin{figure}[h]
\centerline{ \mbox{
\labellist
\small \hair 2pt
\pinlabel {\footnotesize $f(L_0)$} at 93 160
\pinlabel {\footnotesize$f(c_0)$} at 156.784 192.233
\pinlabel {$L$} at 160.174 169.068
\pinlabel {\footnotesize $f(L_1)$} at 144.354 227.827
\endlabellist
\includegraphics[width=5.5cm]{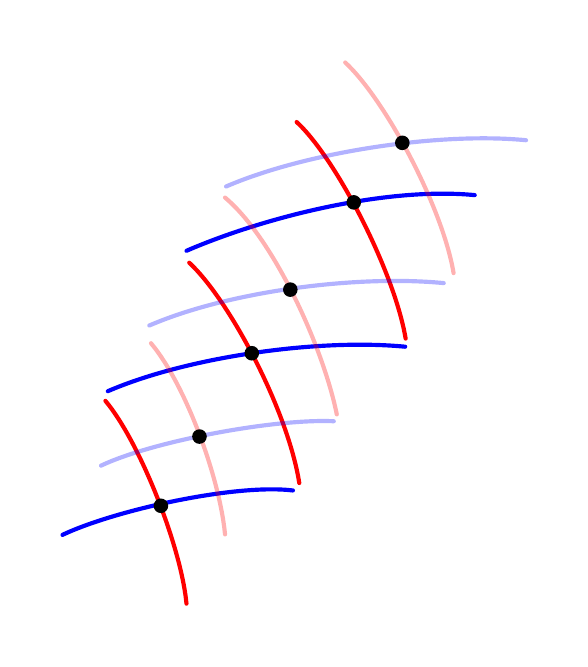}}}
\caption{If $\cC'$ was not simple then it must be a string.}
\label{fig_non_simple_string}
\end{figure}
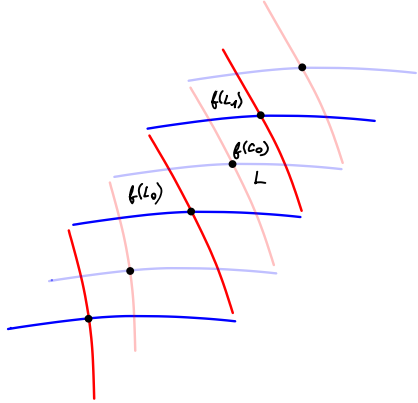 

As a consequence, the stabilizer of $\cC'$ is (algebraically) small:  the action on $\cT(\cC')$ gives a morphism to the semidirect product of $\bZ$ with $\bZ/2\bZ$.  The kernel of this action fixes all corners of $\cC$, hence is cyclic.  Thus, the stabilizer, up to an index two subgroup, is $\bZ$ or a semi-direct product of $\bZ$ and $\bZ$. This implies that $P$ is either a solid torus, $T^2\times I$ or a twisted $I$-bundle over the Klein bottle. The first case is impossible because solid tori are not pieces of a JSJ decomposition. In the second case, i.e., $P=T^2 \times I$, we would have that $M$ must be a mapping torus, and therefore $\flow$ would be a suspension, contradicting the fact that $P$ is periodic. So we must be in the last case, where $P$ is a twisted $I$-bundle over the Klein bottle.  Consider the boundary torus $T$ of $P$. Its fundamental group $\pi_1(T)$ is an index $2$ subgroup of $\pi_1(P)$ so must preserve the string $\cC'$. By Theorem \ref{thm_embedded_tori_to_QT}, $T$ is homotopic to a weakly embedded quasi-transverse surface whose trace is $\cC'$.  But this would mean that $\cC'$ is simple (by Theorem \ref{thm_weakly_to_simple_general} or directly from Claim \ref{claim_chain_simple}), contradicting our assumption.\footnote{Note that this argument does \emph{not} say that the spine is never a string, it says that even if it is a string, it will be simple.}

Having shown $\cC'$ is simple, we may apply the technique of the proof of Theorem \ref{thm_embedded_tori_to_QT}  to build a union $B$ of elementary Birkhoff annuli, one for each lozenge in a fundamental domain for the action of $\pi_1(P)$ on $\cT(\cC')$, glued along their periodic orbits, and then resolve any interior intersections to make this union weakly embedded. 

Let $\wt B$ be a lift of $B$ to $\wt M$.  Then $\pi_1(P)$ acts on $\wt B$, and we claim this action is free:  indeed, if $h \in \pi_1(P)$ were to fix some point $p \in \wt B$, then the projection of $p$ to $\orb$ would be a fixed point in $\cC'$, and hence necessarily a corner.  Thus, $p$ represents a periodic orbit, and $h$ is the identity (otherwise, it would translate nontrivially along this orbit, fixing no point).   We conclude that $B$ is homotopic to $P$ in $M$ and thus any small thickening of it to an open subset of $M$ is a representative of $P$.   
\end{proof}

\subsection{Constructing periodic pieces} 
Theorem \ref{thm_spines_exist} shows that all periodic Seifert pieces are obtained by taking a neighborhood of a union of Birkhoff annuli glued along their boundaries.  This suggests a ``building-block" approach to constructing examples:  first describe a convenient local model for a neighborhood of a Birkhoff annulus, then understand how such neighborhoods glue together when annuli are identified along boundaries.   We describe this approach here, following \cite{BFeM23, BF13}. 

\begin{construction}[Block neighborhood of a Birkhoff annulus]  \label{const_block}
Let $N= I \times I \times S^1$, where $I= [-\pi/2,\pi/2]$ and $S^1 = \R/\bZ$, with coordinates $(x,y,z)$. Fix some $\lambda > 1$, and define a vector field $X_\lambda$ in coordinates by 
\[
\begin{cases} 
 \dot{x} = 0 ,\\
 \dot y = \cos^2(x) + \sin^2(y)\sin^2(x)\\
 \dot z = + \lambda \sin (x) \cos (y). 
\end{cases}
\] 
Let $\psi_\lambda$ denote the flow generated by $X_\lambda$.
\end{construction}
\begin{figure}[h]
 \labellist 
  \small\hair 2pt
     \pinlabel $y$ at 49 70 
     \pinlabel $x$ at 220 15 
   \pinlabel $z$ at 20 200
 \endlabellist
\centerline{ \mbox{
\includegraphics[width=5.5cm]{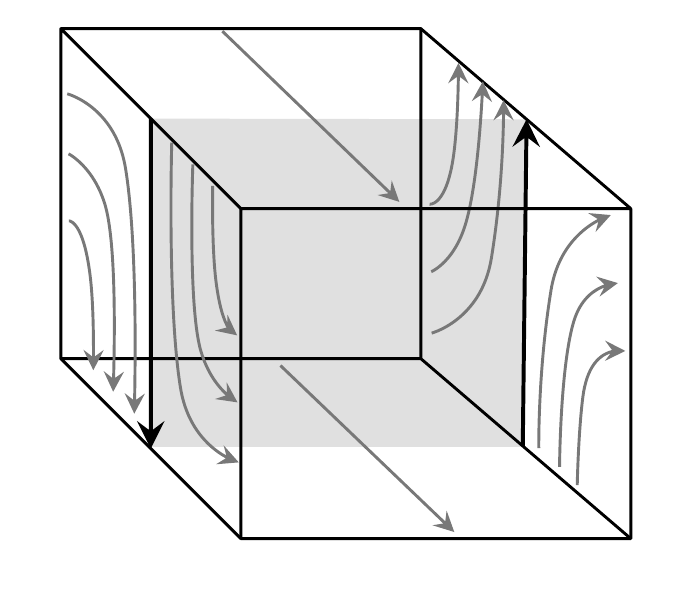}}}
\caption{The model block with flow $\psi_\lambda$. The shaded region is a Birkhoff annulus.}
\label{fig:model_block}
\end{figure} 

Note that $\psi_\lambda$ has exactly two closed orbits, $\alpha_1 = \{-\pi/2\} \times \{0\} \times S^1$ and $\alpha_2 = \{\pi/2\} \times \{0\} \times S^1$, with $\alpha_1$ going in the negative $z$-direction and $\alpha_2$ in the positive $z$-direction, and these bound a  Birkhoff annulus $I  \times \{0\} \times S^1$. 
The boundary $I  \times \{-\pi/2\} \times S^1$ is incoming for $\psi_\lambda$ and the boundary $I  \times \{\pi/2\} \times S^1$ is outgoing, and the other boundaries are tangent to the flow, made of stable and unstable sets for the periodic orbits.  

The parameter $\lambda$ controls the amount of shearing in the $z$-direction as orbits go through the block; the larger the $\lambda$, the more shear. 
Observe that any embedded Birkhoff annulus $A$ of a pseudo-Anosov flow on a manifold $M$ has a neighborhood\footnote{This is technically not an open neighborhood since the boundary periodic orbits remain in the boundary of $N$, but rather a thickening of $A$ in the direction of the flow together with stable and unstable sets for the periodic orbits.  To keep the writing simple, we will follow this conventional abuse of terminology and call it a neighborhood.} that is orbit equivalent to $N$ with the model flow $\psi_\lambda$ (for any $\lambda>1$), via a homeomorphism taking $A$ to the set $[-\pi/2,\pi/2] \times \{0\} \times S^1$ in the model.  
More generally, a Birkhoff annulus that is only weakly embedded admits a neighborhood that is a local embedding of $N$, injective on the interior but possibly non-injective on the faces containing the periodic orbits.  %

The control on $\lambda$ is used to prove hyperbolicity of flows built from gluing such pieces. For most of this section this control will be irrelevant (and when needed, any $\lambda$ sufficiently large will suffice), so we fix some $\lambda>1$ and drop $\lambda$ from the notation for now, writing $\psi$ for $\psi_\lambda$.

Two copies of $(N, \psi)$ can be glued together either just along a periodic orbit, or along the stable or unstable manifold face of a periodic orbit, preserving the $S^1$ fiber factor, provided that they are oriented so that the incoming boundaries are adjacent, and the flows agree on the glued face.  This is achieved by flipping the $z$-coordinate direction of one of the blocks, as shown in Figure \ref{fig:gluings}, left.   Doing this gluing with four blocks around a common periodic orbit corner produces a standard neighborhood of a periodic orbit in a periodic Seifert piece for an Anosov flow; gluing $2k$ produces a standard singular $k$-prong.   Figure \ref{fig:gluings} (right) shows the resulting local picture in the orbit space of the (partially defined) flow. 
 The reader may wish, as an exercise, to write explicit  descriptions of these gluings in coordinates, or may consult Section 4 of \cite{BFeM23} for a detailed description.   

Assembling pieces in this way gives a (trivially) Seifert fibered manifold of the form $S \times S^1$, 
where $S$ is a surface with boundary and the $S^1$ fiber is represented by a periodic orbit, with a flow for which all boundary components are transverse or quasi-transverse, and all Birkhoff annuli are {\em embedded}.  To modify this to obtain weakly embedded Birkhoff annuli instead, one can perform a Dehn-Fried surgery along any periodic orbits (see Section \ref{sec_example_Dehn_Fried}). This has the effect of changing the number of times a Birkhoff annulus boundary covers the periodic orbit, resulting in a singular fiber for the Seifert fibration, and a weakly embedded, rather than embedded, annulus.

\begin{figure}[h]
  \labellist 
    \small\hair 2pt
         \pinlabel $\alpha_1$ at 10 120
     \pinlabel $N_1$ at 110 145
   \pinlabel $\alpha_2$ at 175 155
     \pinlabel $\gamma$ at 115 65
 \pinlabel $N_2$ at 180 210
   \pinlabel {\em in} at 90 187
     \pinlabel {\em in} at 132 210
        \pinlabel {\em out} at 130 25
          \pinlabel {\em out} at 232 215 
 \pinlabel $\alpha_1$ at  370 90
 \pinlabel $\alpha_2$ at 460 110
  \pinlabel $\gamma$ at 405 80
 \pinlabel $L$ at 400 150
  \pinlabel $L'$ at 470 180 
    \endlabellist
\centerline{ \mbox{
\includegraphics[width=11cm]{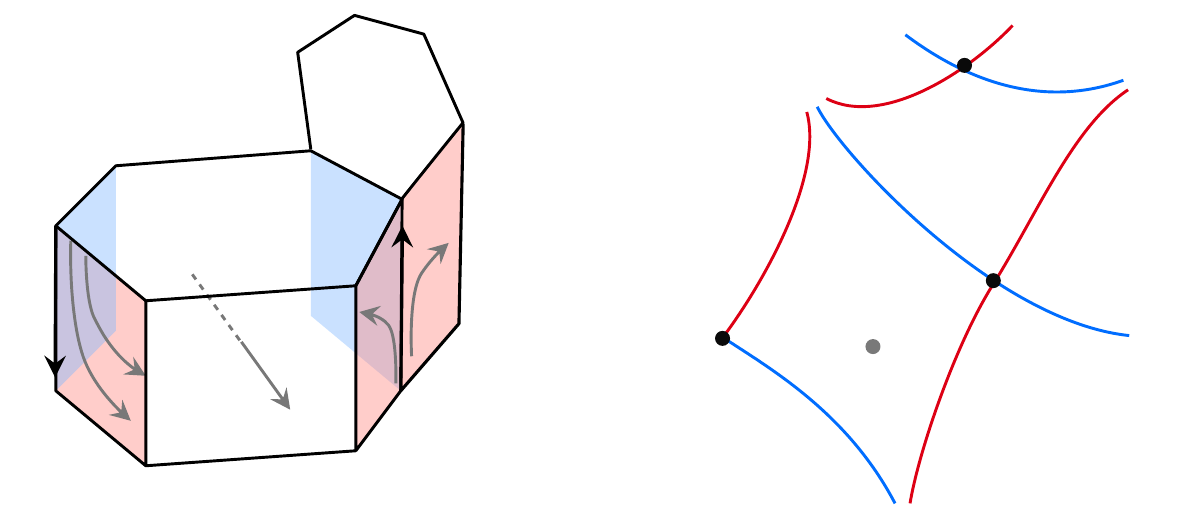}}}
\caption{{\em Left:} two model blocks glued along a half-face, the rear/right block is flipped vertically.  Orbits $\alpha_i$ are labelled according to the block $N_1$.  {\em Right:} When such blocks occur as part of a flow, projecting to the orbit space gives two adjacent lozenges with corners fixed by the element representing the fiber.  In this image, the flow direction is out of the page, towards the reader.}
\label{fig:gluings}
\end{figure}
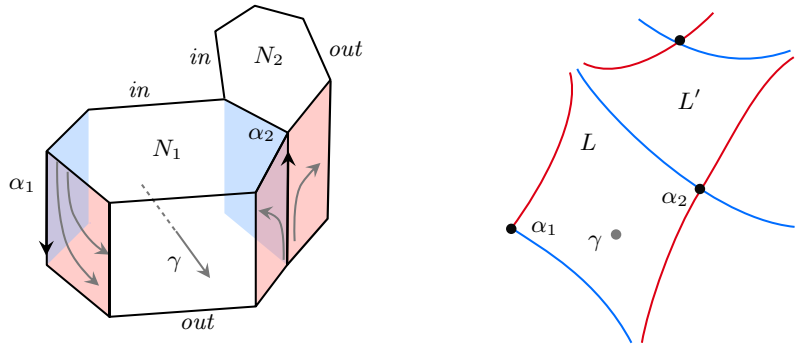 

There is one other type of gluing which is less obvious, and creates reflector arcs in the orbifold base of the Seifert fibration: 

\begin{construction}[Gluing a reflector arc]  \label{const:arc_glue}
Let $N_1$, $N_2$ be copies of $N$, and glue the half-face $\{-\pi/2\} \times [-\pi/2, 0] \times S^1$ of $N_1$ to the half-face $\{\pi/2\} \times [-\pi/2, 0] \times S^1$ of $N_2$ by the map $(y_1,z_1) \mapsto (y_1, -z_1)$ in local coordinates $(x_i, y_i, z_i)$ on $N_i$; as described above. 
Now, glue the half-face $\{-\pi/2\} \times [0, \pi/2] \times S^1$ of $N_1$ to itself by $(y_1, z_1) \mapsto (y_1, z_1 + 1/2)$, and similarly glue  $\{\pi/2\} \times [0, \pi/2] \times S^1$ of $N_2$ to itself by  $(y_2, z_2) \mapsto (y_2, z_2+1/2)$.  This commutes with $X$ so gives a flow on the quotient manifold, which has a Seifert fiber structure with an arc of singular fibers.  
\end{construction} 
The result of Construction \ref{const:arc_glue} is a local neighborhood of a hyperbolic periodic orbit with stable and unstable leaves as in a topological Anosov flow. This easily generalizes to produce a singular hyperbolic periodic orbit by gluing additional copies of $N$ before introducing a self-gluing of faces.  

\subsection{Pieces with transverse boundary}
The gluing constructions described above were worked out as a means of producing examples in \cite[Section 8]{BF13}.  Thanks to Theorem \ref{thm_spines_exist}, these types of gluings should provide a recipe to produce {\em all} examples.  This has been rigorously proved in the setting of periodic Seifert pieces with transverse boundary: 

\begin{theorem}[Periodic pieces have models] \label{thm_periodic_are_models} 
Suppose that $P \subset M$ is a periodic Seifert piece of a pseudo-Anosov flow $\flow$ so that all JSJ boundary surfaces of $P$ can be realized transversely.    Then there exists a subset $P' \subset M$, isotopic to $P$, such that the restriction of $\flow$ to $P'$ is obtained by gluing some number of homeomorphic copies of $(N, \psi)$ on tangent boundary components, then performing Dehn-Fried surgeries on some periodic orbits.  
\end{theorem} 

Theorem \ref{thm_periodic_are_models}, for $M$ assumed orientable, is part of the Main Theorem of \cite{BF15} and was generalized to the non-orientable case in \cite{BFeM23}.  This latter work assumed flows were topologically Anosov, for simplicity, but the proof carries over in the general case.  
The proof also easily extends to cover the case where $M$ and the base surface are orientable, and all Birkhoff annuli are embedded rather than only weakly embedded.  However, the general case for arbitrary Seifert pieces with quasi-transverse boundary has not been treated in the literature, for several reasons.   For one, the necessary modification of doing a Dehn-Fried surgery along a periodic orbit is more cumbersome-- one needs a slightly different description of how to modify the Seifert structure when boundary periodic orbits are present (for instance, first by enlarging the piece, then surgering, then restricting).  In general, the presence of boundary periodic orbits may make it necessary to pass to a neighborhood of $P'$ to obtain a representative of the Seifert piece, which leads to notational difficulties.  %
We describe the main strategy of the proof in the simpler transverse setting.  

\begin{proof}[Sketch of proof of Theorem \ref{thm_periodic_are_models}]
First, recall from Theorem \ref{thm_spines_exist} that 
$P$ can be represented by a neighborhood of a finite union of weakly embedded, elementary Birkhoff annuli. Recall that any individual elementary Birkhoff annulus in a flow has a neighborhood (cut off along the adjacent stable and unstable sets for its bounding periodic orbits) which is orbit-equivalent to an immersed copy of $(N, \psi)$, embedded along the interior.   Thus, what we must show here is that these neighborhood blocks can be glued together, then surgered, to produce $P'$.

Since the interiors of the Birkhoff annuli of the spine are embedded, there is no obstruction to embedding a union of copies of the interior of $N$ -- what requires care is gluing along the periodic orbits and boundaries.  For this, one must pay attention to the faces of the local stable and unstable manifolds of the periodic orbits bounding the quadrants containing the annuli.  

The situation is simplest when all annuli are embedded: in this case, local stable and unstable manifolds of adjacent annuli will either coincide, or be disjoint, and this can be seen from the configuration of lozenges in the trace of the spine. When they coincide, the model neighborhoods of each annulus in $M$ get identified along half-faces, as in Figure \ref{fig:gluings}. Thus, one obtains a union $P'$ of copies of $N$ in $M$, identified in pairs along some (possibly all) half-faces corresponding to stable and unstable sets for the periodic orbits of the spine, and by construction $P'$ is homotopic to and contains the spine, hence any small neighborhood will be a representative for $P$. 

When some or all annuli in the spine are not embedded but only weakly embedded, the situation is more complicated.  As described above, each annulus $A$ has a neighborhood which is orbit-equivalent to an immersed copy $N_A$ of $N$, embedded on its interior.  There are two kinds of failures of embedding which may occur.  First, it may be that the complement of the periodic orbits in $N_A$ is embedded.  Or, it is also possible that some of the unstable and stable sets in the boundary of $N_A$ also fail to be embedded.   In this latter case, since the interior of $N$ is embedded, one can see that this failure of embedding corresponds to a reflector arc in the base orbifold of $P$, and a half-face of the boundary of $N$ is mapped two-to-one into $M$ giving a self-gluing as in Construction \ref{const:arc_glue}.  

If instead the only failure occurs on the periodic orbits themselves, then one can embed a copy of $N \setminus (\alpha_1 \cup \alpha_2)$ for each such annulus $A$ in $M$, gluing boundary faces along stable/unstable sets of the periodic orbits when the local stable and unstable faces of the corresponding lozenges coincide.   Since all boundary tori of $P$ are transverse, all the stable/unstable sets get pairwise identified, and the resulting space has a product structure of the form $S \times S^1$ where $S$ is a surface homeomorphic to a surface $\sigma$ obtained by choosing a section for the Seifert fibered structure of $P$ in the complement of the singular fibers, then puncturing at any remaining periodic orbits.  To obtain the piece $P'$, one performs a Dehn filling of integral slope according to the multiplicity of the map of the Birkhoff annulus on each periodic orbit (with a trivial (1,0) surgery for embedded orbits), or equivalently the multiplicity of the singular fiber in the space $P'$.   In general, one must combine these two gluing strategies to account for both reflector arcs and cone points.  
\end{proof}

\subsection{Fatgraphs} 
A {\em fatgraph} 
$(X,\Sigma)$ is by definition a graph $X$ embedded in a surface $\Sigma$ that deformation retracts to $X$. 

Thinking of the section $I \times I \times \{0\}$ of $N$ as a 2-dimensional thickening of the ``edge" $I \times \{0\} \times \{0\}$ corresponding to the Birkhoff annulus in this section, assemblies of $N$ (glued together as trivial $S^1$ bundles, without Dehn surgeries on periodic orbits or face-gluing over reflector arcs) can be specified by the data of a fatgraph.  
The following definition, from \cite{BF13}, picks out the kind of fatgraphs that can be used to produce orientable periodic Seifert pieces with transverse 
torus boundary:  
 
\begin{definition}
A fatgraph $(X, \Sigma)$ is called \emph{admissible}\footnote{``admissible" is the terminology from \cite{BF13}} if it satisfies the following properties:
\begin{enumerate} 
\item the valence of every vertex is even, 
\item the set of boundary components of $\Sigma$ can be partitioned into two subsets (``outgoing" and ``incoming") so that for each edge $e$ of $X$, the two sides of $e$ lie in different elements of this partition, and 
\item each loop in $X$ corresponding to a boundary component contains an even number of edges.
\end{enumerate}
\end{definition}

Given an admissible fatgraph $X$, one may associate a copy of a model flow $\psi$ on $N$  to each edge of $X$,
 and glue appropriate stable or unstable leaves of the vertical orbits by flipping the $z$-coordinate direction of one of the blocks, as described above (and illustrated in Figure \ref{fig:gluings}) so that the outgoing (resp.~incoming) labels on the components of $\partial \Sigma$ correspond to the outgoing (resp.~incoming) annuli of $N$. 
 The result is a compact Seifert manifold with boundary that is a circle bundle over $\Sigma$, with a flow that is incoming on some boundary tori and outgoing on others. Item (3) in the definition ensures that the boundary components are tori rather than Klein bottles, which comes from the desire for orientability.  
 
If all vertices have valence 4, the resulting partial flow is Anosov;  higher valence vertices give singular orbits of pseudo-Anosov flows.   Vertices of valence 2 were allowed in \cite{BF13} to produce a class of examples called ``one-prong pseudo-Anosov flows" which can have quite different behavior than the pseudo-Anosov flows considered in this text. We will assume going forward that admissible fatgraphs have valence at least 4 at each vertex. 

The following theorem, originally proved in \cite{BF13}, shows that this procedure produces all possible pieces with transverse torus boundary, up to Dehn surgery.  
\begin{theorem}[\cite{BF13, BF15,  BFeM23}] \label{thm_admissible_graphs_realization}
Let $(X, \Sigma)$ be an admissible fatgraph with $\Sigma$ orientable\footnote{A similar statement should hold without the orientability conditions, but has not yet been treated in the literature.} and all vertices of valence at least 4.  Then there exists a periodic Seifert piece $P$ of a pseudo-Anosov flow on an orientable 3-manifold
whose spine can be built from the above procedure, starting with the fatgraph $(X, \Sigma)$.  Moreover, one may also modify these examples by applying an integral slope Dehn surgery along any collection of periodic orbits of $P$.

Conversely, if $P$ is a periodic Seifert piece of a pseudo-Anosov flow $\flow$ on an oriented 3-manifold with all boundary tori transverse to $\flow$, then there exists an admissible fatgraph $(X, \Sigma)$, with $\Sigma$ orientable and vertices of valence at least 4, so that $P$ is constructed by gluing a model according to $(X, \Sigma)$, then performing Dehn surgery along periodic orbits. 
\end{theorem}

The second implication follows directly from the proof of Theorem \ref{thm_periodic_are_models}, paying attention to what combinatorial instructions for gluing result in an orientable piece with transverse boundary.
Barbot and Fenley prove a special case of the first implication in \cite[Section 8]{BF13} by 
showing that, provided the number of incoming boundary components of $(X, \Sigma)$ is equal to the number of outgoing boundary components, and the parameter $\lambda$ (which we have been suppressing) from the construction of $\psi_\lambda$ in \ref{const_block} is large enough, these components can be glued in pairs by linear maps to produce a hyperbolic flow.  One thus obtains a graph manifold (with a single JSJ piece) modeled on $(X, \Sigma)$.  However, the strategy applies more generally: having constructed an assembly of blocks from the data of $(X, \Sigma)$, one could take a second copy of it with the flow reversed and glue outgoing to incoming boundary tori using the same strategy of proof to show it is pseudo-Anosov.  This is sufficient to prove Theorem \ref{thm_admissible_graphs_realization}.

The machinery introduced in \cite{BBY17} gives many more options for inserting a piece modeled on a fatgraph $(X, \Sigma)$ into an Anosov flow.  Multiple Seifert pieces can be glued together in many configurations, and one can also glue any number of additional {\em attractors} and {\em repellers} onto any remaining boundary components.   This flexibility is described by the following proposition.  

\begin{proposition}[See \cite{BFeM23}]\label{prop_gluing_fatgraphs} 
Let $X_1, \dots, X_n$ be admissible fatgraphs with valence at least 4. 
Let $(M_i,\phi_i)$, $i=1, \dots, n$, be partial flows built from this data, and call $S_{in}$ and $S_{out}$ the union of the boundary surfaces that are respectively entering and outgoing. Suppose that $h$ is a diffeomorphism from a subset $\hat S_{in}$ of $S_{in}$ to a subset $\hat S_{out}$ of $S_{out}$, such that the image by $h$ of the stable boundary lamination is not isotopic to the unstable boundary lamination. 

Then there exist (not necessarily connected) manifolds $A, E$ with torus and Klein bottle boundary components, flows $\psi_A$ and $\psi_E$ on $A$ and $E$ respectively which are incoming on $\partial A$ and outgoing on $\partial E$, diffeomorphisms $f_A\colon \partial A \to S_{out}\smallsetminus\hat S_{out}$ and $f_E\colon \partial E \to S_{in}\smallsetminus\hat S_{in}$ such that the flow obtained by gluing $(M_i,\phi_i)$, $(A,\psi_A)$ and $(E,\psi_E)$ using the diffeomorphisms $f_A$, $f_E$, and $h$ (or a modification of $h$ by an isotopy) is a pseudo-Anosov flow on a closed manifold.
\end{proposition}

The techniques described above give a  powerful tool to build examples of pseudo-Anosov flows with a wide range of properties, which can often be specified by giving a particular structure to the fatgraph before assembling blocks into a periodic Seifert piece.  
To give some examples, in  \cite{BFeM23} this  is used to build flows which are not orbit equivalent, but have the same free homotopy classes of periodic orbits.  In \cite[Section 8]{BBM24b}, it is used to build exotic examples of nontransitive pseudo-Anosov flows, with nonwandering set disagreeing with the chain recurrent set in different ways.  

\subsection{Classifying totally periodic flows} 
Now that we know how to build examples of flows out of periodic Seifert pieces, we would like to have an invariant to distinguish them up to orbit equivalence.  Since, on an individual piece, this is determined by the spine, one might naively guess that knowing the spine on all pieces of a {\em totally periodic flow}, that is, a flow on a graph manifold where all pieces are periodic Seifert fibered, is enough to determine the flow.  However, this is not quite enough.  Examples from \cite{BF15} show that flows may have identical spines, but not be orbit equivalent by any map isotopic to the identity.  In \cite[Section 7.1]{BFeM23} there are examples which have the same spine but are not orbit equivalent at all.  The additional data needed to get a complete invariant is to preserve the direction of the orbits in the boundaries of the Birkhoff annuli of the spine.  Following standard terminology, these are called {\em vertical orbits}.  More formally, we have the following: 

\begin{theorem}[Barbot--Fenley \cite{BF15}]\label{thm_classification_totally_periodic}
Let $\flow_1,\flow_2$ be two totally periodic pseudo-Anosov flows on $M$. Let $P_i$ be the Seifert pieces of the decomposition of $M$.  Let $Z_{1,i}$ and $Z_{2,i}$ be the spines of respectively $\flow_1$ and $\flow_2$ in $P_i$ for each $i$.
Then $\flow_1$ and $\flow_2$ are orbit equivalent by a homeomorphism preserving direction of the flow if and only if there exists a homeomorphism of $M$ mapping the collection of spines $\{Z_{1,i}\}$ to $\{Z_{2,i}\}$ while preserving the orientation of each vertical orbit.
\end{theorem}
Immediate from the statement is that $\flow_1$ and $\flow_2$ are orbit equivalent by a homeomorphism {\em reversing} direction of the flow if and only if there exists a homeomorphism of $M$ mapping the collection of spines $\{Z_{1,i}\}$ to $\{Z_{2,i}\}$ while reversing the orientation of each vertical orbit.  The pathological examples from \cite[Section 7.1]{BFeM23} are constructed such that any candidate for orbit equivalence would need to preserve some, and reverse other, directions of vertical orbits.  

We now comment briefly on the proof.  
The fact that two flows that are orbit equivalent will have the same spines and orientations (up to global reversal) of vertical orbits 
follows directly from the construction of the spine in terms of the orbit space of the flow, as we did in the proof of Theorem \ref{thm_spines_exist}.  So the difficult direction of the theorem is to show that the data of the spine and vertical orbit direction is enough to build the orbit equivalence.

To prove this, Barbot and Fenley show that one can uniquely characterize each orbit in $\wt M$ (except for the vertical ones) by
the bi-infinite sequence of the lifts of the boundary surfaces $\partial P_i$ that the orbit intersects.  This sequence is called an {\em itinerary}. The key fact here is that the boundary tori of two periodic Seifert pieces can always be realized transverse to the flow (by Proposition \ref{prop_scalloped_surface}), since the Seifert fiber elements coming from the two sides are algebraically independent and are both represented by periodic orbits of the flow.  This also motivated the theory of admissible fatgraphs to describe such pieces.  

Thus, an isomorphism of the spines gives an identification of itineraries, and hence a natural group-equivariant bijection between the orbit space of one flow and the other. To finish the proof, one needs to show that this bijection is actually continuous.  This is the most technical part of the proof in \cite{BF15}.

We note also that there is an alternative approach in the special case of transitive flows, by using the main theorem of \cite{BFM22}.   
Without considering all itineraries, one can show much more easily that the {\em free homotopy data} of two flows with homeomorphic spines, via a homeomorphism preserving direction on vertical orbits, agree, and then apply the classification theorem for transitive flows via their free homotopy data.  See \cite[Theorem 1.6]{BFeM23}.

\section{Free Seifert pieces}\label{sec_free_seifert} 

Having described the structure of periodic Seifert pieces, we now turn to free pieces, which are the subject of 
 \cite{BF21}. 
 
 We have already seen one example (in fact a family of examples) of a flow built from free Seifert pieces.  This is the 
Handel--Thurston construction, described in Sections \ref{sec_example_Handel_Thurston} and \ref{sec_Handel_Thurston_skew}.  Recall this flow is obtained by (re)-gluing together one or two pieces of a geodesic flow along a quasi-transverse torus.  This quasi-transverse torus becomes a JSJ torus in the re-glued manifold, and Lemma \ref{lem_HT_free} shows that each JSJ piece is free. 

Pieces of geodesic flows can be inserted into other Anosov flows in a more general way using the gluing theorem of \cite{Pau24}, producing more examples of flows with free Seifert pieces.  We describe this procedure below.   Remarkably, this kind of ``insertion of geodesic flow" describes essentially all possible examples: the main result of \cite{BF21} says (in a way stated more precisely in Theorem \ref{thm_free_pieces}) that all free pieces, up to passing to covers, resemble glued-in parts of geodesic flows.  

\subsection{Building free pieces} 
Let $S$ be a hyperbolic surface with geodesic boundary, and $T^1S$ its unit tangent bundle.  Let $\psi$ be the geodesic flow on $T^1 S$; of course this is only a partial flow since it exits and enters the boundary, which consists of a union of quasi-transverse tori or Klein bottles.  For simplicity, in this discussion we will assume the boundary components are tori.  
Each torus boundary component is the union of two elementary Birkhoff annuli, one entering $T^1S$ the other exiting $T^1S$, and so is both alternating and maximally periodic in the sense of Definition \ref{def_quasi_transverse_subclasses}.   
Note also that, lifting to any finite fiberwise cover of $T^1 S$, the boundary remains alternating and maximally periodic. 

The Gluing Theorem \cite[Theorem A]{Pau24} gives conditions under which one can glue such a flow $\psi$ (or a lift to a fiberwise cover) to different ``building blocks" to obtain an Anosov flow.  A {\em building block} in \cite{Pau24} is a compact 3-manifold $P$ with boundary, with a $C^1$ vector field such that the maximal invariant set $\Lambda_P$ of the generated flow is
(saddle) hyperbolic, and such that $\partial P$ is alternating quasi-transverse.  Such building blocks generalize the notion of hyperbolic plugs of \cite{BBY17} described in the previous section, by replacing the condition of having a transverse boundary with more general alternating quasi-transverse boundaries.  The geodesic flow pieces and lifts to covers described above, and more generally any Seifert piece of an Anosov flow with alternating quasi-transverse boundary, are all examples. 
 
Paulet shows that building blocks can be glued together to produce Anosov flows provided that the one-dimensional laminations on $\partial P$ induced by $\cF^s(\Lambda_P)$ and $\cF^u(\Lambda_P)$ are ``filling" (a technical condition that is often easily checkable in practice) and the 
gluing map $f$ is {\em strongly quasi-transverse}, meaning that the one-dimensional laminations from each piece are glued in such a way that they can be extended to a pair of foliations on the glued torus that are transverse except along the periodic orbits in $\partial P$, where they coincide.
  
 The gluing in the Handel-Thurston example satisfies these conditions; the stable and unstable laminations between two periodic orbits are shown in Figure \ref{fig_different_number} (left).  But even more general cases apply.  For instance, one can glue two quasi-transverse tori with a different number of closed leaves of $\cF^s$ and $\cF^u$, as in Figure \ref{fig_different_number} (right). 
\begin{figure} 
\centerline{ \mbox{
\includegraphics[width=12cm]{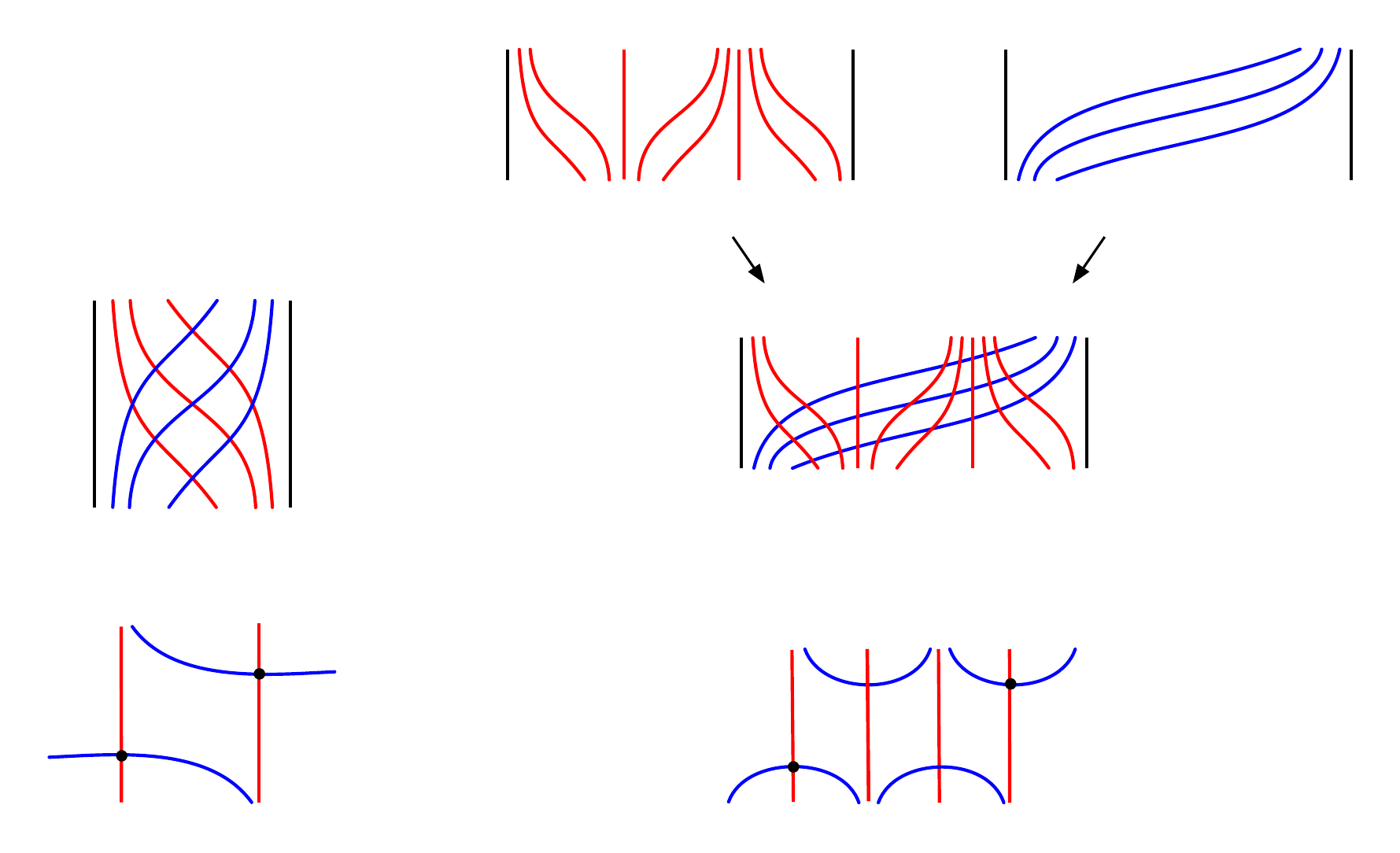}}}
\caption{On the left, the Handel-Thurston gluing.  On the right, an example of a gluing of annuli between two periodic orbits with topologically distinct foliations. The resulting trace of the Birkhoff annuli is a line.}
\label{fig_different_number}
\end{figure}

Applying this gluing procedure to the geodesic flow on a surface $S$ with boundary (by gluing it to itself, if there are multiple boundary components, and/or to another block) gives a flow $\flow$ on a manifold $M$ with the following properties:
\begin{enumerate}
\item $T^1 S$ is a free Seifert piece for $\flow$;
\item The maximal hyperbolic set $\Lambda = \Lambda_{T^1 S}$ for $T^1S$ is invariant by $\flow$ and the restriction of $\flow$ to $\Lambda$ agrees with the restriction of $\psi$.  \footnote{In fact, one can build an orbit equivalence between the partial flows $\flow|_{T^1 S}$ and $\psi|_{T^1S }$ on the whole of $T^1S$, but that orbit equivalence will only preserve the stable and unstable foliations of $\Lambda$.}
\item Each boundary component $T$ is an alternating quasi-transverse torus for $\flow$, which is the union of two Birkhoff annuli;
\item Depending on the choice of the building block glued to $\psi$ (see \cite[Section 7]{Pau24}), the Birkhoff annulus contained in $T$ where the flow exits $T^1 S$ can have any even number of closed leaves of the stable foliation. Similarly, the entering Birkhoff annulus can have any even number of closed leaves of the unstable foliation. 
\end{enumerate}
Equivalent to the last condition is that the trace of the Birkhoff annuli in $T$ can be a line of lozenges instead of a single lozenge, as shown in Figure \ref{fig_different_number}.

This shows that one can always build examples of Anosov flows with free Seifert pieces, such that the restriction of the flow to the maximal invariant of the piece is (up to finite covers) a lift of a geodesic flow on a surface with boundary, and where the induced stable and unstable foliations on the boundary surfaces of the piece may (or may not) have extra closed leaves.  We next discuss a converse to this statement.  

\subsection{Classification of free pieces} 
The following is a special case of the main theorem of \cite{BF21}, under the simplifying assumption that the union of quasi-transverse tori in $P$ are embedded.  

\begin{theorem}[Barbot--Fenley \cite{BF21}]\label{thm_free_pieces}
Let $P$ be a free Seifert piece of a pseudo-Anosov flow $\flow$ on a manifold $M$ with base (topologically) a hyperbolic surface $S$. Assume that all the boundaries of $P$ are embedded quasi-transverse tori, and let  $\Lambda$ denote the maximal invariant of $\flow|_P$.  Then there exists a flow $\psi$ on $P$ which is orbit equivalent to a lift of the geodesic flow on $S$, whose maximal invariant is $\Lambda$, and such that $\psi|_\Lambda$ is orbit equivalent to $\flow|_\Lambda$.
\end{theorem}

If boundary tori are not embedded, one can instead obtain an analogous result after lifting to a cover of $P$.  See \cite[Theorem 7.1]{BF21} for a detailed statement.   
However, if one is willing to pass to covers, one can do even better --  \cite[Lemma 9.2]{FPPR21}, together with the precise description of non-embedded tori in \cite[Theorem 6.10]{BF13} shows that one can find a finite cover $\hat M$ of $M$ such that all the boundary tori of the lift $\hat P$ in $\hat M$ are embedded, and then Theorem \ref{thm_free_pieces} applies directly.  

As an immediate corollary, we have:
\begin{corollary}
If $P$ is a free Seifert piece of the Barbot-Fenley modified JSJ decomposition of $\flow$, then $\flow$ does not have any singular orbit in the interior of $P$.
\end{corollary}
Note however that the quasi-transverse boundary of $P$ may contain singular periodic orbits, and such examples do indeed exist.  

Giving a complete proof of Theorem \ref{thm_free_pieces} is beyond the scope of this text, but much like we did for 
Theorem \ref{thm_classification_totally_periodic}, we are well equipped here to describe the main ideas. 
The guiding principle is to show that the action of $\pi_1(P)$ on $\orb$ has an invariant copy of $\bR$ in each leaf space, allowing one to produce an action of $\pi_1(P)$ on $\bR$ which resembles that of an action coming from a skew plane.  By ``resembles" we mean that one can modify this action in a prescribed way to get an {\em extended convergence group} action in the sense of Thurston \cite{Thu97}, and thus, it is induced by a skew flow, as in Theorem \ref{thm_extended_convergence} (more precisely, one uses a version of Thurston's characterization applicable to manifolds with boundary).  
The control on the precise type of modification is enough to see that the maximal invariant survives, and hence there is an orbit equivalence with this skew flow on the maximal invariant. Finally, because the fiber element is central in $\pi_1(P)$, the same line of argument as in Theorem \ref{thm_normal_abelian_implies_rigidity} shows that some power of the one step up map is represented by the central element of $\pi_1(P)$ and so this skew flow is in fact a cover of a geodesic flow on the unit tangent bundle of a hyperbolic orbifold.  

We now explain how the invariant $\bR$ and the action are obtained.
Consider $h\in \pi_1(P)$ the element representing a regular Seifert fiber element. By definition, $h$ acts freely and thus by Theorem \ref{thm_axes_exist} preserves axes $A^{s}\subset \Lambda(\wt\cF^{s})$  and $A^{u} \subset \Lambda(\wt\cF^{u})$ in the stable and unstable leaf spaces, respectively.  

For concreteness, we focus on $A^s$.  Recall from Proposition \ref{prop:axis_structure} and Lemma \ref{lem_commuting_axis_2} that if $A^{s}$ is not a properly embedded copy of $\bR$, then it is either the stable saturation of a scalloped region $U$ stabilized by $h$, or a union of intervals indexed by $\bZ$, such that consecutive intervals have nonseparated boundary points.  
In the first case we deduce that $\pi_1(P)$ must stabilize $U$ and hence must contain a finite-index subgroup isomorphic to $\bZ^2$ by Proposition \ref{prop_stabilizer_scalloped_withC}. In the second case, one obtains an epimorphism of $\pi_1(P)$ onto $\bZ$, with kernel contained in the stabilizers of the non-separated leaves, which is virtually cyclic.   So in either case, $\pi_1(P)$ is virtually a direct or (semi)-direct product $\bZ\rtimes\bZ$. This contradicts the fact that the base is a hyperbolic orbifold.   

The same argument applies to $A^u$, so we can conclude that both axes $A^s$ and $A^u$ are properly embedded copies of $\bR$. Thus, one can apply Theorem \ref{thm_both_axes_R} to get that the set
\[
\Omega = \{ x\in \orb : \bfs(x) \in A^s \text{ and }\bfu(x) \in A^u\}
\]
is a skew-like piece of $\orb$, bounded by the monotone and $\pi_1(P)$-equivariant maps $i(l^s)=\inf\{l^u\in A^u \mid l^s \cap l^u \neq \emptyset\}$ and $s(l^s)=\sup\{l^u\in A^u \mid l^s \cap l^u \neq \emptyset\}$.

One can similarly build maps  $i',s' \colon A^u \to A^s$ that are also monotone and $\pi_1(P)$-equivariant. With this skew-like structure, one builds a suitably modified version of the  ``one-step-up map'' $A^s \to A^s$ such that the restriction of the action of $\pi_1(P)$ on $A^s$ to its minimal subset commutes with this map and is semiconjugate to an extended convergence group action of $\pi_1(P)$ on $\bR$, and applies Thurston's theorem as described above. \qed

%% file: main_monograph_EMS_Barthelme_Mann_vector.bbl
\begin{thebibliography}{100}
\providecommand{\url}[1]{\texttt{#1}}
\providecommand{\urlprefix}{URL }
\providecommand{\eprint}[2][]{\url{#2}}

\bibitem{FLP79}
\emph{Travaux de {T}hurston sur les surfaces}. Ast\'erisque 66-67,
  Soci\'et\'e{} Math\'ematique de France, Paris, 1979

\bibitem{Ada87}
T.~Adachi, Closed orbits of an {Anosov} flow and the fundamental group.
  \emph{Proc. Am. Math. Soc.} \textbf{100} (1987), 595--598

\bibitem{Ago11}
I.~Agol, \href{https://doi.org/10.1090/conm/560/11087}{Ideal triangulations of
  pseudo-{A}nosov mapping tori}. In \emph{Topology and geometry in dimension
  three}, pp. 1--17, Contemp. Math. 560, Amer. Math. Soc., Providence, RI, 2011

\bibitem{Ago13}
I.~Agol, The virtual {H}aken conjecture. \emph{Doc. Math.} \textbf{18} (2013),
  1045--1087

\bibitem{AT24}
I.~Agol and C.~C. Tsang,
  \href{https://doi.org/10.2140/agt.2024.24.3401}{Dynamics of veering
  triangulations: infinitesimal components of their flow graphs and
  applications}. \emph{Algebr. Geom. Topol.} \textbf{24} (2024), no.~6,
  3401--3453

\bibitem{Ano63}
D.~V. Anosov, Ergodic properties of geodesic flows on closed {R}iemannian
  manifolds of negative curvature. \emph{Dokl. Akad. Nauk SSSR} \textbf{151}
  (1963), 1250--1252

\bibitem{Ano67}
D.~V. Anosov, Geodesic flows on closed {R}iemannian manifolds of negative
  curvature. \emph{Trudy Mat. Inst. Steklov.} \textbf{90} (1967), 209

\bibitem{AS67}
D.~V. Anosov and J.~G. Sina\u{i}, Some smooth ergodic systems. \emph{Uspehi
  Mat. Nauk} \textbf{22} (1967), no. 5(137), 107--172

\bibitem{Asa08}
M.~Asaoka, \href{https://doi.org/10.1007/s00222-008-0151-9}{On invariant
  volumes of codimension-one {A}nosov flows and the {V}erjovsky conjecture}.
  \emph{Invent. Math.} \textbf{174} (2008), no.~2, 435--462

\bibitem{ABM24}
M.~Asaoka, C.~Bonatti, and T.~Marty, Oriented {B}irkhoff sections of {A}nosov
  flows. \emph{J. Topol.} \textbf{17} (2024), no.~4, Paper No. e12356, 45

\bibitem{Bar95a}
T.~Barbot, Caract\'erisation des flots d'{A}nosov en dimension 3 par leurs
  feuilletages faibles. \emph{Ergodic Theory Dynam. Systems} \textbf{15}
  (1995), no.~2, 247--270

\bibitem{Bar95b}
T.~Barbot, Mise en position optimale de tores par rapport \`a un flot
  d'{A}nosov. \emph{Comment. Math. Helv.} \textbf{70} (1995), no.~1, 113--160

\bibitem{Bar96}
T.~Barbot, \href{https://doi.org/10.5802/aif.1556}{Flots d'{A}nosov sur les
  vari\'et\'es graph\'ees au sens de {W}aldhausen}. \emph{Ann. Inst. Fourier
  (Grenoble)} \textbf{46} (1996), no.~5, 1451--1517

\bibitem{Bar98b}
T.~Barbot, Actions de groupes sur les 1-vari\'{e}t\'{e}s non s\'{e}par\'{e}es
  et feuilletages de codimension un. \emph{Ann. Fac. Sci. Toulouse Math. (6)}
  \textbf{7} (1998), no.~4, 559--597

\bibitem{Bar98}
T.~Barbot, \href{https://doi.org/10.4310/CAG.1998.v6.n4.a5}{Generalizations of
  the {Bonatti}-{Langevin} example of {Anosov} flow and their classification up
  to topological equivalence}. \emph{Commun. Anal. Geom.} \textbf{6} (1998),
  no.~4, 749--798

\bibitem{Bar05}
T.~Barbot, \emph{De l'hyperbolique au globalement hyperbolique}. Habilitation
  \`a diriger des recherches, Universit\'e Claude Bernard de Lyon, 2005,
  \urlprefix\url{https://theses.hal.science/tel-00011278/document}

\bibitem{BF13}
T.~Barbot and S.~Fenley, Pseudo-{A}nosov flows in toroidal manifolds.
  \emph{Geom. Topol.} \textbf{17} (2013), no.~4, 1877--1954

\bibitem{BF15}
T.~Barbot and S.~Fenley, Classification and rigidity of totally periodic
  pseudo-{A}nosov flows in graph manifolds. \emph{Ergodic Theory Dynam.
  Systems} \textbf{35} (2015), no.~6, 1681--1722

\bibitem{BF21}
T.~Barbot and S.~R. Fenley, Free {S}eifert pieces of pseudo-{A}nosov flows.
  \emph{Geom. Topol.} \textbf{25} (2021), no.~3, 1331--1440

\bibitem{BBG21}
T.~Barthelm\'e, C.~Bonatti, A.~Gogolev, and F.~Rodriguez~Hertz,
  \href{https://doi.org/10.1112/plms.12321}{Anomalous {A}nosov flows
  revisited}. \emph{Proc. Lond. Math. Soc. (3)} \textbf{122} (2021), no.~1,
  93--117

\bibitem{BBM24b}
T.~Barthelm\'e, C.~Bonatti, and K.~Mann, Non-transtive pseudo-{A}nosov flows.
  \emph{ArXiv}  (2024)

\bibitem{BBM24a}
T.~Barthelm\'e, C.~Bonatti, and K.~Mann,
  \href{https://doi.org/10.1007/s00208-025-03200-y}{Completing prelaminations}.
  \emph{Math. Ann.}  (2025)

\bibitem{BFeM23}
T.~Barthelm\'e, S.~Fenley, and K.~Mann,
  \href{https://doi.org/10.3934/jmd.2025006}{Anosov flows with the same
  periodic orbits}. \emph{J. Mod. Dyn} \textbf{21} (2025), 361--400

\bibitem{BFM25}
T.~Barthelm\'e, S.~Fenley, and K.~Mann, Reconstructing flows from the orbit
  space. \emph{arXiv preprint}  (2025)

\bibitem{BFP23}
T.~Barthelm{\'e}, S.~R. Fenley, and R.~Potrie,
  \href{https://doi.org/10.4171/CMH/557}{Collapsed {Anosov} flows and self
  orbit equivalences}. \emph{Comment. Math. Helv.} \textbf{98} (2023), no.~4,
  771--875

\bibitem{BFM22}
T.~Barthelm{\'e}, S.~Frankel, and K.~Mann,
  \href{https://doi.org/10.1007/s00222-025-01332-1}{Orbit equivalences of
  pseudo-{Anosov} flows}. \emph{Invent. Math.} \textbf{240} (2025), no.~3,
  1119--1192

\bibitem{BMB24}
T.~Barthelm\'e and K.~Mann, \href{https://doi.org/10.2140/gt.2024.28.867}{Orbit
  equivalences of {{\(\mathbb{R}\)}}-covered {Anosov} flows and hyperbolic-like
  actions on the line. {With} an appendix written jointly with {Jonathan}
  {Bowden}}. \emph{Geom. Topol.} \textbf{28} (2024), no.~2, 867--899

\bibitem{BBY17}
F.~B\'{e}guin, C.~Bonatti, and B.~Yu, Building {A}nosov flows on 3-manifolds.
  \emph{Geom. Topol.} \textbf{21} (2017), no.~3, 1837--1930

\bibitem{Bon02}
F.~Bonahon, Geometric structures on 3-manifolds. In \emph{Handbook of geometric
  topology}, pp. 93--164, North-Holland, Amsterdam, 2002

\bibitem{Bon23}
C.~Bonatti, \href{https://doi.org/10.4171/lem/1086}{Action on the circle at
  infinity of foliations of {$\mathbb{R}^2$}}. \emph{Enseign. Math.}
  \textbf{72} (2026), no. 1-2, 85--146

\bibitem{BGP16}
C.~Bonatti, A.~Gogolev, and R.~Potrie,
  \href{https://doi.org/10.1007/s00222-016-0663-7}{Anomalous partially
  hyperbolic diffeomorphisms {II}: stably ergodic examples}. \emph{Invent.
  Math.} \textbf{206} (2016), no.~3, 801--836

\bibitem{BI23}
C.~Bonatti and I.~Iakovoglou,
  \href{https://doi.org/10.1017/etds.2021.170}{Anosov flows on 3-manifolds: the
  surgeries and the foliations}. \emph{Ergodic Theory Dynam. Systems}
  \textbf{43} (2023), no.~4, 1129--1188

\bibitem{BL94}
C.~Bonatti and R.~Langevin, An example of transitive {Anosov} flows transversal
  to a torus and not conjugated to a suspension. \emph{Ergodic Theory Dyn.
  Syst.} \textbf{14} (1994), no.~4, 633--643

\bibitem{BW72}
R.~Bowen and P.~Walters, Expansive one-parameter flows. \emph{J. Differential
  Equations} \textbf{12} (1972), 180--193

\bibitem{Bro12}
L.~E.~J. Brouwer, \href{https://doi.org/10.1007/BF01456888}{Beweis des ebenen
  {T}ranslationssatzes}. \emph{Math. Ann.} \textbf{72} (1912), no.~1, 37--54

\bibitem{Brunella}
M.~Brunella, Separating the basic sets of a nontransitive {A}nosov flow.
  \emph{Bull. London Math. Soc.} \textbf{25} (1993), no.~5, 487--490

\bibitem{Cal01}
D.~Calegari, \href{https://doi.org/10.2140/agt.2001.1.579}{Leafwise smoothing
  laminations}. \emph{Algebr. Geom. Topol.} \textbf{1} (2001), 579--585

\bibitem{Cal06}
D.~Calegari, \href{https://doi.org/10.2140/gt.2006.10.2271}{Universal circles
  for quasigeodesic flows}. \emph{Geom. Topol.} \textbf{10} (2006), 2271--2298

\bibitem{Cal07}
D.~Calegari, \emph{Foliations and the geometry of 3-manifolds}. Oxford
  Mathematical Monographs, Oxford University Press, Oxford, 2007

\bibitem{CD03}
D.~Calegari and N.~M. Dunfield,
  \href{https://doi.org/10.1007/s00222-002-0271-6}{Laminations and groups of
  homeomorphisms of the circle}. \emph{Invent. Math.} \textbf{152} (2003),
  no.~1, 149--204

\bibitem{Cam25}
M.~Camargo, \emph{On bifoliated planes, their structure and group actions}.
  Ph.D. thesis, 2025

\bibitem{CC03}
A.~Candel and L.~Conlon,
  \emph{\href{https://doi.org/10.1090/gsm/060}{Foliations. {II}}}. Graduate
  Studies in Mathematics 60, American Mathematical Society, Providence, RI,
  2003

\bibitem{CJ94}
A.~Casson and D.~Jungreis,
  \href{https://doi.org/10.1007/BF01231540}{Convergence groups and {S}eifert
  fibered {$3$}-manifolds}. \emph{Invent. Math.} \textbf{118} (1994), no.~3,
  441--456

\bibitem{dlH00}
P.~de~la Harpe, \emph{Topics in geometric group theory}. Chicago Lectures in
  Mathematics, University of Chicago Press, Chicago, IL, 2000

\bibitem{DS19}
P.~Dehornoy and M.~Shannon, Almost equivalence of algebraic {A}nosov flows.
  \emph{ArXiv}  (2019)

\bibitem{FM12}
B.~Farb and D.~Margalit, \emph{A primer on mapping class groups}. Princeton
  Mathematical Series 49, Princeton University Press, Princeton, NJ, 2012

\bibitem{Fen94}
S.~Fenley, Anosov flows in {$3$}-manifolds. \emph{Ann. of Math. (2)}
  \textbf{139} (1994), no.~1, 79--115

\bibitem{Fen95}
S.~Fenley, \href{https://doi.org/10.1007/BF02566007}{One sided branching in
  {A}nosov foliations}. \emph{Comment. Math. Helv.} \textbf{70} (1995), no.~2,
  248--266

\bibitem{Fen95b}
S.~Fenley, Quasigeodesic {A}nosov flows and homotopic properties of flow lines.
  \emph{J. Differential Geom.} \textbf{41} (1995), no.~2, 479--514

\bibitem{Fen97}
S.~Fenley, Homotopic indivisibility of closed orbits of {$3$}-dimensional
  {A}nosov flows. \emph{Math. Z.} \textbf{225} (1997), no.~2, 289--294

\bibitem{Fen98}
S.~Fenley, The structure of branching in {A}nosov flows of {$3$}-manifolds.
  \emph{Comment. Math. Helv.} \textbf{73} (1998), no.~2, 259--297

\bibitem{Fen03}
S.~Fenley, \href{https://doi.org/10.1023/A:1024953221158}{Pseudo-{A}nosov flows
  and incompressible tori}. \emph{Geom. Dedicata} \textbf{99} (2003), 61--102

\bibitem{Fen12}
S.~Fenley, Ideal boundaries of pseudo-{A}nosov flows and uniform convergence
  groups with connections and applications to large scale geometry. \emph{Geom.
  Topol.} \textbf{16} (2012), no.~1, 1--110

\bibitem{Fen16}
S.~Fenley, Quasigeodesic pseudo-{A}nosov flows in hyperbolic 3-manifolds and
  connections with large scale geometry. \emph{Adv. Math.} \textbf{303} (2016),
  192--278

\bibitem{FMP26}
S.~Fenley, K.~Mann, and R.~Potrie, Exotic codimension one anosov flows. 2026,
  \urlprefix\url{https://arxiv.org/abs/2605.25082},
  \href{https://arxiv.org/abs/2605.25082}{arXiv:2605.25082},
  \eprint{2605.25082}

\bibitem{FM01}
S.~Fenley and L.~Mosher, Quasigeodesic flows in hyperbolic 3-manifolds.
  \emph{Topology} \textbf{40} (2001), no.~3, 503--537

\bibitem{FH19}
T.~Fisher and B.~Hasselblatt, \emph{Hyperbolic flows}. Zurich Lectures in
  Advanced Mathematics, 2019

\bibitem{FH13}
P.~Foulon and B.~Hasselblatt, Contact {A}nosov flows on hyperbolic 3-manifolds.
  \emph{Geom. Topol.} \textbf{17} (2013), no.~2, 1225--1252

\bibitem{Fra13}
S.~Frankel, Quasigeodesic flows and {M}\"obius-like groups. \emph{J.
  Differential Geom.} \textbf{93} (2013), no.~3, 401--429

\bibitem{FSS19}
S.~Frankel, S.~Schleimer, and H.~Segerman, From veering triangulations to link
  spaces and back again. Preprint, {arXiv}:1911.00006 [math.{GT}] (2019), 2019,
  \urlprefix\url{https://arxiv.org/abs/1911.00006}

\bibitem{Fra70}
J.~Franks, Anosov diffeomorphisms. In \emph{Global {A}nalysis ({P}roc.
  {S}ympos. {P}ure {M}ath., {V}ols. {XIV}, {XV}, {XVI}, {B}erkeley, {C}alif.,
  1968)}, pp. 61--93, Proc. Sympos. Pure Math. XIV-XVI, Amer. Math. Soc.,
  Providence, RI, 1970

\bibitem{Fra92}
J.~Franks, \href{https://doi.org/10.1017/S0143385700006702}{A new proof of the
  {B}rouwer plane translation theorem}. \emph{Ergodic Theory Dynam. Systems}
  \textbf{12} (1992), no.~2, 217--226

\bibitem{FW80}
J.~Franks and B.~Williams, Anomalous {A}nosov flows. In \emph{Global theory of
  dynamical systems ({P}roc. {I}nternat. {C}onf., {N}orthwestern {U}niv.,
  {E}vanston, {I}ll., 1979)}, pp. 158--174, Lecture Notes in Math. 819,
  Springer, Berlin, 1980

\bibitem{Fri82}
D.~Fried, \href{https://doi.org/10.1016/0040-9383(82)90017-9}{The geometry of
  cross sections to flows}. \emph{Topology} \textbf{21} (1982), no.~4, 353--371

\bibitem{Fri83}
D.~Fried, Transitive {A}nosov flows and pseudo-{A}nosov maps. \emph{Topology}
  \textbf{22} (1983), no.~3, 299--303

\bibitem{FPPR21}
S.~Friedl, J.~Park, B.~Petri, J.~Raimbault, and A.~Ray,
  \href{https://doi.org/10.1512/iumj.2021.70.8357}{On distinct finite covers of
  3-manifolds}. \emph{Indiana Univ. Math. J.} \textbf{70} (2021), no.~2,
  809--846

\bibitem{Gab92}
D.~Gabai, \href{https://doi.org/10.2307/2946597}{Convergence groups are
  {F}uchsian groups}. \emph{Ann. of Math. (2)} \textbf{136} (1992), no.~3,
  447--510

\bibitem{GK97}
D.~Gabai and W.~H. Kazez,
  \href{https://doi.org/10.4310/MRL.1997.v4.n4.a14}{Order trees and laminations
  of the plane}. \emph{Math. Res. Lett.} \textbf{4} (1997), no.~4, 603--616

\bibitem{GO89}
D.~Gabai and U.~Oertel, \href{https://doi.org/10.2307/1971476}{Essential
  laminations in {$3$}-manifolds}. \emph{Ann. of Math. (2)} \textbf{130}
  (1989), no.~1, 41--73

\bibitem{Ghy84}
E.~Ghys, \href{https://doi.org/10.1017/S0143385700002273}{Flots d'{A}nosov sur
  les {$3$}-vari\'et\'es fibr\'ees en cercles}. \emph{Ergodic Theory Dynam.
  Systems} \textbf{4} (1984), no.~1, 67--80

\bibitem{Ghy92}
E.~Ghys, \href{https://doi.org/10.5802/aif.1290}{Deformations of {Anosov} flows
  and of {Fuchsian} groups}. \emph{Ann. Inst. Fourier} \textbf{42} (1992), no.
  1-2, 209--247

\bibitem{Ghy01}
{\'E}.~Ghys, Groups acting on the circle. \emph{Enseign. Math. (2)} \textbf{47}
  (2001), no. 3-4, 329--407

\bibitem{Goo83}
S.~Goodman, Dehn surgery on {A}nosov flows. In \emph{Geometric dynamics ({R}io
  de {J}aneiro, 1981)}, pp. 300--307, Lecture Notes in Math. 1007, Springer,
  Berlin, 1983

\bibitem{Gro00}
M.~Gromov, Three remarks on geodesic dynamics and fundamental group.
  \emph{Enseign. Math. (2)} \textbf{46} (2000), no. 3-4, 391--402

\bibitem{Gue16}
F.~Gueritaud, \href{https://doi.org/10.1112/jtopol/jtw016}{Veering
  triangulations and {C}annon-{T}hurston maps}. \emph{J. Topol.} \textbf{9}
  (2016), no.~3, 957--983

\bibitem{Hae84}
A.~Haefliger, Groupo\"ides d'holonomie et classifiants. \emph{Ast\'erisque}
  (1984), no. 116, 70--97

\bibitem{HR57}
A.~Haefliger and G.~Reeb, Vari\'et\'es (non s\'epar\'ees) \`a{} une dimension
  et structures feuillet\'ees du plan. \emph{Enseign. Math. (2)} \textbf{3}
  (1957), 107--125

\bibitem{HT80}
M.~Handel and W.~P. Thurston, Anosov flows on new three manifolds.
  \emph{Invent. Math.} \textbf{59} (1980), no.~2, 95--103

\bibitem{Hat23}
A.~Hatcher, \emph{Notes on basic 3-manifold topology}. 2023

\bibitem{Hay01}
S.~Hayashi, Stability of dynamical systems [translation of {S}\=ugaku {\bf 50}
  (1998), no. 2, 149--162]. In \emph{Sugaku expositions}, pp. 15--29, 14, 2001

\bibitem{HH86}
G.~Hector and U.~Hirsch,
  \emph{\href{https://doi.org/10.1007/978-3-322-90115-6}{Introduction to the
  geometry of foliations. {P}art {A}}}. Second edn., Aspects of Mathematics 1,
  Friedr. Vieweg \& Sohn, Braunschweig, 1986

\bibitem{Hed39}
G.~Hedlund, \href{https://doi.org/10.1090/S0002-9904-1939-06945-0}{The dynamics
  of geodesic flows}. \emph{Bull. Amer. Math. Soc.} \textbf{45} (1939), no.~4,
  241--260

\bibitem{Hop83}
H.~Hopf, \emph{\href{https://doi.org/10.1007/978-3-662-21563-0}{Differential
  geometry in the large}}. Lecture Notes in Mathematics 1000, Springer-Verlag,
  Berlin, 1983

\bibitem{Iak24}
I.~Iakovoglou, \href{https://doi.org/10.5802/crmath.682}{Markov partitions for
  non-transitive expansive flows}. \emph{C. R., Math., Acad. Sci. Paris}
  \textbf{363} (2025), 437--444

\bibitem{IM90}
T.~Inaba and S.~Matsumoto, Nonsingular expansive flows on 3-manifolds and
  foliations with circle prong singularities. \emph{Jpn. J. Math., New Ser.}
  \textbf{16} (1990), no.~2, 329--340

\bibitem{JS79}
W.~H. Jaco and P.~B. Shalen, \href{https://doi.org/10.1090/memo/0220}{Seifert
  fibered spaces in {$3$}-manifolds}. \emph{Mem. Amer. Math. Soc.} \textbf{21}
  (1979), no. 220, viii+192

\bibitem{Joh79}
K.~Johannson, \emph{Homotopy equivalences of {$3$}-manifolds with boundaries}.
  Lecture Notes in Mathematics 761, Springer, Berlin, 1979

\bibitem{Kap41}
W.~Kaplan, Regular curve-families filling the plane, {II}. \emph{Duke Math. J.}
  \textbf{8} (1941), 11--46

\bibitem{KS81}
H.~B. Keynes and M.~Sears,
  \href{https://doi.org/10.1017/s0143385700009214}{Real-expansive flows and
  topological dimension}. \emph{Ergodic Theory Dynam. Systems} \textbf{1}
  (1981), no.~2, 179--195

\bibitem{Kov99}
N.~Kova\v{c}evi\'{c},
  \href{https://doi.org/10.1090/S0002-9947-99-02188-1}{Examples of
  {M}\"obius-like groups which are not {M}\"obius groups}. \emph{Trans. Amer.
  Math. Soc.} \textbf{351} (1999), no.~12, 4823--4835

\bibitem{LMT23}
M.~P. Landry, Y.~N. Minsky, and S.~J. Taylor, Flows, growth rates, and the
  veering polynomial. \emph{Ergodic Theory Dynam. Systems} \textbf{43} (2023),
  no.~9, 3026--3107

\bibitem{LMT24}
M.~P. Landry, Y.~N. Minsky, and S.~J. Taylor, Transverse surfaces and
  pseudo-{A}nosov flows. 2024,
  \urlprefix\url{https://arxiv.org/abs/2406.17717}, \eprint{2406.17717}

\bibitem{Man77}
R.~Ma\~n\'e, \href{https://doi.org/10.2307/1998515}{Quasi-{A}nosov
  diffeomorphisms and hyperbolic manifolds}. \emph{Trans. Amer. Math. Soc.}
  \textbf{229} (1977), 351--370

\bibitem{Man98}
B.~S. Mangum,
  \href{https://doi.org/10.1016/S0166-8641(97)00120-X}{Incompressible surfaces
  and pseudo-{A}nosov flows}. \emph{Topology Appl.} \textbf{87} (1998), no.~1,
  29--51

\bibitem{Man74}
A.~Manning, \href{https://doi.org/10.2307/2373551}{There are no new {A}nosov
  diffeomorphisms on tori}. \emph{Amer. J. Math.} \textbf{96} (1974), 422--429

\bibitem{Mar00}
G.~Margulis, \href{https://doi.org/10.1016/S0764-4442(00)01694-3}{Free
  subgroups of the homeomorphism group of the circle}. \emph{C. R. Acad. Sci.
  Paris S\'er. I Math.} \textbf{331} (2000), no.~9, 669--674 \MR{1797749}

\bibitem{Mar24}
T.~Marty, \href{https://doi.org/10.1007/s11784-024-01147-2}{Reeb-{A}nosov flows
  are skewed {$\mathbb R$}-covered, shorter proof}. \emph{J. Fixed Point Theory
  Appl.} \textbf{26} (2024), no.~4, Paper No. 58, 5

\bibitem{Mat82}
J.~Mather, Foliations of surfaces. {I}: {An} ideal boundary. \emph{Ann. Inst.
  Fourier} \textbf{32} (1982), no.~1, 235--261

\bibitem{Mos92a}
L.~Mosher, Dynamical systems and the homology norm of a {$3$}-manifold. {I}.
  {E}fficient intersection of surfaces and flows. \emph{Duke Math. J.}
  \textbf{65} (1992), no.~3, 449--500

\bibitem{Mos92b}
L.~Mosher, \href{https://doi.org/10.1007/BF01231890}{Dynamical systems and the
  homology norm of a {$3$}-manifold. {II}}. \emph{Invent. Math.} \textbf{107}
  (1992), no.~2, 243--281

\bibitem{Nav11}
A.~Navas, \emph{Groups of circle diffeomorphisms}. Spanish edn., Chicago
  Lectures in Mathematics, University of Chicago Press, Chicago, IL, 2011

\bibitem{NS97}
W.~D. Neumann and G.~A. Swarup,
  \href{https://doi.org/10.2140/gt.1997.1.21}{Canonical decompositions of
  {$3$}-manifolds}. \emph{Geom. Topol.} \textbf{1} (1997), 21--40

\bibitem{New70}
S.~E. Newhouse, \href{https://doi.org/10.2307/2373372}{On codimension one
  {A}nosov diffeomorphisms}. \emph{Amer. J. Math.} \textbf{92} (1970), 761--770

\bibitem{Nie44}
J.~Nielsen, Surface transformation classes of algebraically finite type.
  \emph{Danske Vid. Selsk. Mat.-Fys. Medd.} \textbf{21} (1944), no.~2, 89

\bibitem{Nov65}
S.~P. Novikov, The topology of foliations. \emph{Trudy Moskov. Mat. Ob\v s\v
  c.} \textbf{14} (1965), 248--278

\bibitem{Pal78}
C.~F.~B. Palmeira, Open manifolds foliated by planes. \emph{Ann. Math. (2)}
  \textbf{107} (1978), no.~1, 109--131

\bibitem{Pat93}
M.~Paternain, \href{https://doi.org/10.1007/BF01237676}{Expansive flows and the
  fundamental groups}. \emph{Bol. Soc. Bras. Mat., Nova S{\'e}r.} \textbf{24}
  (1993), no.~2, 179--199

\bibitem{Pau23}
N.~Paulet, \href{https://doi.org/10.5802/mrr.17}{Construction of {A}nosov flows
  in dimension 3 by gluing blocks}. \emph{Math. Res. Rep.} \textbf{4} (2023),
  47--62

\bibitem{Pau24}
N.~Paulet, Anosov flows in dimension 3 from gluing building blocks with
  quasi-transverse boundary. \emph{J. Mod. Dyn.~to appear}  (2024)

\bibitem{Pla81}
J.~F. Plante, {A}nosov flows, transversely affine foliations, and a conjecture
  of {V}erjovsky. \emph{J. London Math. Soc. (2)} \textbf{23} (1981), no.~2,
  359--362

\bibitem{PT72}
J.~F. Plante and W.~P. Thurston,
  \href{https://doi.org/10.1016/0040-9383(72)90002-X}{Anosov flows and the
  fundamental group}. \emph{Topology} \textbf{11} (1972), 147--150

\bibitem{Pol87}
M.~Pollicott, \href{https://doi.org/10.2307/2374558}{Symbolic dynamics for
  {S}male flows}. \emph{Amer. J. Math.} \textbf{109} (1987), no.~1, 183--200

\bibitem{Pot25}
R.~Potrie, \href{https://doi.org/10.1007/s11784-025-01175-6}{Anosov flows in
  dimension 3: an outside look}. \emph{J. Fixed Point Theory Appl.} \textbf{27}
  (2025), no.~1, 47

\bibitem{RS01}
R.~Roberts and M.~Stein,
  \href{https://doi.org/10.1016/S0166-8641(00)00060-2}{Group actions on order
  trees}. \emph{Topology Appl.} \textbf{115} (2001), no.~2, 175--201

\bibitem{Sal23}
F.~Salmoiraghi, Goodman surgery and projectively {A}nosov flows. 2023,
  \urlprefix\url{https://arxiv.org/abs/2202.01328}, \eprint{2202.01328}

\bibitem{Sal24}
F.~Salmoiraghi, Surgery on {A}nosov flows using bi-contact geometry. 2024,
  \urlprefix\url{https://arxiv.org/abs/2104.07109}, \eprint{2104.07109}

\bibitem{RS17}
B.~Scardua and C.~A.~M. Rojas,
  \emph{\href{https://doi.org/10.1142/10366}{Geometry, dynamics and topology of
  foliations}}. World Scientific Publishing Co. Pte. Ltd., Hackensack, NJ, 2017

\bibitem{SS23}
S.~Schleimer and H.~Segerman, From veering triangulations to dynamic pairs.
  Preprint, {arXiv}:2305.08799 [math.{GT}], 2023,
  \urlprefix\url{https://arxiv.org/abs/2305.08799}

\bibitem{Sha21}
M.~Shannon, Hyperbolic models for transitive topological {A}nosov flows in
  dimension three. 2021, \eprint{2108.12000}

\bibitem{Sma67}
S.~Smale, Differentiable dynamical systems. \emph{Bull. Amer. Math. Soc.}
  \textbf{73} (1967), 747--817

\bibitem{Sol82}
V.~V. Solodov, Components of topological foliations. \emph{Mat. Sb. (N.S.)}
  \textbf{119(161)} (1982), no.~3, 340--354, 447

\bibitem{Thu97}
W.~Thurston, 3-manifolds, foliations and circles {I}. 1997,
  \urlprefix\url{https://arxiv.org/pdf/math/9712268}

\bibitem{Thu22}
W.~Thurston, \emph{The geometry and topology of three-manifolds. {V}ol. {IV}}.
  American Mathematical Society, Providence, RI, 2022

\bibitem{Thu82}
W.~P. Thurston,
  \href{https://doi.org/10.1090/S0273-0979-1982-15003-0}{Three-dimensional
  manifolds, {K}leinian groups and hyperbolic geometry}. \emph{Bull. Amer.
  Math. Soc. (N.S.)} \textbf{6} (1982), no.~3, 357--381

\bibitem{Tit70}
J.~Tits, Sur le groupe des automorphismes d'un arbre. Essays {Topol}. {Relat}.
  {Top}., {M{\'e}m}. d{\'e}di{\'e}s {\`a} {Georges} de {Rham}, 188-211 (1970).,
  1970, \urlprefix\url{https://eudml.org/doc/109776}

\bibitem{Tom68}
P.~Tomter, Anosov flows on infra-homogeneous spaces. In \emph{Global {A}nalysis
  ({P}roc. {S}ympos. {P}ure {M}ath., {V}ol. {XIV}, {B}erkeley, {C}alif.,
  1968)}, pp. 299--327, Amer. Math. Soc., Providence, R.I., 1968

\bibitem{Tsa24}
C.~C. Tsang, Horizontal {G}oodman surgery and almost equivalence of
  pseudo-{A}nosov flows. 2024,
  \urlprefix\url{https://arxiv.org/abs/2401.01847}, \eprint{2401.01847}

\bibitem{Tuk88}
P.~Tukia, \href{https://doi.org/10.1515/crll.1988.391.1}{Homeomorphic
  conjugates of {F}uchsian groups}. \emph{J. Reine Angew. Math.} \textbf{391}
  (1988), 1--54

\bibitem{Ver74}
A.~Verjovsky, Codimension one {A}nosov flows. \emph{Bol. Soc. Mat. Mexicana
  (2)} \textbf{19} (1974), no.~2, 49--77

\bibitem{Wil70}
B.~Williams, The {$``{\rm DA}''$} maps of {S}male and structural stability. In
  \emph{Global {A}nalysis ({P}roc. {S}ympos. {P}ure {M}ath., {V}ols. {XIV},
  {XV}, {XVI}, {B}erkeley, {C}alif., 1968)}, pp. 329--334, Proc. Sympos. Pure
  Math. XIV-XVI, Amer. Math. Soc., Providence, RI, 1970

\end{thebibliography}
